%% file: Erlang_Weighted_Tree.tex
\documentclass[11pt]{article}
\usepackage{fullpage}

\usepackage{amsmath,amsthm,amsfonts,amssymb,bbm}
\usepackage{graphicx,enumerate}
\usepackage[title]{appendix}
\usepackage{authblk}
\usepackage{xspace}
\usepackage{mathalfa}
\usepackage{ulem}

\usepackage[
CJKbookmarks=true,
bookmarksnumbered=true,
bookmarksopen=true,
colorlinks=true,
citecolor=red,
linkcolor=blue,
anchorcolor=red,
urlcolor=blue
]{hyperref}
\makeatletter
\DeclareFontEncoding{LS1}{}{}
\DeclareFontSubstitution{LS1}{stix}{m}{n}
\DeclareMathAlphabet{\mathscr}{LS1}{stixscr}{m}{n}
\makeatother
\usepackage[capitalise]{cleveref}
\usepackage{autonum}

\Crefname{equation}{}{}

\usepackage{tikz}
\usepackage{tikz-qtree}
\usepackage{mathtools}
\usetikzlibrary{calc}
\usetikzlibrary{decorations.pathreplacing,decorations.markings}
\tikzstyle{dot}=[circle,fill,black,inner sep=1pt]
\newcommand{\myquad}[1][1]{\hspace*{#1em}\ignorespaces}

\usepackage{mathrsfs}
\usepackage{enumitem}
\usepackage{physics}
\usepackage{float}

\usepackage{subcaption}

\DeclareMathAlphabet{\mathpzc}{OT1}{pzc}{m}{it}

\theoremstyle{definition}
\newtheorem{definition}{Definition}
\newtheorem{theorem}{Theorem}[section]

\newtheorem{proposition}[theorem]{Proposition}
\newtheorem{lemma}[theorem]{Lemma}
\newtheorem{corollary}[theorem]{Corollary}
\newtheorem{remark}{Remark}
\renewcommand{\root}{\textrm{\o}}

\newcommand{\expect}{\mathbb{E}}
\newcommand{\prob}{\mathbb{P}}

\newcommand{\bs}{\boldsymbol}
\newcommand{\widesim}[2][1.5]{\mathrel{\overset{#2}{\scalebox{#1}[1]{$\sim$}}}}

\newcommand{\G}{G_\circ}
\newcommand{\N}{N_\circ}
\newcommand{\rtree}{\mathbb{T}_\circ}
\newcommand{\rtreealt}{\mathrm{T}_\circ}

\newcommand{\Topt}{{\mathcal{T}}}

\newcommand{\pd}{{\mathscr{P}}}
\newcommand{\Apset}{{\mathcal{A}}}
\newcommand{\Bpset}{{\mathcal{B}}}

\newcommand{\thresh}{{\mathpzc{T}}} 
\newcommand{\potdeg}{{\mathpzc{P}}} 

\newcommand{\eexp}{{\rm e}}

\newcommand{\At}{{\mathpzc{A}}}
\newcommand{\Bt}{{\mathpzc{B}}}
\newcommand{\Ct}{{\mathpzc{C}}}
\newcommand{\Dt}{{\mathpzc{D}}}
\newcommand{\Rt}{{\mathpzc{R}}}
\newcommand{\It}{{\mathpzc{I}}}
\newcommand{\Ut}{{\mathpzc{U}}}
\newcommand{\St}{{\mathpzc{S}}}
\newcommand{\permt}{\mathpzc{m}}

\newcommand{\erlangdist}{\mathrm{Erlang}}
\newcommand{\bindist}{\mathrm{Bin}}
\newcommand{\expdist}{\mathrm{Exp}}
\newcommand{\geomdist}{\mathrm{Geom}}
\newcommand{\poissdist}{\mathrm{Poiss}}
\newcommand{\ertreedist}{\mathrm{Er}}

\newcommand{\onefunc}{\mathbbm{1}}

\newcommand\numberthis{\addtocounter{equation}{1}\tag{\theequation}}

\newcommand*\xbar[1]{\hbox{\vbox{
			\hrule height 0.5pt 
			\kern0.5ex
			\hbox{\kern-0.1em
				\ensuremath{#1}%
				\kern-0.1em
			}}}
}

\DeclareFontFamily{U}{BOONDOX-calo}{\skewchar\font=45 }
\DeclareFontShape{U}{BOONDOX-calo}{m}{n}{
	<-> s*[1.05] BOONDOX-r-calo}{}
\DeclareFontShape{U}{BOONDOX-calo}{b}{n}{
	<-> s*[1.05] BOONDOX-b-calo}{}
\DeclareMathAlphabet{\mathcalboondox}{U}{BOONDOX-calo}{m}{n}
\SetMathAlphabet{\mathcalboondox}{bold}{U}{BOONDOX-calo}{b}{n}
\DeclareMathAlphabet{\mathbcalboondox}{U}{BOONDOX-calo}{b}{n}
\makeatletter
\def\keywords{\xdef\@thefnmark{}\@footnotetext}
\makeatother
\begin{document}
\title{The Erlang Weighted Tree, A New Branching Process}
\author[1]{M. Moharrami}
\author[2]{V. Subramanian}
\author[2]{M. Liu}
\author[3]{R. Sundaresan}
\affil[1]{Coordinated Science Lab, University of Illinois at Urbana-Champaign}
\affil[2]{Electrical and Computer Engineering, University of Michigan}
\affil[3]{Indian Institute of Science and Indian Statistical Institute}
\date{}

\maketitle

\begin{abstract}
	In this paper, we study a new discrete tree and the resulting branching process, which we call the \textbf{E}rlang \textbf{W}eighted \textbf{T}ree(\textbf{EWT}). The EWT appears as the local weak limit of a random graph model proposed in~\cite{La2015}. In contrast to the local weak limit of well-known random graph models, the EWT has an interdependent structure. In particular, its vertices encode a multi-type branching process with uncountably many types.

	We derive the main properties of the EWT, such as the probability of extinction, growth rate, etc. We show that the probability of extinction is the smallest fixed point of an operator. We then take a point process perspective and analyze the growth rate operator. We derive the Krein--Rutman eigenvalue $\beta_0$ and the corresponding eigenfunctions of the growth operator, and show that the probability of extinction equals one if and only if $\beta_0 \leq 1$.
\end{abstract}

%
%

\section*{Basic Notation}
\input{Sections/Notation}

\section{Introduction}\label{sec:intro}
\input{Sections/Introduction}

\section{Finite Graph Model}\label{sec:finitegraph}
\input{Sections/FiniteGraphModel}

\section{Numerical Example}\label{sec:numsim}
\input{Sections/NumSim}

\section{Open Problems}\label{sec:openprob}
\input{Sections/openprob}

\section{Properties of Erlang Weighted Tree}\label{sec:properties}

\subsection{Degree Distribution}\label{sec:DegreeDist}
\input{Sections/PropEWT/DegreeDist}
\subsection{Probability of Extinction}\label{sec:probext}
\input{Sections/PropEWT/ExtincProb}
\subsection{Expected Number of Vertices at Depth $l$}
\input{Sections/PropEWT/ExpNumDepthl}
\subsection{Krein--Rutman Eigenvalue and the Corresponding Eigenfunctions}\label{sec:PFeigenvalue}
\input{Sections/PropEWT/GrowthRate}
\subsection{Analysis of the Second Moments and Asymptotic Results for $\beta_0 > 1$}\label{sec:SecMom}
\input{Sections/PropEWT/SecMom}
\subsection{Transience of $Z_l$}\label{sec:PhTr}
\input{Sections/PropEWT/PhaseTrans}
\subsection{Probability of Extinction Revisited}\label{sec:PropExtRev}
\input{Sections/PropEWT/ProbExtRev}
\subsection{Asymptotic Degree Distribution for $\beta_0 > 1$}\label{sec:DegreeDistRev}
\input{Sections/PropEWT/DegreeDistRev}

\section*{Acknowledgements}
We are very grateful to Charles Bordenave, Remco van Der Hofstad, Richard La, and Marc Lelarge for helpful conversations. Mehrdad Moharrami acknowledges support from AST-1516075, CNS-1616575, CNS-1739517 and Rackham Graduate Predoctoral Fellowship. The majority of the work was done while the first author was at the University of Michigan.
Vijay Subramanian acknowledges support from NSF via grants AST-1343381, AST-1516075, IIS-1538827, ECCS-1608361, EECS 2038416, CNS 1955777 and CCF 2008130.
Rajesh Sundaresan acknowledges support from the Cisco-IISc Centre for Networked Intelligence, Indian Institute of Science.
Mingyan Liu acknowledges support from NSF via grants CNS-1616575, CNS-1739517, CNS-1939006, and ARO W911NF1810208.
\bibliographystyle{plain}
\bibliography{references}

\appendix
\section*{Appendix}
\section{Background Material}\label{sec:background}
\subsection{Random Graphs and Local Weak Convergence}\label{sec:back_randomgrpah}
\input{Sections/Background_RandomGraph}
\subsection{Point Process Perspective of a Branching Process}\label{sec:back_Branch}
\input{Sections/Background_PointProc}
\subsection{Spectral Theorem for Compact Self-adjoint Bounded Linear Operators}\label{sec:back_opt}
\input{Sections/Background_Operator}

\section{EWT is a Weak Limit, Proof of Theorem~\ref{thm:weaklimit}} \label{appendix:pfconv}
\input{Sections/ProofWeakConv}
\section{Unimodularity of EWT, Proof of Corollary~\ref{cor:unimod}} \label{appendix:pfunimod}
\input{Sections/Unimod}

\section{Other Technical Proofs}
\subsection{Proof of Lemma~\ref{lem:propt}}\label{appendix:lemextinc}
\input{Sections/ProofLemmaExtinc}
\subsection{Proof of Theorem~\ref{thm:ConDegree}}\label{appendix:degreedist}
\input{Sections/ProofDegreeDist}
\subsection{Proof of Theorem~\ref{thm:expectedz_l}}\label{appendix:EZ}
\input{Sections/ProofEZ}
\subsection{Proof of Proposition~\ref{prop:geodist} part~\ref{prop:geodist_1}, part~\ref{prop:geodist_2}, and part~\ref{prop:geodist_4}}\label{appendix:geodist}
\input{Sections/ProofGeoDist}
\end{document}

%% file: Sections/Notation.tex
Bold symbols are used for sequences, while random variables are denoted by capital letters and their realization by small letters. We use $\mathbb{C}$ to denote the set of complex numbers, and $\mathbb{R}_+$ to denote the set of non-negative real numbers. Similarly, $\mathbb{Z}_+$ denotes the set of non-negative integers. The set of positive integers is denoted by $\mathbb{N}$. The set of all finite sequences of $\mathbb{N}$ is denoted by $\mathbb{N}^f = \cup_{i=0}^{\infty} \mathbb{N}^i$ with the convention $\mathbb{N}^0 = \{\root\}$. The set of positive integers less than or equal to $n$ is denoted by $[n]$, i.e., $[n]=\{1,2,\dots,n\}$. 
Let $L(\mathbb{R}_+;[0,1])$ be the set of Lebesgue measurable functions from $\mathbb{R}_+$ to $[0,1]$. Let $C^1(\mathbb{R}_+;[0,1])$ be the set of continuously differentiable functions from $\mathbb{R}_+$ to $[0,1]$. The Erlang distribution with parameters $k\in \mathbb{N}$ and $\lambda > 0$ is denoted by $\erlangdist(k,\lambda)$, where its commutative distribution function is given by $F(x;k,\lambda) = 1 - \sum_{n=0}^{k-1} \frac{1}{n!}\eexp^{-\lambda x} (\lambda x)^n$. The probability mass function of the Binomial distribution with parameters $n\in\mathbb{N}$ and $p\in[0,1]$ is denoted by $\bindist(\cdot\,;n,p)$, the Poisson distribution with parameter $\lambda$ is denoted by $\poissdist(\lambda)$, the geometric distribution with parameter $p$ and support on $\mathbb{N}$ is denoted by $\geomdist(p)$, and the exponential distribution with mean $n$ is denoted by $\expdist(1/n)$. The abbreviation {\em i.i.d.} stands for independent and identically distributed. A real-valued random variable $X$ is said to have a moment generating function at $\theta\in\mathbb{R}$, if $\expect[\eexp^{\theta X}]<\infty$. For a set $S$, $\mathcal{P}(S)$ is the set of all Borel probability measures defined on $S$.

%% file: Sections/Introduction.tex
This paper studies a random tree object called the Erlang Weighted Tree (EWT) which arises as the local weak limit of a sequence of finite random graphs, each of which is obtained using a certain bilateral agreement procedure~\cite{La2015}. We begin with a description of the EWT and then follow it up with a description of the random graph family and the bilateral agreement procedure.

The construction of the EWT begins with the construction of a so-called ``backbone tree'' whose edges will be pruned to obtain the EWT. 
The vertices of the backbone tree are endowed with labels from the set $\mathbb{N}^f$.
Each label $\boldsymbol{i}\in\mathbb{N}^f$ is associated with three types of random variables: $1)$ $n_{\boldsymbol{i}}$ which is the potential number of descendants of the vertex $\boldsymbol{i}$, $2)$ $v_{\boldsymbol{i}}$ which is the value associated with the vertex $\boldsymbol{i}$, and, $3)$ $\{\zeta_{({\boldsymbol{i}},j)}\}_{j=1}^{n_{\boldsymbol{i}}}$ which represents the cost of the potential edges $\left\{{\boldsymbol{i}},({\boldsymbol{i}},j)\right\}$ for $j\in\{1,2,\dots,n_{\boldsymbol{i}}\}$.
The probability distribution of $n_{\root}$ is given by $P\in \mathcal{P}(\mathbb{N})$ which is assumed to have a finite mean and $P(1)<1$. The probability distribution of $n_{\boldsymbol{i}}$ for ${\boldsymbol{i}}\in \mathbb{N}^f\setminus\mathbb{N}^0$ is given by the shifted distribution $\widehat{P}\in \mathcal{P}(\mathbb{Z}_+)$, that is, for all $k \in\mathbb{Z}_+$ the value of  $\widehat{P}(k)$ is set to be $P(k+1)$.
Conditioned on $n_{\boldsymbol{i}}$, $v_{\boldsymbol{i}}$ is distributed as $\erlangdist(n_{\boldsymbol{i}}+1,\lambda)$ for a positive and fixed real $\lambda$.
Conditioned on $n_{\boldsymbol{i}}$ and $v_{\boldsymbol{i}}$, $\{\zeta_{({\boldsymbol{i}},j)}\}_{j=1}^{n_{\boldsymbol{i}}}$ are $n_{\boldsymbol{i}}$ independent and uniformly distributed random variables over the interval $[0,v_{\boldsymbol{i}}]$. When $n_{\boldsymbol{i}} =0$, there are no potential edges emanating from vertex ${\boldsymbol{i}}$. The backbone tree is the connected component of $\root$ with the potential edges as its edge set.

The edges of the backbone tree are pruned as follows to obtain the EWT: preserve the edge between the vertices $\boldsymbol{i}$ and $({\boldsymbol{i}},j)$ if and only if $\zeta_{({\boldsymbol{i}},j)} < v_{({\boldsymbol{i}},j)}$, and consider the connected component of $\root$. This object is a rooted tree $\rtree=(V,E,\root,w_v,w_e)$ rooted at $\root$ with vertex marks $w_v$ and edge marks $w_e$ defined as follows,
\begin{align}
w_v&:V\to\mathbb{N}\times \mathbb{R}_+,\qquad w_v(\boldsymbol{i}) =
\begin{cases}
(n_{\root},v_{\root})	 & \text{if }\boldsymbol{i} = \root\\
(n_{\boldsymbol{i}}+1,v_{\boldsymbol{i}}) & \text{otherwise}
\end{cases},
\\
w_e&:E\to\mathbb{R_+}, \qquad w_e(\{\boldsymbol{i},(\boldsymbol{i},j)\}) = \zeta_{(\boldsymbol{i},j)}.
\end{align}
Let $[\rtree]$ denote the equivalence class of $\rtree$ up to isomorphisms (over vertex relabelings that preserve the root). Denote by $\ertreedist(P,\lambda)$ the probability distribution of $[\rtree]$ in $G_*$, which denotes the set of rooted marked graphs up to isomorphisms\footnote{See Appendix~\ref{sec:back_randomgrpah} for a formal definition of $G_*$ and related background material.}. We call $\ertreedist(P,\lambda)$ the EWT with \textit{potential degree distribution} $P$. The parameter $\lambda$ in the definition of $\ertreedist(P,\lambda)$ appears only as a scaling factor. Henceforth $\lambda$ is set to $1$, and for ease of notation, $\ertreedist(P)\equiv \ertreedist(P,1)$ and $\erlangdist(k,\lambda)\equiv \erlangdist(k)$.

A non-root vertex $\boldsymbol{i}$ with the mark $(n_{\boldsymbol{i}}+1,v_{\boldsymbol{i}})$ is referred to as a vertex of \textit{type} $(n_{\boldsymbol{i}},v_{\boldsymbol{i}})$ where $n_{\boldsymbol{i}}$ denotes the potential number of descendants of vertex $\boldsymbol{i}$. The generation sizes of EWT then forms a multi-type branching process~\cite[Chapter III]{Harris1963} with mark space $\mathbb{N}\times \mathbb{R}_+$. In particular, a vertex of type $(m,x)$ has a $\bindist(m, p)$ number of descendants, where $p = \prob (x \cdot U \leq Z)$ for
\begin{itemize}
	\item $K-1 \sim \widehat{P}$ and conditional on $K$, we have $Z \sim \erlangdist(K)$,
	\item $U$ is a uniformly distributed random variable over $[0,1]$, independent of $K$ and $Z$.
\end{itemize}
Each of these descendants has a type distributed as $(K-1, Z)$ conditional on $\{ x \cdot U \leq Z\}$, for $K$, $Z$, $U$ as described above, independent of its siblings.

We will show that the EWT appears as the local weak limit of a random graph sequence constructed in La and Kabkab~\cite{La2015}. The graph construction starts with a complete graph $K_n = ([n],E_n)$, a sequence of positive integers $\bs{d}_n=(d_1(n), d_2(n), \cdots, d_n(n))$ and a random cost function $C_n$ that assigns non-negative real values to each edge of $K_n$. The value of $d_i(n)$ indicates the number of neighbors with which vertex $i$ wants a connection. The numbers $d_i(n)$ are chosen as follows: choose an \emph{i.i.d.} sequence of natural numbers $(\widehat{d}_1(n), \widehat{d}_2(n), \cdots, \widehat{d}_n(n))$ with underlying distribution $P$ and then set $d_i(n)=\min(\widehat{d}_i(n),n-1)$ for $i=1,\dotsc,n$. The value assigned to each edge by $C_n$ is an independent exponentially distributed random variable with parameter $1/n$ that represents the cost of the edge; the random cost function $C_n$ results in a random preference order of the edges of $K_n$. Each vertex $i$ then selects the $d_i(n)$ lowest cost incident edges and declares them to be its preferred edges. The random graph $G_n = ([n],\widetilde{E}_n)$ is then constructed by keeping only those edges of $E_n$ that are preferred by both end-vertices. Notice that an edge is preserved only if there is a bilateral agreement between the end-vertices.

\paragraph{Main Results:} In this work, we derive the following properties of the EWT:
\begin{enumerate}[label=(\roman*)]
	\item Theorem~\ref{thm:weaklimit}: EWT is the local weak limit of the random graph family $[G_n]$ indexed by $n$.
	\item Theorem~\ref{thm:ConDegree}: The degree distribution of the root is given by
	\begin{align}
	&\prob(D_{\root}=d) = \sum_{m=1}^\infty P(m) \int_{0}^{\infty} \frac{ \eexp^{- x}x^m}{m!} Bi\left(d;m,\int_{0}^{x} \frac{1}{x}\sum_{k=1}^{\infty}P(k)\bar{F}_{k}(y)\,dy\right)\, dx\allowdisplaybreaks\\
	&\expect[D_{\root}] = \sum_{m=1}^\infty\sum_{k=1}^{\infty} P(m)P(k) \int_{0}^{\infty} \bar{F}_{k}(y) \bar{F}_{m}(y)\, dy,
	\end{align}
	where $\bar{F}_k{(\cdot)}$ denote the complementary cumulative distribution function of $\erlangdist(k)$. 
	\item Theorem~\ref{thm:probext} and Lemma~\ref{lem:uniq}: The probability of extinction is given by
	\begin{align}
	\prob(\{\text{extinction}\})= \sum_{m=1}^{\infty} P(m)\int_{x=0}^{\infty} \frac{\eexp^{-x}x^m}{m!} \left(q(x)\right)^m \,dx,
	\end{align}
	where $q(\cdot)\in C^1(\mathbb{R}_+;[0,1])$ is the smallest fixed point (point-wise smaller than all the other fixed points) of the operator $T:L(\mathbb{R}_+;[0,1]) \to C^1(\mathbb{R}_+;[0,1])$, which is defined as
	\begin{align}
	T(f)(x) \coloneqq
	\begin{dcases*}
	\begin{aligned}
	&\frac{1}{x}\sum_{k=1}^{\infty}P(k)\int_{y=0}^{x} \Bigg(\int_{z=0}^{y}\frac{\eexp^{-z}z^{k-1}}{(k-1)!}\,dz + \\ &\myquad[10]\int_{z=y}^{\infty}\frac{\eexp^{-z}z^{k-1}}{(k-1)!}\left(f(z)\right)^{k-1}\,dz\Bigg)\,dy,
	\end{aligned} & $x >0$, \\
	\sum_{k=1}^{\infty}P(k) \int_{z=0}^{\infty}\frac{\eexp^{-z}z^{k-1}}{(k-1)!}\left(f(z)\right)^{k-1}\,dz, & $x=0$.
	\end{dcases*}
	\end{align}
	This fixed point is also the pointwise limit of $T^l(\boldsymbol{0})(\cdot)$ as $l$ goes to infinity where for all $x\in \mathbb{R}_+$ the value of  $\boldsymbol{0}(x)$ is set to be $0$ (the all-zero function).
    Moreover, for any function $f(\cdot)\in L(\mathbb{R}_+;[0,1])$ such that the Lebesgue measure of the set $\{x\in\mathbb{R}_+: f(x) < 1\}$ is positive, $T^l(f)(\cdot) \to q(\cdot)$ pointwise. If the probability of extinction equals 1, then the function $\mathbf{1}(x) \equiv 1$ for all $x\in\mathbb{R}_+$ (the all-one function) is the unique fixed point of $T$ (upto sets of measure $0$). If the probability of extinction is smaller than 1, then assuming that the moment generating function of $n_{\root}$ exists at some $\theta> 0$, the operator $T$ has exactly two fixed points: $q(\cdot)$ and $\boldsymbol{1}(\cdot)$.
	\item Theorem~\ref{thm:pfeigenval} and Corollary~\ref{cor:EZ_l}: Assume that the moment generating function of $n_{\root}$ exists at some $\theta> 0$.
	Let $g_2(x) = \eexp^{-x} \sum_{k=2}^{\infty} P(k)\frac{x^{k-2}}{(k-2)!}$ and then define functions $\{G_i(\cdot)\}_{i\in\mathbb{Z}_+}$ over $\mathbb{R}_+$ recursively via
	\begin{align}
	G_0(x) &\equiv 1,\allowdisplaybreaks\\
	G_i(x) &= \int_{x}^{\infty}\int_{z=y}^{\infty}g_2(z) G_{i-1}(z) \,dz \,dy, \qquad\forall i>0.
	\end{align}
	Further, for all $x\in\mathbb{R}_+$ and $\beta > 0$ define
	\begin{align}
	L(\beta,x) &= \sum_{i=0}^{\infty} G_i(x) \left(\frac{-1}{\beta} \right)^i.
	\end{align}
	Let $Z_l$ denote the number of vertices in generation $l$ of the EWT. Then we have
	\begin{align}
	\frac{\expect[Z_l]}{{\beta_0}^l} \xrightarrow{l\to\infty} \left(\int_{0}^{\infty} \sum_{k=1}^\infty P(k)\frac{\eexp^{-z}z^{k-1}}{(k-1)!} f_0(z)\,dz\right)^2 <\infty,
	\end{align}
	where $\beta_0$ is the smallest zero of the function $L(\beta,0)$, $f_0(x) = L(\beta_0,x)/\sqrt{C_N}$, and $C_N = (\int_{0}^{\infty} g_2(y) \left(L(\beta_0,y)\right)^2 \,dy)^{-1}$ is the normalization factor so that \label{item:marinres_partiv}
    \begin{align}
        \int_{0}^{\infty} g_2(y) {f_0(y)}^2 dy = 1.
    \end{align}
    \item Theorem~\ref{thm:pfeigenval}: Define the growth operator $M_1: (\mathbb{Z}_+\times \mathbb{R}_+\mapsto \mathbb{R}_+) \mapsto (\mathbb{Z}_+\times \mathbb{R}_+\mapsto \mathbb{R}_+)$ via $M_1((m,x);\mathcal{A}) = \expect[Z_1(\mathcal{A})\,\vert\,(n_\root,v_\root) = (m,x)]$ which is the expected number of children of type belonging to a Borel subset $\mathcal{A}$ of $\mathbb{Z}_+\times\mathbb{R}_+$ in generation~$1$, conditioned on the root vertex being of type $(m,x)\in \mathbb{Z}_+\times\mathbb{R}_+$.
	Then the Krein--Rutman eigenvalue of the growth operator $M_1$ is $\beta_0$, and the corresponding left and right eigenfunctions (see Definition~\ref{def:leftrighteigen}) are given as follows:
	\begin{align}
		&\text{Right eigenfunction: }  \forall(m,x)\in \mathbb{Z}_+ \times \mathbb{R}_+,\qquad \mu(m,x) = \frac{m}{x} f_0(x),\allowdisplaybreaks\\
		&\text{Left eigenfunction: } \forall(k,z)\in \mathbb{N}\times\mathbb{R}_+,\qquad \nu(k-1,z) = P(k)\frac{\eexp^{-z}z^{k-1}}{(k-1)!} f_0(z).
	\end{align}
	When it exists, the Krein--Rutman eigenvalue of an operator is its principal eigenvalue, i.e., its largest eigenvalue in magnitude. It is also simple. See Appendix~\ref{sec:back_opt} for more details, especially the conditions under which such an eigenvalue exists.
	\item Theorem~\ref{thm:L2conv} and Theorem~\ref{thm:growthrate_aux}: Suppose that the assumption of part~\ref{item:marinres_partiv} hold and assume $\beta_0 > 1$. For every Borel subset $\mathcal{A}$ of $\mathbb{Z}_+\times \mathbb{R}_+$, define $Z_l(\mathcal{A})$ to be the number of vertices in generation $l$ with types in $\mathcal{A}$.
	Then, conditioned on the type of the root vertex, for every Borel set $\mathcal{A}$, there is a random variable $W(\mathcal{A})$ such that $Z_l(\mathcal{A})/\beta_0^l$ converges to $W(\mathcal{A})$ almost surely and in $L^2$.
    In particular, $Z_l(\Omega) \sim \beta_0^l W(\Omega)$ where $\Omega = \mathbb{Z}_+\times\mathbb{R}_+$, that is, $\beta_0$ is the growth rate of $Z_l(\Omega)$. Further, we have that $W(\mathcal{A}) = c(\mathcal{A}) \, W(\Omega)$ almost surely, where $c(\cdot)$ is a non-random probability measure depending only on the left eigenfunction, i.e., for any Borel subset $\mathcal{A}$,
    \begin{align}
        c(\mathcal{A}) \propto \underset{(k-1,z)\in\Apset}{\sum \int } \nu(k-1,z) \,dz,\qquad c(\Omega) =1.
    \end{align}
     In other words, the proportion of various types asymptotically collapses to a non-random limit. Notice that the Borel $\sigma$-algebra on $\Omega$ is countably generated, and hence, our result implies that, almost surely, the random measure $Z_l(\cdot)/\beta_0^l$ converges set-wise to $c(\cdot)W(\Omega)$.

	\item Theorem~\ref{thm:asymdegdist}: Let $D_l$ denote the number of descendants of a randomly selected vertex in generation $l$ given the number of vertices in generation $l$ is positive. Suppose that the assumption of part~\ref{item:marinres_partiv} hold and assume $\beta_0 > 1$. The asymptotic distribution of $D_l$ is given by:
	\begin{align}
		\lim_{l \to\infty} \prob(D_l = d\,\vert\,Z_l > 0) = \dfrac{\sum_{k=1}^\infty \int_{0}^\infty  \nu(k-1,z) \times Bi\left(d;k-1,\int_{0}^{z} \frac{1}{z}\sum_{k'=1}^{\infty}P(k')\bar{F}_{k'}(y)\,dy\right)\,dz}{\sum_{k=1}^\infty \int_{0}^\infty\nu(k-1,z) \,dz},
	\end{align}
	where $\nu$ is the left eigenfunction.
	\item Lemma~\ref{lem:transZ} and Corollary~\ref{cor:beta0_probextinc}: Suppose that the assumption of part~\ref{item:marinres_partiv} hold. If $\beta_0 > 1$, the probability of extinction is less than 1; otherwise, it equals $1$. Moreover, if $\beta_0 > 1$, then the number of vertices in generation $n$ goes to either $0$ or $\infty$ as $n\to\infty$.
\end{enumerate}

\paragraph{Literature Review:} Cooper and Frieze in~\cite{Cooper1995} studied the $k$-th nearest neighbor graphs model in which a connection survives as long as at least one individual involved in the connection is interested in it. The La-Kabkab random graph model~\cite{La2015} is more intricate because it requires bilateral agreement for a connection to survive. This makes the analysis challenging.
One way to overcome this challenge is to study the local weak limit of the sequence of graphs and then understand its implications on the pre-limit.
Such an approach has been taken, for example, for the configuration model~\cite{Bordenave2016}: the unimodular \textbf{G}alton-\textbf{W}atson \textbf{T}ree~(\textbf{GWT})~\cite{Bienayme1845,Galton1875} is the local weak limit of sparse random graphs generated by the configuration model. It is well-known~\cite{Molly1998} that the limiting fraction of vertices in the largest connected component of the corresponding random graph is one minus the probability of extinction of the GWT. This connection to the GWT has also been established for the random graph models of Erd\H{o}s and R\'enyi~\cite{Erdos1959} and Gilbert~\cite{Gilbert1959}.


In our results, we establish that the EWT is the local weak limit of the La-Kabkab random graph model, and then we focus on understanding the EWT as a branching process\footnote{Before the formalization of local weak limits, locally tree-like random graphs were studied using branching processes~\cite{Bollobas2008,Durrett2006}.}. Numerically we illustrate the connection between the fraction of vertices in the largest component of the finite graph model and the probability of extinction of the EWT; a rigorous proof of this connection is still open. 

EWT is a multitype branching process in which the mark space is $\mathbb{N}\times \mathbb{R}_+$. Such branching processes were analyzed by Harris in~\cite[Chapter III]{Harris1963} using a point process perspective: a general multitype branching process is considered as a point distribution on $\mathcal{X}\subset\mathbb{R}^d$ that evolves in a Markovian fashion. The analysis involves the study of the growth operator $M_l$ defined by $M_l(x,\mathcal{A}) = \expect[Z_l(\mathcal{A})\,\vert\,x_0 = x]$, the expected number of objects of type $\mathcal{A}\subseteq \mathcal{X}$ in generation $l$ conditioned on the root vertex being of type $x$. Harris assumes the following condition (see~\cite[Condition 10.1, Chapter III]{Harris1963}):
\begin{gather}
	\begin{aligned}
		\textnormal{\textbf{C}: }
		& \text{There exists $n_0 \in\mathbb{N}$ such that $M_{n_0}$ has a density $m_{n_0}$ (with respect to Lebesgue measure)}\\
		& \text{which is uniformly bounded above and uniformly positive: for all $x,y\in\mathcal{X}$ we have }
	\end{aligned}\\
	0<a\leq m_{n_0}(x,y) \leq b .
\end{gather}
Under condition \textbf{C}, Harris proves the following: the existence and uniqueness of the Krein--Rutman eigenvalue, the asymptotic formula for $m_l$, the analysis of correlation measure $\expect[Z_l(\mathcal{A})Z_l(\mathcal{B})\,\vert\,x_0=x]$ (where $\mathcal{A}$, $\mathcal{B}$ are Borel subsets of $\mathcal{X}$), and the connection between the growth rate of $\mathbb{E}[Z_l(\mathcal{X})]$ and the probability of extinction of the branching process.
The Krein--Rutman eigenvalue and the corresponding left and right eigenfunctions are not explicitly identified.

EWT does not satisfy condition \textnormal{\textbf{C}}. Given the root is of type $(m,x)$, the expected number of objects of type $(k-1,z)$, $k\in \mathbb{N}$ and $z\in A\subset\mathbb{R}_+$, in the first generation is given by
\begin{align}
	M_1(m,x;k-1,A) &= \int_{z \in A}\sum_{k\in \mathbb{N}} \frac{m}{x} \min(x,z) P(k) \frac{\eexp^{-z}z^{k-1}}{(k-1)!}\,dz\\
	&= \int_{z \in A}\sum_{k\in \mathbb{N}} m_1(m,x;k-1,z)\,dz,
\end{align}
where $m_1(m,x;k-1,z)=\tfrac{m}{x} \min(x,z) P(k) \tfrac{\eexp^{-z}z^{k-1}}{(k-1)!}$ is the density of $M_1$; notice that $P(k) \tfrac{\eexp^{-z}z^{k-1}}{(k-1)!}\,dz$ is the probability that a vertex $j$ in the first generation of the backbone tree is of type $(k-1,z)$, and ${\min(x,z)}/{x}$ is the probability that the edge $(\root,j)$ in the backbone tree survives given $v_{\root} = x$ and $v_{j} = z$. In Section~\ref{sec:PFeigenvalue}, using the Chapman-Kolmogorov equations, we will argue that for all $l>0$, $m_l(m,x;k-1,\cdot)$, the density of $M_l(m,x;k-1,\cdot)$ (the expected number of objects of type $(k-1,\cdot)$ in the $l^{\mathrm{th}}$ generation given the root is of type $(m,x)$), is continuous and $m_l(m,x;k-1,0) = 0$, and $m_l(m,x;k-1,z)\to 0$ as either $k\to\infty$ or $z\to\infty$. Hence, $m_l$ is not uniformly positive, i.e., bounded away from $0$, for any value of $l$, so condition \textnormal{\textbf{C}} fails, and the results of~\cite[Chapter III]{Harris1963} do not apply. To overcome the difficulty that arises from condition \textnormal{\textbf{C}} failing, we take a different approach by characterizing the Krein--Rutman eigenvalue and the corresponding left and right eigenfunctions explicitly.

The importance of the Krein--Rutman eigenvalue arises from the Chapman-Kolmogorov equations. Notice that for all $l \in\mathbb{N}$, we have
\begin{align}
	m_l(m,x;k-1,z) &= \int_{z^\prime = 0}^{\infty} \sum_{k^\prime=1}^{\infty} m_{l-1}(m,x;k^\prime-1,z^\prime)m_1(k^\prime-1,z^\prime;k-1,z)\,dz^\prime \allowdisplaybreaks\\
					&= \int_{z^\prime = 0}^{\infty} \sum_{k^\prime=1}^{\infty} m_{1}(m,x;k^\prime-1,z^\prime)m_{l-1}(k^\prime-1,z^\prime;k-1,z)\,dz^\prime.
\end{align}
Hence, if there is an asymptotic expression for $m_l$, it should be non-negative and involve both a left and a right eigenfunction of $M_1$ (see Definition~\ref{def:leftrighteigen}). Invoking the Krein--Rutman Theorem~\ref{thm:optth3} (if it applies), $M_1$ has a unique eigenvalue with non-negative left and right eigenfunctions. Moreover, this eigenvalue is simple, real, and larger in magnitude than all the other eigenvalues of $M_1$. Hence, a natural guess for $m_l$ would be
\begin{align}
	m_l \propto (\text{Krein--Rutman eigenvalue})^l \times \text{Left eigenfunction}\times \text{Right eigenfunction},
\end{align}
and in particular, the Krein--Rutman eigenvalue captures the growth/extinction rate of the population. Furthermore, the left eigenfunction captures the asymptotic distribution of the type of vertices, whereas the right eigenfunction captures the influence of the root vertex on the asymptotic size of the generations.

To formalize the above ideas, we first introduce a function $L(\beta,x)$ (see Theorem~\ref{thm:propL}) and show that $\beta_0$, the Krein--Rutman eigenvalue of the $M_1$, is the smallest positive zero of $L(\beta,0)$. This is done via a probabilistic approach using the ergodicity and stability results of Baxendale~\cite{Baxendale2005} for continuous-state Markov chains. We also derive the corresponding Krein--Rutman eigenfunctions in terms of the function $L(\beta_0,x)$. 
Then, we extend some of the results of~\cite[Chapter III]{Harris1963} to the EWT by exploiting its structural properties. As we proceed with the paper, we will highlight the differences between our approach and those presented in~\cite[Chapter III]{Harris1963}. 
Our analysis might apply to a broader class of multitype branching processes, although we do not pursue this in our paper.


\paragraph{Mathematical Background:} The necessary background for the rest of the paper is presented in Appendix~\ref{sec:background}, and an informed reader can skip this material. 
The background on ``random graphs and local weak convergence'' is mostly based on lecture notes by Bordenave~\cite{Bordenave2016} and the work of Aldous and Lyons~\cite{Aldous2007}, and is used in Section~\ref{sec:finitegraph}. 
The background on the ``point process perspective of a branching process'' is based on chapter 3 of Harris's book~\cite{Harris1963}.
We use this background in Sections~\ref{sec:PFeigenvalue}-\ref{sec:PropExtRev}. 
The background on the ``spectral theorem for compact self-adjoint bounded linear operators'' is based on a classic textbook in functional analysis by Lax~\cite{Lax2002} and the work of Toland~\cite{Toland1996}. These results will be used in Section~\ref{sec:PFeigenvalue}; however, we will provide probabilistic proof of the main theorems presented in this section to show explicitly that $\beta_0$ is the Krein-Rutman eigenvalue of the growth operator $M_1$.
We will follow the notation presented in Appendix~\ref{sec:background} in the rest of the paper.

\paragraph{Organization of paper:} The organization of the rest of the paper is as follows. In Section~\ref{sec:finitegraph}, we describe the finite graph model and discuss the local weak convergence of the finite graph model to the EWT. In Section~\ref{sec:numsim}, we present some numerical illustrations of our main results and discuss the connection between the proportion of vertices in the largest connected component of the finite graph model and the probability of extinction of the EWT. In Section~\ref{sec:openprob}, we present some open problems.
In Section~\ref{sec:properties}, we begin with basic properties of the EWT: the degree distribution of the root vertex and the expected number of vertices in generation $l$. Then, we derive the probability of extinction. Finally, we discuss the point process perspective and derive the growth rate of the branching process. Background material, proof of local weak convergence of the finite graph model to EWT, proof of unimodularity of EWT, and some algebraic proofs are presented in the Appendix for convenience.

%% file: Sections/FiniteGraphModel.tex
\subsection{The La-Kabkab Random Graph Model}
Let $K_n = ([n],E_n)$ denote a complete graph, that is, $E_n = \{\{i,j\}:i,j\in[n],i\neq j \}$. Consider some probability mass function $P(\cdot)$ defined over $\mathbb{N}$. Let $\boldsymbol{d}_n=(d_1(n),d_2(n),\dots,d_n(n))\in\mathbb{N}^n$ denote the sequence of potential degrees such that $d_i(n)\leq n-2$ and
assume that, as $n\to \infty$, its empirical distribution converges to $P(\cdot)$, i.e.,
\begin{align}
	&\forall k\in \mathbb{N},\qquad P_{\boldsymbol{d}_n}(k) = \frac{1}{n} \sum_{i=1}^{n}\delta_{d_i(n)}(k) \xrightarrow{n\to\infty}P(k).
\end{align}
Often, we write $d_i$ for $d_i(n)$ when the number of vertices $n$ is clear from the context. Let $C_{n}:E_n\to\mathbb{R}_+$ denote a random function that assigns {\em i.i.d.} random variables distributed as $\expdist(1/n)$ to the edges of $K_n$. The value of an edge corresponds to the cost of the edge.
\begin{remark}\label{rem:naivevalues}
	Without loss of generality, we assume the cost of all the edges in $K_n$ are different.
\end{remark}
For each vertex $i\in[n]$, let $\thresh_i$ and $\potdeg_i$ denote the threshold and the set of potential neighbors of the vertex $i$,
\begin{align}
\thresh_i &= \text{${d_i+1}^{st}$ smallest value in }\left\{C_{n}(\{i,j\}) : j\in[n] \setminus \{i\} \right\} \label{eq:thresh}\\
\potdeg_i &= \{j\in[n]\setminus \{i\}:C_{n}(\{i,j\}) < \thresh_i \}. \label{eq:potneigh}
\end{align}
Vertices of the graph have the following self-optimizing behavior: they are willing to form an edge only if the cost of the edge is less than each of their thresholds in \cref{eq:thresh} and an edge is formed only if both endpoint vertices are willing. Call the resulting random graph the random graph $G_n=([n],\widetilde{E}_n)$ with
\begin{align}
\widetilde{E}_n = \left\{ \{i,j\}\in E_n : i\in \potdeg_j \text{ and }j\in \potdeg_i \right\}.
\end{align}
 The bilateral agreement required for establishing an edge causes an interdependence structure; more precisely, inclusion of an edge into $\widetilde{E}_n$ depends on the preference of both ends, which is in turn dictated by the values of all the incident edges. This makes the analysis of the finite graph intricate; however, it is possible to study the model, using the framework of local weak convergence.
\subsection{Convergence to the EWT}
 Consider the random network $N_n = ([n],\widetilde{E}_n,\widetilde{W}_{v,n},\widetilde{W}_{e,n})$, where the mark functions are defined as follows:
\begin{align}
&\widetilde{W}_{v,n}:[n]\to\mathbb{N}\times\mathbb{R},&&\forall i\in[n],\,\widetilde{W}_{v,n}(i) = (d_i, \thresh_i),\\
&\widetilde{W}_{e,n}:\widetilde{E}_n\to\mathbb{R},&&\forall \{i,j\}\in\widetilde{E}_n,\,\widetilde{W}_{e,n}(\{i,j\}) = C_{n}(\{i,j\}).
\end{align}
Let $\mathcal{N}(n,\boldsymbol{d}_n)$ denote the law of the random network $N_n$. Define the random probability measure $U(N_n)$ over $G_*$ as follows,
\begin{align}
	U(N_n) = \frac{1}{n}\sum_{i\in [n]} \delta_{[N_{n,\circ}(i)]},
\end{align}
where $N_n\sim \mathcal{N}(n,\boldsymbol{d}_n)$ and $N_{n,\circ}(i)$ is the connected component of $i$ in $N_n$ rooted at $i$. Taking expectation with respect to the randomness of the network, for every event $A\in G_*$,
\begin{align}
	\expect U(N_n)(A)\coloneqq\expect\left[U(N_n) (A)\right] &= \frac{1}{n}\sum_{i\in [n]} \expect\left[ \delta_{[N_{n,\circ}(i)]}(A) \right]\\
	&= \frac{1}{n}\sum_{i\in [n]} \prob([N_{n,\circ}(i)]\in A).
\end{align}
Hence, $\expect U(N_n)$ is the law of $[N_{n,\circ}(\root)]$ where $\root\in [n]$ is a random vertex chosen uniformly from $[n]$. Then the primary motivation of our work is the claim that the sequence of random networks $N_n$ converges locally weakly to the EWT, i.e., $\expect U(N_n) \xrightarrow{w} \ertreedist(P)$.

As is suggested by Aldous and Steele in~\cite{Aldous2004}, the first step to establish local weak convergence is to guess the object that the finite graph model converges to. Next, we provide an argument to justify the EWT guess.

Aldous~\cite{Aldous2001} proved that the complete graph $K_n$ with {\em i.i.d.} edge weights distributed as $\expdist(1/n)$ is locally tree-like, and it converges to the Poisson Weighted Infinite Tree(PWIT). The idea is to modify the structure of PWIT to capture the behavior of the finite graph model while preserving unimodularity of the asymptotic object.
In our graph family the root vertex $\root$ is potentially connected to $n_{\root}$ other vertices; hence, the ${n_{\root}+ 1}^{st}$ edge weight in the PWIT is considered as the threshold of the vertex $n_{\root}$. On the other hand, any non-root vertex with label $\boldsymbol{i}$, needs to know the edge weight of its $n_{\boldsymbol{i}}^{\mathrm{th}}$ descendant to decide whether to connect to its ``parent'' or not. Hence, the edge weight of the $n_{\boldsymbol{i}}^{\mathrm{th}}$ descendant in the PWIT is taken to be its real-valued threshold mark if $\boldsymbol{i}$ belongs to the connected component of $\root$.
Moreover, a pruning process is added to include the fact that the survival of an edge is based on the marks at both endpoint vertices. Finally, the labels of the descendants of each vertex are permuted to remove the order. This is an essential step to make the object unimodular.

However, there are quite a few technical issues to resolve to make the above intuitive argument precise. For example, there is interdependence beyond just pairs. The fact that this interdependence can be ignored as was done in the intuitive reasoning that led to the pruned PWIT needs a rigorous proof, as is presented in the Theorem below.

\begin{theorem}\label{thm:weaklimit}
	Suppose that $(\boldsymbol{d}_n)_{n\geq 1}$ is such that $P_{\boldsymbol{d}_n}$ converges weakly to some distribution $P(\cdot)$ and let $N_n \sim \mathcal{N}(n, \boldsymbol{d}_n)$. Then $$\expect U(N_n) \xrightarrow{w} \ertreedist(P).$$
\end{theorem}
\begin{proof}[Sketch of the proof]
	The main body of the proof consists of four steps:
	\begin{enumerate}
		\item Recall that $\expect U(N_n)$ is the law of $[N_{n,\circ}(r)]$ for a uniformly chosen $r\in[n]$. The first step is to redefine the construction of the random network $N_n$, as viewed from $r$.
		\item The random network $N_n$ has an interdependence structure; however as $n$ grows, the dependency weakens. The second step is to exploit this weak dependence and to prove that as $n$ goes to infinity, the connected component of the vertex $r$ becomes locally tree-like.
		\item As the dependency weakens, the local structure of $[N_{n,\circ}(r)]$ gets close to the local structure of a rooted tree distributed under $\ertreedist(P)$. The third step is to prove that for every finite rooted network $\rtree\in G_*$ with depth $t$, the measure assigned to $	A_{\rtree} = \left\{[\N]\in G_*: d_{G_*}([\N],{\rtree})< ({1+t})^{-1} \right\}$ by $\expect U(N_n)$ converges to the measure assigned to $A_{\rtree}$ by $\ertreedist(P)$.
		\item Finally, since $G_*$ is a Polish space, the Portmanteau Theorem is applied to show the desired convergence.
	\end{enumerate}
	The formal proof of the theorem is lengthy and technical. For consistency and to keep the focus of the manuscript on the branching processes itself, the formal proof is given in Appendix~\ref{appendix:pfconv}.
\end{proof}
\subsection{Consequences of the Convergence to EWT}
The local weak convergence result has implications on the global properties of the finite graph model. For example, if $\ertreedist(P)$ assigns probability one to finite rooted networks, then the size of the giant component of the finite graph model is asymptotically $O(1)$ with high probability.

\begin{corollary}\label{cor:sizeofgiant}
	Let $\prob(\{\text{extinction}\})$ denote the probability that the component containing the root of the EWT is finite. Then the expected the proportion of vertices with component of size $O(1)$ converges to $\prob(\{\text{extinction}\})$. In particular, the expected the proportion of vertices in the giant component of the finite graph model is asymptotically bounded by $1-\prob(\{\text{extinction}\})$.
\end{corollary}

\begin{proof}
	For every finite rooted network $\rtree\in G_*$ with depth $t$, define
	\begin{align}
		\tilde{A}_{\rtree} \coloneqq \left\{[\N]\in G_*: d_{G_*}([\N],{\rtree})< ({2+t})^{-1} \right\}.
	\end{align}
	Notice that the depth of all elements of $\tilde{A}_{\rtree}$ is less than $t+1$. Let $\tilde{\mathcal{S}}_{t,d} \subset G_*$ denote the set of all rooted tree networks of depth $t$ and degrees less than $d$, and let ${\mathcal{S}}_{t,d}$ denote a countable dense subset of $\tilde{\mathcal{S}}_{t,d}$. Notice that $\tilde{\mathcal{S}}_{t,d}$ is measurable, since
	\begin{align}
		\tilde{\mathcal{S}}_{t,d} = \bigcup_{{\rtree} \in {\mathcal{S}}_{t,d}} \tilde{A}_{\rtree},
	\end{align}
	and in particular, $\tilde{\mathcal{S}}_{t,d}$ is a continuity set. Hence, by Theorem~\ref{thm:weaklimit}, we have $\lim_{n\to\infty}\expect U(N_n)(\tilde{\mathcal{S}}_{t,d}) = \ertreedist(P)(\tilde{\mathcal{S}}_{t,d})$. Taking $d\to\infty$ and then $t\to\infty$, and using the monotone convergence theorem, we get $\lim_{n\to\infty} \expect U(N_n)(\tilde{\mathcal{S}}) = \ertreedist(P)(\tilde{\mathcal{S}})$, where $\tilde{\mathcal{S}} = \cup_{t,d}\,\tilde{\mathcal{S}}_{t,d}$. Finally, notice that $\ertreedist(P)(\tilde{\mathcal{S}}) = \prob(\{\text{extinction}\})$, and that $\lim_{n\to\infty} \expect U(N_n)(\tilde{\mathcal{S}})$ is the expected proportion of vertices with component of size $O(1)$.
\end{proof}
Usually, there is a stronger relation between $\prob(\{\text{extinction}\})$ and the (limiting) size of the giant component of the finite graph model: we expect the size of the giant component of the finite graph model to be (approximately) $n\times (1 - \prob(\{\text{extinction}\}))$. Such a relation has been established between the configuration model random graphs and the associated unimodular Galton--Watson tree, and also between Erd\H{o}s-R\'enyi random graphs and the Poisson Galton-Watson tree. However, a more detailed analysis is required to prove this relation for the La-Kabkab random graphs and the EWT, and we leave it as an open problem for future work (see open problem~\ref{openprob:sizeofgiant}). We numerically validate this assertion in Section~\ref{sec:numsim}. From general results on local weak convergence, Theorem~\ref{thm:weaklimit} implies that $\ertreedist(P)$ is unimodular. However, unimodularity of $\ertreedist(P)$ can be proved directly too, and our direct proof provides more insight into the structure of the EWT.

\begin{corollary} \label{cor:unimod}
	If $P\in\mathcal{P}(\mathbb{N})$ has a positive and finite mean, then $\ertreedist(P)$ is a unimodular measure in $\mathcal{P}(G_*)$
\end{corollary}
\begin{proof}
	An independent proof is given in Appendix~\ref{appendix:pfunimod}.
\end{proof}

%% file: Sections/NumSim.tex
In this section, we present some numerical results when $P$, the distribution of $n_{\root}$, is the geometric distribution. We start by explicitly determining the degree distribution of the root vertex, simplifying the operator $T$ associated with the probability of extinction, and also determining explicitly the Krein--Rutman eigenvalue and the corresponding eigenfunctions of $M_1$. Then, we investigate various properties of the resulting EWT and compare its structural properties with related unimodular GWTs~\cite{Bordenave2016}.
\begin{proposition}\label{prop:geodist}
	Assume $P$ is the geometric distribution with parameter $p$, i.e., for all $k\in\mathbb{N}$, we have $P(k) = (1-p)^{k-1}p$. Then, the following hold:
	\begin{enumerate}[label=(\roman*)]
		\item The probability distribution of the root vertex is given as follows:
		\begin{align}
			\forall d\in\mathbb{Z}_+,\qquad\prob(D_{\root}=d) = \frac{p}{(1-p)^2} \left(1 - \sum_{m=0}^d \frac{\left(\frac{1-p}{p}\right)^m }{m!}\eexp^{-\frac{1-p}{p}}\right) - \frac{p}{1-p}\onefunc\{d=0\}.
		\end{align}
		\label{prop:geodist_1}
		\item The extinction operator $T$ is given as follows:
		\begin{align}
			T(f)(x) \coloneqq
			\begin{dcases*}
				\frac{px-1+\eexp^{-px}}{px} + \frac{p}{x}\int_{z=0}^{\infty} \min(x,z) \exp\left(-z\left(1 - (1-p)f(z) \right)\right)\,dz, & \!$x>0$ \\
				p\int_{z=0}^{\infty}\exp\left(-z\left(1 - (1-p)f(z) \right)\right)\,dz, &\!$x=0$
			\end{dcases*}
		\end{align}
		\label{prop:geodist_2}
		\item The Krein--Rutman eigenvalue and the corresponding eigenfunctions of $M_1$ are given as follows:
		\begin{align}
			&\text{Eigenvalue: }\beta_0 = \frac{4(1-p)}{{r_0}^2\,p},\allowdisplaybreaks\\
			&\text{Right eigenfunction: } \mu(m,x) = \frac{m}{x} \,\frac{J_0\left({r_0}\eexp^{-\frac{p}{2}x}\right)}{\sqrt{\int_{0}^{\infty} p(1-p)\eexp^{-py} \,\left(J_0\left({r_0}\eexp^{-\frac{p}{2}y}\right)\right)^2\,dy}},\allowdisplaybreaks\\
			&\text{Left eigenfunction: } \nu(k-1,z) = P(k)\frac{\eexp^{-z}z^{k-1}}{(k-1)!}\, \frac{J_0\left({r_0}\eexp^{-\frac{p}{2}z}\right)}{\sqrt{\int_{0}^{\infty} p(1-p)\eexp^{-py} \,\left(J_0\left({r_0}\eexp^{-\frac{p}{2}y}\right)\right)^2\,dy}},
		\end{align}
		where $J_0(\cdot)$ is the zeroth-order Bessel function of first kind, i.e., $J_0(x) = \sum_{i=0}^{\infty} \frac{(-1)^i}{i!~i!} \left(\frac{x}{2}\right)^{2i}$, and $r_0 \approx 2.4048$ is the smallest positive zero of $J_0(\cdot)$.
        \label{prop:geodist_3}
	  \item The asymptotic degree distribution is given as follows:
		\begin{align}
    		\lim_{l \to\infty} \prob(D_l = d\,\vert\,Z_l > 0)
            &=\frac{r_0}{2J_1(r_0)} \int_{0}^{1} \frac{ \left(\frac{1-p}{p}\right)^{d} w^d \exp\left(-\frac{1-p}{p}w\right) J_0\left({r_0}\sqrt{1-w}\right)}{d!}\,dw\\
            &=\frac{1}{J_1(r_0)} \left(\frac{2(1-p)}{r_0p}\right)^{d} \sum_{k=1}^\infty \left(-\frac{2(1-p)}{r_0p}\right)^{k}{ d+k \choose k} J_{d+k+1}(r_0),
        \end{align}
        where for non-negative integer $v$, $J_{v}(\cdot)$ is the $v^{\mathrm{th}}$-order Bessel function of first kind, i.e., $J_v(x) = \sum_{i=0}^{\infty} \frac{(-1)^i}{i!~(i+v)!} \left(\frac{x}{2}\right)^{2i+v}$. Here, $D_l$ is the degree of a vertex at generation $l$, chosen uniformly at random.
    	\label{prop:geodist_4}
        \end{enumerate}
\end{proposition}
\begin{proof}
	The proofs of part~\ref{prop:geodist_1},~\ref{prop:geodist_2}, and~\ref{prop:geodist_4} are elementary and are presented in Appendix~\ref{appendix:geodist}.
	\begin{enumerate}[label=(\roman*)]
		\setcounter{enumi}{2}
		\item Recall the definition of $g_2(\cdot)$ and $G_i(\cdot)$ given in Section~\ref{sec:intro}. We have
		\begin{align}
		g_2(x) = \eexp^{-x} \sum_{k=2}^{\infty} (1-p)^{k-1}p\frac{x^{k-2}}{(k-2)!} = p(1-p)\eexp^{-px}.
		\end{align}
		Using the above equality together with a simple induction argument, we get
		\begin{align}
		G_i(x) = \left(\frac{1-p}{p}\right)^i \frac{ \eexp^{-ipx}}{i!\,i!}
		\end{align}
		Substituting the above equality into the definition of $L(\beta,x)$, we have
		\begin{align}
		L(\beta,x) = \sum_{i=0}^\infty \left(\frac{4(1-p)\eexp^{-px}}{2 p \beta }\right)^i \frac{ (-1)^i}{i!\,i!} = J_0\left(\sqrt{\frac{4(1-p)\eexp^{-px}}{p \beta }}\right)
		\end{align}
		Notice that $J_0\left(\sqrt{\frac{4(1-p)\eexp^{-px}}{p \beta }}\right)$ is the solution of the following differential equation:
		\begin{align}
		\beta\, \dv[2]{q(x)}{x} + p(1-p)\eexp^{-px}\, q(x) = 0,
		\end{align}
		 as we will discuss in Theorem~\ref{thm:propL} part~\ref{part:propL_ii} for the function $L(\beta,x)$.

		As we pointed out in Section~\ref{sec:intro}, $\beta_0$ is the smallest root of $L(\cdot,0)$, and $f_0(\cdot)$ is given by $L(\beta_0,\cdot)\sqrt{C_N}$ where $C_N\! =\! (\int_{0}^{\infty} g_2(y) \left(L(\beta_0,y)\right)^2 dy)^{-1}$. Then, by simple algebra
		\begin{align}
		\beta_0 = \frac{4(1-p)}{{r_0}^2\,p}, \qquad \text{ and } \qquad f_0(x) = \frac{J_0\left({r_0}\eexp^{-\frac{p}{2}z}\right)}{\sqrt{\int_{0}^{\infty} p(1-p)\eexp^{-py} \,\left(J_0\left({r_0}\eexp^{-\frac{p}{2}y}\right)\right)^2\,dy}}.
		\end{align}
	\end{enumerate}
This completes the proof.
\end{proof}
The simple form of the geometric distribution makes it easier to study the associated EWT. Next, we numerically compare the degree distribution of EWT with a related unimodular GWT ($\text{GWT}_*$). A $\text{GWT}_*$ with degree distribution $Q\in\mathcal{P}(\bs{N})$ is a rooted tree, rooted at $\root$, such that the number of descendants of the root is distributed as $Q$, and for all the other vertices, the offspring distribution is given by the size-biased distribution $Q_*$:
\begin{align}\label{eq:sizebiased}
Q_*(k-1) = \frac{k\,Q(k)}{\sum_{r}r \,Q(r)}.
\end{align}

In Figure~\ref{fig:DegreeDist}, we compare the degree distribution of the zeroth and the first generation of the EWT with $\text{GWT}_*$. We consider a $\text{GWT}_*$ that has a Poisson degree distribution with parameter $\lambda'$, and a $\text{GWT}_*$ that has a geometric degree distribution with parameter $p'$. Both $p'$ and $\lambda'$ are chosen so that the expected degree of the root vertex is the same as in the EWT with a geometric distribution for $n_\phi$ with parameter $0.08$. We also consider the size-biased distribution of the root vertex of EWT, using \cref{eq:sizebiased} and Theorem~\ref{thm:ConDegree}. 
The degree distributions of the EWT have different behavior compared with $\text{GWT}_*$. Most notably, the number of descendants of a randomly selected vertex in the first generation is not the size-biased distribution of the root vertex, as we will also discuss in Section~\ref{sec:DegreeDist}. Since there is no closed form for the degree distribution of the first generation for an EWT, we numerically derive this distribution by averaging over $2\times 10^6$ rooted trees.

\begin{figure}[h]
	\centering
	\includegraphics[width=0.8\textwidth]{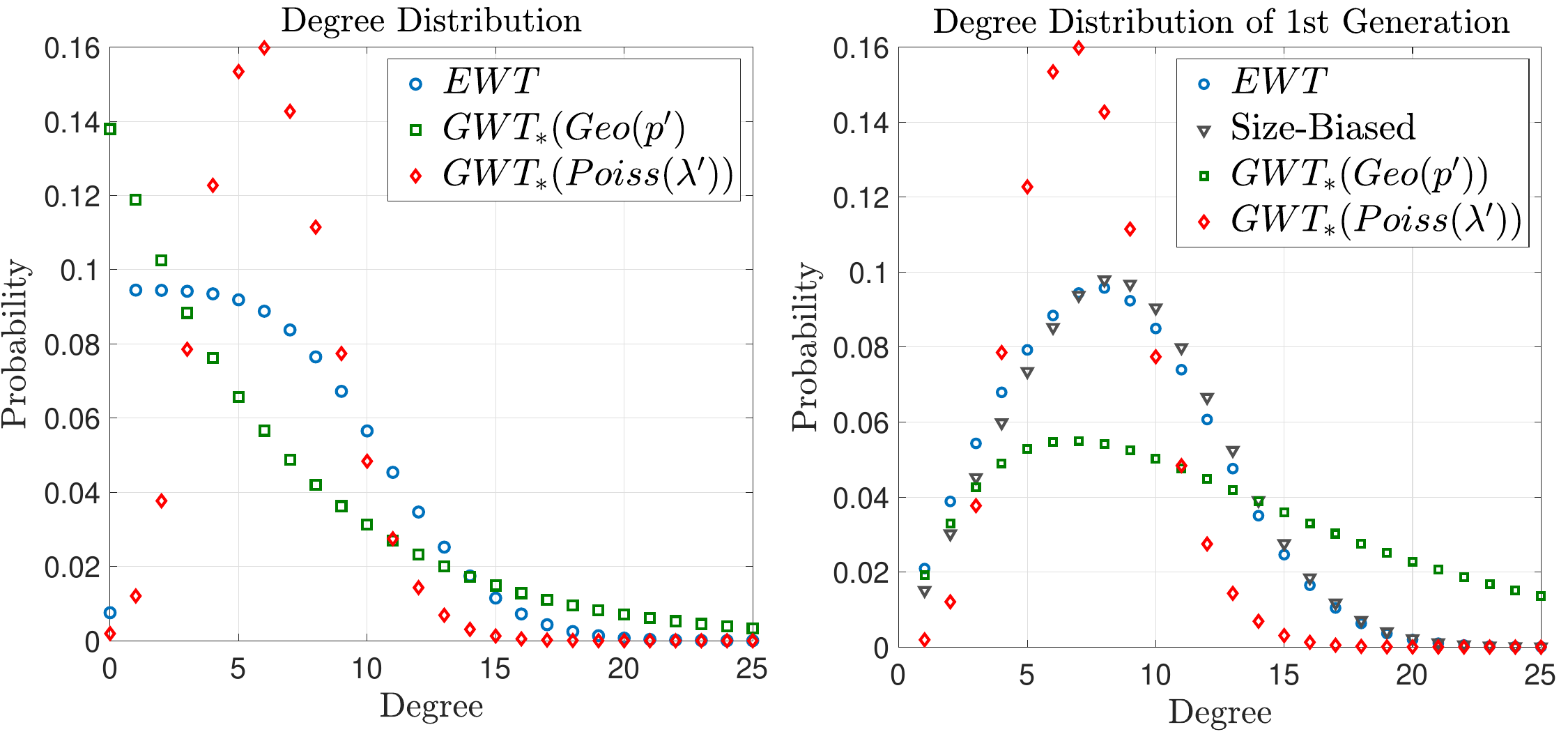}
	\caption{The degree distribution of the root vertex (zeroth generation) and the first generation of the EWT (with potential degree distribution $\geomdist(p)$), unimodular Galton--Watson Trees (with degree distribution $\poissdist(\lambda')$ and $\geomdist(p')$), and the size-biased degree distribution of the root of EWT. We select $p=0.08$, and then the parameters $p'$ and $\lambda'$ are chosen so that the expected degree of the root vertex is the same as in EWT.}
	\label{fig:DegreeDist}
\end{figure}

Next, we compare the degree distributions of different generations of the EWT. In Figure~\ref{fig:DegreeDistGen}, we illustrate the degree distributions of the root node, a node in the first generation, the second generation, and the third generation, and the asymptotic degree distribution of the EWT with potential degree distribution $\geomdist(0.08)$.
Similar to Figure~\ref{fig:DegreeDist}, we numerically derive the degree distribution of the first three generations of the EWT by averaging over $2\times 10^6$ rooted trees.
The error bars are also included in Figure~\ref{fig:DegreeDistGen}. The asymptotic degree distribution is given by Proposition~\ref{prop:geodist}. Notice that the degree distribution of the first generation in Figures~\ref{fig:DegreeDist} and~\ref{fig:DegreeDistGen} are the same. Given the interdependence structure of the EWT, the degree distributions of different generations are not the same. Interestingly, numerically the degree distributions are stochastically ordered as we proceed down the generations. Also, notice that the size-biased distribution of the root vertex is close to the asymptotic degree distribution; however, the two distributions are not the same. This suggests that the growth rate and probability of extinction of the EWT with the potential degree distribution $\geomdist(p)$ should also be close to the growth rate and probability of extinction of $\text{GWT}_*$ with probability distribution given by the degree distribution of the root vertex in EWT; see Figures~\ref{fig:GrowthRate} and \ref{fig:Extinction}.

\begin{figure}[h]
	\centering
	\includegraphics[width=0.85\textwidth]{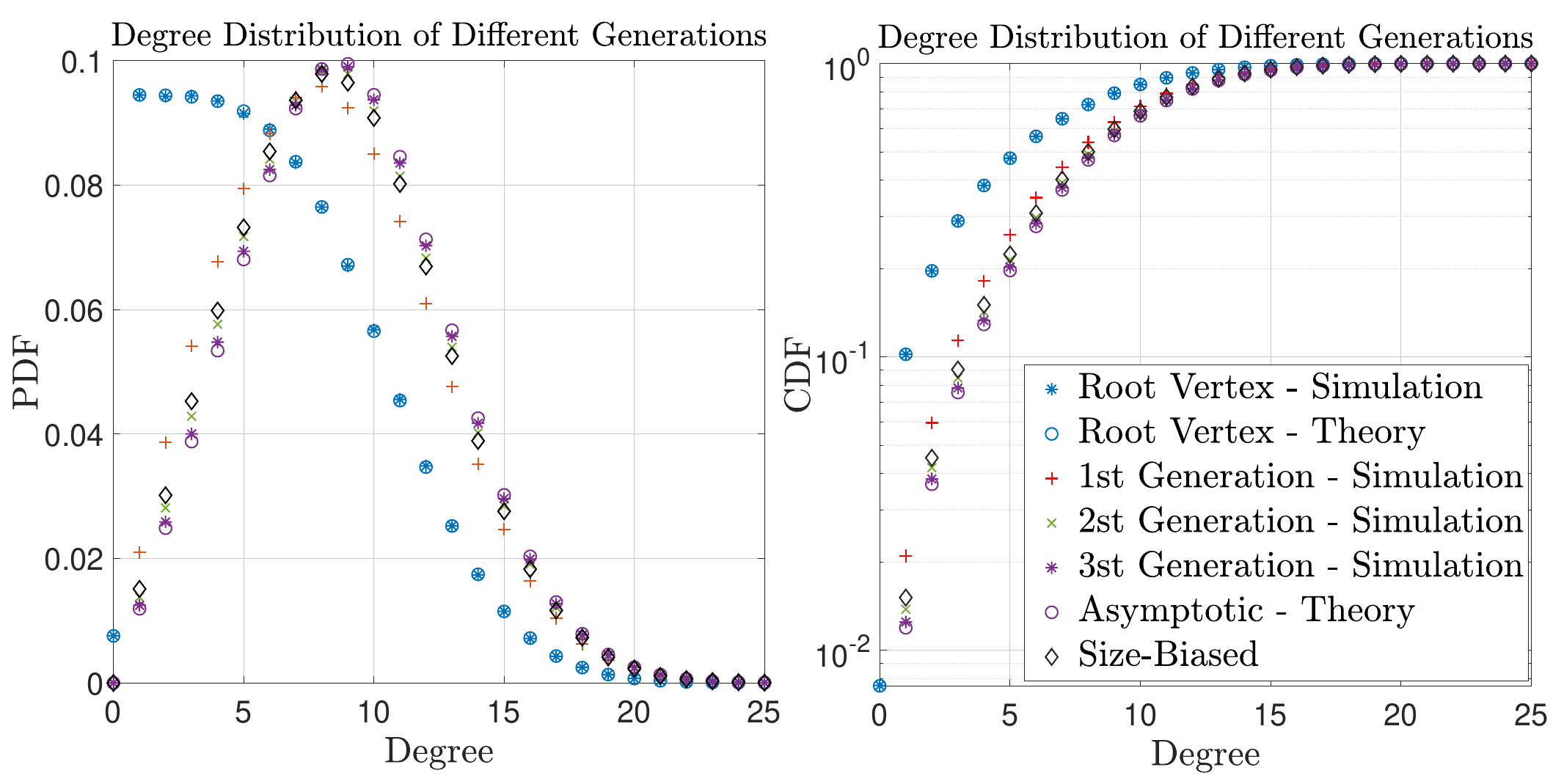}
	\caption{The degree distribution of the root vertex (zeroth generation), the first generation, the second generation, the third generation, the asymptotic degree distribution, and the size-biased distribution of the root vertex of the Erlang Weighted Tree with potential degree distribution $\geomdist(0.08)$. Both plots share the same legend.}
	\label{fig:DegreeDistGen}
\end{figure}

Next, in Figure~\ref{fig:DegreeDistCond} we compare the conditional degree distributions of the first generation, conditioned on the degree of the root vertex. Since the EWT is the random weak limit of the finite graph model, we numerically derive the conditional degree distributions of the EWT by averaging over $1000$ graphs with $10000$ vertices.
As we pointed out earlier, the degree distribution of the first generation, $D_1$, depends on the degree distribution of the root vertex, $D_\root$. In particular, a larger value of $D_\root$ increases the probability of observing larger values of $D_1$.

\begin{figure}[h]
	\centering
	\includegraphics[width=0.7\textwidth]{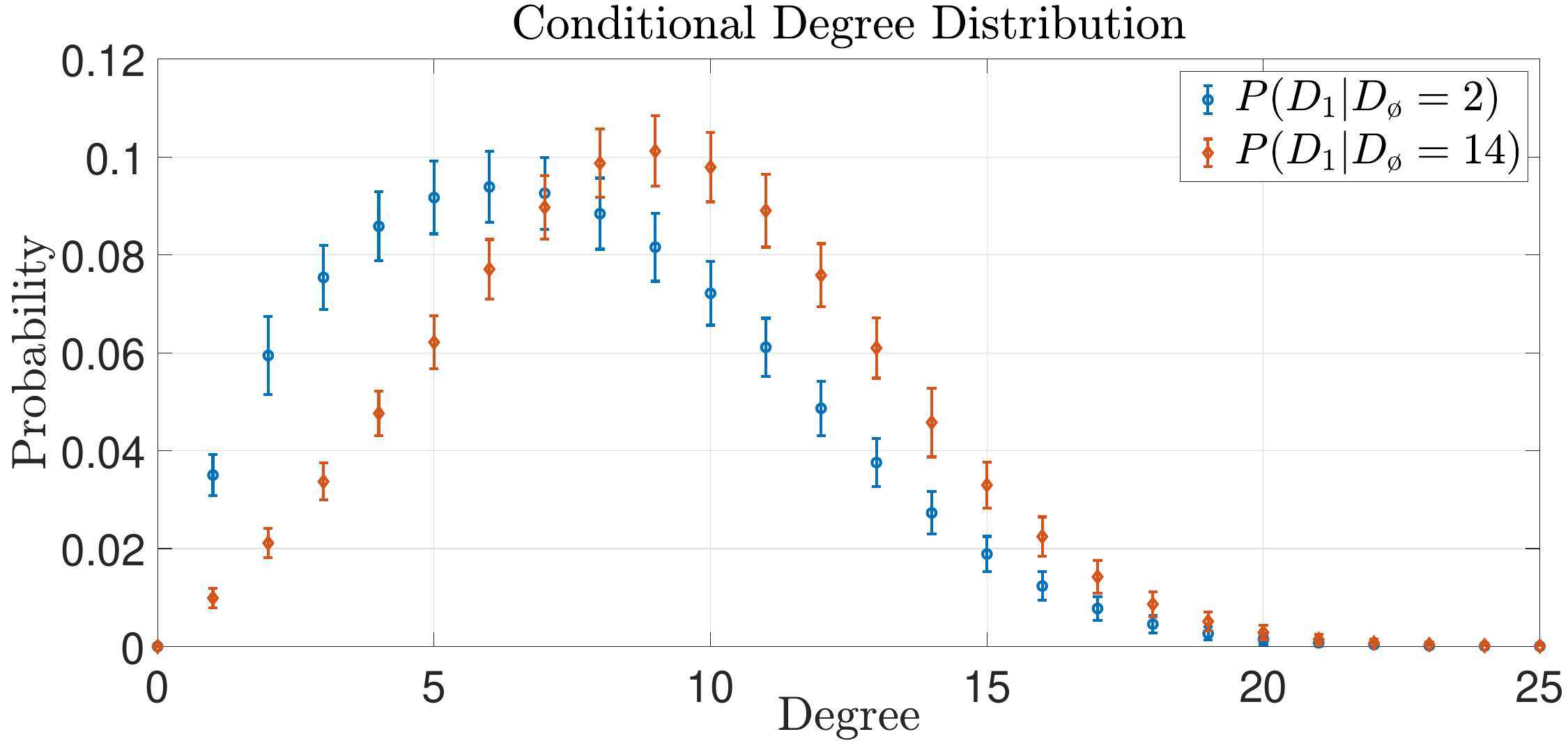}
	\caption{The conditional degree distribution of the first generation, conditioned on the degree of the root vertex of the Erlang Weighted Tree with potential degree distribution $\geomdist(0.08)$.}
	\label{fig:DegreeDistCond}
\end{figure}

In Figure~\ref{fig:GrowthRate} we compare the growth/extinction rate of the EWT with different choices of $\text{GWT}_*$. We consider a $\text{GWT}_*$ that has a Poisson degree distribution with parameter $\lambda'$, a $\text{GWT}_*$ that has a geometric degree distribution with parameter $p'$, and a $\text{GWT}_*$ with degree distribution given by the degree distribution of the root vertex of EWT. As we mentioned earlier, the growth/extinction rate of the EWT is close to the growth/extinction rate of $\text{GWT}_*$ with degree distribution given by the degree distribution of the root vertex of EWT; however, they are not the same.

\begin{figure}[h]
	\centering
	\includegraphics[width=0.7\textwidth]{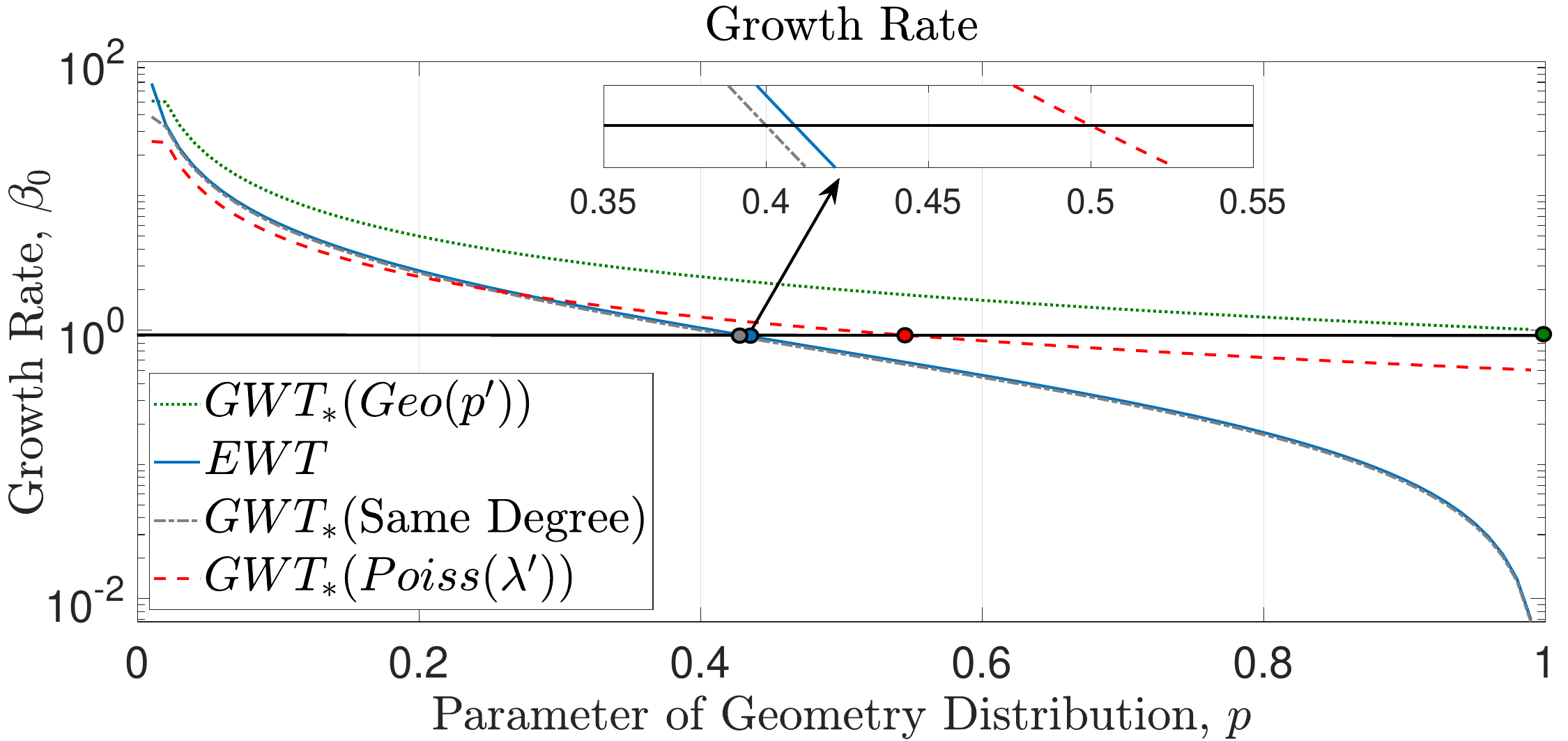}
	\caption{The growth rate of the EWT (with potential degree distribution $\geomdist(p)$) and unimodular Galton--Watson Trees (with degree distribution $\poissdist(\lambda')$, $\geomdist(p')$, and the degree distribution of the root vertex in EWT, respectively). We select $p=0.08$, and the the parameters $p'$ and $\lambda'$ are chosen so that the expected degree of the root vertex is the same as in the EWT.}
	\label{fig:GrowthRate}
\end{figure}

Finally, in Figure~\ref{fig:Extinction} we compare the probability of extinction of the EWT with different $\text{GWT}_*$ choices. We consider the same set of unimodular GWTs as those used in Figure~\ref{fig:GrowthRate}. We also compare the proportion of vertices in the giant component of the finite graph model (with potential degree distribution $\geomdist(p)$) with random graphs generated by the configuration model (using the same degree distribution as in the associated $\text{GWT}_*$), and the Erd\H{o}s-R\'enyi random graph (with parameter $\lambda'/n$, where $n$ is the number of vertices), in Figure~\ref{fig:GianComp}. We derive the size of the giant component of the finite graph model by averaging over $50$ graphs with $50000$ vertices. The error bars are also included. The configuration model generates a random graph by uniformly pairing the half-edges assigned to vertices of the graph, where the number of half-edges assigned to a vertex is given by a fixed degree distribution. The Erd\H{o}s-R\'enyi random graph with parameter $\lambda'/n$ is given by connecting pairs of vertices to each other with probability $\lambda'/n$. For the configuration model and the Erd\H{o}s-R\'enyi random graph model, this ratio equals $1-\prob(\{\text{extinction}\})$, where $\prob(\{\text{extinction}\})$ is the probability of extinction of the associated $\text{GWT}_*$~\cite{Bordenave2016}. Figures~\ref{fig:Extinction} and~\ref{fig:GianComp} suggest that this is also true for the La-Kabkab random graph model and the EWT.

\begin{figure}[h]
	\centering
	\includegraphics[width=0.8\textwidth]{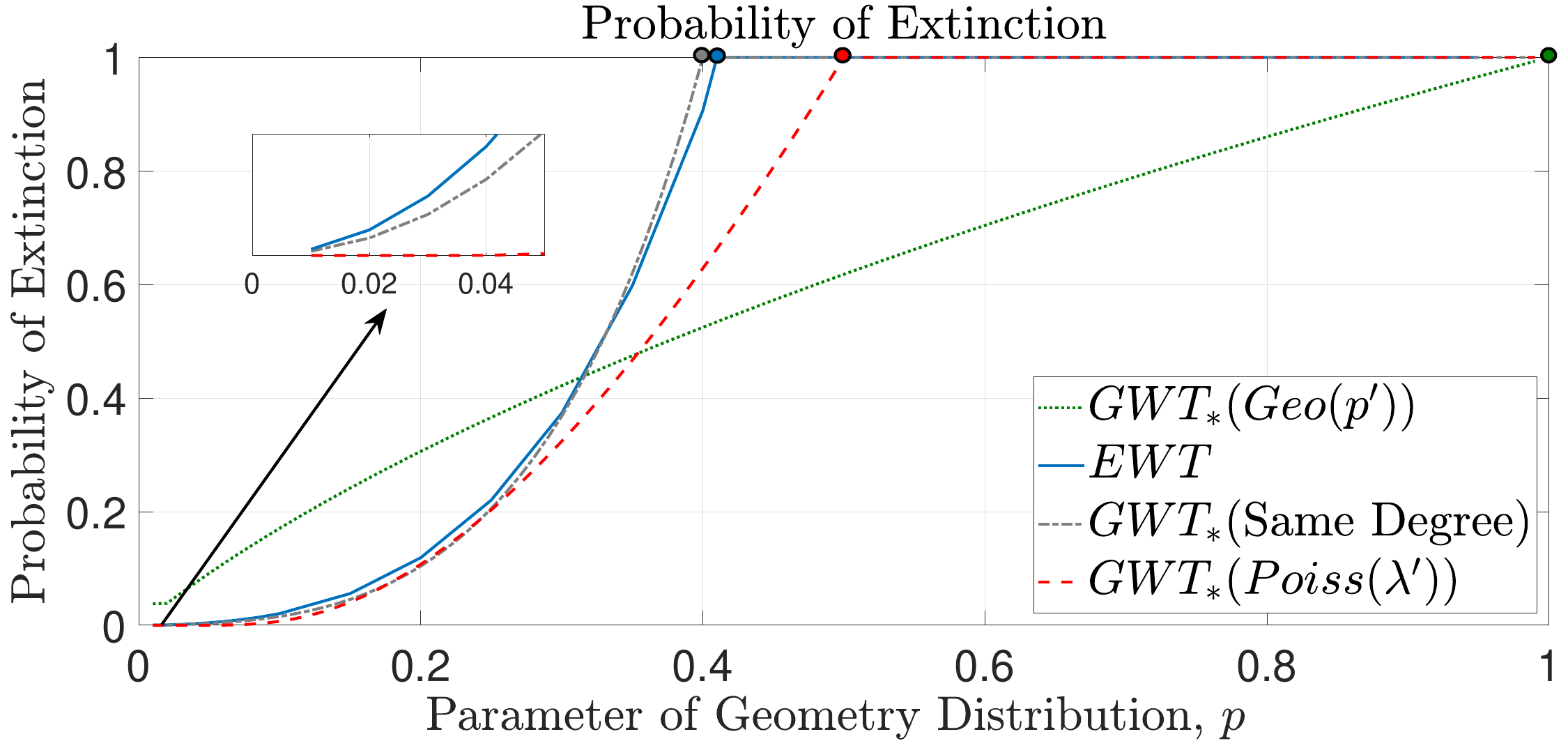}
	\caption{The probability of extinction of Erlang Weighted Tree (with potential degree distribution $\geomdist(p)$) and unimodular Galton--Watson Trees (with degree distribution $\poissdist(\lambda')$, $\geomdist(p')$, and the degree distribution of the root vertex in EWT ). The parameters $p'$ and $\lambda'$ are chosen so that the expected degree of the root vertex is the same as in EWT.}
	\label{fig:Extinction}
\end{figure}

\begin{figure}[h]
	\centering
	\includegraphics[width=0.8\textwidth]{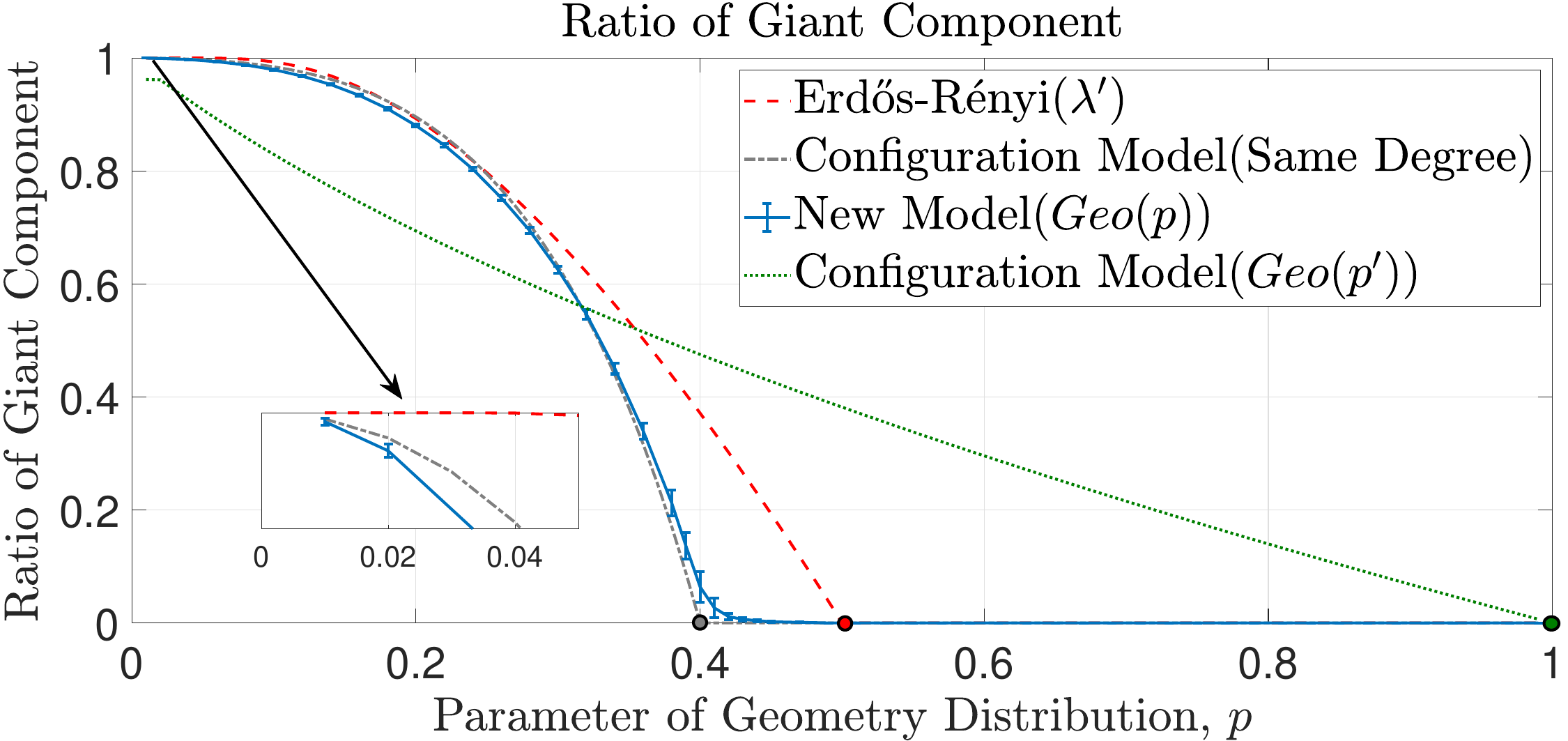}
	\caption{The proportion of vertices in the giant component of the finite graph model (with potential degree distribution $\geomdist(p)$), random graphs generated by the configuration model (with degree distributions $\geomdist(p')$, and the degree distribution of the root vertex in the EWT ), and the Erd\H{o}s-R\'enyi random graph (with parameter $\lambda'/n$). The parameters $p'$ and $\lambda'$ are chosen so that the degree distributions are the same.}
	\label{fig:GianComp}
\end{figure}

%% file: Sections/openprob.tex
We conclude the discussion of the main results with some open problems:
\begin{enumerate}[wide]
	\item Conditioned on $Z_l > 0$, for $\beta_0 \leq 1$, what is the asymptotic distribution of $Z_l$ as $l$ grows without bound? And, what is the asymptotic distribution of $Z_l$ conditioned on $Z_{l+m}>0$ as either $l$ or $m$ increase without bound? Starting with Yaglom~\cite{yaglom1947} such limits were studied for the Galton-Watson branching process culminating in the famous $L\log(L)$ criterion~\cite{heathcote1967,Joffe1967,Athreya1972,lyons1995,Pakes1999}: the limiting distributions, called the Yaglom limits, exist if and only if $\mathbb{E}[Z_1 \log^+(Z_1) | Z_0=1] < + \infty$. Similar problems have also been studied for multi-type branching processes~\cite{Penission2016}, continuous-state branching processes~\cite{Lambert2007,li2000,Grey1974} and superprocesses~\cite{Liu2022}. Formally, the question is whether there exists a probability measure $\varrho$ on $\Omega$ such that
	\begin{align}
		\lim_{l\to\infty} \prob(Z_l \in \Apset \,\vert\, Z_l(\Omega) > 0) = \varrho(\Apset),
	\end{align}
	for any Borel $\Apset \subset \Omega$. 
    See \cite{Goldstein1971,Mullikin1963,seneta1968,Labbe2013,liu2021} for further discussion.

	\item In \cref{sec:PFeigenvalue}, we introduce a continuous state Markov process related to the study of the generations of the EWT. The transition probability kernel of this continuous state Markov process is given by:
	\begin{align}
		\forall x,y\in\mathbb{R}_+, \qquad p(x,y) \coloneqq \frac{\min(x,y)g_2(y)f_0(y)}{\beta_0f_0(x)}.
	\end{align}
	Based on the results we obtained for this process, the following question arises: What is the connection between the reversibility of this continuous state Markov process and the unimodularity of the branching process? Exploring this connection can provide a general framework to study unimodular branching process, an important class of branching processes.

	\item What is the connection between the probability of extinction and the proportion of vertices in the giant component in the finite graph model? For other random graph models (e.g. configuration model, Erd\H{o}s-R\'enyi model, etc.), the the proportion of vertices in the giant component converges to $1-\prob(\{\text{extinction}\})$, where $\prob(\{\text{extinction}\})$ is the probability the associated branching process goes extinct eventually. We have observed the same relation via numerical simulation in Figures~\ref{fig:Extinction} and~\ref{fig:GianComp} between the finite graph model and the EWT. Notice that by Corollary~\ref{cor:sizeofgiant}, the the proportion of vertices in the giant component is bounded above by $1-\prob(\{\text{extinction}\})$; however, the reverse inequality needs a separate proof.
    \label{openprob:sizeofgiant}

    \item Our main results on the growth and extinction of the EWT were obtained by assuming that the potential degree distribution $P$ has a finite moment generating function for some $\theta>0$. To what extent can this assumption on the potential degree distribution be relaxed?

	\item What is the local weak limit if vertices in the finite graph model iterate to use all their budget, given by their potential degree? Naturally, one can imagine a scenario in which after the realization of $G_n$, all vertices $i$ with degree less than $d_i(n)$ have a second chance to find more neighbors by announcing an updated set of potential neighbors. Of particular interest is the case when vertices can iterate as many times as possible until they achieve $d_i(n)$ or have checked all other $n-1$ vertices.

	\item To what extent can the methodology we developed in this work be applied to finding the Krein-Rutman eigenvalue and the corresponding eigenfunction for other operators? Is it possible to extend this methodology to more general multi-type branching processes with uncountably many types?

\end{enumerate}

%% file: Sections/PropEWT/DegreeDist.tex
We begin the analysis of the EWT by characterizing its degree distribution. The conditional degree distribution of a vertex conditioned on its type and the degree distribution of the root vertex is given as follows.
\begin{theorem}\label{thm:ConDegree}
	Let $D_{\boldsymbol{i}}$ denote the number of descendants of the vertex $\boldsymbol{i} \in T\sim \ertreedist(P)$. The conditional distribution of $D_{\boldsymbol{i}}$, conditioned on the type of the vertex $\boldsymbol{i}$ is given as follows:
	\begin{align}
	\prob\left(D_{\boldsymbol{i}}=d \,\vert\,n_{\boldsymbol{i}} = m, v_{\boldsymbol{i}} = x \right) = Bi\left(d;m,\int_{0}^{x} \frac{1}{x}\sum_{k=1}^{\infty}P(k)\bar{F}_{k}(y)\,dy\right),
	\end{align}
	where $\bar{F}_{k}(\cdot)$ is the complementary cumulative distribution function of $\erlangdist(k)$ and $Bi(d;m,\eta) = C(m,d) \eta^d (1-\eta)^{m-d}$, with $C(m,d) = m! / (d! (m-d)!)$. Consequently, the degree distribution of the root vertex and its mean are given as follows:
	\begin{align}
	&\prob(D_{\root}=d) = \sum_{m=1}^\infty P(m) \int_{0}^{\infty} \frac{ \eexp^{- x}x^m}{m!} Bi\left(d;m,\int_{0}^{x} \frac{1}{x}\sum_{k=1}^{\infty}P(k)\bar{F}_{k}(y)\,dy\right) \,dx,\allowdisplaybreaks\\
	&\expect[D_{\root}] = \sum_{m=1}^\infty\sum_{k=1}^{\infty} P(m)P(k) \int_{0}^{\infty} \bar{F}_{k}(y) \bar{F}_{m}(y) \,dy =\sum_{m=1}^\infty\sum_{k=1}^{\infty} P(m)P(k) \sum_{n=0}^{k-1} \sum_{l=0}^{m-1} {n+l \choose n} 2^{-n-l-1}.
	\end{align}
\end{theorem}
\begin{proof}
The proof is presented in Appendix~\ref{appendix:degreedist}.
\end{proof}
It is easy to derive in closed form the degree distribution of the root vertex. However, the degree distribution of a vertex at depth $l$ is rather complex. To see why, let us focus on the vertices at the first generation, i.e., the neighbors of the root vertex $\root$. For a unimodular measure $\rho$ with support on rooted trees, the following equality holds,
\begin{align} \label{eq:unimodprop}
\expect_{\rho}\left[\sum_{v\sim\root} \boldsymbol{1}_{deg(\root) = k}\right] =
\expect_{\rho}\left[\sum_{v\sim\root} \boldsymbol{1}_{deg(v) = k}\right].
\end{align}
The above relation is obtained by using the following function in the definition of the unimodularity (see Definition~\ref{def:unimod}),
\begin{align}
f_{k}([N_{\circ\circ}(\root,v)])=
\begin{cases}
1 &\text{ if }deg(\root) = k \text{ and } v\sim\root\\
0 &\text{ otherwise}
\end{cases}.
\end{align}
It is easy to check that the function $f_{k}$ is a Borel function from $G_{**}$ to $\mathbb{R}$, where $G_{**}$ is the set of isomorphism classes of connected locally-finite networks with an ordered pair of distinct vertices; see \cref{sec:back_randomgrpah}. Let $D_1$ and $D_{\root}$ denote the number of descendants of a vertex at the first generation and the degree of the root vertex, respectively. Simplifying \cref{eq:unimodprop}, we have
\begin{align}
k\prob(D_{\root} = k) =
\expect_{\rho}\left[D_{\root} \prob(D_1 + 1=k\,\vert\,D_{\root})\right].
\end{align}
From the last display we can check that if $D_1$ and $D_{\root}$ are independent, then $D_1+1$ would have the size-biased distribution corresponding to $D_{\root}$. This is the case for the unimodular GWT~\cite{Bordenave2016}. However, in our setting, $D_1$ and $D_{\root}$ are not independent. Another interesting observation is that the degree distributions of different generations are not the same since the probability of the events $n_{\boldsymbol{i}} = m$ and $v_{\boldsymbol{i}}=x$ depends on the depth of the vertex $\boldsymbol{i}$. Owing to this interdependence structure, at present, we do not have a characterization of the degree distribution at any level other than the root.
We will revisit this problem in Section~\ref{sec:DegreeDistRev}, and characterize the number of descendants of a typical vertex in generation $l$ as $l\to\infty$ (assuming extinction doesn't occur), after presenting the point process perspective.

%% file: Sections/PropEWT/ExtincProb.tex
The next natural quantity to study is the probability that the component containing the root is finite, i.e., the probability of extinction. This is an important quantity associated with the EWT which should be related to the size of the giant component in the finite graph model, as in the unimodular GWT. Let us start with the definition of the probability of extinction.
\begin{definition}\label{def:extinction}
	Let $Z_l$ denote the number of vertices at depth $l$. The probability of extinction is defined as:
	\begin{align}
	\prob(\{\text{extinction}\}) &\coloneqq \prob\left( \bigcup_{l=1}^\infty \{Z_l = 0\}\right)\allowdisplaybreaks.
	\end{align}
\end{definition}
Observe that the event $\{Z_i=0\}$ is a subset of the event $\{Z_j=0\}$ for every $j<i$; hence, the continuity of probability measures implies that
\begin{align}
\prob(\{\text{extinction}\}) &= \lim_{l\to\infty} \prob(\{Z_l=0\}).
\end{align}
Using this, we can characterize the probability of extinction. For this we will define an operator $T$ associated with the EWT that maps Lebesgue measurable functions on $\mathbb{R}_+$ taking values in $[0,1]$ to continuously differentiable functions on $\mathbb{R}_+$ taking values in $[0,1]$. For a valid function $f(\cdot)$, given $x\in \mathbb{R}_+$, the value $f(x)$ can be considered as a candidate for $\mathbb{P}(\{\text{extinction}\}\,\vert\, n_{\root}=1,v_{\root}=x)$. Then $T(f)(x)$ for $x\in\mathbb{R}_+$ will be another candidate for $\mathbb{P}(\{\text{extinction}\}\,\vert\, n_{\root}=1,v_{\root}=x)$ obtained by iterating $f(\cdot)$ on the EWT as detailed in \eqref{eq:defT} below: the first term comes from the first edge (from the root vertex) not forming; and the second contribution comes from the first edge forming and then noticing that the subtrees from vertex $1$ are statistically the same and independent conditioned on the type of vertex $1$, where we apply $f(\cdot)$ using the type of vertex $1$. Then, fixed points of $T$ will help determine the probability of extinction. For specific choices of $f(\cdot)$ we can also use $T$ to define a sequential relationship on probabilities of quantities associated with the EWT. For $x\in\mathbb{R}_+$ and $l\in\mathbb{Z}_+$, set $f_l(x) \coloneqq \prob\left(\{Z_l = 0\}\,\vert\,n_{\root} = 1, v_{\root} = x \right)$. Notice that $Z_l=0$ if for all $i\in[n_\root]$ either the potential edge $\{i,\root\}$ does not survive, or $\{i,\root\}$ survives and the subtree rooted at $i$ goes extinct after at most $l-1$ generations. Also, notice that the survival of different branches of the root vertex, conditioned on its type, are independent of each other. These two observations will be used to prove 
the following recursive relationship between the associated probabilities:
\begin{align}
   \forall l\in\mathbb{N}, x\in\mathbb{R}_+, \qquad f_l(x)  =
    T(f_{l-1})(x).
\end{align}
\begin{theorem} \label{thm:probext}
	Consider the operator $T: L(\mathbb{R}_+;[0,1]) \to C^1(\mathbb{R}_+;[0,1])$ defined as
	\begin{align}
	T(f)(x) \coloneqq
	\begin{dcases*}
	\begin{aligned}
	&\frac{1}{x}\sum_{k=1}^{\infty}P(k)\int_{y=0}^{x} \Bigg(\int_{z=0}^{y}\frac{\eexp^{-z}z^{k-1}}{(k-1)!}\,dz \\
	&\myquad[10]+\int_{z=y}^{\infty}\frac{\eexp^{-z}z^{k-1}}{(k-1)!}\left(f(z)\right)^{k-1}dz\Bigg)dy,
	\end{aligned}
	 & $x>0$ \\
	\sum_{k=1}^{\infty}P(k) \int_{z=0}^{\infty}\frac{\eexp^{-z}z^{k-1}}{(k-1)!}\left(f(z)\right)^{k-1}\,dz, &$x=0$
	\end{dcases*}\numberthis \label{eq:defT}
	\end{align}
	with the convention $0^0 = 1$. Then the probability of extinction is
	\begin{align}
	\prob(\{\text{extinction}\})= \sum_{m=1}^{\infty} P(m)\int_{x=0}^{\infty} \frac{\eexp^{-x}x^m}{m!} \left(q(x)\right)^m \,dx,
	\end{align}
	where the function $q(\cdot)\in C^1(\mathbb{R}_+;[0,1])$ is the smallest fixed point of the operator $T$, that is, for any other fixed point of $T(\cdot)$ say $f(\cdot)\in C^1(\mathbb{R}_+;[0,1])$, and for all $x\in\mathbb{R}_+$, we have $f(x) \geq q(x)$. Equivalently, the function $q(\cdot)$ is the point of convergence of $T^l(\boldsymbol{0})(\cdot)$ as $l$ goes to infinity, where $\boldsymbol{0}(\cdot)$ is the null function, that is, $\boldsymbol{0}(x)\equiv0$ for all $x$.
\end{theorem}
\begin{proof}[Sketch of the proof]\renewcommand{\qedsymbol}{}
The main idea is to find the probability of the event $\{Z_l = 0\}$ and then, let $l$ increase to infinity. This can be done through the following steps.
\begin{enumerate}
	\item Observe that conditioned on the type of the root vertex to be $(m,x)$, there are $m$ potential branches and the probability that the depth of each branch is less than or equal to $l-1$ depends only on the value of $x$.
	\item Starting from the first generation, all the vertices have the same behavior, i.e., for any non-root vertex $\boldsymbol{i}$, the distribution of $n_{\boldsymbol{i}}$ is given by $\widehat{P}$. Hence, it is possible to write the probability that the depth of a branch is less than or equal to $l-1$ via a recursion.
	\item Taking the limit and using monotonicity, the result follows.
\end{enumerate}
\end{proof}
\begin{proof}
	We now fill in the details. The theorem claims that the range of $T$ is $C^1(\mathbb{R}_+;[0,1])$ and that there exists a fixed point $q(\cdot)$ such that for any other fixed point $f(\cdot)$ of $T$,
	\begin{align}
	\forall x\in\mathbb{R}_+,\qquad q(x)=T(q)(x)\leq T(f)(x)=f(x),
	\end{align}
	i.e., it is the smallest fixed point of the operator $T$. The theorem also claims that
	\begin{align}
	q(\cdot) = \lim_{l\to\infty} T^l(\boldsymbol{0}).
	\end{align}
	We start by proving the following important properties of the operator $T$. Let $\boldsymbol{1}(\cdot)$ be the constant function with value $1$ everywhere.
	\begin{lemma}\label{lem:propt}
	The following hold:
		\begin{enumerate}[label=(\roman*)]
			\item For every $f(\cdot)\in L(\mathbb{R}_+;[0,1])$, the function $T(f)(\cdot)$ is non-decreasing and it belongs to $C^1(\mathbb{R}_+;[0,1])$. Moreover, $T(f)(\cdot)\equiv 1$ if and only if $f(x) = 1$ for almost every $x\in\mathbb{R}_+$.\label{part:propt_i}
			\item The largest fixed point of the operator $T$ is the constant function $\boldsymbol{1}(\cdot)$. 
			Moreover, if $f(\cdot)\neq \boldsymbol{1}(\cdot)$ is a fixed point of $T$, then $f(\cdot)$ is strictly increasing.
			\item For every pair of functions $f(\cdot),g(\cdot)\in C^1(\mathbb{R}_+,[0,1])$ with the property that for all $x\in\mathbb{R}_+$ the inequality $f(x)< g(x)$ holds, we have
                \label{part:propt_iii}
			\begin{align}
			\forall x\in\mathbb{R}_+,\qquad T(f)(x)< T(g)(x).
			\end{align}
			\item The function $T^l(\boldsymbol{0})$ converges point-wise to some function $q(\cdot)\in C[0,1]$ as $l$ goes to infinity, which is the smallest fixed point of the operator $T$.
		\end{enumerate}
	\end{lemma}
	\begin{proof}[Proof of Lemma~\ref{lem:propt}.] The proof is algebraic and does not use the connection between the operator $T(\cdot)$ and the probability of extinction. The proof is presented in Appendix~\ref{appendix:lemextinc}.
	\end{proof}
	We now get back to the proof of the main theorem. As we mentioned, the main idea is to characterize the probability of the event $\{Z_l = 0\}$. Define $Z_{l,i}$ to be the number of children at depth $l$ in the subtree connected to the root via vertex $i\in[n_\root]$. Fix an $l > 1$. Notice that $Z_{l} = 0$ if for all $i\in [n_{\root}]$ either $(\text{i})$ $v_i < \zeta_i$, i.e., the $i^{\mathrm{th}}$ edge does not form, or $(\text{ii})$ the $i^{\mathrm{th}}$ edge forms but there are no children at its $l^{\mathrm{th}}$ level, i.e., $Z_{l,i} = 0$. Recall that for $i\in[n_{\root}]$, $\zeta_i$ is the cost of the potential edge $\{\root,i\}$. Hence, for $l\geq 2$ we have
	\begin{align}
	&\prob\left(\{Z_l = 0\}\,\vert\, n_{\root} = m, v_{\root} = x\right) \allowdisplaybreaks\\
	&\qquad =\prod_{i=1}^m \prob\Big(\big\{v_i < \zeta_i\big\}\cup\big\{\{v_i \geq \zeta_i\} \cap \{Z_{l,i} = 0\}\big\}\,\vert\, n_{\root} = m, v_{\root} = x\Big)\allowdisplaybreaks\\
	&\qquad =\Big( \prob\big(\{v_1 < \zeta_1\big\}\,\vert\, n_{\root} = m, v_{\root} = x\big) + \prob\big(\{v_1 \geq \zeta_1\} \cap \{Z_{l,1} = 0\}\,\vert\, n_{\root} = m, v_{\root} = x\big)\Big)^m\allowdisplaybreaks\\
	&\qquad =\Bigg( \sum_{k=1}^\infty \widehat{P}(k-1) \int_{y=0}^{x}\frac{1}{x}\int_{z=0}^{y}\frac{\eexp^{-z}z^{k-1}}{(k-1)!}\,dz\,dy  \allowdisplaybreaks\\
	&\qquad\qquad +\sum_{k=1}^\infty \widehat{P}(k-1) \int_{y=0}^{x}\frac{1}{x}\int_{z=y}^{\infty}\frac{\eexp^{-z}{z}^{k-1}}{(k-1)!} \prob\left(\{Z_{l,1} = 0\}\,\vert\,n_1 = k-1, v_1 = z\right) \,dz\,dy \Bigg)^m.\numberthis\label{eq:ProbZ_l0}
	\end{align}
	Conditioning on the type of the vertex $1$, the probability distribution of $Z_{l,1}$ for $l>1$ is exactly the same as the probability distribution of $Z_{l-1}$ conditioned on the corresponding type of the root vertex; in particular,
	\begin{align}\label{eq:propext_obs}
	\prob\left(\{Z_{l,1} = 0\}\,\vert\,n_1 = k-1, v_1 = z\right) = \prob\left(\{Z_{l-1} = 0\}\,\vert\, n_{\root} = k-1, v_{\root} = z\right).
	\end{align}
	A crucial observation is that $\prob\left(\{Z_l = 0\}\,\vert\, n_{\root} = m, v_{\root} = x\right)$ depends on $m$ only through the exponent, that is,
 \begin{align}
 \prob\left(\{Z_l = 0\}\,\vert\, n_{\root} = m, v_{\root} = x\right) = \left( \prob\left(\{Z_l = 0\}\,\vert\,n_{\root} = 1, v_{\root} = x \right)\right)^m = \left(f_l(x)\right)^m.
 \end{align}
 Using \eqref{eq:ProbZ_l0} we get the following expression
	\begin{align}
	&f_l(x) 
 =\sum_{k=1}^\infty \widehat{P}(k-1) \int_{y=0}^{x}\frac{1}{x}\int_{z=0}^{y}\frac{\eexp^{-z}z^{k-1}}{(k-1)!}\,dz\,dy  \allowdisplaybreaks\\
	&\qquad\qquad\qquad +\sum_{k=1}^\infty \widehat{P}(k-1) \int_{y=0}^{x}\frac{1}{x}\int_{z=y}^{\infty}\frac{\eexp^{-z}{z}^{k-1}}{(k-1)!} \prob\left(\{Z_{l,1} = 0\}\,\vert\,n_1 = k-1, v_1 = z\right) \,dz\,dy.
	\end{align}
	Using \cref{eq:propext_obs} and the definition of the function $f_l(\cdot)$, for every $l>0$, we have
	\begin{align}
	f_l(x) &= \sum_{k=1}^\infty P(k) \int_{y=0}^{x}\frac{1}{x}\left(\int_{z=0}^{y}\frac{\eexp^{-z}z^{k-1}}{(k-1)!}\,dz + \int_{z=y}^{\infty}\frac{\eexp^{-z}{z}^{k-1}}{(k-1)!} \left(f_{l-1}(z)\right)^{k-1} \,dz\right)\,dy \allowdisplaybreaks \\
	&~=T(f_{l-1})(x),
	\end{align}
	where $f_1(\cdot)$ should be taken to be $T(\boldsymbol{0})(\cdot)$ for consistency with \cref{eq:ProbZ_l0} at $l=2$. Lemma~\ref{lem:propt} implies that $f_l(\cdot) = T^l(\boldsymbol{0})(\cdot)$ converges to $q(\cdot)$, the smallest fixed point of $T$, point-wise. Hence,
	\begin{align}
	\prob\left(\{\text{extinction}\}\,\vert\,n_{\root} = m, v_{\root} = x \right) &= \lim_{l\to\infty} \prob\left(\{Z_l = 0\}\,\vert\,n_{\root} = m, v_{\root} = x \right) \allowdisplaybreaks\\
	&= \lim_{l\to\infty} \big(T^l(\boldsymbol{0})(x)\big)^m \allowdisplaybreaks\\
	&=\left(q(x)\right)^m.
	\end{align}
	Taking expectation with respect to $n_{\root}$ and $v_{\root}$ and using the monotone convergence theorem, we have
	\begin{align}
	\prob(\{\text{extinction}\})= \sum_{m=1}^{\infty} P(m)\int_{x=0}^{\infty} \frac{\eexp^{-x}x^m}{m!} \left(q(x)\right)^m \,dx,
	\end{align}
 which completes the proof.
\end{proof}
The above theorem suggests that for all $f(\cdot) \in L(\mathbb{R}_+;[0,1])$, the function $T^l(f)(\cdot)$ converges point-wise to a fixed point of $T$, as $l$ goes to infinity; however, it is not clear how many fixed points the operator $T$ has and, if there is more than one fixed point, to which one does $T^l(f)(\cdot)$ converge.
An immediate corollary is the following.
\begin{corollary}\label{cor:pext=1}
	$\prob(\{\text{extinction}\}) = 1$ if and only if $\boldsymbol{1}(\cdot)$ is the unique fixed point of the operator $T$.
\end{corollary}
A sufficient condition to check $\prob(\{\text{extinction}\}) < 1$ is to find a test function $f(\cdot) \in L(\mathbb{R}_+;[0,1])$ such that for all $x\in\mathbb{R}_+$ we have $T(f)(x)\leq f(x)$. One natural choice is
\begin{align}
f_{x_0,\epsilon}(x) \coloneqq \begin{cases}
1 - \epsilon, & \text{if } x\leq x_0 \\
1, & \text{otherwise}.
\end{cases}
\end{align}
Choosing $\epsilon > 0$ to be small enough, we get the following corollary.
\begin{corollary}\label{cor:extcheck}
	Assume that there is an $x_0 > 0$ such that for all $x\in[0,x_0]$,
	\begin{align}
	\int_{z=0}^{x_0} z\frac{\min(x,z)}{x} g_2(z) \,dz > 1,
	\end{align}
	where $g_2(z) = \sum_{k=2}^{\infty} P(k) \frac{\eexp^{-z}z^{k-2}}{(k-2)!}$, and $\min(x,z)/x$ for $x=0$ is interpreted as $1$. Then, it follows that $\prob(\{\text{extinction}\}) \allowbreak < 1$.
\end{corollary}
\begin{proof}
	Notice that for $\epsilon\in(0,1)$ we have $(1-\epsilon)^{k-1} \leq (1+(k-1)\epsilon)^{-1}$. Using this, for all $x\leq x_0$ we have
	\begin{align}
	f_{x_0,\epsilon}(x) - T(f_{x_0,\epsilon})(x) &= \frac{1}{x} \sum_{k=1}^{\infty} P(k) \int_{z=0}^{x_0} \frac{\eexp^{-z}z^{k-1}}{(k-1)!} \min(x,z)(1-(1-\epsilon)^{k-1}) \,dz - \epsilon \\
	&\geq \epsilon \left(\sum_{k=1}^{\infty}\frac{(k-1)}{1+(k-1)\epsilon}\times P(k) \int_{z=0}^{x_0} \frac{\eexp^{-z}z^{k-1}}{(k-1)!} \frac{\min(x,z)}{x} \,dz - 1\right)
	\end{align}
	We want to show that the given condition in \cref{cor:extcheck} implies that the inequality $f_{x_0,\epsilon}(x) - T(f_{x_0,\epsilon})(x) \geq 0$ holds for all $x\in\mathbb{R}_+$. It is sufficient to prove that the right-hand side of the above inequality is non-negative for all $x\leq x_0$ when $\epsilon$ is small enough. Equivalently, we want to show the following inequality holds
	\begin{align}
	\lim_{\epsilon\downarrow 0} \sum_{k=1}^{\infty}\frac{(k-1)}{1+(k-1)\epsilon}\times P(k) \int_{z=0}^{x_0} \frac{\eexp^{-z}z^{k-1}}{(k-1)!} \frac{\min(x,z)}{x} \,dz > 1
	\end{align}
	Using the monotone convergence theorem, the result follows by changing the order of summation and the limit.
\end{proof}
The assumption of the corollary is not tight, i.e., there are examples where $\prob(\{\text{extinction}\}) < 1$, but the assumption of the above corollary fails. 
Two natural follow-up questions are: $1)$ Is there a general test function $f(\cdot)$ such that $\prob(\{\text{extinction}\}) < 1$ if and only if $f \geq T(f)$? $2)$ If the answer is yes, what is the closed form of $f$?

The idea of using test functions, as simple as it seems, combined with point process perspective turns out to be a powerful tool for analyzing the branching process. We revisit this idea in Section~\ref{sec:PropExtRev}.

%% file: Sections/PropEWT/ExpNumDepthl.tex
Let $Z_l$ and $W_l$ denote the number of vertices and the number of potential vertices, respectively, at depth $l$. The expected value of $Z_l$ and $W_l$ are related to the growth rate of the EWT. These are also closely related to the probability of extinction via the following claim:
\begin{align}
\expect[Z_l] < \text{Const} \text{ for all $l$ if and only if } \prob(\{\text{extinction}\}) = 1\label{eq:classprop}.
\end{align}
The proof of the claim in \cref{eq:classprop} is based on a classical property of branching processes that $Z_n$ goes to either $0$ or $\infty$. We will revisit this property later on in Section~\ref{sec:PropExtRev}. For now, we state the following.
\begin{theorem}\label{thm:expectedz_l}
	We have
	\begin{align}
	\expect[W_l] = \MoveEqLeft \expect[n_{\root}]\times \left(\expect[n	_{\root}-1]\right)^{l-1}\allowdisplaybreaks\\
	\expect[Z_l] = \MoveEqLeft \sum_{m=1}^\infty  P(m) \sum_{k_1=2}^{\infty} P(k_1)\dots \sum_{k_{l-1}=2}^\infty P(k_{l-1})\sum_{k_l=1}^{\infty}P(k_l)\allowdisplaybreaks\\
	& \int_{y_l=0}^{\infty}\int_{y_{l-1}=0}^{\infty}\dots \int_{y_1=0}^{\infty} \bar{F}_{m}(y_1) \bar{F}_{k_1-1}(\max(y_1,y_2)) \dots \allowdisplaybreaks\\
	& \qquad\qquad\qquad\qquad\qquad\qquad \bar{F}_{k_{l-1}-1}(\max(y_{l-1},y_l)) \bar{F}_{k_l}(y_l) \,dy_1\,dy_2 \dots \,dy_l.
	\end{align}
	where, as before, $\bar{F}_{k}(\cdot)$ is the complementary cumulative distribution function of the $\erlangdist(k)$ distribution.
\end{theorem}
\begin{proof}
	The proof is presented in Appendix~\ref{appendix:EZ}.
\end{proof}
A necessary but not sufficient condition for $\prob(\{\text{extinction}\})$ to be non-zero is stated in the following corollary.
\begin{corollary}
	If the expected number of the potential neighbors of the root vertex, i.e., $\expect[n_{\root}]$, is smaller than $2$, then the population will eventually go extinct.
\end{corollary}
\begin{proof}
	If $\expect[n_{\root}] < 2$, then
	$\expect[Z_l] \leq \expect[W_l] = \expect[n_{\root}]\times \left(\expect[n_{\root}-1]\right)^{l-1} \xrightarrow{l\to\infty} 0$. Hence, by \eqref{eq:classprop} we have $\prob(\{\text{extinction}\}) = 1$.
\end{proof}
Theorem~\ref{thm:expectedz_l} does not provide an easy way to check whether $\expect[Z_l]$ goes to zero or not. There is no recursive representation for the quantities provided by the theorem either; however, using the point process perspective leads to a full characterization of the growth rate and provides a necessary and sufficient condition for the probability of extinction to be less than $1$.

%% file: Sections/PropEWT/GrowthRate.tex
To obtain the growth rate of the EWT more work needs to be done. We follow the discussion of Chapter 3 of Harris~\cite{Harris1963}. Harris analyzes general branching processes from a point process perspective. Although we use the same idea, our assumptions are different and the results from Harris's book~\cite{Harris1963} do not apply to our setting, and hence, it requires a generalization.

Abusing notation, let $Z_l(k-1,A)$ denote the number of vertices at depth $l$ of type $(k-1,z)$ where $k\in\mathbb{N}$ and $z\in A$ with $A\subset \mathbb{R}_+$ being a Borel set. Let $M_l(m,x;k-1,A)$ denote the expected value of $Z_l(k-1,A)$, conditioned on $n_{\root} = m$ and $v_{\root} = x$, i.e.,
\begin{align}
M_l(m,x;k-1,A) \coloneqq \expect[Z_l(k-1,A)\,\vert\, n_{\root} = m, v_{\root} = x].
\end{align}
Let $m_l(m,x;k-1,z)$ denote the density of $M_l(m,x;k-1,A)$ at $(k-1,z)$:
\begin{align}
M_l(m,x;k-1,A) = \int_{z\in A}m_l(m,x;k-1,z) \,dz.
\end{align}
We will show that $\beta^{-l}M_l(m,x;\mathbb{Z}_+,\mathbb{R}_+)$ converges to some fixed function, for a suitable $\beta$. Moreover, we show that $\beta^{-l}m_l(m,x;k-1,z)$ converges to $\mu(m,x)\nu(k-1,z)$. The quantity $\beta$ is the largest eigenvalue of $M_1$, and the functions $\mu(\cdot\,,\cdot)$ and $\nu(\cdot\,,\cdot)$ are the unique right and left eigenfunctions corresponding to the eigenvalue $\beta$, respectively.
\begin{definition}\label{def:leftrighteigen}
	Let $m_1$ denote the density of $M_1$. If there exists a non-zero function $\mu(\cdot\,,\cdot)$ and a $\beta\in\mathbb{R}$ such that
	\begin{align}\label{eq:righteigenfunction}
	\beta \mu(m,x) = \int_{z=0}^{\infty}\sum_{k=1}^\infty m_1(m,x;k-1,z) \mu(k-1,z) \,dz,
	\end{align}
	then $\mu(\cdot\,,\cdot)$ is called a right eigenfunction of $M_1$ corresponding to the eigenvalue $\beta$. Similarly, a left eigenfunction corresponding to the eigenvalue $\beta$ is defined as follows,
	\begin{align}\label{eq:lefteigenfunction}
	\beta \nu(k-1,z) = \int_{x=0}^{\infty}\sum_{m=0}^\infty \nu(m,x) m_1(m,x;k-1,z) \,dx.
	\end{align}
\end{definition}
The interpretations of the left and right eigenfunctions associated with an eigenvalue $\beta$ are as follows:
	\begin{enumerate}[label = --]
		\item The left eigenfunction represents the population density. If the population density of type $(m,x)$ in the current generation is a constant multiple of $\nu(m,x)$, then the population density of type $(k-1,z)$ among the descendants is a constant multiple of $\beta\nu(m,x)$.

		\item The right eigenfunction can be seen as the score function. Assuming the score of a vertex of type $(k-1,z)$ is a constant multiple of $\mu(k-1,z)$, then the expected score of the descendants of a vertex of type $(m,x)$ is a constant multiple of $\beta\mu(m,x)$.
	\end{enumerate}
	Intuitively, $\mu(\cdot,\cdot)$ captures the influence of the root vertex, while $\nu(\cdot,\cdot)$ captures the distribution of vertex types. These interpretations are sensible given that the eigenfunctions are non-negative. Consequently, a potential candidate for the asymptotic behavior of $m_l(m,x;k-1,z)$ could be $\beta^l \mu(m,x)\nu(k-1,z)$.

The main goal of this section is to prove a result analogous to the Perron--Frobenius theorem. We show that a version of Krein--Rutman Theorem by Toland~\cite{Toland1996} applies to our setting, which proves the existence of a unique eigenvalue for which the left and right eigenfunctions are positive. However, it does not provide an easy way to find the spectral radius. The specific structure of the EWT makes it possible to directly prove the convergence of $\beta^{-l}m_l(m,x;k-1,z)$ to $\mu(m,x)\nu(k-1,z)$ and to show that $\beta^{-l}M_l(m,x;\mathbb{R}_+,\mathbb{Z}_+)$ converges to some function that only depends on $x$ and $m$.

Before presenting the main theorems and their proofs, let us simplify the operator of interest,
\begin{align}
M_1(m,x;k-1,A) &= m\int_{y=0}^{x}\frac{1}{x}\int_{z\geq y, z\in A} \widehat{P}(k-1)f_k(z) \,dz\,dy \allowdisplaybreaks\\
&= \frac{m}{x} \int_{z \in A} \min(x,z) P(k) \frac{\eexp^{-z}z^{k-1}}{(k-1)!}\,dz,
\end{align}
where $f_{k}(\cdot)$ is the probability density function of $\erlangdist(k)$. Hence,
\begin{align}\label{eq:ml}
m_1(m,x;k-1,z) = \frac{m}{x} \min(x,z) P(k)\frac{\eexp^{-z}z^{k-1}}{(k-1)!}.
\end{align}
Let $\beta$ be an arbitrary eigenvalue of $M_1$. By \cref{eq:righteigenfunction} a right eigenfunction of $\beta$ then satisfies the following equation:
\begin{align}
\beta \mu(m,x) = \int_{z=0}^{\infty}\sum_{k=1}^\infty \frac{m}{x} \min(x,z) P(k)\frac{\eexp^{-z}z^{k-1}}{(k-1)!} \mu(k-1,z) \,dz.
\end{align}
Dividing both sides by $m$, the right-hand side is independent of $m$ (notice that $\mu(0,x) = 0$); hence, $\mu(m,x)$ is linear in $m$ and we can write
\begin{align}\label{eq:righteigf}
x\mu(m,x)/m \eqqcolon \widetilde{\mu}(x),
\end{align}
where $\widetilde{\mu}(\cdot)$ is a solution to the following equation
\begin{align}\label{eq:mu_tild}
\beta \widetilde{\mu}(x) &= \int_{z=0}^{\infty}g_2(z) \min(x,z) \widetilde{\mu}(z) \,dz,
\end{align}
and $g_2(x) = \eexp^{-x} \sum_{k=2}^{\infty} P(k)\frac{x^{k-2}}{(k-2)!}$. Notice that if $\widetilde{\mu}(\cdot)$ satisfies the above relation, then a right eigenfunction of $M_1$ corresponding to the eigenvalue $\beta$ is given by $\mu(x,m)\coloneqq m\widetilde{\mu}(x)/x$.
Similarly, for a left eigenfunction, we have
\begin{align}
\beta \nu(k-1,z) &= \int_{x=0}^{\infty}\sum_{m=0}^\infty m_1(m,x;k-1,z) \nu(m,x) \,dx \allowdisplaybreaks\\
&= \int_{x=0}^{\infty}\sum_{m=0}^\infty \frac{m}{x} \min(x,z) P(k)\frac{\eexp^{-z}z^{k-1}}{(k-1)!}\nu(m,x) \,dx \allowdisplaybreaks\\
&= P(k)\frac{\eexp^{-z}z^{k-1}}{(k-1)!} \int_{x=0}^{\infty}\sum_{m=0}^\infty \frac{m}{x} \min(x,z) \nu(m,x) \,dx \allowdisplaybreaks\\
&= P(k)\frac{\eexp^{-z}z^{k-1}}{(k-1)!} \int_{y=0}^z\int_{x=y}^{\infty}\sum_{m=0}^\infty \frac{m}{x} \nu(m,x) \,dx\,dy.
\end{align}
Notice that the dependence of $\nu(k-1,z)$ in $k$ is through the term $P(k)\frac{\eexp^{-z}z^{k-1}}{(k-1)!}$. Hence, we can write
\begin{align}
\nu(k-1,z) = \widetilde{\nu}(z)P(k)\frac{\eexp^{-z}z^{k-1}}{(k-1)!} \label{eq:lefteigf},
\end{align}
for a suitable $\widetilde{\nu}(\cdot)$ that is a solution to the following equation,
\begin{align}
\beta \widetilde{\nu}(z) &= \int_{y=0}^z\int_{x=y}^{\infty}\sum_{m=0}^\infty \frac{m}{x} P(m+1)\frac{\eexp^{-x}x^{m}}{m!} \widetilde{\nu}(x) \,dx\,dy \allowdisplaybreaks\\
&= \int_{y=0}^z\int_{x=y}^{\infty}\sum_{m=2}^\infty P(m)\frac{\eexp^{-x}x^{m-2}}{(m-2)!} \widetilde{\nu}(x) \,dx\,dy \allowdisplaybreaks\\
&= \int_{y=0}^z\int_{x=y}^{\infty}g_2(x) \widetilde{\nu}(x) \,dx\,dy\allowdisplaybreaks\\
&= \int_{x=0}^\infty \min(x,z)g_2(x) \widetilde{\nu}(x) \,dx.
\end{align}
Observe that $\widetilde{\nu}(\cdot)$ satisfies the same equation as $\widetilde{\mu}(\cdot)$ does. To study this equation, we define a new operator and rely on the background materials discussed in Section~\ref{sec:back_opt}.

Let $\mathcal{H} = L^2(\mathbb{R}_{+},\upsilon)$ denote the set of real-valued square integrable functions with respect to a measure $\upsilon$. It is easy to prove that $L^2(\mathbb{R}_{+},\upsilon)$ together with the inner product $\langle f,g\rangle = \int_0^\infty f(x)g(x) d\upsilon(x)$ is a real Hilbert space. Let $H_1\in\mathcal{L}\left(\mathcal{H},\mathcal{H}\right)$ be an integral operator with integrand $\min(\cdot\,,\cdot)\in L^2(\mathbb{R_+}\times\mathbb{R_+},\upsilon\otimes\upsilon)$, i.e.,
\begin{align}
H_1f(x) = \int_{0}^{\infty} \min(x,y) f(y) d\upsilon(y),
\end{align}
where $\upsilon(\cdot)$ is a finite measure with Radon--Nikodym derivative $g_2(\cdot)$ with respect to Lebesgue measure. The integral operator $H_1$ is self-adjoint since its integrand is symmetric. Moreover, $H_1$ is compact since $\mathcal{H}$ is separable (the proof follows by the fact that $\mathcal{H}$ has a countable orthonormal basis). With these facts in hand, $H_1$ is a compact self-adjoint operator.

Let $\mathcal{K}$ denote the set of all non-negative functions in $\mathcal{H}$. The set $\mathcal{K}$ is closed and convex. Moreover, for all $\lambda \in\mathbb{R}_+$, we have $\lambda \mathcal{K}\subset \mathcal{K}$ and $\mathcal{K}\cap (-\mathcal{K}) = \{\boldsymbol{0}\}$; hence, $\mathcal{K}$ is a cone. Actually, it is a total cone, i.e., $\mathcal{H} = \mathcal{K}-\mathcal{K}$. The following theorem is a direct implication of Theorems~\ref{thm:optth1}--\ref{thm:optth3}.
\begin{theorem}\label{thm:opth}
	The largest eigenvalue of $H_1$ in magnitude is,
	\begin{align}
	\mathscr{X}(H_1) = \max_{\substack{f(\cdot)\in\mathcal{H}, \norm{f}_{\mathcal{H}}=1,\\\text{$f(\cdot)$ is non-negative}}} \int_{0}^\infty\int_{0}^{\infty} \min(x,y)f(x)f(y) d\upsilon(y)d\upsilon(x).
	\end{align}
	$\mathscr{X}(H_1) > 0$ is a simple eigenvalue and corresponds to a non-negative eigenfunction. Moreover, all the eigenvalues of $H_1$ are real, and if $\zeta(\cdot)$ is an eigenfunction of $H_1$ with some eigenvalue $\mu\neq \mathscr{X}(H_1)$, we have
	\begin{align}
	\int_{0}^{\infty} f(y)\zeta(y) d\nu(y) = 0.
	\end{align}
\end{theorem}
Having established the existence of the Krein--Rutman eigenvalue of $H_1$ and the corresponding eigenfunction, the following simplification will help in finding them. Changing the order of integration, the operator $H_1$ can be written as follows,
\begin{align}\label{eq:H_1falt}
H_1f(x) = \int_{y=0}^{x}\int_{z=y}^{\infty} f(z) d\upsilon(z)\,dy.
\end{align}
Then, define the operator $\widetilde{H_1}$ as follows,
\begin{align}
	\widetilde{H_1}f(x) = \int_{y=x}^{\infty}\int_{z=y}^{\infty} f(z) d\upsilon(z)\,dy.
\end{align}
Using \cref{eq:H_1falt}, we have
\begin{align}
H_1f(x) + \widetilde{H_1}f(x) &= \int_{0}^{\infty}\int_{y}^{\infty} f(z) d\upsilon(z)\,dy\allowdisplaybreaks\\
&= \int_{0}^{\infty}z f(z) d\upsilon(z) = \langle f,\mathbbm{I}\rangle_{\mathcal{H}},
\end{align}
where $\mathbbm{I}(\cdot)$ is the identity function, i.e., $\forall x\in\mathbb{R}_+,\,\mathbbm{I}(x) = x$. The Krein--Rutman eigenvalue of $H_1$ and the corresponding eigenfunction are related to the operator $\widetilde{H_1}$.
\begin{theorem} \label{thm:propL}
	Consider the function $L(\beta,x)$ for $x\in \mathbb{R}_+$ and $\beta\in \mathbb{C}$ defined as follows,
	\begin{align}
	L(\beta,x) &\coloneqq \sum_{i=0}^{\infty} G_i(x) \left(\frac{-1}{\beta} \right)^i,
	\end{align}
	where the function $G_i(x)$ is defined recursively via
	\begin{align}
	G_0(x) &\coloneqq 1, \allowdisplaybreaks\\
	G_i(x) &\coloneqq \int_{y=x}^{\infty}\int_{z=y}^{\infty}g_2(z) G_{i-1}(z) \,dz \,dy = \widetilde{H}_1 G_{i-1}(x) \qquad\forall i>0,
	\end{align}
	where $$g_2(x) \coloneqq \eexp^{-x} \sum_{k=2}^{\infty} P(k)\frac{x^{k-2}}{(k-2)!}.$$
	Assuming the moment generating function of $n_{\root}$ exists at some $\theta> 0$, the function $L(\beta,x)$ satisfies the following properties,
	\begin{enumerate}[label=(\roman*)]
		\item For all $\beta\in\mathbb{C}$ and $x\in\mathbb{R}_+$, the function $L(\beta,x)$ is well-defined, that is, the series converges absolutely. \label{part:propL_i}
		\item The second partial derivative of $L(\beta,x)$ with respect to $x$, satisfies the following equality, \label{part:propL_ii}
		\begin{align}
		\beta \pdv[2]{L(\beta,x)}{x} = - g_2(x) L(\beta,x).
		\end{align}
		\item For every fixed $x\in \mathbb{R}_+$, all the zeros of the function $L(\beta,x)$ are real-valued. \label{part:propL_iii}
		\item There exists a real value $\beta_0 \in \left(\max\limits_x\left(x\upsilon([x,\infty))\right), \expect[n_{\root}]-1\right)$ such that for every real $\beta > \beta_0$, the function $L(\beta,x)$ is uniformly positive, i.e., $\exists\ \varepsilon_{\beta}>0$ such that $\forall x\in\mathbb{R}_+$, we have $L(\beta,x) > \varepsilon_{\beta}$. Moreover, for all $x\in\mathbb{R}_+$ the function $L(\beta_0,x)$ is non-negative and $L(\beta_0,0) = 0$. Finally, for all $\beta \geq \beta_0$ and all $x_0 \in\mathbb{R}_+$,
		\begin{align}
		\left.\pdv{L(\beta,x)}{x}\right|_{x = x_0} > 0.
		\end{align}
		\label{part:propL_iv}
		\item For all $x\in\mathbb{R}_+$, the function {\Large$\frac{x}{L(\beta_0,x)}$} is well-defined, is strictly positive, and is strictly increasing.\label{part:propL_v}
	\end{enumerate}
\end{theorem}
\begin{proof} In the course of the proof, it will become apparent that $L(\beta,x)$ and the Bessel function of the first kind of zeroth order $J_0(x)$ share similar properties.
	\begin{enumerate}[label=(\roman*)]
		\item Using the Chernoff bound,
		\begin{align}
		\prob(n_{\root}\geq k) \leq \frac{\expect[\eexp^{\theta n_{\root}}]}{\eexp^{\theta k}} < \infty.
		\end{align}
		We then have
		\begin{align}
		g_i(z) &\coloneqq \sum_{k=i}^{\infty} P(k) \frac{\eexp^{-z}z^{k-i}}{(k-i)!}\allowdisplaybreaks\\
		&\leq \sum_{k=i}^{\infty} \prob(n_{\root}\geq k) \frac{\eexp^{-z}z^{k-i}}{(k-i)!} \allowdisplaybreaks\\
		&\leq \sum_{k=i}^{\infty} \frac{\expect[\eexp^{\theta n_{\root}}]}{\eexp^{\theta k}} \frac{\eexp^{-z}z^{k-i}}{(k-i)!}\allowdisplaybreaks\\
		&= \frac{\expect[\eexp^{\theta n_{\root}}]}{\eexp^{\theta i}} \sum_{k=i}^{\infty} \frac{\eexp^{-z}z^{k-i}}{\eexp^{\theta (k-i)}(k-i)!}\allowdisplaybreaks\\
		&= \frac{\expect[\eexp^{\theta n_{\root}}]}{\eexp^{\theta i}} \exp(-z(1-\eexp^{-\theta })) \numberthis \label{lem:ineqgi}.
		\end{align}
		Let $C = \expect[\eexp^{\theta n_{\root}}]/\eexp^{2\theta }$ and $\Upsilon = 1-\eexp^{-\theta }$. It is easy to prove that $G_i(x)$ is upper bounded by $C^{i}\,\eexp^{-i\Upsilon x}/(\Upsilon^{2i}i!i!)$, by induction. Indeed,
		\begin{align}
		G_0(x) &= 1\allowdisplaybreaks\\
		G_{i+1}(x) &= \int_{y=x}^{\infty}\int_{z=y}^{\infty} G_{i}(z) g_2(z)\,dz\,dy\allowdisplaybreaks\\
		&\leq \int_{y=x}^{\infty}\int_{z=y}^{\infty} \frac{C^{i}}{\Upsilon^{2i}} \frac{\eexp^{-i\Upsilon z}}{i!i!} \times C \eexp^{-z\Upsilon} \,dz\,dy\allowdisplaybreaks\\
		&=\frac{C^{i+1}}{\Upsilon^{2(i+1)}} \frac{\eexp^{-(i+1)\Upsilon x}}{(i+1)!(i+1)!},
		\end{align}
		which proves the upper bound by induction. Collectively we then have
		\begin{align}
		\sum_{i=0}^{\infty} G_i(x) \left(\frac{1}{|\beta|} \right)^i \leq \sum_{i=0}^{\infty} \frac{C^{i}}{\Upsilon^{2i}} \frac{\eexp^{-i\Upsilon x}}{i!i!} \left(\frac{1}{|\beta|} \right)^i &= J_0\left(\sqrt{\frac{-4C\eexp^{-\Upsilon x}}{\Upsilon^2 |\beta|}}\right) \allowdisplaybreaks\\
		&= I_0\left(\sqrt{\frac{4C\eexp^{-\Upsilon x}}{\Upsilon^2 |\beta|}}\right)\in (0,\infty),
		\end{align}
		where $J_0(\cdot)$ is the Bessel function of the first kind of order $0$ (which is defined on the complex plane using its power series expansion), and $I_0(\cdot)$ is the modified Bessel function of the first kind of order $0$. This establishes that the series converges absolutely.
		\item Using the definition of $L(\beta,x)$ and part~$\ref{part:propL_i}$,
		\begin{align}
		\beta \pdv[2]{L(\beta,x)}{x} &= \beta \sum_{i=0}^{\infty} \dv[2]{G_i(x)}{x} \left(\frac{-1}{\beta} \right)^i\allowdisplaybreaks\\
		&= \sum_{i=1}^{\infty} -G_{i-1}(x)g_2(x) \left(\frac{-1}{\beta} \right)^{i-1}\allowdisplaybreaks\\
		&= -g_2(x)L(\beta,x).
		\end{align}
		\item Fix some $x\in\mathbb{R}_+$. Consider the function $H_1(\beta,x)$ defined as follows,
		\begin{align}
		H_1(\beta,x) \coloneqq L(\beta,x) \pdv{L(\bar{\beta},x)}{x} - L(\bar{\beta},x) \pdv{L(\beta,x)}{x},
		\end{align}
		where $\bar{\beta}$ is the complex conjugate of $\beta$. The partial derivative of $H_1(\beta,x)$ with respect to $x$, using part~\ref{part:propL_ii}, is given as follows,
		\begin{align}
		\pdv{H_1(\beta,x)}{x} &= \pdv{L(\beta,x)}{x}\pdv{L(\bar{\beta},x)}{x} - {\bar{\beta}}^{-1} L(\beta,x) L(\bar{\beta},x) g_2(x) \allowdisplaybreaks\\
		&\qquad -\pdv{L(\bar{\beta},x)}{x}\pdv{L(\beta,x)}{x} + {\beta}^{-1} L(\bar{\beta},x) L(\beta,x) g_2(x)\allowdisplaybreaks\\
		&=({\beta}^{-1}-{\bar{\beta}}^{-1})|L(\beta,x)|^2 g_2(x),
		\end{align}
		where the last equality is obtained by the fact that $\overline{L(\beta,x)} = L(\bar{\beta},x)$. Notice that,
		\begin{align}
		&\lim_{x\to\infty} L(\beta,x) = 1 \text{ and }\lim_{x\to\infty}\pdv{L(\beta,x)}{x} = 0,
		\end{align}
		since $\lim_{x\to\infty} G_i(x) = \lim_{x\to\infty} \dv{G_i(x)}{x} = 0$ for all $i>0$ and $L(\beta,x)$ is absolutely summable. Hence, $\lim_{x\to\infty} H_1(\beta,x) = 0$. Therefore,
		\begin{align}
		\int_{x}^{\infty}({\beta}^{-1}-{\bar{\beta}}^{-1})|L(\beta,y)|^2 g_2(y)\,dy =- H_1(\beta,x).
		\end{align}
		Since for every fixed $x\in\mathbb{R}_+$ the coefficients of $L(\beta,x)$ are real-valued, $L(\beta,x) = 0$ implies $L(\bar{\beta},x) = 0$. Moreover, if $L(\beta,x) = 0$ for some $x\in\mathbb{R}_+$ and $\beta\in\mathbb{C}$, then $H_1(\beta,x) = 0$; hence,
		\begin{align}
		\int_{x}^{\infty}({\beta}^{-1}-{\bar{\beta}}^{-1})|L(\beta,y)|^2 g_2(y)\,dy = 0,
		\end{align}
		from which we conclude that $\beta = \bar{\beta}$, i.e., $\beta \in \mathbb{R}$.
		\item Pick any real-valued $\beta \geq \expect[n_{\root}] -1$. For all $i\geq 1$, we have,
		\begin{align}
		&G_i(x) \left(\frac{1}{\beta}\right)^i - G_{i+1}(x) \left(\frac{1}{\beta}\right)^{i+1} \allowdisplaybreaks\\
		&\qquad=\int_{y=x}^{\infty}\int_{z=y}^{\infty} g_2(z)\frac{1}{\beta} \left(G_{i-1}(z) \left(\frac{1}{\beta}\right)^{i-1} - G_{i}(z) \left(\frac{1}{\beta}\right)^{i}\right)\,dz\,dy,
		\end{align}
		and for $i=0$,
		\begin{align}
		1 - G_1(x)\frac{1}{\beta} &= 1 - \int_{y=x}^{\infty}\int_{z=y}^{\infty} g_2(z)\frac{1}{\beta}\,dz\,dy.
		\end{align}
		For each $i\in\mathbb{N}$, the function $G_i(x)$ is decreasing; hence, the function $1 - \beta^{-1} G_1(x)$ is increasing and it achieves its minimum at $x=0$, so
		\begin{align}
		1 - G_1(0)\frac{1}{\beta} &= 1 - \int_{y=0}^{\infty}\int_{z=y}^{\infty} g_2(z)\frac{1}{\beta}\,dz\,dy \allowdisplaybreaks\\
		&= 1 - \int_{z=0}^{\infty} z g_2(z)\frac{1}{\beta}\,dz \allowdisplaybreaks\\
		&= 1 - \int_{z=0}^{\infty} z \sum_{k=2}^{\infty} P(k) \frac{\eexp^{-z}z^{k-2}}{(k-2)!}\frac{1}{\beta}\,dz \allowdisplaybreaks\\
		& = 1 - \frac{1}{\beta} \sum_{k=2}^{\infty} (k-1) P(k) \int_{z=0}^{\infty} \frac{\eexp^{-z}z^{k-1}}{(k-1)!}\,dz \allowdisplaybreaks\\
		&= 1- \frac{\expect[n_{\root}]-1}{\beta} \geq 0.
		\end{align}
		By induction, for all $x\in\mathbb{R}_+$ and all $i\in\mathbb{N}$,
		\begin{align}
		G_i(x) \left(\frac{1}{\beta}\right)^i - G_{i+1}(x) \left(\frac{1}{\beta}\right)^{i+1}>0.
		\end{align}
		Hence, for every real $\beta \geq \expect[n_{\root}] -1$, by rewriting $L(\beta,x)$, we get
		\begin{align}
		L(\beta,x) = \sum_{i=0}^{\infty} \left(G_{2i}(x) \left(\frac{1}{\beta}\right)^{2i} - G_{2i+1}(x) \left(\frac{1}{\beta}\right)^{2i+1} \right) > 0.
		\end{align}
		Moreover, if for a fixed real $\beta > 0$ and for all $x\in\mathbb{R}_+$ the function $L(\beta,x)$ is non-negative, then the function $L(\beta,x)$ is strictly increasing:
		\begin{align}
		\pdv{L(\beta,x)}{x} &= \sum_{i=1}^{\infty} \dv{G_i(x)}{x} \left(\frac{-1}{\beta}\right)^{i} \allowdisplaybreaks\\
		&=\sum_{i=1}^{\infty} \int_{y=x}^{\infty}-g_2(y)G_{i-1}(y)\,dy \left(\frac{-1}{\beta}\right)^{i} \allowdisplaybreaks\\
		&=\frac{1}{\beta}\int_{y=x}^{\infty} g_2(y)\sum_{i=1}^{\infty}G_{i-1}(y) \left(\frac{-1}{\beta}\right)^{i-1}\,dy \allowdisplaybreaks\\
		&=\frac{1}{\beta}\int_{y=x}^{\infty} g_2(y)L(\beta,y)\,dy > 0. \numberthis\label{eq:f0prime}
		\end{align}
		Next, we prove that for some $\beta\in \mathbb{R}_+$ and $x\in\mathbb{R}_+$, the function $L(\beta,x)$ is negative. Let us rewrite the function $L(\beta,x)$,
		\begin{align}
			L(\beta,x) &= 1 + \left(\frac{-1}{\beta}\right)\sum_{i=1}^{\infty} G_{i}(x) \left(\frac{-1}{\beta}\right)^{i-1} \allowdisplaybreaks\\
			&= 1 - \frac{1}{\beta} \int_{y=x}^{\infty}\int_{z=y}^{\infty} L(\beta,z) g_2(z) \,dz \,dy\numberthis\label{eq:Lequals}
		\end{align}
		where the last equality is based on the recursive relation between $G_i(x)$ and $G_{i-1}(x)$. Using the above equality we have,
		\begin{align}
			L(\beta,x) - L(\beta,0) &= \frac{1}{\beta} \left(\int_{y=0}^{\infty}\int_{z=y}^{\infty} L(\beta,z) g_2(z) \,dz \,dy - \int_{y=x}^{\infty}\int_{z=y}^{\infty} L(\beta,z) g_2(z) \,dz \,dy\right) \allowdisplaybreaks\\
			&= \frac{1}{\beta} \int_{y=0}^{x}\int_{z=y}^{\infty} L(\beta,z) g_2(z) \,dz \,dy \allowdisplaybreaks\\
			&= \frac{1}{\beta}\int_{0}^{\infty} \min(x,z) L(\beta,z) g_2(z) \,dz\allowdisplaybreaks\\
			&= \frac{1}{\beta}\int_{0}^{x} z L(\beta,z) g_2(z) \,dz + \frac{1}{\beta}\int_{x}^{\infty} x L(\beta,z) g_2(z) \,dz, \allowdisplaybreaks \numberthis \label{eq:Lpositive}
		\end{align}
		where the third equality follows by changing the order of integration.

		Suppose that for all $\beta\in\mathbb{R}_+$ and all $x\in\mathbb{R}_+$, the function $L(\beta,x)$ is non-negative. Hence, for any fixed $\beta\in\mathbb{R}_+$, the function $L(\beta,x)$ is strictly increasing and,
		\begin{align}
		\forall x\in\mathbb{R}_+,\qquad -L(\beta,0) -\frac{1}{\beta}\int_{0}^{x} z L(\beta,z) g_2(z) \,dz \geq L(\beta,x) \left( \frac{x}{\beta} \int_{y=x}^{\infty} g_2(y) \,dy - 1\right).
		\end{align}
		However, the left-hand side of the above equation is negative for all $\beta\in\mathbb{R}_+$ and the right-hand side, for small enough $\beta$, is positive, which is a contradiction. The above argument shows that if there exist some $\hat{x}>0$ such that $\beta \leq \hat{x} \upsilon([\hat{x},\infty])$, then the function $L(\beta,\cdot)$ takes negative values.
		Moreover, for every $\beta \geq \expect[n_{\root}] -1$ the function $L(\beta,x)$ is strictly positive. Combining these together and considering the fact that $L(\beta,x)$ is a continuous function of $x\in\mathbb{R}_+$ and $\beta\in\mathbb{R}_+$, we conclude that there exists a largest $\beta_0 > 0$ such that the function $L(\beta_0,x)$ is non-negative, and $L(\beta_0,x_0) = 0$ for some $x_0\in\mathbb{R}_+$. The already established strictly increasing property of $L(\beta_0,x)$ implies that $x_0 = 0$, and the proof is complete.
		\item Using~\cref{eq:Lpositive}, for all $x>0$, we have $L(\beta_0,x) > 0$. Moreover, using the L'Hospital rule,
		\begin{align}
		\lim_{z\to 0} \frac{z}{L(\beta_0,z)} 
		= \frac{1}{\pdv{L(\beta_0,x)}{x}\big|_{x=0}},
		\end{align}
		which is well-defined since $\pdv{L(\beta_0,x)}{x}\big|_{x=0}$ is strictly positive. Next, taking the derivative of {\Large$\frac{x}{L(\beta_0,x)}$}, we get
		\begin{align}\label{eq:derivativeOfx/L}
		\pdv{\left(x/L(\beta_0,x)\right)}{x} = \frac{L(\beta_0,x) - x\pdv{L(\beta_0,x)}{x}}{\left(L(\beta_0,x)\right)^2},
		\end{align}
		Notice that $L(\beta_0,0) = 0$ and $L(\beta_0,x)$ is a strictly concave function due to parts~\ref{part:propL_ii} and~\ref{part:propL_iv}; therefore,
		\begin{align}
		\forall x>0,\qquad L(\beta_0,0) < L(\beta_0,x) + \pdv{L(\beta_0,x)}{x} (0-x).
		\end{align}
		Hence the expression \cref{eq:derivativeOfx/L} is strictly positive for every $x > 0$, and we have established that the function {\Large$\frac{x}{L(\beta_0,x)}$} is strictly increasing.
		\end{enumerate}
\end{proof}
The following immediate corollary guarantees the existence of an eigenfunction $f(\cdot)$ and an eigenvalue $\beta$ of the operator $H_1$.
\begin{corollary} \label{cor:existf&beta}
	Let $\beta_0$ be the largest zero of $L(\cdot\,,0)$. For all $x\in\mathbb{R}_+$, let $f_0(x) = L(\beta_0,x)$. Then the constant $\beta_0$ and the function $f_0(\cdot)$ satisfy the following fixed point equation,
	\begin{align}\label{eq:fixedpoint}
	\beta_0 f_0(x) &= \int_{y=0}^x\int_{z=y}^{\infty}g_2(z) f_0(z) \,dz \,dy\allowdisplaybreaks .
	\end{align}
\end{corollary}
\begin{proof}
	 Substituting the function $L(\beta_0,x)$ in the above equation, we get,
	 \begin{align}
	 &\int_{y=0}^{x}\int_{z=y}^{\infty}g_2(z)L(\beta_0,z)\,dz\,dy \allowdisplaybreaks\\
	 &\qquad= \sum_{i=0}^{\infty}\int_{y=0}^{x}\int_{z=y}^{\infty}g_2(z)G_i(z)\,dz\,dy \left(\frac{-1}{\beta_0}\right)^i \allowdisplaybreaks\\
	 &\qquad= \sum_{i=0}^{\infty}\left(G_{i+1}(0)-G_{i+1}(x)\right)\left(\frac{-1}{\beta_0}\right)^i\allowdisplaybreaks\\
	 &\qquad= -\beta_0(L(\beta_0,0) - L(\beta_0,x)) = \beta_0 L(\beta_0,x),
	 \end{align}
	 where the last equality follows from part~\ref{part:propL_iv} of Theorem~\ref{thm:propL}, since $L(\beta_0,0) = 0$.
\end{proof}
Using the Corollary~\ref{cor:existf&beta} and the equations \cref{eq:righteigf} and \cref{eq:lefteigf}, a left and a right eigenfunction of $M_1$ for the eigenvalue $\beta_0$ are obtained.

Observe that, from \cref{eq:ml}, $m_l(m,x;k-1,z)$ satisfies the following recursive equation:
\begin{align}
m_l(m,x;k-1,z) &= \int_{z^\prime = 0}^{\infty} \sum_{k^\prime=1}^{\infty} m_{l-1}(m,x;k^\prime-1,z^\prime)m_1(k^\prime-1,z^\prime;k-1,z)\,dz^\prime \allowdisplaybreaks\\
&= \int_{z^\prime = 0}^{\infty} \sum_{k^\prime=2}^{\infty} m_{l-1}(m,x;k^\prime-1,z^\prime)\frac{k^\prime - 1}{z^\prime} \min(z^\prime,z) P(k)\frac{\eexp^{-z}z^{k-1}}{(k-1)!}\,dz^\prime \allowdisplaybreaks\\
&= P(k)\frac{\eexp^{-z}z^{k-1}}{(k-1)!}\int_{z^\prime = 0}^{\infty} \sum_{k^\prime=2}^{\infty} m_{l-1}(m,x;k^\prime-1,z^\prime)\frac{k^\prime - 1}{z^\prime} \min(z^\prime,z) \,dz^\prime.
\end{align}
The terms related to the values of $k$ and $m$ can be factored out. However, to avoid dividing by zero, we consider the function $h_l(\cdot\,,\cdot)$ defined recursively as follows:
\begin{equation}
\begin{aligned}
h_l(x,z) &= \int_{z^\prime = 0}^{\infty} h_{l-1}(x,z^\prime) g_2(z^\prime)h_1(z^\prime,z) \,dz^\prime \quad l \geq 2,\allowdisplaybreaks\\
h_1(x,z) &= \min(x,z).
\end{aligned}\label{eq:hl}
\end{equation}
It is easy to see that the function $m_l$ is related to the function $h_l$ via the following equation; indeed, the relation holds between $m_1$ and $h_1$, which is just \cref{eq:ml}, and for a general $l$ the proof holds via induction:
\begin{align}\label{eq:mlhl}
m_l(m,x;k-1,z) = h_l(x,z) \times \frac{P(k)\eexp^{-z}z^{k-1}}{(k-1)!} \frac{m}{x}.
\end{align}
Recall that the kernel of the operator $H_1$ is symmetric, hence, any right eigenfunction is also a left eigenfunction. Moreover, Corollary~\ref{cor:existf&beta} implies that $f_0(\cdot)$ is an eigenfunction of $H_1$ with eigenvalue $\beta_0$, i.e.,
\begin{align}
\beta_0 f_0(x) &= \int_{z=0}^{\infty} \min(x,y) f_0(y) d\upsilon(y)\label{eq:f0rel}.
\end{align}
Hence, the question of whether or not $\beta_0$ is the Krein--Rutman eigenvalue of $M_1$ with right eigenfunctions $\mu(\cdot\,,\cdot)$ and left eigenfunction $\nu(\cdot\,,\cdot)$, is equivalent to the same question for $H_1$ with right and left eigenfunction $f_0(\cdot)$.

To show that $\beta_0$ is the Krein--Rutman eigenvalue of $H_1$, we define a continuous state Markov chain and prove uniform geometric ergodicity for the chain. Consider a continuous state Markov chain with the following transition probability kernel:
\begin{align}\label{eq:defmchain}
\forall x,y\in\mathbb{R}_+, \qquad p(x,y) \coloneqq \frac{h_1(x,y)g_2(y)f_0(y)}{\beta_0f_0(x)},
\end{align}
where the transition probability at $x=0$ is defined by taking the limit of $p(x,\cdot)$ as $x$ goes to $0$, namely,
\begin{align}
p(0,y) \coloneqq \lim_{x\to 0} \frac{h_1(x,y)g_2(y)f_0(y)}{\beta_0f_0(x)} = \frac{g_2(y)f_0(y)}{\beta_0f_0^\prime(0)}.
\end{align}
By Theorem~\ref{thm:propL} part~\ref{part:propL_iv}, the term $f_0^\prime(0)$ is strictly positive; hence, the function $p(\cdot\,,\cdot)$ is well-defined. Moreover, the function $p(\cdot\,,\cdot)$ is indeed a valid transition probability kernel since
\begin{align}
\int_{z=0}^{\infty}p(x,z)\,dz &= \int_{z=0}^{\infty}\frac{\min(x,z)g_2(z)f_0(z)}{\beta_0f_0(x)}\,dz = \int_{y=0}^{x}\int_{z=y}^{\infty}\frac{g_2(z)f_0(z)}{\beta_0f_0(x)}\,dz =1\allowdisplaybreaks\\
\int_{z=0}^{\infty}p(0,z)\,dz &= \int_{z=0}^{\infty}\frac{g_2(z)f_0(z)}{\beta_0f_0^\prime(0)}\,dz \stackrel{(*)}{=} \frac{1}{\beta_0f'_0(0)} \left.\pdv{\left(\beta_0 f_0(x)\right)}{x} \right|_{x=0} = 1,
\end{align}
where $(*)$ follows from \cref{eq:f0prime}. By induction, it is easy to observe from \cref{eq:hl} that the $l$ step transition probability kernel is related to the function $h_l(\cdot\,,\cdot)$ via the following equation,
\begin{align}\label{eq:pl}
p^{(l)}(x,y) = \int_{z=0}^{\infty} p^{(l-1)}(x,z)p(z,y)\,dz = \frac{h_l(x,y)g_2(y)f_0(y)}{{\beta_0}^l f_0(x)}.
\end{align}
The stationary density of the Markov chain can now be verified to be $\pi(y) = C_N g_2(y)\left(f_0(y)\right)^2$, where $C_N$ is the normalization factor. Indeed, from \cref{eq:f0rel} and \cref{eq:defmchain}, we have
\begin{align}
\int_{x=0}^{\infty} \pi(x) p(x,y)\,dx = C_N\int_{x=0}^{\infty} g_2(x)f_0(x) \frac{\min(x,y)g_2(y)f_0(y)}{\beta_0}dx = \pi(y).
\end{align}
Observe that the stationary distribution equals the product of the left and the right eigenfunctions of $H_1$ upto a normalization factor. Recall that $g_2(\cdot)$ is the Radon--Nikodym derivative of $\upsilon(\cdot)$. Moreover, the Markov chain is reversible with respect to the stationary distribution $\pi(\cdot)$, i.e., $\pi(x)p(x,y) = \pi(y)p(y,x)$.

It is natural to expect $p^{(l)}(x,y)$ to converge point-wise to $\pi(y)$ as $l$ goes to infinity. To prove this, we invoke the following result by Baxendale~\cite{Baxendale2005}.
\begin{theorem}[Baxendale 2005] \label{thm:bax}
	Let $\{X_n: n>0\}$ be a time homogeneous Markov chain on a state space $(\mathcal{S},\mathscr{S})$. For $x\in S$ and $A\in\mathscr{S}$, let $P(x,A)$ denote the transition probability and by abusing notation let $P$ denote the corresponding operator on measurable functions $\mathcal{S}\to\mathbb{R}$. Assume that the following assumptions hold:
	\begin{enumerate}[label = (A\arabic*)]
		\item {\it Minorization condition}: There exists $C\in\mathscr{S}$, $\widetilde{\beta} > 0$ and a probability measure $\nu$ on $(\mathcal{S},\mathscr{S})$ such that for all $x\in C$ and $A\in\mathscr{S}$,
		\begin{align}
		P(x,A) \geq \widetilde{\beta}\nu(A).
		\end{align}\label{assum:Baxendale2005_A1}
		\item {\it Drift condition}: There exist a measurable function $V:\mathcal{S}\to[1,\infty)$ and constants $\lambda < 1$ and $K< \infty$ satisfying,
		\begin{align}
		PV(x)\leq
		\begin{cases}
		\lambda V(x),&\text{if } x\notin C\allowdisplaybreaks\\
		K,&\text{if } x\in C
		\end{cases}.
		\end{align}\label{assum:Baxendale2005_A2}
		\item {\it Strong aperiodicity condition}: There exists $\widehat{\beta} > 0$ such that $\widetilde{\beta}\nu(C)\geq \widehat{\beta}$.\label{assum:Baxendale2005_A3}
	\end{enumerate}
	Then $\{X_n: n>0\}$ has a unique stationary probability measure $\pi$, say, and $\int V\,d\pi < \infty$. Moreover, there exists $\rho < 1$ depending only (and explicitly) on $\widehat{\beta}$, $\widetilde{\beta}$, $\lambda$ and $K$ such that whenever $\rho< \gamma < 1$ there exists $M<\infty$ depending only (and explicitly) on $\gamma$, $\widehat{\beta}$, $\widetilde{\beta}$, $\lambda$ and $K$ such that for all $x\in\mathcal{S}$ and $n\in\mathbb{Z}_+$,
	\begin{align}
	\sup_{|g|\leq V}\left|(P^n g)(x)-\int g \,d\pi\right| \leq M V(x) \gamma^n,
	\end{align}
	where the supremum is taken over all measurable functions $g:\mathcal{S}\to\mathbb{R}$ satisfying $|g(x)| \leq V(x)$. In particular, $P^ng(x)$ and $\int g\,d\pi$ are both well-defined whenever
	\begin{align}
	\norm{g}_V \equiv \sup\{|g(x)|/V(x):x\in \mathcal{S}\} < \infty.
	\end{align}
\end{theorem}
Baxendale~\cite{Baxendale2005} provides explicit values for $\rho$ and $M$ and improves the constants if the corresponding Markov chain is reversible, which holds in our case. In the following lemma, we prove that the Markov chain with transition probability $p(x,y)$ from \cref{eq:defmchain} satisfies assumptions \ref{assum:Baxendale2005_A1}--\ref{assum:Baxendale2005_A3}.
\begin{lemma} \label{lem:bax}
	Assume the moment generating function of $n_{\root}$ exists at some $\theta > 0$. Then, the Markov chain defined by the transition probability kernel $p(x,y)$ on state space $(\mathbb{R}_+,\mathscr{B})$ satisfies the assumptions \ref{assum:Baxendale2005_A1}--\ref{assum:Baxendale2005_A3} of \cref{thm:bax} where the set $C$, the constants $\widetilde{\beta}$, $\lambda$, $K$, $\widehat{\beta}$, the function $V:\mathbb{R}\to[1,\infty)$ and the probability measure $\nu(x)$ are given as follows:
	\begin{gather}
	\begin{aligned}
	C\coloneqq [0,c], && \widetilde{\beta} \coloneqq \int_{0}^{\infty} W(y) \,dy,&&\lambda \coloneqq \frac{1}{2},
	\end{aligned}\allowdisplaybreaks\\
	\begin{aligned}
	K \coloneqq f_0^\prime(0)\frac{c}{f_0(c)} \frac{\expect[\eexp^{\theta n_{\root}}]}{\beta_0 \eexp^{2\theta }} \frac{1}{(1-\eexp^{-\theta } - \eta)^2}, &&\!\! \widehat{\beta} \coloneqq \min\left(\frac{1}{\beta_0 f^\prime (0)},\frac{c/2}{\beta_0 f_0(c)}\right)\!\! \int_{c/2}^{c} g_2(y)f_0(y)\,dy,
	\end{aligned}\allowdisplaybreaks\\
	V(x) \coloneqq f_0^\prime(0) \eexp^{\eta x} \frac{x}{f_0(x)}, \numberthis \label{eq:V}\allowdisplaybreaks\\
	\nu(A) \coloneqq \frac{1}{\widetilde{\beta}} \int_A W(y) \,dy,
	\end{gather}
	where the constants $\eta$ and $c$, and the function $W(y)$ are defined as follows,
	\begin{gather}
	\begin{aligned}
	\eta \coloneqq \frac{1 - \eexp^{-\theta}}{2},&&c \coloneqq \max\left(\frac{1}{\eta}\ln\left(\frac{\expect[\eexp^{\theta n_{\root}}]}{\beta_0 \eexp^{2\theta}} \frac{2}{(1-\eexp^{-\theta} - \eta)^2}\right),1\right),
	\end{aligned}
	\allowdisplaybreaks\\
	W(y) \coloneqq
	\begin{cases}
	\frac{1}{\beta_0 f_0^\prime(0)} f_0(y)g_2(y) &\text{if } y\notin [0,c]\allowdisplaybreaks\\
	\min(\frac{1}{\beta_0 f_0^\prime(0)},\frac{y}{\beta_0 f_0(c)}) f_0(y)g_2(y) &\text{if } y\in [0,c]
	\end{cases}\numberthis \label{eq:W(y)}.
	\end{gather}
\end{lemma}
\begin{proof}
	First, we prove that assumption \ref{assum:Baxendale2005_A2} holds and derive the constants $c$, $\lambda$ and $K$, and the function $V(\cdot)$. Next, we show that assumption \ref{assum:Baxendale2005_A1} holds and derive the probability measure $\nu$ and the constant $\widetilde{\beta}>0$. Finally, we prove that assumption \ref{assum:Baxendale2005_A3} holds and derive the constant $\widehat{\beta}$.
	\begin{itemize}[leftmargin=*]
		\item[]{\it Assumption \ref{assum:Baxendale2005_A2}:} Define the operator $P$ by its action on non-negative measurable functions as follows:
		\begin{align}
		PV(x) &\coloneqq \int_{0}^{\infty} V(y) \frac{\min(x,y)g_2(y)f_0(y)}{\beta_0 f_0(x)} \,dy \allowdisplaybreaks\\
		&\leq \frac{x}{\beta_0 f_0(x)} \int_{0}^{\infty} V(y) g_2(y)f_0(y) \,dy.
		\end{align}
		Assuming the moment generating function of $n_{\root}$ exists at some $\theta > 0$ and using inequality \cref{lem:ineqgi}, we have
		\begin{align}
		PV(x) &\leq \frac{x}{\beta_0 f_0(x)} \int_{0}^{\infty} V(y) \frac{\expect[\eexp^{\theta n_\phi}]}{\eexp^{2\theta }} \exp(-y(1-\eexp^{-\theta })) f_0(y) \,dy. \label{eq:PV_ineq}
		\end{align}
		Let $V(x) = f_0^\prime(0) \eexp^{\eta x} \frac{x}{f_0(x)}$ where the constant $\eta > 0$ is small enough such that $1-\eexp^{-\theta } - \eta > 0$. Part~\ref{part:propL_v} of Theorem~\ref{thm:propL} states that the function $\frac{x}{f_0(x)}$ is strictly increasing. Hence, $V(\cdot)$ is a strictly increasing function and its range is $[1,\infty)$. Substituting the function $V(\cdot)$ into \cref{eq:PV_ineq}, we get
		\begin{align}
		PV(x) &\leq f_0^\prime(0) \frac{x}{f_0(x)} \frac{\expect[\eexp^{\theta n_\phi}]}{\beta_0 \eexp^{2\theta }} \int_{0}^{\infty} y \exp(-y(1-\eexp^{-\theta } - \eta)) \,dy\allowdisplaybreaks\\
		&= f_0^\prime(0) \frac{x}{f_0(x)} \frac{\expect[\eexp^{\theta n_\phi}]}{\beta_0 \eexp^{2\theta }} \frac{1}{(1-\eexp^{-\theta } - \eta)^2}. \numberthis \label{eq:PVineq}
		\end{align}
		Consider the constants $c$, $K$ and $\lambda$ as in the statement of the Theorem.
		For every $x\leq c$, the right-hand side of the \cref{eq:PVineq} is bounded by $K$. Moreover, for every $x>c$, the following inequality holds
		\begin{align}
		\frac{x}{f_0(x)} \frac{\expect[\eexp^{\theta n_{\root}}]}{\beta_0 \eexp^{2\theta }} \frac{1}{(1-\eexp^{-\theta } - \eta)^2} \leq \frac{x}{f_0(x)} \frac{1}{2} \eexp^{\eta c} \leq \frac{1}{2} \frac{x}{f_0(x)} \eexp^{\eta x}.
		\end{align}
		Hence, the assumption \ref{assum:Baxendale2005_A2} is satisfied.

		\item[]{\it Assumption \ref{assum:Baxendale2005_A1}:} Recall that $P(x,A)$ is defined as follows,
		\begin{align}
		P(x,A) = \int_{y\in A}p(x,y)\,dy &= \frac{1}{\beta_0 f_0(x)} \int_{y\in A} f_0(y)g_2(y) \min(x,y)\,dy.
		\end{align}
		For $x\in\mathbb{R}_+$, define the set $A_x = A\cap [0,x]$ and $A_{\bar{x}} = A\cap (x,\infty)$. Using $A=A_x\cup A_{\bar{x}}$, we have,
		\begin{align}
		P(x,A) &= \frac{1}{\beta_0 f_0(x)} \int_{y\in A_x} f_0(y)g_2(y) y\,dy + \frac{x}{\beta_0 f_0(x)} \int_{y\in A_{\bar{x}}} f_0(y)g_2(y) dy.
		\end{align}
		Consider the function $W(y) = \min_{x\in[0,c]} p(x,y)$. Using the fact that $\frac{x}{f_0(x)}$ and $f_0(x)$ are increasing functions, the function $W(\cdot)$ is given as in \cref{eq:W(y)}.
		Notice that $W(\cdot)$ is integrable since it is upper bounded by the integrable function ${\beta_0 f_0^\prime(0)}^{-1} g_2(y)f_0(y)$. Define the probability measure $\nu$ as follows,
		\begin{align}
		\nu(A) = \frac{1}{\widetilde{\beta}}\int_{y\in A} W(y)\,dy,
		\end{align}
		where $\widetilde{\beta}$ is the normalization factor. For all $x\in[0,c]$, the inequality $P(x,A) \geq \widetilde{\beta}\nu(A)$ holds because of the following inequalities:
		\begin{align}
		&
		\begin{aligned}
		\frac{1}{\beta_0 f_0(x)} \int_{y\in A_x} f_0(y)g_2(y) y\,dy &\geq \int_{y\in A_x} \min(\frac{1}{\beta_0 f_0^\prime(0)},\frac{y}{\beta_0 f_0(c)}) f_0(y)g_2(y) dy 	\\
		&= \int_{y\in A_x} W(y) \,dy,
		\end{aligned}\allowdisplaybreaks\\
		&\frac{x}{\beta_0 f_0(x)} \int_{y\in A_{\bar{x}}} f_0(y)g_2(y) \,dy \geq \int_{y\in A_{\bar{x}}} \frac{1}{\beta_0 f_0^\prime(0)} f_0(y)g_2(y)\,dy \geq \int_{y\in A_{\bar{x}}} W(y) \,dy. \allowdisplaybreaks
		\end{align}
		From here, the assumption \ref{assum:Baxendale2005_A1} immediately follows.
		\item[]{\it Assumption \ref{assum:Baxendale2005_A3}:} Using the definition of the probability measure $\nu$, we have,
		\begin{align}
		\widetilde{\beta}\nu([0,c]) = \int_{0}^{c} W(y) \,dy \geq \min\left(\frac{1}{\beta_0 f_0^\prime(0)},\frac{c/2}{\beta_0 f_0(c)}\right)\int_{c/2}^{c} g_2(y)f_0(y) \,dy = \widehat{\beta} > 0.
		\end{align}
	\end{itemize}
\end{proof}
\begin{remark}
	The function $V(\cdot)$ in \cref{eq:V} provides us with more freedom, i.e., it is possible to choose a function $g:\mathbb{R}_+\to\mathbb{R}$ that goes to infinity.
\end{remark}
Lemma~\ref{lem:bax} implies that the Theorem~\ref{thm:bax} holds for the continuous state Markov chain with transition probability $p(x,y)$. The first implication is that the stationary distribution $\pi(x) = C_Ng_2(x) \left(f_0(x)\right)^2$ is unique. Moreover, there exists $M<\infty$ and $0 < \gamma < 1$ such that all the measurable functions $g:\mathbb{R}_+ \to \mathbb{R}$ with the property that $|g(x)| \leq V(x)$ for all $x\in\mathbb{R}_+$, satisfy
\begin{align}
\left|(P^n g)(x)-\int g \,d\pi\right| \leq M V(x) \gamma^n.
\end{align}
Since $V(0)=1$, and $V(x)$ is increasing as can be gleaned from \cref{eq:V} and Theorem~\ref{thm:propL} part~\ref{part:propL_v},
geometric ergodicity follows by restricting the function $g(\cdot)$ to satisfy $|g(x)| \leq 1$ for all $x\in\mathbb{R}_+$, that is
\begin{align}
\norm{P^n(x,\cdot) - \pi}_{TV} \leq M V(x) \gamma^n.
\end{align}
However, it is possible to prove uniform ergodicity by another appropriate choice of function $V(\cdot)$.
\begin{lemma} \label{lem:bddV}
	Let $V(x) = 1 + a \times {\onefunc}_{x > x_0}$. Let
	\begin{align}
	\lambda \coloneqq \frac{3}{4} 	\qquad K \coloneqq 1+a \qquad c= x_0,
	\end{align}
	where the constant $a$ is defined as follows,
	\begin{align}
	&a \coloneqq \frac{8}{\beta_0}\frac{\expect[\eexp^{\theta n_\phi}]}{\eexp^{2\theta }} \times \frac{1}{(1-\eexp^{-\theta })^2},
	\end{align}
	and the constant $x_0$ is large enough such that $f_{0}(x_0) \geq 0.5$ and the following inequality is satisfied for all $x>x_0$:
	\begin{align}
	&\frac{2}{\beta_0}\frac{\expect[\eexp^{\theta n_\phi}]}{\eexp^{2\theta }} \times \frac{(x+1)\eexp^{-x(1-\eexp^{-\theta})}}{(1-\eexp^{-\theta})^2} < \frac{1}{4} .
	\end{align}
	Then, for a suitable $\widetilde{M} > 0$ and $\widetilde{\gamma} < 1$, the following inequality holds for all $x\in\mathbb{R}_+$:
	\begin{align}
	\norm{P^n(x,\cdot) - \pi}_{TV} \leq \widetilde{M} (1+a){\widetilde{\gamma}}^n.
	\end{align}
\end{lemma}
\begin{proof}
	Again, we apply Theorem~\ref{thm:bax} (Baxendale's Theorem), but this time the function $V(\cdot)$ is bounded.
	The only assumption affected by the choice of the function $V(\cdot)$ is assumption \ref{assum:Baxendale2005_A2}. Recall that the transition probability is given by,
	\begin{align}
	p(x,y) = \frac{\min(x,y)g_2(y)f_0(y)}{\beta_0 f_0(x)},
	\end{align}
	Hence, the operator $P$ applied on the measurable function $V(\cdot)$ yields,
	\begin{align}
	PV(x) &= \int_{0}^{\infty} V(y) \frac{\min(x,y)g_2(y)f_0(y)}{\beta_0 f_0(x)} \,dy \allowdisplaybreaks\\
	&= \frac{1}{\beta_0 f_0(x)} \int_{0}^{x} V(y) y g_2(y)f_0(y) \,dy + \frac{x}{\beta_0 f_0(x)} \int_{x}^{\infty} V(y)g_2(y)f_0(y) \,dy.
	\end{align}
	Recall that the function $f_0(x)$ is an increasing function, $f_0(0) =0$ and $\lim_{x\to\infty} f_0(x) = 1$, see for e.g., \cref{eq:Lequals}. Consider the function $V(x) = 1 + a\times{\onefunc}_{\{x>x_0\}}$, where $a$ and $x_0$ are constants to be specified later. Substituting the choice of function $V(\cdot)$, we get
	\begin{align}
	PV(x) &=
	\begin{dcases}
	\begin{aligned}
	&\frac{1}{\beta_0 f_0(x)} \int_{0}^{x}y g_2(y)f_0(y) \,dy + \frac{x}{\beta_0 f_0(x)} \int_{x}^{x_0} g_2(y)f_0(y) dy\\
	&\qquad +\frac{(1+a)x}{\beta_0 f_0(x)} \int_{x_0}^{\infty} g_2(y)f_0(y) \,dy
	\end{aligned}
	& \text{if } x\leq x_0
	\allowdisplaybreaks\\
	\begin{aligned}
	&\frac{1}{\beta_0 f_0(x)} \int_{0}^{x_0}y g_2(y)f_0(y) \,dy + \frac{1+a}{\beta_0 f_0(x)} \int_{x_0}^{x}y g_2(y)f_0(y) \,dy \\
	&\qquad+\frac{(1+a)x}{\beta_0 f_0(x)} \int_{x}^{\infty} g_2(y)f_0(y) \,dy
	\end{aligned}
	& \text{if } x>x_0
	\end{dcases}\allowdisplaybreaks\\
	&\leq
	\begin{dcases}
	\frac{1+a}{\beta_0 f_0(x)}\! \left( \frac{1}{1+a} \int_{0}^{x}y g_2(y) \,dy + x\int_{x}^{x_0} g_2(y) \,dy + x\int_{x_0}^{\infty} g_2(y) \,dy \right)&\!\!\! \text{if } x\leq x_0\allowdisplaybreaks\\
	\frac{1+a}{\beta_0 f_0(x)}\! \left( \frac{1}{1+a} \int_{0}^{x_0}y g_2(y) \,dy + \int_{x_0}^{x}y g_2(y) \,dy + x\int_{x}^{\infty} g_2(y) \,dy \right)&\!\!\! \text{if } x>x_0
	\end{dcases}.
	\end{align}
	Assume $x>x_0$. Using the inequality \cref{lem:ineqgi}, we have,
	\begin{align}
	PV(x) &\leq \frac{(1+a)\expect[\eexp^{\theta n_\phi}]}{\eexp^{2\theta }\beta_0 f_0(x)} \Bigg( \frac{1}{1+a} \int_{0}^{x_0}y \eexp^{-y(1-\eexp^{-\theta })} \,dy + \int_{x_0}^{x}y \eexp^{-y(1-\eexp^{-\theta })} \,dy  \allowdisplaybreaks\\
	&\myquad[20] + x\int_{x}^{\infty} \eexp^{-y(1-\eexp^{-\theta })} \,dy \Bigg) \allowdisplaybreaks\\
	& \begin{aligned}
	&\leq \frac{1+a}{\beta_0 f_0(x)}\frac{\expect[\eexp^{\theta n_\phi}]}{\eexp^{2\theta }} \Bigg( \frac{1}{1+a} \frac{1}{(1-\eexp^{-\theta })^2} + \frac{(x_0 (1-\eexp^{-\theta })+1)\eexp^{-x_0 (1-\eexp^{-\theta })}}{(1-\eexp^{-\theta })^2} \allowdisplaybreaks\\
	&\myquad[25] +x \frac{\eexp^{-x(1-\eexp^{-\theta })}}{1 - \eexp^{-\theta}}\Bigg).
	\end{aligned}\numberthis \label{eq:PV(x)leq}
	\end{align}
	The last inequality follows by evaluating the integrals and removing the negative terms. The constants $a$ and $x_0$ are chosen such that $f_0(x_0) \geq 0.5$, and all the following inequalities are satisfied for all $z>x_0$:
	\begin{align}
	&\frac{2}{\beta_0}\frac{\expect[\eexp^{\theta n_\phi}]}{\eexp^{2\theta }} \times \frac{1}{1+a} \frac{1}{(1-\eexp^{-\theta })^2} < \frac{1}{4} \label{eq:aineq}\allowdisplaybreaks\\
	&\frac{2}{\beta_0}\frac{\expect[\eexp^{\theta n_\phi}]}{\eexp^{2\theta }} \times \frac{(z+1)\eexp^{-z(1-\eexp^{-\theta})}}{(1-\eexp^{-\theta})^2} < \frac{1}{4}.\label{eq:x0ineq}
	\end{align}
	Notice that the left-hand side of \cref{eq:aineq} is decreasing in $a$ and the left-hand side of \cref{eq:x0ineq} can be made arbitrary small by setting $x_0$ to be large enough.
	Noticing that \cref{eq:x0ineq} upper bounds the last two terms in \cref{eq:PV(x)leq}, for all $x > x_0$ we have $PV(x) \leq (3/4)(1+a)$. Since the function $V(\cdot)$ is bounded by $1+a$, for $x \leq x_0$ we have $PV(x) \leq 1+a$. Given the above choice of constants $a$ and $x_0$, for $\lambda = \frac{3}{4}$ and $K = 1+a$, taking $C = \{x : x \leq x_0\}$, assumption \ref{assum:Baxendale2005_A2} is satisfied; i.e.,
	\begin{align}
	PV(x) \leq
	\begin{cases}
	\frac{3}{4} (1+a) &\text{if }x > x_0\allowdisplaybreaks\\
	1+a &\text{if } x \leq x_0.
	\end{cases}
	\end{align}
	An application of Baxendale's Theorem then completes the proof.
\end{proof}
An immediate consequence of uniform ergodicity and Lemma~\ref{lem:bddV} is the following.
\begin{corollary}
	For any $x,y\in\mathbb{R}_+$ and $l>1$, we have
	\begin{align}\label{eq:pointwiseineq}
	\left|p^{(l)}(x,y) - \pi(y)\right| < 2\widetilde{M} (1+a){\widetilde{\gamma}}^{l-1}.
	\end{align}
\end{corollary}
\begin{proof}
	The idea of the proof follows Doob~\cite[pages 216-217]{Doob1953}. Notice that $\pi(\cdot)$ is the unique stationary distribution. Hence, for any
	\begin{align}
	\left|p^{(l)}(x,y) - \pi(y)\right| &= \left|\int_{z = 0}^{\infty} p^{(1)}(z,y)\left(p^{(l-1)}(x,z) - \pi(z)\right)\,dz \right|\allowdisplaybreaks\\
	&\leq \left|\int_{p^{(l-1)}(x,z) > \pi(z)} \left(p^{(l-1)}(x,z) - \pi(z)\right)\,dz \right| \\
	&\qquad+ \left|\int_{p^{(l-1)}(x,z) < \pi(z)} \left(p^{(l-1)}(x,z) - \pi(z)\right)\,dz \right|\allowdisplaybreaks\\
	&\leq 2\widetilde{M} (1+a){\widetilde{\gamma}}^{l-1},
	\end{align}
 which completes the proof.
\end{proof}
To get rid of the constant factor $C_N$, from now on, we assume the function $f_0$ is normalized such that,
\begin{align}
	\int_{0}^{\infty} g_2(y) \left(f_0(y)\right)^2 \,dy = 1.
\end{align}
That is $f_0(y) = L(\beta_0,y)/\sqrt{C_N}$, where $C_N = (\int_{0}^{\infty} g_2(y) \left(L(\beta_0,y)\right)^2 \,dy)^{-1}$.

Then inequality \cref{eq:pointwiseineq} implies that for every $x\in\mathbb{R}_+$ and $y>0$,
\begin{align}
h_l(x,y) = {\beta_0}^l f_0(x) f_0(y) \left(1 + \frac{2\widetilde{M} (1+a)O(\widetilde{\gamma}^{l-1})}{g_2(y)\left(f_0(y)\right)^2}\right)\qquad l\geq 2 \label{eq:hlest}.
\end{align}
Harris~\cite{Harris1951} assumes that the density of the $M_1$ is uniformly positive and bounded, and deduces that the corresponding eigenfunction is uniformly positive as well. However, in our setting $f_0(0) = 0$ and $g_2(y)\to 0$ as $y\to\infty$. As a result, the error term for $h_l(x,y)/\beta_0^l$ explodes as $y$ goes to $0$ or $\infty$. On the other hand, induction using \cref{eq:hl} implies $h_l(x,0) = h_l(0,y) = 0$. Hence, we should expect a uniform bound. The idea is to use the function $V(\cdot)$ in \cref{eq:V} and apply \cref{eq:hl}.
\begin{lemma}\label{lem:hlbound}
	For some constant $\widehat{M}>0$, we have
	\begin{align}
	h_{l}(x,y) = {\beta_0}^{l} f_0(y)f_0(x) \left(1 + \widehat{M} O({\gamma}^{l-2})\frac{x }{{\beta_0}^2f_0(y)f_0(x)}\right)\qquad l\geq 2. \label{eq:hlest2}
	\end{align}
\end{lemma}
\begin{proof}
	Fix $z\in\mathbb{R}_+$ and define the function $g(\cdot)$ as follows,
	\begin{align}
	g(x) =
	\begin{dcases}
	\frac{h_1(x,z)}{f_0(x)} \times f_0^\prime(0)&\text{if $x\neq 0$}\\
	1&\text{if $x= 0$}
	\end{dcases}.
	\end{align}
	The function $g(\cdot)$ is a well-defined continuous function by Theorem~\ref{thm:propL} part~\ref{part:propL_v}. Moreover, for all $x\in\mathbb{R}_+$, we have $|g(x)|\leq V(x)$ where $V(\cdot)$ is given by \cref{eq:V}. Now using Lemma~\ref{lem:bax} and Theorem~\ref{thm:bax} (Baxendale's Theorem), we have
	\begin{align}
	&\left|\int_{0}^{\infty} \frac{h_l(x,y) g_2(y)f_0(y) }{{\beta_0}^l f_0(x)} \times\frac{h_1(y,z)}{f_0(y)}f_0^\prime(0) \,dy - \int_{0}^{\infty} g_2(y) \left(f_0(y)\right)^2 \times \frac{h_1(y,z)}{f_0(y)} f_0^\prime(0) \,dy\right| \allowdisplaybreaks\\
	&\myquad[20]\leq M{\gamma}^l V(x).
	\end{align}
	Using \cref{eq:hl} and \cref{eq:f0rel}, we get
	\begin{align}
	\left|\frac{h_{l+1}(x,z)}{{\beta_0}^l f_0(x)} - \beta_0 f_0(z)\right| \leq M{\gamma}^l \frac{x \eexp^{\eta x}}{f_0(x)},
	\end{align}
	hence,
	\begin{align}
	h_{l+1}(x,y) = {\beta_0}^{l+1} f_0(y)f_0(x) \left(1 + MO({\gamma}^{l})\frac{x \eexp^{\eta x}}{\beta_0f_0(y)f_0(x)}\right).
	\end{align}
	Now using \cref{eq:hl} again, we have
	\begin{align}
	h_{l+2}(x,y) &= \int_{0}^{\infty} h_{1}(x,z)h_{l+1}(z,y) g_2(z)\,dz\allowdisplaybreaks\\
	&= \int_{0}^{\infty} \min(x,z) {\beta_0}^{l+1} f_0(y)f_0(z) \left(1 + MO({\gamma}^{l})\frac{z \eexp^{\eta z}}{\beta_0f_0(y)f_0(z)}\right) \,dz \allowdisplaybreaks\\
	&= \int_{0}^{\infty} \min(x,z) {\beta_0}^{l+1} f_0(y)f_0(z) g_2(z)\,dz + \allowdisplaybreaks\\
	&\myquad[15]MO({\gamma}^{l}) \int_{0}^{\infty} \min(x,z) {\beta_0}^{l} z \eexp^{\eta z} g_2(z) \,dz.
	\end{align}
	Applying inequality \cref{lem:ineqgi}, we get
	\begin{align}
	&\left|h_{l+2}(x,y) - \beta_0^{l+2} f_0(x) f_0(y) \right| \leq \\
	&\myquad[10]M{\gamma}^{l} \times \beta_0^l
	\frac{\expect[\eexp^{\theta n_{\root}}]}{\eexp^{2\theta}}
	\int_0^{\infty} \min(x,z) z \eexp^{\eta z} \exp(-z (1 - \eexp^{-\theta})) \,dz.
	\end{align}
	Now the result follows by the fact that $\min(x,z)\leq x$, and the fact that $\eta < 1-\eexp^{-\theta}$. Notice that
	\begin{align}
	\widehat{M} = M \times
	\frac{\expect[\eexp^{\theta n_{\root}}]}{\eexp^{2\theta}}
	\int_0^{\infty} z \eexp^{\eta z} \exp(-z (1 - \eexp^{-\theta})) \,dz.
	\end{align}
\end{proof}
\begin{remark}
	In the proof of Lemma~\ref{lem:hlbound}, we bound $\min(x,z)$ by $x$ instead of $z$. This gives us a uniform error bound for $m_l$. Specifically, as $x\to 0$ the error term in \cref{eq:mlunifbound} stays bounded.
\end{remark}
Combining \cref{eq:mlhl} and \cref{eq:hlest2}, we get a similar bound for $m_l(m,x;k-1,z)$: for every $x\in\mathbb{R}_+$ and $z>0$, we have
\begin{align}
\begin{aligned}
&m_l(m,x;k-1,z) =\frac{P(k)\eexp^{-z}z^{k-1}}{(k-1)!} \frac{m}{x} \times\allowdisplaybreaks\\
&\myquad[10]{\beta_0}^l f_0(x) f_0(z) \left(1 + \widehat{M} O({\gamma}^{l-2})\frac{x}{{\beta_0}^2f_0(z)f_0(x)}\right) ~ l\geq 2.
\end{aligned}\label{eq:mlunifbound}
\end{align}
Notice that the error term is uniformly bounded for all $x,z\in\mathbb{R}_+$ and $k\in\mathbb{N}$ (naturally, it is not uniform in $m$). Next we prove that $\beta_0$ is the Krein--Rutman eigenvalue of $H_1$ with the eigenfunction given by $f_0(x)$.
\begin{theorem} \label{thm:f&beta}
	Assume that the moment generating function of $n_{\root}$ exists at some $\theta>0$. Then $\beta_0 \in \left(\max\limits_x\left(x\upsilon([x,\infty))\right), \expect[n_{\root}]-1\right)$ is an eigenvalue of $H_1$ larger in magnitude than any other eigenvalue of $H_1$. The corresponding eigenfunction is $f_0(\cdot)$. Moreover, this is the only non-negative eigenfunction of $H_1$ up to a normalization factor.
\end{theorem}
\begin{proof}
	Assume there exists a real-valued function $\zeta(\cdot)$ and $\beta^\prime\neq 0$ such that,
	\begin{align}
	\beta^\prime \zeta(x) &= \int_{z=0}^{\infty}h_1(x,z)g_2(z) \zeta(z) \,dz.
	\end{align}
	Clearly, $\zeta(x)$ satisfies the following inequality:
	\begin{align}
	|\zeta(x)| &\leq \frac{1}{|\beta^\prime|} \int_{z=0}^{\infty}h_1(x,z)g_2(z) |\zeta(z)| \,dz \allowdisplaybreaks\\
	&\leq \frac{x}{|\beta^\prime|} \int_{z=0}^{\infty}g_2(z) |\zeta(z)| \,dz = \texttt{Const}\times x.
	\end{align}
	Moreover, $\zeta(0) =0 $ since $h_1(0,z) = 0$; hence, for all $x\in\mathbb{R}_+$, the function $g(x) = \zeta(x)/f_0(x)$ is well-defined.
	Letting $V(x) = f_0^\prime(0) \eexp^{\eta x}\frac{x}{f_0(x)} \times \max(\frac{\texttt{Const}}{f_0^\prime(0)},1)$ in Lemma~\ref{lem:bax}, for all $x\in\mathbb{R}_+$ we have $|g(x)| \leq V(x)$. Using Theorem~\ref{thm:bax} (Baxendale's Theorem), we have
	\begin{align}
	\left|\int_{0}^{\infty} \frac{h_l(x,y) g_2(y) f_0(y) }{{\beta_0}^lf_0(x)} \frac{\zeta(y)}{f_0(y)}\,dy - \int_{0}^{\infty} g_2(y). \left(f_0(y)\right)^2\frac{\zeta(y)}{f_0(y)}\,dy\right| < M \gamma^l V(x)
	\end{align}
	Hence,
	\begin{align}
	\left|\frac{{\beta^\prime}^l \zeta(x)}{{\beta_0}^lf_0(x)} - \int_{0}^{\infty} g_2(y) {f_0(y)}\zeta(y)\,dy\right| < M \gamma^l V(x).
	\end{align}
	As $l$ goes to infinity, the right-hand side of the above inequality goes to zero. If $|\beta^\prime| > \beta_0$, then the left-hand side explodes. If $|\beta^\prime| = \beta_0$, then the left-hand side does not go to zero for all $x$. Hence, $|\beta^\prime| < \beta_0$ and $\zeta(\cdot)$ and $f_0(\cdot)$ are orthogonal to each other, i.e.,
	\begin{align}
	\int_{0}^{\infty} {f_0(y)}\zeta(y)d\upsilon(y) = 0.
	\end{align}
	The above equality also proves that $f_0(\cdot)$ is the only non-negative eigenfunction.
\end{proof}
We summarize the key conclusions in the following theorem.
\begin{theorem}\label{thm:pfeigenval}
	Assume the moment generating function of $n_{\root}$ exists at some $\theta > 0$. Let $\beta_0$ and $f_0(\cdot)$ to be as in Theorem~\ref{thm:f&beta}. Then $\beta_0 \in \left(\max\limits_x\Big(x\upsilon([x,\infty))\right), \expect[n_{\root}]\allowbreak-1\Big)$ is the largest eigenvalue of $M_1$ in magnitude. The corresponding eigenfunctions are given as follows
	\begin{align}
	&\text{Right eigenfunction: } \mu(m,x) = \frac{m}{x} f_0(x),\allowdisplaybreaks\\
	&\text{Left eigenfunction: } \nu(k-1,z) = P(k)\frac{\eexp^{-z}z^{k-1}}{(k-1)!} f_0(z).
	\end{align}
	These eigenfunctions are the unique non-negative right and left eigenfunctions, respectively. Moreover, there exists $0<\gamma<1$ and a constant $\widehat{M} > 0$ independent of $x$, $m$, $z$ and $k$ such that for all $x\in\mathbb{R}_+$, $y>0$, $k\geq 1$ and $m\geq 0$,
	\begin{align}
	\begin{aligned}
	&m_l(m,x;k-1,z) =\frac{P(k)\eexp^{-z}z^{k-1}}{(k-1)!} \frac{m}{x} \times \allowdisplaybreaks\\
	&\myquad[10]{\beta_0}^l f_0(x) f_0(z) \left(1 + \widehat{M}O({\gamma}^{l-2})\frac{x}{{\beta_0}^2f_0(z)f_0(x)}\right),\myquad[1] l \geq 2.
	\end{aligned} \label{eq:mlaprox}
	\end{align}
	Finally, $m_l(m,x;k-1,z)$ is related to the function $h_l(x,y)$ via the following equation,
	\begin{align}
	m_l(m,x;k-1,z) = h_l(x,z) \times \frac{P(k)\eexp^{-z}z^{k-1}}{(k-1)!} \frac{m}{x},
	\end{align}
	and for all functions $g:\mathbb{R}_+ \to \mathbb{R}$ satisfying $|g(x)| \leq V(x)$ for all $x\in\mathbb{R}_+$, we have
	\begin{align}
	\left|\int_{0}^{\infty} \frac{h_l(x,y) g_2(y)f_0(y) g(y)}{{\beta_0}^l f_0(x)} \,dy - \int_{0}^{\infty} g_2(y) \left(f_0(y)\right)^2 g(y)\,dy\right| \leq M{\gamma}^l V(x),\qquad l \geq 2,
	\end{align}
	where $V(x) = f_0^\prime (0) \exp(\eta x) \frac{x}{f_0(x)}$ and $\eta = (1 - \eexp^{-\theta})/2$. The constants $M$ and $0<\gamma<1$ are independent of $x$ and $l$.
\end{theorem}
Using the above theorem, we get similar bounds for $M_l$ which are useful for large $l$.
\begin{corollary}\label{cor:growthrate}
	The growth rate of $M_l(m,x;\mathbb{R}_+,\mathbb{\mathbb{Z}_+})$ equals $\beta_0$ which is given by Theorem~\ref{thm:pfeigenval}, i.e.,
	\begin{align}
	\left|\frac{M_l(m,x;\mathbb{R}_+,\mathbb{Z}_+)}{{\beta_0}^l} - \frac{m}{x}f_0(x) \int_{0}^{\infty} \sum_{k=1}^\infty P(k)\frac{\eexp^{-z}z^{k-1}}{(k-1)!} f_0(z)\,dz\right| = \frac{m}{{\beta_0}^2}\widehat{M}O({\gamma}^{l-2}), \qquad \!\!\! l\geq 2,
	\end{align}
	where the constant $0<\gamma<1$ is independent of $x$, $m$ and $l$.
\end{corollary}
\begin{proof}
	By Theorem~\ref{thm:pfeigenval}, we have
	\begin{align}
	&M_l(m,x;\mathbb{R}_+,\mathbb{Z}_+)/{\beta_0}^l \allowdisplaybreaks\\
	&\myquad[5]= \int_{0}^{\infty}\sum_{k=1}^{\infty} m_l(m,x;k-1,z)\,dz/{\beta_0}^l\allowdisplaybreaks\\
	&\myquad[5]=\int_{0}^{\infty}\sum_{k=1}^{\infty} \frac{P(k)\eexp^{-z}z^{k-1}}{(k-1)!} \frac{m}{x} \times \\
	&\myquad[10] f_0(x) f_0(z) \left(1 + \widehat{M}O({\gamma}^{l-2})\frac{x}{{\beta_0}^2 f_0(z)f_0(x)}\right)\,dz \allowdisplaybreaks\\
	&\myquad[5]= \frac{m}{x}f_0(x) \int_{0}^{\infty} \sum_{k=1}^\infty P(k)\frac{\eexp^{-z}z^{k-1}}{(k-1)!} f_0(z)\,dz +\frac{m}{{\beta_0}^2} \widehat{M}O({\gamma}^{l-2})\int_{0}^{\infty} g_1(z)dz.
	\end{align}
\end{proof}
Recall that $Z_l$ denotes the number of vertices in generation $l$. As an immediate Corollary, the growth/extinction rate of $\expect[Z_l]$ is $\beta_0$ as well.
\begin{corollary}\label{cor:EZ_l}
	We have
	\begin{align}
	&\expect[Z_l]/{\beta_0}^l \xrightarrow{l\to\infty}\allowdisplaybreaks\\
	 &\myquad[5]\left(\sum_{m=1}^{\infty}P(m) \int_{x=0}^{\infty}\frac{\eexp^{-x} x^{m}}{m!}\times \frac{m}{x}f_0(x)\,dx\right) \left(\int_{0}^{\infty} \sum_{k=1}^\infty P(k)\frac{\eexp^{-z}z^{k-1}}{(k-1)!} f_0(z)\,dz\right).
	\end{align}
\end{corollary}
If $\beta_0 > 1$, the expected number of vertices in generation $l$ explodes as $l$ goes to infinity. If $\beta_0 = 1$, the expected number of vertices in generation $l$ stays bounded. If $\beta_0 < 1$, the expected number of vertices in generation $l$ goes to zero.

%% file: Sections/PropEWT/SecMom.tex
A follow-up question is the limit of the random variable $Z_l / \beta_0^l$: $1)$ If $\beta_0 < 1$, it is clear that $Z_l \to 0$ almost surely as $l \rightarrow \infty$ since the population will become extinct; however, conditioned on $Z_l > 0$, the distribution of the total number of vertices might be of interest. We leave this problem for future work. $2)$ If $\beta_0 > 1$, one way to study the limit is to analyze the second moment. This methodology was introduced by Harris in~\cite{Harris1948} and was generalized to finite type branching processes in~\cite{Harris1951}. In~\cite{Harris1959}, Harris pointed out that a similar generalization is possible for general branching processes and discussed this further in~\cite[Chapter 3]{Harris1963}. We follow his argument closely in this section. $3)$ The case $\beta = 1$ is tricky and is discussed in Section~\ref{sec:PhTr}. We will prove that $Z_l \to 0$ almost surely as $l \rightarrow \infty$; however, a similar question as in $1)$ is left for future work.

Let $Z_l(\Apset)$ denote the number of vertices at depth $l$ of type $(k-1,\zeta)\in \Apset, \Apset\subset \Omega$. By the discussion of Section~\ref{sec:back_Branch}, $Z_l(\cdot)$ is a set function. For Borel sets $\Apset_1, \Apset_2 \subset \Omega$, define,
\begin{align}
&M_l^{(2)}(m,x;\Apset_1;\Apset_2) \coloneqq \expect[Z_l(\Apset_1)Z_l(\Apset_2)\,\vert\, n_{\root} = m, v_{\root} = x],\qquad\forall l=0,1,\cdots, \allowdisplaybreaks\\
&v(m,x;\Apset_1;\Apset_2) \coloneqq M_1^{(2)}(m,x;\Apset_1;\Apset_2) - M_1(m,x;\Apset_1)M_1(m,x;\Apset_2). \label{eq:vindentity}
\end{align}
Notice that $M_l^{(2)}(m,x;\Apset_1;\Apset_2)$ is the correlation of $Z_l(\Apset_1)$ and $Z_l(\Apset_2)$ given $\{n_{\root} = m, v_{\root} = x\}$, and $v(m,x;\Apset_1;\Apset_2)$ is the covariance of $Z_l(\Apset_1)$ and $Z_l(\Apset_2)$ conditioned on $\{n_{\root} = m, v_{\root} = x\}$.
The conditionally independent structure of the EWT implies
\begin{align}
M_1^{(2)}(m,x;\Apset_1;\Apset_2) =
\small
\begin{dcases*}
\begin{aligned}
&\frac{m(m-1)}{x} \Bigg(\underset{(k-1,\zeta) \in \Apset_1}{\sum \int} P(k) \min(x,\zeta) \frac{\eexp^{-\zeta} {\zeta}^{k-1}}{(k-1)!} \,d\zeta\Bigg)\times\\
&\myquad[6]\Bigg(\underset{(k-1,\zeta) \in \Apset_2}{\sum \int} P(k) \min(x,\zeta) \frac{\eexp^{-\zeta} {\zeta}^{k-1}}{(k-1)!} \,d\zeta\Bigg)\\
&\myquad[2]+ \frac{m}{x} \Bigg(\underset{(k-1,\zeta) \in \Apset_1\cap \Apset_2}{\sum \int} P(k) \min(x,\zeta) \frac{\eexp^{-\zeta} {\zeta}^{k-1}}{(k-1)!} \,d\zeta\Bigg),
\end{aligned}
 & \text{if $x>0$}\\
\begin{aligned}
&m(m-1) \Bigg(\underset{(k-1,\zeta) \in \Apset_1}{\sum \int} P(k) \frac{\eexp^{-\zeta} {\zeta}^{k-1}}{(k-1)!} \,d\zeta\Bigg)\\
&\myquad[6]\Bigg(\underset{(k-1,\zeta) \in \Apset_2}{\sum \int} P(k) \frac{\eexp^{-\zeta} {\zeta}^{k-1}}{(k-1)!} \,d\zeta\Bigg) \\
&\myquad[2]+ m \Bigg(\underset{(k-1,\zeta) \in \Apset_1\cap \Apset_2}{\sum \int}P(k) \frac{\eexp^{-\zeta} {\zeta}^{k-1}}{(k-1)!} \,d\zeta\Bigg),
\end{aligned}& \text{if $x=0$}
\end{dcases*}
\end{align}
where for a measurable and integrable function $f(\cdot,\cdot): \mathbb{N}\times\mathbb{R}_+\mapsto\mathbb{R}$, the shorthand $\underset{(k-1,\zeta) \in \Apset}{\sum \int} f(k,\zeta)\,d\zeta$ stands for $\sum_{k=1}^\infty \int_{\zeta:(k-1,\zeta) \in \Apset} f(k,\zeta)\,d\zeta$. To get the above equality, notice that
\begin{align}
M_1^{(2)}(m,x;\Apset_1;\Apset_2) &=\expect[Z_l(\Apset_1)Z_l(\Apset_2) - Z_l(\Apset_1\cap \Apset_2)\,\vert\, n_{\root} = m, v_{\root} = x]\\
&\qquad\qquad +\expect[Z_l(\Apset_1\cap \Apset_2)\,\vert\, n_{\root} = m, v_{\root} = x],
\end{align}
and also notice that $Z_l(\Apset_1)Z_l(\Apset_2) - Z_l(\Apset_1\cap \Apset_2)$ equals the number of ways to select two different descendants of the root successively: first, a descendant of a type belongs to $\Apset_1$ and then a descendant of a type belongs to $\Apset_2$.

For any fixed $(m,x)$, we can interpret $M_l^{(2)}(m,x;\Apset_1;\Apset_2)$ as the measure of the ``rectangle'' $\Apset_1\times\Apset_2$, i.e., the measure of points $(k_1-1,\zeta_1;k_2-2,\zeta_2)$ such that $(k_1-1,\zeta_1)\in \Apset_1$, $(k_2-1,\zeta_2)\in\Apset_2$, and $(k_i-1,\zeta_i)\in Z_l$ for $i\in\{1,2\}$, where $Z_l$ (abusing notation) is the point distribution of vertices in generation $l$. To make the notions rigorous, we need to define bivariate measures and random double integrals.
\begin{definition}
	A function $F(\Apset,\Bpset)$, where $\Apset$ and $\Bpset$ are subsets of $\Omega$, is called a bivariate measure if it satisfies the following conditions:
	\begin{enumerate}[label=(\alph*)]
		\item it is finite and non-negative;
		\item if ${\Apset}_1, {\Apset}_2,\dots {\Apset}_k$ are disjoint subsets of $\Omega$, then $F\left(\cup_j {\Apset}_j,\Bpset\right) = \sum_{j} F({\Apset}_j,\Bpset)$;
		\item if ${\Bpset}_1, {\Bpset}_2,\dots {\Bpset}_k$ are disjoint subsets of $\Omega$, then $F\left(\Apset,\cup_j {\Bpset}_j \right) = \sum_{j} F(\Apset,{\Bpset}_j)$;
	\end{enumerate}
	$F$ is called a signed bivariate measure if $F=F_1 - F_2$, where $F_1$ and $F_2$ are bivariate measures.
\end{definition}
\begin{definition}
	For a function $f(k_1-1,\zeta_1;k_2-1,\zeta_2)$ defined over $\Omega\times \Omega$, the random double integral is defined as follows:
	\begin{align}
	&\underset{(k_2-1,\zeta_2)\in\Omega}{\sum \int }\,\underset{(k_1-1,\zeta_1)\in\Omega}{\sum \int } f(k_1-1,\zeta_1;k_2-1,\zeta_2) d\omega(\zeta_1,k_1) d\omega(\zeta_2,k_2) \\
	&\myquad[20]= \sum_{i,j} a_i a_j f(m_i,x_i;m_j,x_j)
	\end{align}
	where $\omega$ is the point distribution $( (m_1,x_1),a_1; (m_2,x_2),a_2;\dots;(m_k,x_k),a_k )$.
\end{definition}
By definition, $M_l^{(2)}(m,x;\Apset_1;\Apset_2)$ and $M_1(m,x;\Apset_1)M_2(m,x;\Apset_2)$ are bivariate measures, and $v(m,x;\Apset_1;\Apset_2)$ is a signed bivariate measure.
Define a map $\Topt $ from the set of signed bivariate measures to itself as follows,
\begin{align}
&\Topt F(\Apset_1;\Apset_2) = \\
&\myquad[2]\underset{(k_2-1,\zeta_2)\in\Omega}{\sum \int }\,\underset{(k_1-1,\zeta_1)\in\Omega}{\sum \int } M_1(k_1-1,\zeta_1;\Apset_1)M_1(k_2-1,\zeta_2;\Apset_2) \, dF(k_1-1,\zeta_1;k_2-1,\zeta_2).
\end{align}
To derive a recurrence relation between $M_l^{(2)}$ and $M_{l+1}^{(2)}$, write
\begin{align}
M_{l+1}^{(2)}(m,x;\Apset_1;\Apset_2) = \expect[\, \expect[Z_{l+1}(\Apset_1)Z_{l+1}(\Apset_2)\,\vert\,Z_l = \omega]\,\vert\,n_{\root} = m, v_{\root} = x].
\end{align}
Conditioned on $Z_l = \omega \in\pd$, the expected value of $Z_{l+1}(\Apset_1)Z_{l+1}(\Apset_2)$ is given by the following random integrals,
\begin{align}
&\underset{(k_2-1,\zeta_2)\in\Omega}{\sum \int }\,\underset{(k_1-1,\zeta_1)\in\Omega}{\sum \int } \expect_{k_1-1,\zeta_1}[\widetilde{Z}_1(\Apset_1)]\, \expect_{k_2-1,\zeta_2}[\widetilde{Z}_1(\Apset_2)] \,dZ_l(k_1-1,\zeta_1)\,dZ_l(k_2-1,\zeta_2) \allowdisplaybreaks\\
&\myquad[4]-\underset{(k-1,\zeta)\in\Omega}{\sum \int } \expect_{k-1,\zeta}[\widetilde{Z}_1(\Apset_1)] \,\expect_{k-1,\zeta}[\widetilde{Z}_1(\Apset_2)]\,dZ_l(k-1,\zeta) \allowdisplaybreaks\\
&\myquad[4] +\underset{(k-1,\zeta)\in\Omega}{\sum \int } \expect_{k-1,\zeta}[\widetilde{Z}_1(\Apset_1)\widetilde{Z}_1(\Apset_2)]\,dZ_l(k-1,\zeta),
\end{align}
where $\widetilde{Z}_1$ is an {\em i.i.d.} copy of the point distribution $Z_1$ and $\expect_{m,x}$ is the expected value conditioned on the type of the root to be $(m,x)$. Now, taking expectation of the above random integrals with respect to the point distribution $Z_l$, we derive the following recurrence relation,
\begin{align} \label{eq:M_l2rec}
\begin{aligned}
&M_{l+1}^{(2)}(m,x;\Apset_1;\Apset_2) = \\
&\myquad[4]\Topt M_l^{(2)}(m,x;\Apset_1;\Apset_2) + \underset{(k-1,\zeta)\in\Omega}{\sum \int } v(k-1,\zeta;\Apset_1;\Apset_2) \,d M_l(m,x;k-1,\zeta).
\end{aligned}
\end{align}
Repeatedly using \cref{eq:M_l2rec} and then applying \cref{eq:vindentity}, we get the following relation
\begin{align}
\begin{aligned}
M_{l+1}^{(2)}(m,x;\Apset_1;\Apset_2) &= \Topt ^{l} M_1(m,x;\Apset_1)M_1(m,x;\Apset_2) \\
&\myquad[2]+\sum_{\hat{l}=0}^{l} \Topt ^{l-\hat{l}} \left(\underset{(k-1,\zeta)\in\Omega}{\sum \int } v(k-1,\zeta;\Apset_1;\Apset_2) \,d M_{\hat{l}}(m,x;k-1,\zeta)\right),
\end{aligned}
\label{eq:M_l2T}
\end{align}
where $\Topt^0$ is the identity map. Finally, observe that
\begin{align}\label{eq:relscrT}
\begin{aligned}
&\Topt ^l F(\Apset_1;\Apset_2) = \\
&\myquad[2]\underset{(k_1-1,\zeta_1)\in\Omega}{\sum \int } \, \underset{(k_2-1,\zeta_2)\in\Omega}{\sum \int } M_l(k_1-1,\zeta_1;\Apset_1)M_l(k_2-1,\zeta_2;\Apset_2) \,dF(k_1-1,\zeta_1;k_2-1,\zeta_2),
\end{aligned}
\end{align}
which can be proved by induction and the following equality:
\begin{align}
\begin{aligned}
&d\Topt F(k-1,\zeta;\widetilde{k}-1,\widetilde{\zeta}) = \allowdisplaybreaks\\
&\myquad[5]\underset{(k_2-1,\zeta_2)\in\Omega}{\sum \int }\,\underset{(k_1-1,\zeta_1)\in\Omega}{\sum \int } m_1(k_1-1,\zeta_1;k-1,\zeta)\,d\zeta\\
&\myquad[10] m_1(k_2-1,\zeta_2;\widetilde{k}-1,\widetilde{\zeta}) \,d\widetilde{\zeta}\,dF(k_1-1,\zeta_1;k_2-1,\zeta_2).
\end{aligned}
\end{align}
Now, we can use the analysis of the previous section to approximate $M_{l}^{(2)}(m,x;\Apset_1;\Apset_2)$ for large values of $l$. This is basically the same result as in~\cite[page 72, eqn. (13.5)]{Harris1963}.
\begin{theorem}\label{thm:ml2conv}
	With $\beta_0 > 1$ and $\mu(m,x)$ as specified in Theorem~\ref{thm:pfeigenval}, the growth rate of $M_l^{(2)}(m,x;\allowbreak\Apset_1;\Apset_2)$ equals ${\beta_0}^2$, i.e.,
	\begin{align}\label{eq:ml2conv}
	M_l^{(2)}(m,x;\Apset_1,\Apset_2)/{\beta_0}^{2l} = &U(m,x)
	\left( \underset{(k-1,z)\in\Apset_1}{\sum \int } \nu(k-1,z) \,dz\right)\!\!\left( \underset{(k-1,z)\in\Apset_2}{\sum \int } \nu(k-1,z) dz\right) \allowdisplaybreaks\\
	&\qquad + m^2 \xbar{M}O({\gamma}^{l-2}),\qquad l \geq 2, \nonumber
	\end{align}
	where the constants $\xbar{M} > 0$ and $0<{\gamma}<1$ are independent of $x$, $l$, $\Apset_1$, and $\Apset_2$. The function $U(m,x)$ is defined as follows,
	\begin{align}
	\begin{aligned}
	&U(m,x) \coloneqq \\
	&\myquad[2]\left(\mu(m,x)\right)^2 + \sum_{\hat{l}=1}^\infty {\beta_0}^{-2\hat{l}} \underset{(k-1,z)\in\Omega}{\sum \int } \Bigg(\underset{(k_1-1,z_1)\in\Omega}{\sum \int }~ \underset{(k_2-1,z_2)\in\Omega}{\sum \int } \mu(k_1-1,z_1)\mu(k_2-1,z_2)\\
	&\myquad[10]dv(k-1,z;k_1-1,z_1;k_2-1,z_2)\Bigg) \,d M_{\hat{l}-1}(m,x;k-1,z).
	\end{aligned}\label{eq:U}
	\end{align}
\end{theorem}
	\begin{remark}
		Notice that using Theorem~\ref{thm:pfeigenval}, the summand in the definition of $U(m,x)$ is $O(\beta_0^{-\hat{l}-1})$; hence, the sum is finite and $U(m,x)$ is well-defined.
	\end{remark}
\begin{proof}
	As we pointed out in the proof of Corollary~\ref{cor:growthrate}, using Theorem~\ref{thm:pfeigenval} for any $\Apset\subset \Omega$ we have,
	\begin{align}
	M_l(m,x;\Apset)/{\beta_0}^l =\underset{(k-1,z)\in\Apset}{\sum \int } \mu(m,x)\nu(k-1,z) \,dz + \frac{m}{{\beta_0}^2} \widehat{M}O({\gamma}^{l-2}).
	\end{align}
	Substituting the above equality in \cref{eq:relscrT}, after some simple algebra, we have
	\begin{align}
	&\Topt ^l F(\Apset_1;\Apset_2)/{\beta_0}^{2l} = \allowdisplaybreaks\\
	&\myquad[2]\left(\underset{(k_1-1,z_1)\in\Omega}{\sum \int } \, \underset{(k_2-1,z_2)\in\Omega}{\sum \int } \mu(k_1-1,z_1)\mu(k_2-1,z_2)\,dF(k_1-1,z_1;k_2-1,z_2) \right)\times \allowdisplaybreaks\\
	&\myquad[6]\left( \underset{(k-1,z)\in\Apset_1}{\sum \int } \nu(k-1,z) \,dz\right)\left( \underset{(k-1,z)\in\Apset_2}{\sum \int } \nu(k-1,z) \,dz\right) + C_F \,\widehat{M} O({\gamma}^{l-2}),
	\end{align}
	Now the result follows by combining \cref{eq:M_l2rec}, the above equality, and the following relation
	\begin{align}
	&\Topt ^{l-1} M_1(m,x;\Apset_1)M_1(m,x;\Apset_2)= M_{l}(m,x;\Apset_1)M_{l}(m,x;\Apset_2).
	\end{align}
	which can be proved using induction similar to \cref{eq:relscrT}.
	The constant $C_F$ depends on the choice of the function $F$. It is easy to check that for $v(m,x;\Apset_1,\Apset_2)$, we can replace $C_F \widehat{M}$ with $m^2 \xbar{M}$ for some $\xbar{M} > 0$ independent of $x$, $l$, $\Apset_1$, and $\Apset_2$ (notice that $\min(x,z)/x \leq 1$).
\end{proof}
\begin{remark}\label{rem:l+ltild}
	Fix the value of $\tilde{l} > 0$ and consider $\expect[Z_l(\Apset_1)Z_{l+\tilde{l}}(\Apset_2)\,\vert\, n_{\root} = m, v_{\root} = x]$. Using the same argument as above, the conditional expectation converges to the same value as in \cref{eq:ml2conv} with the error bounded by $m^2 \xbar{M}_{\tilde{l}}\, O({\gamma}^{l-2})$.
\end{remark}
Now, combining the above theorem and remark, we get a similar result as in~\cite[Theorem 14.1, page 72]{Harris1963}.
\begin{theorem}\label{thm:L2conv}
	Suppose that $\beta_0 > 1$. Let $\Apset \subset \Omega$ and set $W_l(\Apset) = Z_l(\Apset)/{\beta_0}^l$.
	Then, conditioned on $n_{\root} = m$ and $v_{\root} = x$, where $x\in\mathbb{R}_+$ and $m\in\mathbb{N}$, there is a random variable $W(\Apset)$ such that $W_l(\Apset)$ converges to $W(\Apset)$ in $L^2$ and almost surely.
	The first and the second moments of $W(\Apset)$ are,
	\begin{align}
	&\expect\left[W(\Apset)\,\vert\,n_{\root} = m,v_{\root} = x\right] = \frac{m}{x}f_0(x) \left(\underset{(k-1,z)\in\Apset}{\sum \int } P(k)\frac{\eexp^{-z}z^{k-1}}{(k-1)!} f_0(z)\,dz\right),\allowdisplaybreaks\\
	&\expect\left[\left(W(\Apset)\right)^2\,\vert\, n_{\root} = m,v_{\root} = x\right] = U(m,x)
	\left( \underset{(k-1,z)\in\Apset}{\sum \int } \nu(k-1,z) \,dz\right)^2,
	\end{align}
	where the function $U(x,m)$ is given by \cref{eq:U}. Furthermore, if $\Apset$ and $\Bpset$ are subsets of $\Omega$ such that $\underset{(k-1,z)\in\Apset}{\sum \int } \nu(k-1,z) \,dz > 0$, then
	\begin{align}
	W(\Bpset) = \frac{\underset{(k-1,z)\in\Bpset}{\sum \int } \nu(k-1,z) \,dz}{\underset{(k-1,z)\in\Apset}{\sum \int } \nu(k-1,z) \,dz}W(\Apset)\qquad \textit{a.s.}
	\end{align}
\end{theorem}
\begin{proof}
	Remark~\ref{rem:l+ltild} and equation \cref{eq:ml2conv} imply that $\expect[(W_l(\Apset) - W_{l+\hat{l}}(\Apset))^2 ] = m^2\allowbreak \xbar{M}_{\tilde{l}}\, O({\gamma}^{l-2})$. Hence, $\{W_l(\Apset)\}_{l}$ satisfies the Cauchy criteria and converges to $W(\Apset)$ in $L^2$. Since for any $\hat{l} > 0$
	\begin{align}
	\sum_{l=1}^\infty \expect[(W_l(\Apset) - W_{l+\hat{l}}(\Apset))^2 ] < \infty,
	\end{align}
	and $\{W_l(\Apset)\}_{l}$ converges to $W(\Apset)$ almost surely as well. Finally, the relation between $W(\Bpset)$ and $W(\Apset)$ follows by the following relation between $W_l(\Apset)$ and $W_l(\Bpset)$:
	\begin{align}
	\expect\left[\,\left(W_l(\Bpset) - \frac{\underset{(k-1,z)\in\Bpset}{\sum \int } \nu(k-1,z) \,dz}{\underset{(k-1,z)\in\Apset}{\sum \int } \nu(k-1,z) \,dz}W_l(\Apset)\right)^2\,\right] = m^2 \xbar{M}\, O({\gamma}^{l-2}),
	\end{align}
	which follows by Theorem~\ref{thm:ml2conv}. Notice that $\expect\left[W_l(\Bpset) W_l(\Apset)\right] = \expect[M_l^{(2)}(n_{\root},v_{\root} ;\Bpset,\Apset)/{\beta_0}^{2l}]$, $\expect\left[W_l(\Apset)^2\right] = \expect[M_l^{(2)}(n_{\root},v_{\root} ;\Apset,\Apset)/{\beta_0}^{2l}]$, and  $\expect\left[W_l(\Bpset)^2\right] = \expect[M_l^{(2)}(n_{\root},v_{\root} ;\Bpset,\Bpset)/{\beta_0}^{2l}]$.
\end{proof}
An immediate corollary of the above theorem and Corollary~\ref{cor:EZ_l} is the following, which connects the growth rate and the probability of extinction.
\begin{corollary}\label{cor:phasetransincomp}
	If $\beta_0 > 1$, then the probability of extinction is less than $1$. If $\beta_0 < 1$, then the probability of extinction equals $1$.
\end{corollary}
\begin{proof}
	By Theorem~\ref{thm:L2conv} if $\beta_0 > 1$, then $W(\Apset)$ is positive with non-zero probability. Hence, the probability of extinction is less than $1$. The second part follows by Markov inequality and Corollary~\ref{cor:EZ_l}.
\end{proof}

%% file: Sections/PropEWT/PhaseTrans.tex
To analyze the case of $\beta_0 = 1$ and to show that $Z_l \sim {\beta_0}^l\, W$ we need to show transience of $Z_n$, i.e., $Z_n$ either goes to zero or infinity. Consider the generalized Markov Chain introduced in Section~\ref{sec:back_Branch}. Recall that $Z_l(\Apset)$ is the number of vertices $(k-1,\zeta)\in\Apset$, and $Z_l(\Omega)$ is the total number of vertices in generation $l$. Notice that for any $\kappa\in\mathbb{N}$, the probability of extinction after $\kappa$ steps conditioned on $n_\root = m$, $v_\root = x$ can be arbitrary small when $m$ is large. As a result, the same proof technique as in~\cite[Theorem 11.2, page 69]{Harris1963} does not work in our problem setting. To show the transience of $Z_l$, more work needs to be done. The following lemma establishes the transience of $Z_l$. We follow the notation introduced in Section~\ref{sec:back_Branch}.
\begin{lemma}\label{lem:transZ}
	For all $k\geq 1$ and for all $\omega \in \pd_{\Omega}$ we have,
	\begin{align}
	\prob(\{0 < Z_l(\Omega) \leq k, \text{ infinitely often}\} ) = 0.
	\end{align}
\end{lemma}
\begin{proof}
	Define $\pd_{\Omega_0}$ to be the set of non-null point distributions with at most $k$ vertices, $$\pd_{\Omega_0} \coloneqq \{\omega\in \pd_{\Omega} \,\vert\,0<\omega(\Omega)\leq k \}.$$
	Let $\pd_{\Omega_0,m}$ be the set of point distributions $\omega\! =\! ( (m_1,x_1),a_1; (m_2,x_2),a_2;\dots;(m_{\widetilde{k}},x_{\widetilde{k}}),a_{\widetilde{k}} ) \in \pd_{\Omega_0}$ such that $m_i \leq m$ for all $i$.
	Recall that $\emptyset$ denotes the null point distribution.

	\textit{Step 1:} Using the same argument as in~\cite[Theorem 11.2, page 69]{Harris1963}, we show that $\prob(Z_l\in\pd_{\Omega_0,m}) = 0$. Define $R_m(\omega)$ for $\omega \in \pd_{\Omega}$ as follows:
	\begin{align}
	R_m(\omega) = \prob(\{ Z_l \in \pd_{\Omega_0,m}, \text{ infinitely often}\} \,\,\vert\, Z_0 = \omega ).
	\end{align}
	For $\pd \subset \pd_{\Omega_0,m}$ let $Q_{m,2}(\omega,\pd)$ be the conditional probability that, conditioned on $Z_0 = \omega$, at least one of the point distributions $Z_2$, $Z_3$, $\cdots$ are in $\pd_{\Omega_0,m}$ and if $Z_l$ is the first one, then $Z_l\in\pd$. Then
	\begin{align}
	R_m(\omega) = \int_{\pd_{\Omega_0,m}} R_m(\omega^\prime) dQ_{m,2}(\omega,\omega^\prime).
	\end{align}
	Let $\xbar{R}_m\coloneqq \sup_{\omega \in \pd_{\Omega_0,m}} R_m(\omega)$. We have
	\begin{align}
	R_m(\omega) \leq \xbar{R}_m \int_{\pd_{\Omega_0,m}} dQ_{m,2}(\omega,\omega^\prime) = \xbar{R}_m\, Q_{m,2}(\omega,\pd_{\Omega_0,m}). \label{eq:R_mstep1}
	\end{align}
	In the proof of Theorem~\ref{thm:probext}, we show that, if $Z_0 = (m,x)$, the probability of extinction after 2 generations is given by $\left(T^2(\boldsymbol{0})(x)\right)^{m}$. Recall that $T^2(\boldsymbol{0})(\cdot)$ is a decreasing and strictly positive function. Hence,
	\begin{align}
	Q_{m,2}(\omega,\pd_{\Omega_0,m}) \leq 1 - \prod_{i=1}^{\widetilde{k}}\left(T^2(\boldsymbol{0})(x_i)\right)^{a_im_i} \leq 1 - \left(T^2(\boldsymbol{0})(0)\right)^{mk} < 1 - \epsilon.\label{eq:Qm2bound}
	\end{align}
	where $\omega = ((x_1,m_1),a_1;(x_2,m_2),a_2;\cdots;(x_{\widetilde{k}},m_{\widetilde{k}}),a_{\widetilde{k}})$ and $\epsilon > 0$ is a constant which depends only on $m$ and $k$. Contradiction follows by taking supremum from both sides of \cref{eq:R_mstep1}.
	\begin{remark}
		In \textit{Step 1}, we proved that the probability of the event $\{ Z_l \in \pd_{\Omega_0,m}, \text{ infinitely }\allowbreak\text{often}\}$ is zero. However (as $\pd_{\Omega_0,m} = \cup _m \pd_{\Omega_0,m}$) this approach cannot rule out the possibility of the event $\{ Z_l \in \pd_{\Omega_0}, \text{ infinitely }\allowbreak\text{often}\}$. As an example, there might be a sequence $\{l_i\}_{i=1}^{\infty}$ such that $Z_{l_i} \in \pd_{\Omega_0,m_{l_i}}$ where $m_{l_1}< m_{l_2}<m_{l_3}<\cdots$. In \textit{Step 2}, we will prove that such sequences are unlikely.
	\end{remark}
	\textit{Step 2}: For the sake of notational simplicity, we prove the result for $k=1$, and then discuss the general case. Fix the value of $m$. Notice that by the first step, the probability of hitting $\pd_{\Omega_0,m}$ infinitely often is zero. Hence, we need to show that the probability of hitting $\widetilde{\pd}_{\Omega_0,m}\coloneqq \pd_{\Omega_0} \setminus \pd_{\Omega_0,m}$ infinitely often is zero to complete the proof.
	\begin{remark}\label{rem:altertrankernal}
		By conditional independence, the transition kernel of the generalized Markov chain from the point distribution $\omega = ( (m_1,x_1),a_1; (m_2,x_2),a_2;\dots;(m_{\widetilde{k}},x_{\widetilde{k}}),a_{\widetilde{k}} ) \in \pd_{\Omega_0}$ is exactly the same as the transition kernel from the point distribution $\omega = ( (1,x_1),m_1\times a_1; (1,x_2),m_2\times a_2;\dots;(1,x_{\widetilde{k}}),m_{\widetilde{k}}\times a_{\widetilde{k}} )$ which may or may not be in $\pd_{\Omega_0}$.
	\end{remark}
	Assume $k = 1$ and let $\kappa \in \mathbb{N}$. Define $\widetilde{Q}_{m,\kappa}(\omega,\pd)$ and $\widetilde{R}_m(\omega)$ similar to $Q_{m,2}(\omega,\pd)$ and $R_m(\omega)$ by considering the set $\widetilde{\pd}_{\Omega_0,m}$ instead of $\pd_{\Omega_0,m}$. Specifically, let $\widetilde{Q}_{m,\kappa}(\omega,\pd)$ be the conditional probability that, conditioned on $Z_0 = \omega$, at least one of the point distributions $Z_\kappa, Z_{\kappa+1},Z_{\kappa+2},\cdots$ are in $\widetilde{\pd}_{\Omega_0,m}$ and if $Z_l$ is the first one, then $Z_l \in \pd$. Similarly, define $\widetilde{R}_m(\omega)$ for $\omega \in \widetilde{\pd}_{\Omega}$ as follows:
	\begin{align}
		\widetilde{R}_m(\omega) = \prob(\{ Z_l \in \widetilde{\pd}_{\Omega_0,m}, \text{ infinitely often}\} \,\,\vert\, Z_0 = \omega ).
	\end{align}
	Assume $Z_0 = \omega = (m_1,x_1)$, where $m_1 \geq m$. Notice that the first time $Z_l \in \widetilde{\pd}_{\Omega_0,m}$ for some $l>0$, $m_1-1$ out of $m_1$ branches of $Z_0$ go extinct. Hence, by Remark~\ref{rem:altertrankernal}
	\begin{align}\label{eq:Zntransk=1}
	\widetilde{R}_m(\omega) = m_1 \left(q(x_1)\right)^{m_1 -1 } \widetilde{R}_m(\widehat{\omega})
	\end{align}
	where $q(\cdot)$ is the smallest fixed point of the operator $T$ defined in Theorem~\ref{thm:probext}, and $\widehat{\omega} = (1,x_1)$ is a point distribution with only one point of type $(1,x_1)$. Next, using the same argument as in \textit{step $1$}, we have
	\begin{align}
	\widetilde{R}_m(\widehat{\omega}) \leq \xbar{\widetilde{R}}_m \, \widetilde{Q}_{m,\kappa}(\widehat{\omega},\widetilde{\pd}_{\Omega_0,m}) \leq \xbar{\widetilde{R}}_m (1 - T^\kappa(\boldsymbol{0})(x_1)),
	\end{align}
	where $\xbar{\widetilde{R}}_m\coloneqq \sup_{\omega \in \pd_{\Omega_0,m}} \widetilde{R}_m(\omega)$. Notice that $\widetilde{R}_m(\widehat{\omega})$ does not depend on $\kappa$. Now, if we take $\kappa$ to infinity, we have
	\begin{align} \label{eq:Zntransk=1_tmp}
	\widetilde{R}_m(\widehat{\omega}) \leq \xbar{\widetilde{R}}_m \, (1-q(x_1)).
	\end{align}
	Combining \cref{eq:Zntransk=1} and \cref{eq:Zntransk=1_tmp}, and taking supremum with respect to $\omega$, the result follows by the fact that $ m_1 \left(q(x_1)\right)^{m_1 -1 }(1-q(x_1)) = Bi(1\,;m_1,1-q(x_1)) \leq 0.5$ for $m_1\geq m >1$.

	Now consider the case $k = 2$ and pick $\omega = ( (m_1,x_1),a_1; (m_2,x_2),a_2)$, where $a_1 + a_2 \leq 2$. Assume $\widetilde{R}_m(\omega) > 0$. Notice that the point distribution $\omega$ has $m_1 a_1 + m_2 a_2$ potential branches dangling from it. Conditioned on $\{ Z_l \in \widetilde{\pd}_{\Omega_0,m}, \text{ infinitely often}\} \cap \{Z_0 =\omega\}$, if two of these potential branches survive, do not go extinct at all, then by Remark~\ref{rem:altertrankernal} we have for some $i$
	\begin{align}
	 \prob(\{ Z_l(\Omega) = 1, \text{ infinitely often}\} \,\,\vert\, Z_0 = (1,x_i) ) > 0,
	\end{align}
	which is a contradiction. That is to say, conditioned on $\{ Z_l \in \widetilde{\pd}_{\Omega_0,m}, \text{ infinitely often}\}\cap \{Z_0 =\omega\}$, if two out of $m_1 a_1 + m_2 a_2$ potential branches survives with positive probability then each one will hit point distributions with only one vertex infinitely often.

	 Hence, only one of these branches can survive. Following a similar logic as before, we have
	\begin{align}\label{eq:Zntransk=2}
	\widetilde{R}_m(\omega) = m_1 \left(q(x_1)\right)^{a_1m_1 -1 } \left(q(x_2)\right)^{a_2 m_2} \widetilde{R}_m(\widehat{\omega}_1) + m_2 \left(q(x_1)\right)^{a_1 m_1 } \left(q(x_2)\right)^{a_2 m_2 - 1} \widetilde{R}_m(\widehat{\omega}_2)
	\end{align}
	where $\widehat{\omega}_1 = (1,x_2)$ and $\widehat{\omega}_2 = (1,x_2)$. The result follows by same argument using \cref{eq:Zntransk=1_tmp}.
	The exact same argument holds for any $k>2$ as well, and we get similar relation as \cref{eq:Zntransk=2}.
\end{proof}
The above lemma, together with Corollary~\ref{cor:phasetransincomp} and Corollary~\ref{cor:EZ_l}, have an important implication that completes the connection between the probability of extinction and the growth rate.
\begin{corollary}\label{cor:beta0_probextinc}
	If $\beta_0 > 1$ then the probability of extinction is less than $1$. If $\beta_0 \leq 1$, then the probability of extinction equals $1$.
\end{corollary}

%% file: Sections/PropEWT/ProbExtRev.tex
To show that growth rate of $Z_n$ is $\beta_0$ when $\beta_0 > 1$, i.e., $Z_n \sim {\beta_0}^n\, W$, we need to show that $\prob(W = 0 \,\vert\, Z_n \to \infty) =0$. As Harris points out in~\cite[Remark 1, page 28]{Harris1963}, if there is a positive probability that $Z_n \to \infty$ at a rate less than $\beta_0$, then $\prob(W = 0\,\,\vert\,Z_n \to \infty) > 0$. To rule out such a scenario, we need to show that $\prob(W = 0\,\vert\,n_{\root} = m, v_{\root} = x) = \prob(\{\text{extinction}\}\,\vert\,n_{\root} = m, v_{\root} = x) = \left(q(x)\right)^m$, where $q(\cdot)$ is given by Theorem~\ref{thm:probext}. In fact, it is easy to see that $\prob(W = 0\,\vert\,n_{\root} = 1, v_{\root} = x)$ is a fixed point of the operator $T$. However, to complete the argument, we need to show that $T(\cdot)$ does not have any fixed point other than $q(\cdot)$ and $\boldsymbol{1}(\cdot)$.

Using the point process perspective, we can rewrite the operator $T$ as follows:
\begin{align}
T(f)(x) = \int\limits_{ \omega = (m_1,x_1)\in \pd_{\Omega} \text{ or } \omega = \emptyset} \left(f(x_1)\right)^{m_1} dP_{{\omega}_0}^{(1)}(\omega)
\end{align}
where $\omega_0 = (1,x)$ is the type of the root vertex and $P_{{\omega}_0}^{(1)}$ is the one step transition probability defined in Section~\ref{sec:back_Branch}. For ease of representation, we define
\begin{align}\label{eq:omegaemptydef}
\int\limits_{ \omega = \emptyset} \left(f(x_1)\right)^{m_1} dP_{{\omega}_0}^{(1)}(\omega) \coloneqq \int\limits_{ \omega = \emptyset} dP_{{\omega}_0}^{(1)}(\omega) = \prob(Z_1(\Omega) = 0\,\vert\,Z_0 = (1,x)).
\end{align}
Inductively, using the same argument as in Remark~\ref{rem:altertrankernal} we have
\begin{align}\label{eq:Trewrite}
T^l(f)(x) = \int\limits_{ \omega = ((m_1,x_1),a_1;\cdots;(m_k,x_k),a_k)\in \pd_{\Omega}, \,k\in\mathbb{Z}_+} \left(f(x_1)\right)^{m_1 a_1}\cdots \left(f(x_k)\right)^{m_k a_k} dP_{{\omega}_0}^{(l)}(\omega),
\end{align}
where by $k=0$ we mean $\omega = \emptyset$ which follows the same definition as in \cref{eq:omegaemptydef}.
The above equality, combined with an appropriate test function, becomes a powerful tool to study the properties of the operator $T$ and the branching process in general. Recall that $q(\cdot)$ is the smallest fixed point of the operator $T$ (Theorem~\ref{thm:probext}).
\begin{lemma}\label{lem:uniq}
	If $\beta > 1$, then the operator $T$ has two fixed points, one of which is $q(\cdot)$, and the other one is $\boldsymbol{1}(\cdot)$. Moreover, for any function $f\in L(\mathbb{R}_+;[0,1])$ such that the Lebesgue measure of the set $\{x\in\mathbb{R}_+: f(x) < 1\}$ is positive, for all $x\in\mathbb{R}_+$ we have $T^l(f)(x) \to q(x)$.
\end{lemma}
\begin{proof}
	Consider the function $\xbar{f}_{x_0,\epsilon}(\cdot)$ defined as follows
	\begin{align}
	\xbar{f}_{x_0,\epsilon}(x) \coloneqq \begin{cases}
	\epsilon & \text{if } x\leq x_0 \\
	1 & \text{otherwise}
	\end{cases}.
	\end{align}
	The goal is to show that for every large enough $x_0$, there is an $\epsilon > 0$ such that for all $x \in \mathbb{R}_+$, $T\left(\xbar{f}_{x_0,\epsilon}\right)(x) \leq q(x)$, where $q(\cdot)$ is the smallest fixed point of the operator $T$. One important implication of this inequality is:
	\begin{align}\label{eq:Tfbar}
	\forall x\in\mathbb{R}_+, \qquad\lim_{l\to\infty} T^l\left(\xbar{f}_{x_0,\epsilon}\right)(x) = q(x).
	\end{align}
	Notice that for $x>0$,
	\begin{align}
	q(x) - T\left(\xbar{f}_{x_0,\epsilon}\right)(x)& = \frac{1}{x} \Bigg(\sum_{k=1}^{\infty} P(k) \int_{z=0}^{x_0} \frac{\eexp^{-z}z^{k-1}}{(k-1)!}\min(x,z)(\left(q(z)\right)^{k-1} - \epsilon^{k-1})\,dz\allowdisplaybreaks\\
	& \qquad\qquad-\sum_{k=1}^{\infty} P(k)\int_{z=x_0}^{\infty} \frac{\eexp^{-z}z^{k-1}}{(k-1)!}\min(x,z)(1 - \left(q(z)\right)^{k-1} )\,dz \Bigg) \allowdisplaybreaks\\
	&\geq \frac{1}{x}\Bigg(\sum_{k=1}^{\infty} P(k) (\left(q(0)\right)^{k-1} - \epsilon^{k-1})\int_{z=0}^{x_0} \frac{\eexp^{-z}z^{k-1}}{(k-1)!}\min(x,z)\,dz\allowdisplaybreaks\\
	& \qquad\qquad-\sum_{k=1}^{\infty} P(k)\int_{z=x_0}^{\infty} \frac{\eexp^{-z}z^{k-1}}{(k-1)!}\min(x,z)\,dz \Bigg)
	\end{align}
	By choosing $x_0$ to be large enough, we can make the second term in the parenthesis arbitrarily small. Fixing $x_0$, we can choose $\epsilon > 0$ to be small enough such that for all $x\in\mathbb{R}_+$ the inequality $q(x) - T\left(\xbar{f}_{x_0,\epsilon}\right)(x) > 0$ holds. Notice that $q(0) > 0$ and $q(x)$ is a strictly increasing function.

	Given \cref{eq:Tfbar}, we next use the alternative representation of $T^l\left(\xbar{f}_{x_0,\epsilon}\right)$ as in \cref{eq:Trewrite} to prove the lemma. Define the set $\pd_{M}$ as
	\begin{align}
		\pd_M \coloneqq \left\{w = ((m_1,x_1),a_1;(m_2,x_2),a_2;\cdots;(m_k,x_k),a_k),\, k\in\mathbb{Z}_+ \,\Bigg\vert\, \sum_{i \text{ s.t. } x_i \leq x_0} m_i a_i < M \right\}.
	\end{align}
	Now, we have
	\begin{align}
	T^l\left(\xbar{f}_{x_0,\epsilon}\right)(x)&= \prob(Z_l = 0 \,\vert\, Z_0 = (1,x) ) \\
	&\qquad + \int\limits_{\omega\in\pd_M} {\epsilon}^{\sum m_i a_i} dP_{{\omega}_0}^{(l)}(\omega_1)\allowdisplaybreaks\\
	&\qquad + \int\limits_{\omega\in \pd_{\Omega} \setminus \pd_M} \left(f(x_1)\right)^{m_1 a_1} \left(f(x_2)\right)^{m_2 a_2} \cdots \left(f(x_k)\right)^{m_k a_k}dP_{{\omega}_0}^{(l)}(\omega_1).
	\end{align}
	Notice that,
	\begin{align}
	\int\limits_{\omega\in\pd_M} {\epsilon}^{\sum m_i a_i} dP_{{\omega}_0}^{(l)}(\omega_1) \geq \epsilon^M \prob(Z_l \in \pd_M\,\vert\,Z_0 = (1,x))
	\end{align}
	However, by \cref{eq:Tfbar}, the left-hand side of the above inequality goes to $0$ as $l$ goes to infinity. Hence, $$\prob(Z_l \in \pd_M\,\vert\,Z_0 = (1,x)) \to 0\text{ as }l\to \infty.$$

	For the sake of contradiction, assume that $T$ has another fixed point $\widetilde{q}(\cdot)$. By Lemma~\ref{lem:propt}, we already know that $\forall x\in\mathbb{R}_+,\, q(x) < \widetilde{q}(x) < 1$ and that $\widetilde{q}(\cdot)$ is strictly increasing. Notice that,
	\begin{align}
	\widetilde{q}(x) = T^l(\widetilde{q})(x) &= \prob(Z_l = 0 \,\vert\, Z_0 = (1,x) ) \\
	&\qquad + \int\limits_{\omega\in\pd_M} \left(\widetilde{q}(x_1)\right)^{m_1 a_1} \left(\widetilde{q}(x_2)\right)^{m_2 a_2} \cdots \left(\widetilde{q}(x_k)\right)^{m_k a_k} dP_{{\omega}_0}^{(l)}(\omega_1)\allowdisplaybreaks\\
	&\qquad + \int\limits_{\omega\in \pd_{\Omega} \setminus \pd_M} \left(\widetilde{q}(x_1)\right)^{m_1 a_1} \left(\widetilde{q}(x_2)\right)^{m_2 a_2} \cdots \left(\widetilde{q}(x_k)\right)^{m_k a_k}dP_{{\omega}_0}^{(l)}(\omega_1).
	\end{align}
	As $l$ goes to infinity, the first term converges to $q(x)$. Using the analysis of $\xbar{f}_{x_0,\epsilon}$, the second term goes to 0 since
	\begin{align}
	\int\limits_{\omega\in\pd_M} \left(\widetilde{q}(x_1)\right)^{m_1 a_1} \left(\widetilde{q}(x_2)\right)^{m_2 a_2} \cdots \left(\widetilde{q}(x_k)\right)^{m_k a_k} dP_{{\omega}_0}^{(l)}(\omega_1) \leq \prob(Z_l \in \pd_M\,\vert\,Z_0 = (1,x)) \xrightarrow{l\to \infty}0.
	\end{align}
	Finally, we can bound the third term as follows,
	\begin{align}
	&\int\limits_{\omega\in \pd_{\Omega} \setminus \pd_M} \left(\widetilde{q}(x_1)\right)^{m_1 a_1} \left(\widetilde{q}(x_2)\right)^{m_2 a_2} \cdots \left(\widetilde{q}(x_k)\right)^{m_k a_k}dP_{{\omega}_0}^{(l)}(\omega_1) \\
	&\myquad[8]\leq \left(\widetilde{q}(x_0)\right)^{M}\prob(Z_l \in \pd_M\,\vert\,Z_0 = (1,x)) \leq \left(\widetilde{q}(x_0)\right)^{M}
	\end{align}
	since $\widetilde{q}(\cdot)$ is non-decreasing. Combining these inequalities, we have,
	\begin{align}
	\widetilde{q}(x) = \lim_{l\to\infty}T^l(\widetilde{q})(x) \leq q(x) + \left(\widetilde{q}(x_0)\right)^{M}
	\end{align}
	The result follows by letting $M$ increase to infinity.

	Finally, if $f\in L(\mathbb{R}_+;[0,1])$ such that the Lebesgue measure of the set $\{x\in\mathbb{R}_+: f(x) < 1\}$ is positive, then by same analysis and the fact that $\forall x\in\mathbb{R}_+,\, T(f)(x) < 1$, we have
	\begin{align}
	\forall x\in\mathbb{R}_+,\qquad\lim_{l\to\infty}T^l(f)(x) = q(x).
	\end{align}
 This completes the proof.
\end{proof}
As we pointed out in Section~\ref{sec:SecMom}, one implication of the above lemma is $Z_n\sim {\beta_0}^n W$.
\begin{theorem}\label{thm:growthrate_aux}
	If $\beta_0 > 1$, then the growth rate of $Z_n$ is $\beta_0$, i.e., $\prob(W=0\,\vert\, Z_n\to\infty) = 0$. Moreover, conditioned on $Z_n\to\infty$, the proportions of different types converges to a constant.
\end{theorem}
\begin{proof}
	Let $f(x) = \prob(W = 0\,\vert\,Z_0 = (1,x))$. Notice that,
	\begin{align}
		\prob(W = 0\,\vert\,Z_0 = (m,x)) = \left(f(x)\right)^m,
	\end{align}
	and
	\begin{align}
	\prob(W = 0\,\vert\,Z_0 = (1,x)) = \int\limits_{ \omega = (m_1,x_1)\in \pd_{\Omega}\text{ or } \omega = \emptyset} \left(f(x_1)\right)^{m_1} dP_{{\omega}_0}^{(1)}(\omega).
	\end{align}
	Hence, $f(x)$ is a fixed point of the operator $T$. On the other hand, by Theorem~\ref{thm:L2conv}, for all $x\in\mathbb{R}_+$, we have $\prob(W = 0\,\vert\,Z_0 = (1,x)) < 1$. Hence, by Lemma~\ref{lem:uniq}, $f(\cdot) = q(\cdot)$, and in particular $\prob(W = 0) = \prob(\{\text{extinction}\})$. Now, the result follows by the law of total probability. The second part is a simple corollary of the first part and Theorem~\ref{thm:L2conv}.
\end{proof}

%% file: Sections/PropEWT/DegreeDistRev.tex
Let $D_l$ denote the number of descendants of a randomly selected vertex in generation $l$ given the number of vertices in generation $l$ is positive. In terms of expectation, for any $d \in\mathbb{Z}_+$, we have
\begin{align}
	\prob(D_l = d\,\vert\,Z_l(\Omega) > 0) &= \expect\left[ \frac{\sum_{\boldsymbol{i} = (i_1,i_2,\cdots,i_l) \in \mathbb{N}^l}\onefunc\left( \{\boldsymbol{i} \in \rtree\} \cap \{D_{\boldsymbol{i}} = d\}\right)}{Z_l(\Omega)}\,\bigg\vert\,Z_l(\Omega) > 0\right]\\
	&=\sum_{\boldsymbol{i} = (i_1,i_2,\cdots,i_l) \in \mathbb{N}^l} \expect\left[ \frac{\onefunc\left( \{\boldsymbol{i} \in \rtree\} \cap \{D_{\boldsymbol{i}} = d\}\right)}{Z_l(\Omega)}\,\bigg\vert\,Z_l(\Omega) > 0\right]
\end{align}
where $\rtree$ is the EWT as defined in Section~\ref{sec:intro}, and $Z_l(\Omega)$ is the total number of vertices in generation $l$ of $\rtree$, following the notation of Section \ref{sec:SecMom}. Notice that conditioned on the type of vertex $\boldsymbol{i}\in\mathbb{N}^l$, the event $\{D_{\boldsymbol{i}} = d\}$ is independent of the event $\{\boldsymbol{i} \in \rtree\}$ and the value of $Z_l(\Omega)$. Hence, for any $\boldsymbol{i}\in\mathbb{N}^l$, we have
\begin{align}
	&\expect\left[ \frac{\onefunc\left( \{\boldsymbol{i} \in \rtree\} \cap \{D_{\boldsymbol{i}} = d\}\right)}{Z_l(\Omega)} \,\bigg\vert\,Z_l(\Omega) > 0\right] \\
    &\myquad[8]= \expect\left[ \expect\left[\frac{\onefunc\left( \{\boldsymbol{i} \in \rtree\} \cap \{D_{\boldsymbol{i}} = d\}\right)}{Z_l(\Omega)}\,\vert\,n_{\boldsymbol{i}}, v_{\boldsymbol{i}} \right] \,\bigg\vert\,Z_l(\Omega) > 0\right] \\
	&\myquad[8]= \expect\left[ \frac{\onefunc\left( \{\boldsymbol{i} \in \rtree\}\right)}{Z_l(\Omega)}\times Bi\left(d;n_{\boldsymbol{i}},\int_{0}^{v_{\boldsymbol{i}}} \frac{1}{v_i}\sum_{k'=1}^{\infty}P(k')\bar{F}_{k'}(y)\,dy\right) \,\bigg\vert\,Z_l(\Omega) > 0\right]
\end{align}
where the last equality follows by Theorem~\ref{thm:ConDegree}. Fix $(k-1,z) \in \Omega$ and let $\delta > 0$ to be small enough. Suppose that $\beta_0 > 1$. By the above equality and simple algebra, we have
\begin{align}
	&\expect\left[ \frac{\sum_{\boldsymbol{i} = (i_1,i_2,\cdots,i_l) \in \mathbb{N}^l} \onefunc\left( \{\boldsymbol{i} \in \rtree\} \cap \{D_{\boldsymbol{i}} = d\} \cap \{n_{\boldsymbol{i}} = k-1, v_{\boldsymbol{i}}\in (z,z+\delta] \} \right) }{Z_l(\Omega)} \,\bigg\vert\,Z_l(\Omega) > 0\right] \allowdisplaybreaks\\
	&\myquad[4] = \expect\left[ \frac{ \sum_{\boldsymbol{i} = (i_1,i_2,\cdots,i_l) \in \mathbb{N}^l} \onefunc\left( \{\boldsymbol{i} \in \rtree\} \cap \{n_{\boldsymbol{i}} = k-1, v_{\boldsymbol{i}}\in (z,z+\delta] \} \right)}{Z_l(\Omega)} \,\bigg\vert\,Z_l(\Omega) > 0\right]\times \\
	&\myquad[20]\left(Bi\left(d;k-1,\int_{0}^{z} \frac{1}{z}\sum_{k'=1}^{\infty}P(k')\bar{F}_{k'}(y)\,dy\right) + O(\delta)\right)\allowdisplaybreaks\\
	&\myquad[4] = \expect\left[ \frac{Z_l\left((k-1,(z,z+\delta])\right)}{Z_l(\Omega)} \,\bigg\vert\,Z_l(\Omega) > 0 \right]\times \allowdisplaybreaks\\
    &\myquad[10]\left(Bi\left(d;k-1,\int_{0}^{z} \frac{1}{z}\sum_{k'=1}^{\infty}P(k')\bar{F}_{k'}(y)\,dy\right) + O(\delta)\right)\allowdisplaybreaks\\
	&\myquad[4] \xrightarrow{\text{as $l\to\infty$}} \frac{\int_z^{z+\delta} \nu(k-1,z') \,dz'}{\underset{(k'-1,z')\in\Omega}{\sum \int } \nu(k'-1,z') \,dz'}\times \left(Bi\left(d;k-1,\int_{0}^{z} \frac{1}{z}\sum_{k'=1}^{\infty}P(k')\bar{F}_{k'}(y)\,dy\right) + O(\delta)\right)
\end{align}
where the convergence follows by Theorem~\ref{thm:L2conv} and~\ref{thm:growthrate_aux}. Using the above analysis and following simple algebraic manipulation, we get the following characterization of the asymptotic degree distribution.
\begin{theorem}\label{thm:asymdegdist}
	Suppose that $\beta_0 > 1$. Let $D_l$ denote the number of descendants of a randomly selected vertex in generation $l$ given the number of vertices in generation $l$ is positive. For any $d\in\mathbb{Z}_+$, we have
	\begin{align}
		\lim_{l \to\infty} \prob(D_l = d \,\vert\, Z_l(\Omega) > 0 ) = \dfrac{\underset{(k-1,z)\in\Omega}{\sum \int } \nu(k-1,z) \times Bi\left(d;k-1,\int_{0}^{z} \frac{1}{z}\sum_{k'=1}^{\infty}P(k')\bar{F}_{k'}(y)\,dy\right)\,dz}{\underset{(k-1,z)\in\Omega}{\sum \int } \nu(k-1,z) \,dz}.
	\end{align}
\end{theorem}

%% file: Sections/Background_RandomGraph.tex
We start with a few graph terminologies that are used in the chapter. Let $G=(V,E)$ denote an undirected graph, where $V$ is the set of vertices (finite or countably infinite), and $E$ is the set of edges. A rooted graph $\G=(V,E,\root)$ is a graph with a distinguished vertex $\root\in V$. Vertices $v_1,v_2\in V$ are said to be neighbors, if $\{v_1,v_2\}\in E$. The degree of a vertex $v\in V$, denoted by $d_v$, is the number of its neighbors. A graph $G$ is said to be locally-finite if the degree of each vertex is finite. A path $p$ of length $n-1$ is an ordered sequence of vertices $(v_1,v_2,\dots,v_n)$ where $\{v_i,v_{i+1}\}\in E,~\forall i<n$. A graph $G$ is said to be connected if there is a path between every pair of vertices.

Two graphs $G=(V,E)$ and $G^\prime=(V^\prime,E^\prime)$ are said to be isomorphic if there is a bijection $\sigma$ from $V$ to $V^\prime$ such that $\{v_1,v_2\}\in E$ if and only if $\sigma(\{v_1,v_2\})\coloneqq\{\sigma(v_1),\sigma(v_2)\}\in E^\prime$. The function $\sigma$ is called an isomorphism from $G$ to $G^\prime$. A rooted-isomorphism between two rooted graphs is an isomorphism that maps the root vertices to each other.

A network $N=(V,E,w_v,w_e)$ is a graph $(V,E)$ with mark functions $w_v:V\to \Omega_1$ and $w_e:E\to \Omega_2$, where $\Omega_1$ and $\Omega_2$ are the mark spaces. A rooted network is a network with a distinguished vertex as the root vertex. In this chapter, the mark spaces are assumed to be $\Omega_1 = \mathbb{N}\times \mathbb{R}_+$ and $\Omega_2 = \mathbb{R}_+$, which are complete separable metric spaces equipped with the following metrics,
\begin{align}
	&\forall m,n\in \mathbb{N},~\forall x,y\in \mathbb{R}_+,&&d_{\Omega_1}\left((m,x),(n,y)\right) = \sqrt{(m-n)^2+(x-y)^2},\\
	&\forall x,y\in \mathbb{R}_+,&&d_{\Omega_2}\left(x,y\right) = |x-y|.
\end{align}

Two networks $N$ and $N^\prime$ are said to be isomorphic if there is a bijection map from $V$ to $V^\prime$ that preserve the edges as well as the marks. A rooted-isomorphism between two rooted networks $\N$ and $\N^\prime$ is an isomorphism that maps the root of one network to the other. For a rooted network $\N=(V,E,\root,w_v,w_e)$, $[\N]$ denotes the class of rooted networks that are isomorphic to $\N$. Let $G_*(\Omega_1,\Omega_2)$ denote the set of all isomorphism classes $[\N]$, where $\N$ ranges over all connected locally-finite rooted networks with mark spaces $\Omega_1$ and $\Omega_2$. For notational simplicity, we use $G_*$ instead of $G_*(\Omega_1,\Omega_2)$.

There is a natural way to define a metric on $G_*$. Consider a connected rooted network $\N=(V,E,\root,w_v,w_e)$\footnote{Strictly speaking, $\N$ is a member of the equivalence class $[\N]$.} and the corresponding rooted graph $\G = (V,E,\root)$. The depth of a vertex $v\in V$ is defined to be the infimum length of the paths from $v$ to the root vertex. Let $(\G)_t = (V_t,E_t,\root)$ denote the subgraph of $\G$ where $V_t$ is the set of vertices in $V$ at a depth less than or equal to $t$ from $\phi$, and $E_t$ is the set of edges in $E$ between the vertices in $V_t$. For any $[\N],[\N^\prime]\in G_*$, a natural way to define a distance is given by
\begin{align}
d_{G_*}([\N],[\N^\prime]) = \frac{1}{R+1},
\end{align}
where
\begin{align}
	R = \sup\left\{t\geq 0: \begin{minipage}{0.75 \textwidth}
	there exists a rooted-isomorphism $\sigma$ from $(\G)_t$ to $(\G^\prime)_t$ such that $\forall \mathcalboondox{v}\in V_t \text{ and } \forall \mathcalboondox{e}\in E_t$, $d_{\Omega_1}(w_v(\mathcalboondox{v}),w_v^\prime(\sigma(\mathcalboondox{v}))) < t^{-1}$ and $d_{\Omega_2}(w_e(\mathcalboondox{e}),w_e^\prime(\sigma(\mathcalboondox{e}))) < t^{-1}$
	\end{minipage}\right\}.
\end{align}
Notice that in the definition of $R$, the isomorphism is between the rooted graphs and not the corresponding rooted networks.
The space $G_*$ equipped with $d_{G_*}$ is a complete separable metric (Polish) space~\cite{Bordenave2016}. Define $\mathcal{P}(G_*)$ as the set of all probability measures on $G_*$ and endow this space with the topology of weak convergence. Since $G_*$ is a Polish space, the space $\mathcal{P}(G_*)$ is a Polish space as well~\cite{Bordenave2016} with the L\'evy-Prokhorov metric.

The members of $G_*$ are unlabeled connected locally-finite rooted networks; however, there is a way to generalize the framework to unrooted, not necessarily connected, finite networks. Consider a finite network $N=(V,E,w_v,w_e)$. For every vertex $v\in V$, define $N(v)$ to be the connected component of the vertex $v$ in the network $N$. Let $\N(v)$ denote the rooted version of $N(v)$, rooted at $v$, and define $\delta_{[\N(v)]} \in \mathcal{P}(G_*)$ to be the Dirac measure that assigns $1$ to $[\N(v)]$ and $0$ to any other member of $G_*$. Define $U(N)\in \mathcal{P}(G_*)$ as follows,
\begin{align}\label{eq:unifmeasure}
	U(N) = \frac{1}{|V|}\sum_{v\in V} \delta_{[\N(v)]}.
\end{align}
The probability measure $U(N)$ is the law of $[\N(\root)]$, where $\root\in V$ is picked uniformly at random. This probability measure captures the local structure of $N$ as viewed from a randomly chosen vertex. The notion of local weak convergence studies the weak limit of $\{U(N_n)\}_{n\in\mathbb{N}}$ for a sequence of finite networks $\{N_n\}_{n\in\mathbb{N}}$.
\begin{definition}({\bf Local Weak Limit})
	 A sequence of finite networks $\{N_n\}_{n\in\mathbb{N}}$ has a {\em local weak limit} $\rho\in \mathcal{P}(G_*)$ if $U(N_n) \xrightarrow{w} \rho$. 
\end{definition}
A necessary condition for a probability measure $\rho \in \mathcal{P}(G_*)$ to be a local weak limit is {\it unimodularity}~\cite{Aldous2007}, which is defined next.
Let $G_{**}(\Omega_1,\Omega_2)$, or more simply $G_{**}$ denote the set of isomorphism classes of connected locally-finite networks with an ordered pair of distinct vertices. Let $N_{\circ\circ}(\root,v)$ denote a network in $G_{**}$. Equip $G_{**}$ with the natural metric $d_{G_{**}}$ which is defined in the same way as $d_{G_*}$.
\begin{definition}({\bf Unimodularity})\label{def:unimod}
	A measure $\rho\in \mathcal{P}(G_*)$ is said to be {\it unimodular} if for all Borel functions $f:G_{**}\to \mathbb{R}_+$,
	\begin{align}\label{eq:unimod}
	\int \sum_{v\in V} f([N_{\circ\circ}(\root,v)])\, d\rho([\N(\root)]) = \int \sum_{v\in V} f([N_{\circ\circ}(v,\root)])\, d\rho([\N(\root)]).
	\end{align}
\end{definition}
The function $f$ in the definition of unimodularity ranges over all Borel functions from $G_{**}$ to $\mathbb{R}_+$; however, it is sufficient to consider Borel functions $f:G_{**}\to\mathbb{R}$ that assign a non-zero value to a doubly rooted network only if the roots are adjacent. This property is known as involution invariance~\cite{Aldous2007}.
\begin{lemma}({\bf Involution Invariance})\label{lem:involution}
	A measure $\rho\in\mathcal{P}(G_*)$ is unimodular if and only if the equality \cref{def:unimod} holds for all Borel functions $f:G_{**}\to \mathbb{R}_+$ such that $f([N_{\circ\circ}(\root,v)]) = 0$ unless $\{\root,v\}\in E$.
\end{lemma}
It is easy to show that the class of local weak limits are unimodular. The question of whether the class of unimodular measures and local weak limits coincide or not is still an open problem.

%% file: Sections/Background_PointProc.tex
Let $\Omega = \mathbb{Z}_{+}\times \mathbb{R}_+$ denote the type space. A point distribution $\omega = ( (m_1,x_1),a_1; (m_2,x_2),a_2;\dots;\allowbreak (m_k,x_k),a_k )$ on type space $\Omega$ is a finite set of vertices that consists of $a_j$ vertices of type $(m_j,x_j)$ for $k \in \mathbb{Z}_+ \setminus \{ 0 \}$, and $k=0$ corresponds to null point distribution. Let $\pd_{\Omega}$ denote the set of all point distributions. A point distribution $\omega\in\pd_{\Omega}$ defines a natural set function $\widetilde{\omega}(\cdot)$ over all subsets of $\Omega$,
\begin{align}
\forall {\Apset}\subset\Omega,\qquad \widetilde{\omega}({\Apset}) \coloneqq \sum_{(m_j,x_j)\in {\Apset}}{a_j}.
\end{align}
It is easy to see that there is a one-to-one correspondence between point distributions and set functions satisfying the following conditions:
\begin{enumerate}[label=(\alph*)]
	\item for any ${\Apset}\subset \Omega$, $\widetilde{\omega}({\Apset})$ is a non-negative integer.
	\item if ${\Apset}_1, {\Apset}_2,\dots {\Apset}_k$ are disjoint subsets of $\Omega$, then $\widetilde{\omega}\left(\cup_j {\Apset}_j\right) = \sum_{j} \widetilde{\omega}({\Apset}_j)$.
	\item if ${\Apset}_1\supset {\Apset}_2 \supset\dots $ are subsets of $\Omega$ and $\cap_j {\Apset}_j = \emptyset$, then $\widetilde{\omega}({\Apset}_j) = 0$ for all sufficiently large $j$.
\end{enumerate}
Abusing notation, we write $\omega(\cdot)$ as the set function generated by the point distribution $\omega\in\pd_{\Omega}$. We now define a $\sigma$-algebra on $\pd_{\Omega}$.

A rational interval is a subset of $\Omega$ with elements of the form $(m,x)$ such that $\underline{q}_1\leq m<\overline{q}_1$ and $\underline{q}_2\leq x < \overline{q}_2$, where $\underline{q}_1$ and $\overline{q}_1$ are non-negative integers, $\underline{q}_2$ and $\overline{q}_2$ are non-negative rational numbers, and $\overline{q}_1$ and $\overline{q}_2$ are allowed to be $\infty$. A basic set is a finite union of rational intervals or the empty set. Given a collection of basic sets ${\Apset}_1, {\Apset}_2, \cdots, {\Apset}_k$ and a set of non-negative integers $r_1, r_2,\cdots, r_k$, a cylinder set in $\pd_{\Omega}$ is defined as follows:
\begin{align}
\mathcal{C}({\Apset}_1, {\Apset}_2, \cdots, {\Apset}_k;r_1, r_2,\cdots, r_k) = \{\omega\in\pd_{\Omega}: \omega({\Apset}_j) = r_j,~ \forall j\in[k]\}.
\end{align}
Let $\mathscr{A}$ denote the $\sigma$-algebra generated by the cylinder sets. The following theorem defines a probability measure on $(\pd_{\Omega},\mathscr{A})$ using a set of probability distributions defined over basic sets. The proof is based on the Kolmogorov extension theorem~\cite{Harris1963}.
\begin{theorem}\label{thm:extprob}
	Let functions $p({\Apset}_1,{\Apset}_2,\cdots,{\Apset}_k;r_1,r_2,\cdots,r_k)$ be given, defined for any collection of basic sets and non-negative integers, satisfying the following.
	\begin{enumerate}[label=(\alph*)]
		\item $p({\Apset}_1,{\Apset}_2,\cdots,{\Apset}_k;r_1,r_2,\cdots,r_k)$ is a probability distribution on $k$-tuples of non-negative integers $r_1, r_2,\cdots,r_k$.
		\item $p({\Apset}_1,{\Apset}_2,\cdots,{\Apset}_k;r_1,r_2,\cdots,r_k)$ is permutation invariant, that is to say $\forall\sigma \in S_k$
		\begin{align}
		p({\Apset}_1,{\Apset}_2,\cdots,{\Apset}_k;r_1,r_2,\cdots,r_k) = p({\Apset}_{\sigma(1)},{\Apset}_{\sigma(2)},\cdots,{\Apset}_{\sigma(k)};r_{\sigma(1)},r_{\sigma(2)},\cdots,r_{\sigma(k)}).
		\end{align}
		\item The functions $p$ are consistent,
		\begin{align}
		p({\Apset}_1,{\Apset}_2,\cdots,{\Apset}_k;r_1,r_2,\cdots,r_k) = \sum_{r_{k+1}=0}^\infty p({\Apset}_1,{\Apset}_2,\cdots,{\Apset}_k,{\Apset}_{k+1};r_1,r_2,\cdots,r_k,r_{k+1}).
		\end{align}
		\item If ${\Apset}_1,{\Apset}_2,\cdots,{\Apset}_k$ are disjoint sets and ${\Apset} = \cup_{j=1}^{k} {\Apset}_j$, then $p({\Apset},{\Apset}_1,{\Apset}_2,\cdots,{\Apset}_k;r,r_1,r_2,\allowbreak\cdots,r_k) = 0$ unless $r=\sum_{j=1}^{k} r_j$.
		\item If ${\Apset}_1\supset {\Apset}_2 \supset \cdots$ and $\cap_{j=1}^\infty {\Apset}_j = \emptyset$, then $\lim_{j\to \infty}p({\Apset}_j;0) = 1$.
	\end{enumerate}
	Then there exists a unique probability measure $P$ on $\mathscr{A}$ that coincides with the functions $p$ whenever ${\Apset}_j$'s are basic sets,
	\begin{align}
	P(\omega({\Apset}_1) = r_1, \omega({\Apset}_2) = r_2,\cdots, \omega({\Apset}_k) = r_k) = p({\Apset}_1,{\Apset}_2,\cdots,{\Apset}_k;r_1,r_2,\cdots,r_k).
	\end{align}
\end{theorem}
For a point distribution $\omega = ( (m_1,x_1),a_1; (m_2,x_2),a_2;\dots;(m_k,x_k),a_k )\in\pd_{\Omega}$ and a function $h:\Omega\to\mathbb{R}$, the random integral $\int hd\omega$ is defined as $\sum_{j=1}^{k} a_j \times h(m_j,x_j)$. The term ``random'' refers to the randomness of $\omega$. Given a probability distribution $P$ on $(\pd_{\Omega},\mathscr{A})$, the Moment Generating Functional (MGF) of $P$ is defined as follows:
\begin{align}
\Phi(s) = \expect \eexp^{-\int s \, d\omega} = \int_{\pd_{\Omega}} \eexp^{-\int s \,d\omega}\, dP(\omega),
\end{align}
where $s:\Omega \to \mathbb{R}_+$ is a non-negative function. Similarly, given some conditions on a functional $\Phi$ defined over non-negative functions $s:\Omega \to \mathbb{R}_+$, there exists a unique probability measure $P$ on $(\pd_{\Omega},\mathscr{A})$ with MGF $\Phi$~\cite{Harris1963}. This correspondence implies the following theorem:
\begin{theorem}\label{thm:mgf}
	Let $\Phi_1, \Phi_2, \cdots, \Phi_k$ be MGF's on $(\pd_{\Omega},\mathscr{A})$. Then the functional $\Phi(s) = \Phi_1(s)\Phi_2(s)\allowbreak\cdots\Phi_k(s)$ defines an MGF on $(\pd_{\Omega},\mathscr{A})$.
\end{theorem}
Now, we revisit the EWT from point processes perspective. For any collection of basic sets $\{{\Apset}_1,{\Apset}_2,\cdots,{\Apset}_k\}$ and non-negative integers $\{r_1,r_2,\cdots,r_k\}$ define $p_{(m,x)} ({\Apset}_1,{\Apset}_2,\cdots,{\Apset}_k;\allowbreak r_1,r_2,\cdots, r_k)$ to be the probability that a vertex of type $(m,x)$ has $r_j$ children of type ${\Apset}_j$ for $j\in[k]$. Then, the functions $p_{(m,x)}$ determines a unique probability measure $P_{(m,x)}^{(1)}$ on $(\pd_{\Omega},\mathscr{A})$ (Theorem~\ref{thm:extprob}). The probability measure $P_{(m,x)}^{(1)}$ determines, in turn, an MGF $\Phi^{(1)}_{(m,x)}$. Notice that $p_{(m,x)}$, for any fixed set of arguments ${\Apset}_i$s and $r_i$s, is a Borel-measurable function of $(m,x)\in \Omega$ where $\Omega$ is equipped with the same metric as $\Omega_1$. Using the Theorem~\ref{thm:mgf}, for any point distribution $\omega = ( (m_1,x_1),a_1; (m_2,x_2),a_2;\dots;\allowbreak(m_k,x_k),a_k )\in\pd_{\Omega}$ the functional $\Phi^{(1)}_{\omega}$
\begin{align}
\Phi^{(1)}_{\omega}(s) = \left(\Phi^{(1)}_{(m_1,x_1)}(s)\right)^{a_1} \left(\Phi^{(1)}_{(m_2,x_2)}(s)\right)^{a_2} \cdots \left(\Phi^{(1)}_{(m_1,x_1)}(s)\right)^{a_k},
\end{align}
is an MGF and induces a probability measure $P_{\omega}^{(1)}$ on $(\pd_{\Omega},\mathscr{A})$. The probability measure $P_{\omega}^{(1)}$ is the transition probability function of a generalized Markov chain defined by the branching process,
\begin{align}
\forall \mathcal{{\Apset}} \in \mathscr{A},\qquad P_{\omega}^{(1)}({\Apset}) = \prob(Z_{l+1} \in {\Apset}\,\vert\,Z_l = \omega),
\end{align}
where $Z_l$ is the point distribution of vertices at depth $l$ (abusing the notation). As in regular Markov chains, the $m+n$-step transition probability function satisfies the following Chapman–Kolmogorov recurrence relation,
\begin{align}
\forall {\Apset} \in \mathscr{A},\qquad P_{\omega}^{(m+n)}({\Apset}) = \int_{\pd_{\Omega}} P^{(n)}_{{\omega}^\prime}({\Apset})\, d P^{(m)}_{\omega}({\omega}^\prime).
\end{align}
The MGF of $P_{\omega}^{(n)}$ is denoted by $\Phi_{\omega}^{(n)}$ which satisfies the following recurrence relation,
\begin{align}
\Phi_{\omega}^{(m+n)} = \Phi_{\omega}^{(n)}(-\log \Phi_{\cdot}^{(m)}).
\end{align}

%% file: Sections/Background_Operator.tex
A linear space ${\mathcal{X}}$ equipped with a norm $\norm{\cdot}_{\mathcal{X}}$ is called normed linear space. A complete normed linear space is called Banach space. Every Banach space is a metric space. A metric space $({\mathcal{X}},d)$ is called separable if it has a countable dense subset, i.e., a set $\{x_1,x_2,x_3,\cdots \}$ with the property that for all $\epsilon >0$ there exists $x_n$ such that $d(x_n,x)<\epsilon$. A linear space equipped with an inner-product is called an inner-product space. We say $S = \{e_\alpha\}_{\alpha\in I}$ is an orthonormal basis of an inner-product space ${\mathcal{X}}$, if $\forall x\in {\mathcal{X}}$ we have $x = \sum_{\alpha\in I} \langle x,e_{\alpha}\rangle $ and $\langle e_{\alpha},e_{\beta} \rangle = 0$ when $\alpha \neq \beta$ and $\langle e_{\alpha},e_{\alpha} \rangle = 1$. A Banach space with a norm induced by an inner-product is called Hilbert space. It is easy to prove that a Hilbert space is separable if and only if it has a countable orthonormal basis.

Let ${\mathcal{X}}$ and ${\mathcal{U}}$ be normed linear spaces over $\mathbb{C}$ with norms $\norm{\cdot}_{\mathcal{X}}$ and $\norm{\cdot}_{\mathcal{U}}$, respectively. A map $M:{\mathcal{X}}\to {\mathcal{U}}$ is called a bounded linear map if it is linear and there exists $b > 0$ such that $\forall x\in {\mathcal{X}},\,\norm{Tx}_{{\mathcal{U}}}\leq b \norm{x}_{{\mathcal{X}}}$. Let ${\mathcal{L}}({\mathcal{X}},{\mathcal{U}})$ denote the set of all bounded linear maps from ${\mathcal{X}}$ to ${\mathcal{U}}$ and equip this space with the natural norm $\norm{M}_{\mathcal{L}} = \sup_{x\in {\mathcal{X}}, \norm{x}_{\mathcal{X}} =1 } \norm{Mx}_{\mathcal{U}}$. Then $({\mathcal{L}}({\mathcal{X}},{\mathcal{U}}),\norm{\cdot}_{\mathcal{L}})$ is a normed linear space. It is easy to check that if ${\mathcal{U}}$ is a Banach space, then ${\mathcal{L}}({\mathcal{X}},{\mathcal{U}})$ is also a Banach space.

Consider ${\mathcal{L}}({\mathcal{X}},{\mathcal{X}})$ together with its natural binary map, i.e., if $N,M\in {\mathcal{L}}({\mathcal{X}},{\mathcal{X}})$ then {for all $x\in {\mathcal{X}}$ define $N \cdot M(x) \coloneqq N(M(x))$}. This forms an algebra over $\mathbb{C}$ which is called a normed algebra. A complete normed algebra is called Banach algebra. A operator $M$ in a Banach algebra is called invertible if $\exists N \in {\mathcal{L}}({\mathcal{X}},{\mathcal{X}})$ such that $N\cdot M = M\cdot N = I$, where $I \in \mathcal{L}(\mathcal{X}, \mathcal{X})$ is the identity map.

Let ${\mathcal{L}}({\mathcal{X}},{\mathcal{X}})$ be a Banach algebra over $\mathbb{C}$ and let $M\in {\mathcal{L}}({\mathcal{X}},{\mathcal{X}})$. The resolvent set of $M$ is given by
\begin{align}
\rho(M) = \{\lambda\in\mathbb{C}: \lambda I - M \text{ is invertible } \}.
\end{align}
The set $\sigma(M) = \mathbb{C}\setminus \rho(M)$ is called the spectrum of $M$. If $\lambda \in \sigma(M)$ then, $1)$ if $\lambda I - M$ is not one-to-one then $\lambda$ is called an eigenvalue of $M$, $2)$ if $\lambda I - M$ is one-to-one, but $\overline{R(\lambda I - M)}\neq {\mathcal{X}}$, where $R(N)$ is the range of $N$, then $\lambda$ is called a residual of $\sigma(M)$, and $3)$ if $\lambda$ is neither an eigenvalue nor a residual, then it is said to be in the continuous spectrum of $M$. The eigenvalues of $M$ are denoted by $\sigma_p(M)$, the residual spectrum of $M$ is denoted by $\sigma_r(M)$, and the continuous spectrum of $M$ is denoted by $\sigma_c(M)$. The spectrum of $M$ is nonempty, bounded, and closed in $\mathbb{C}$. The spectral radius of an operator $M$ is defined as $|\sigma(M)| \coloneqq \max_{\lambda\in\sigma(M)} |\lambda|$.
\begin{theorem}
	We have $|\sigma(M)| = \lim_{n\to\infty} \left(\norm{M^n}_{\mathcal{L}}\right)^{\frac{1}{n}}$
\end{theorem}

Let ${\mathcal{X}}$ and ${\mathcal{U}}$ be Banach spaces. A set $S\subset {\mathcal{X}}$ is called precompact if $\overline{S}$ is compact. A map $M\in {\mathcal{L}}({\mathcal{X}},{\mathcal{U}})$ is a compact operator if $M(B)$, where $B$ is the ball of radius $1$ in ${\mathcal{X}}$, is precompact in ${\mathcal{U}}$. The following theorem is the Riesz-Schauder Theorem which is a spectral theorem for compact operators.
\begin{theorem}\label{thm:optth1}
	Let ${\mathcal{X}}$ be a Banach space and let $M\in {\mathcal{L}}({\mathcal{X}},{\mathcal{X}})$ be a compact operator. Then the spectrum of $M$ satisfies the following:
	\begin{enumerate}[label=(\roman*)]
		\item $0$ is in the spectrum of $M$ unless the dimension of ${\mathcal{X}}$ is finite.
		\item All non-zero elements of $\sigma(M)$ are in $\sigma_p(M)$.
		\item If $\lambda$ is a non-zero eigenvalue of $M$, then $\lambda$ has finite multiplicity, i.e., the dimension of the null space of $\lambda I - M$ is finite.
		\item If $\lambda_0$ is an accumulation point of $\sigma(M)$ then $\lambda_0 = 0$.
	\end{enumerate}
\end{theorem}
Let $\mathcal{H}$ denote a Hilbert space and let $A\in {\mathcal{L}}(H,H)$. The adjoint of $A$, written as $A^*$, is defined by $\langle x,A^*y\rangle _{\mathcal{H}} \coloneqq \langle Ax,y\rangle _{\mathcal{H}}$ for all $x,y\in {\mathcal{H}}$. If $A^* = A$ or equivalently $\forall x,y\in {\mathcal{H}},\,\langle Ax,y\rangle = \langle x,Ay\rangle$, we say $A$ is symmetric or self-adjoint. 
The spectral theorem for compact symmetric operators on a Hilbert space $\mathcal{H}$ is given as follows.
\begin{theorem}\label{thm:optth2}
	Let ${\mathcal{H}}$ be a Hilbert space and let $A\in {\mathcal{L}}({\mathcal{H}},{\mathcal{H}})$ be a compact symmetric operator. Then the spectrum of $A$ satisfies the following properties.
	\begin{enumerate}[label=(\roman*)]
		\item The spectrum of $A$ is a subset of $\mathbb{R}$.
		\item If $\lambda,\lambda^\prime \in \sigma_p(A)$ and $\lambda \neq \lambda^\prime$ then the null space of $\lambda I-A$ is orthogonal to the null space of $\lambda^\prime I - A$.
		\item There exists $x_0\in {\mathcal{H}}$ with $\norm{x_0}_{\mathcal{H}} = 1$ such that $|\langle Ax_0,x_0\rangle _{\mathcal{H}}| = \sup_{\norm{x}_{\mathcal{H}} = 1} |\langle Ax,x\rangle _{\mathcal{H}}| = \norm{A}_{\mathcal{L}}$, and moreover, $x_0$ is an eigenvector of $A$, i.e., $Ax_0 = \lambda x_0$ for some $\lambda \in \mathbb{R}$. The corresponding eigenvalue $\lambda$ is the largest eigenvalue of $A$ in magnitude.
		\item (Hilbert-Schmidt) There exists an orthonormal basis of ${\mathcal{H}}$ consisting of the eigenvectors of $A$.
	\end{enumerate}
\end{theorem}
Let ${\mathcal{H}}$ be a Hilbert space. A cone $\mathcal{K}\subset {\mathcal{H}}$ is a closed convex subset of ${\mathcal{H}}$ such that for all $\lambda \in\mathbb{R}_+$, we have $\lambda \mathcal{K} \subset \mathcal{K}$ and $\mathcal{K}\cap (-\mathcal{K}) = \{\boldsymbol{0}\}$ where $(-1)\mathcal{K}$ is denoted as $-\mathcal{K}$. A closed subset $\mathcal{S}$ of $\mathcal{H}$ is said to be invariant under $A\in {\mathcal{L}}(\mathcal{H},\mathcal{H})$ if $A\mathcal{S} \subseteq\mathcal{S}$. The following theorem by Toland~\cite{Toland1996} is a version of the Krein--Rutman Theorem~\cite{Krein1948} for compact self-adjoint operators.
\begin{theorem}\label{thm:optth3}
Suppose $\mathcal{K}\subseteq \mathcal{H}$ is a closed cone such that $\mathcal{K}^{\bot} \coloneqq \{x\in\mathcal{H}:\langle x,y \rangle = 0,\, \forall y\in\mathcal{K}\}= \{\boldsymbol{0}\}$. Let $A\in {\mathcal{L}}(H,H)$ be a compact self-adjoint operator such that $A:\mathcal{K}\to\mathcal{K}$. Define $\mathscr{X}(A)\coloneqq \sup\{\langle Aw,w\rangle _{\mathcal{H}} : \norm{w}_{\mathcal{H}} =1, w\in\mathcal{K} \}$. We have the following.
\begin{enumerate}[label=(\roman*)]
	\item $\mathscr{X}(A)>0$ is the largest eigenvalue of $A$ in magnitude and $\mathscr{X}(A)$ has an eigenvector in $\mathcal{K}$.
	\item $\mathscr{X}(A)>0$ is a simple eigenvalue of $A$.
\end{enumerate}
\end{theorem}

%% file: Sections/ProofWeakConv.tex
Before presenting the proof, we revisit some basic properties of the order statistics of $n$ independent and identically distributed random variables.
\begin{lemma} \label{lem:orderstat}
	Let $\{X_i\}_{i=1}^m$ denote a set of {\em i.i.d.} random variables. Let $F(\cdot)$ and $f(\cdot)$ represent the cumulative distribution function and probability density function of $X_1$, respectively. Consider the order statistics of $\{X_i\}_{i=1}^m$ and denote it by $\{X^{(i)}\}_{i=1}^m$. For every $x_1\leq x_2\leq \dots\leq x_n$ and $l\leq m$, we have
	\begin{align}
	&f_{X^{(1)},X^{(2)},\dots,X^{(l)}}(x_1,x_2,\dots,x_l) = l! {m \choose l} \times \prod_{i=1}^{l} f(x_i) \times \left(1-F(x_l)\right)^{m-l}, \allowdisplaybreaks\\
	&f_{X^{(l)},X^{(l+1)},\dots,X^{(m)}}(x_{l},x_{l+1},\dots,x_m) = (m-l+1)! {m \choose m-l+1} \times \prod_{i=l}^{m} f(x_i) \times \left(F(x_l)\right)^{l-1}, \allowdisplaybreaks\\
	&f_{X^{(l)}}(x_l) = l{m \choose l} \times f(x_l) \times \left(F(x_l)\right)^{l-1} \times \left(1-F(x_l)\right)^{m-l}, \allowdisplaybreaks\\
	&f_{X^{(1)},X^{(2)},\dots,X^{(l-1)}|X^{(l)}}(x_1,x_2,\dots,x_{l-1}|x_l) = (l-1)! \frac{\prod_{i=1}^{l-1}f(x_i)}{\left(F(x_l)\right)^{l-1}},\allowdisplaybreaks\\
	&f_{X^{(l+1)},X^{(l+2)},\dots,X^{(m)}|X^{(l)}}(x_{l+1},x_{l+2},\dots,x_m|x_l) = (m-l)! \frac{\prod_{i=l+1}^{m}f(x_i)}{\left(1-F(x_l)\right)^{m-l}}.\allowdisplaybreaks\\
	\end{align}
\end{lemma}
\begin{corollary}\label{cor:randorderstat}
	Let $\{Y_i\}_{i=1}^{l-1}$ denote a random permutation of $\{X^{(i)}\}_{i=1}^{l-1}$, i.e., pick a permutation $\sigma\in S_{l-1}$ uniformly at random and set $Y_i = X^{(\sigma(i))}$ for all $1\leq i\leq l-1$. Similarly, let $\{Z_i\}_{i=l+1}^{m}$ denote a random permutation of $\{X^{(i)}\}_{i=l+1}^{m}$. Then we have
	\begin{align}
	&f_{Y_1,Y_2,\dots,Y_{l-1}|X^{(l)}}(y_1,y_2,\dots,y_{l-1}|x_l) = \frac{\prod_{i=1}^{l-1}f(y_i)}{\left(F(x_l)\right)^{l-1}},\allowdisplaybreaks\\
	&f_{Z_{l+1},Z_{l+2},\dots,Z_m|X^{(l)}}(z_{l+1},z_{l+2},\dots,z_m|x_l) = \frac{\prod_{i=l+1}^{m}f(z_i)}{\left(1-F(x_l)\right)^{m-l}}.\allowdisplaybreaks\\
	\end{align}
	Moreover, $\{Y_i\}_{i=1}^{l-1}$ are identically distributed and conditioned on $X^{(l)}$, they are independent. Same holds for $\{Z_i\}_{i=l+1}^{m}$:
	\begin{align}
	&f_{Y_i|X^{(l)}}(y_i|x_l) = \frac{f(y_i)}{F(x_l)}\qquad\qquad&\forall i \leq l-1\allowdisplaybreaks\\
	&f_{Z_i|X^{(l)}}(y_i|x_l) = \frac{f(z_i)}{1-F(x_l)}\qquad\qquad&\forall i \geq l+1\allowdisplaybreaks\\
	\end{align}
\end{corollary}
\begin{corollary}\label{cor:exprrderstat}
	Let $\{X_i\}_{i=1}^m$ be independent exponentially distributed random variables with mean $n$. Consider the random variables $\{Y_i\}_{i<l}$ and $\{Z_i\}_{i>l}$ as are defined in Corollary~\ref{cor:randorderstat}. Then, the conditional distribution of these random variables are given as follows:
	\begin{align}
	&f_{X^{(i)}}(x_i) = i{n \choose i} \times (1 - \eexp^{-\frac{x_i}{n}})^{i-1} \times \frac{1}{n}\eexp^{-(n-i+1)\frac{x_i}{n}} \xrightarrow{n\to\infty} \frac{\eexp^{-x_i}{x_i}^{i-1}}{(i-1)!}\qquad&\forall i\in[n]\allowdisplaybreaks\\
	&f_{Y_i|X^{(l)}}(y_i|x_l) = \frac{\frac{1}{n}\eexp^{-y_i/n}}{1-\eexp^{-x_l/n}}\xrightarrow{n\to\infty} \frac{1}{x_l}\qquad\qquad&\forall i\leq l-1\allowdisplaybreaks\\
	&f_{Z_{i}|X^{(l)}}(z_i|x_l) = \frac{\frac{1}{n}\eexp^{-z_i/n}}{\eexp^{-x_l/n}}=\frac{1}{n}\eexp^{-(z_i-x_l)/n}\qquad\qquad&\forall i\geq l+1
	\end{align}
	Most notably, the conditional distribution of $Y_i$ for $i<l$ conditioned on $X^{(l)}=x_l$ converges to the uniform distribution over $[0,x_l]$, as $n$ goes to infinity. Moreover, the distribution of $X^{(i)}$ converges to $\erlangdist(i)$.
\end{corollary}

As we mentioned, $\expect U(N_n)$ is the law of $[N_{n,\circ}(r)]$ for a uniformly chosen $r\in[n]$. The idea is to first define an exploration process over $K_n$ that realizes the connected component of the vertex $r$ in $N_n$. Then, we show that the connected component is locally tree-like, and the distribution of the connected component up to any finite time step of the exploration process converges to $\ertreedist(P)$. Finally, using the Portmanteau Theorem, we prove $\expect U(N_n) \xrightarrow{w} \ertreedist(P)$.\\
\smallskip\\
{\bf Step 1: Exploration Process}\\
The first step is to define a process that explores $K_n$ and realizes the connected component of a randomly selected vertex $r\in[n]$ in $N_n$. Let $E_n = \left\{\{i,j\}:i\neq j\in[n] \right\}$ denote the set of all edges in $K_n$. In order to track the process, we also construct a map $\phi$ from $\mathbb{N}^f$ to the connected component of $r$ and a fictitious vertex. In particular, $\phi$ maps $\mathbb{N}^f\setminus \{\phi^{-1}(v):\text{ $v$ in connected component of $r$ }\}$ to the fictitious vertex. The exploration is on $E_n$ and the cost of edges in $E_n$; at each step of the exploration process, $E_n$ is partitioned into five sets, defined as follows:
\begin{align}
&\At_t =\left\{\left(\{i,j\},C_n(\{i,j\})\right):\text{$\{i,j\}$ is active at time $t$}\right\}\allowdisplaybreaks\\
&\Ct_t=\left\{\left(\{i,j\},C_n(\{i,j\})\right):\text{$\{i,j\}$ belongs to the connected component at time $t$}\right\} \allowdisplaybreaks\\
&\Dt_t=\left\{\left(\{i,j\},C_n(\{i,j\})\right):\text{$\{i,j\}$ does not belong to the connected component at time $t$}\right\} \allowdisplaybreaks\\
&\Rt_t=\left\{\left(\{i,j\},C_n(\{i,j\})\right):\text{the cost of the non-active edge $\{i,j\}$ has been realized by time $t$}\right\} \allowdisplaybreaks\\
&\Ut_t = \left\{\{i,j\}:\text{the cost of the edge $\{i,j\}$ has not been realized by time $t$}\right\} \allowdisplaybreaks
\end{align}
\begin{remark}
	During the proof, we may abuse the notation by saying $\{i,j\}\in \At_t$ without including the cost. Even though $\At_t$ is a set of edges and their costs, we say $i\in \At_t$ (notationally), if there is a vertex $j\in[n]$ such that $\{i,j\}\in \At_t$. Finally, we say a vertex $i\in[n]$ has been explored by time step $t$, if both the threshold of $i$, $\thresh_i$ given by \cref{eq:thresh}, and the set of potential neighbors of $i$, $\potdeg_i$ given by \cref{eq:potneigh}, have been realized.
\end{remark}
\begin{remark}\label{rem:propset}
	The partition of $E_n$ at time $t$ satisfies the following properties:
	\begin{enumerate} []
		\item {$\At_t$}: The set of active edges, $\At_t$, consists of all the edges $\{v,z\}$ such that: $i)$ The cost of $\{v,z\}$ has been realized; $ii)$ Exactly one of $v$ or $z$ (but not both) belongs to the connected component at time $t$; and $iii)$ If $\potdeg_v$ has been realized, then $z\in \potdeg_v$. If $\potdeg_z$ has been realized, then $v\in \potdeg_z$.
		\item {$\Ct_t$}: The set of voted-in edges, $\Ct_t$, consists of all the edges $\{v,z\}$ such that: $i)$ The cost of $\{v,z\}$ has been realized; $ii)$ The vertices $v$ and $z$ belong to the connected component at time $t$; and $iii)$ Each vertex is a potential neighbor of the other, i.e., $z\in \potdeg_v$ and $v\in \potdeg_z$.
		\item {$\Dt_t$}: The set of erased edges, $\Dt_t$, consists of all the edges $\{v,z\}$ such that: $i)$ The cost of $\{v,z\}$ has been realized; and $ii)$ If only $\potdeg_v$ ($\potdeg_z$) has been realized, then $z\notin \potdeg_v$ ($v\notin \potdeg_z$); if $\potdeg_v$ and $\potdeg_z$ have been realized, then either $z\notin \potdeg_v$ or $v\notin \potdeg_z$ (or both).
		\item {$\Rt_t$}: The set of realized edges, $\Rt_t$, consists of all the edges $\{v,z\}$ such that: $i)$ The cost of $\{v,z\}$ has been realized; $ii)$ Neither $v$ nor $z$ belongs to the connected component at time $t$; and $iii)$ If $\potdeg_v$ has been realized, then $z\in \potdeg_v$; if $\potdeg_z$ has been realized, then $v\in \potdeg_z$.
		\item {$\Ut_t$}: The set of unrealized edges, $\Ut_t$, consists of all the edges $\{v,z\}$ such that the cost of $\{v,z\}$ has not been realized.
	\end{enumerate}
\end{remark}
\begin{remark}
	At each step of the exploration process, we may add at most one vertex to the connected component of $r$. Moreover, if the vertex $v$ is added to the connected component at time $t+1$, i.e., $v\in \Ct_{t+1}$, then $v$ is active at time $t$, i.e., $v\in \At_t$ and the exploration process explores an edge $\{j,v\}$ such that $j\in \Ct_t$.
\end{remark}
\noindent \textbf{\textit{Exploration process details and an alternative viewpoint}}: The exploration process starts by realizing the sets for $t=0$. Set $\phi(r) = \root$ and define $v_0 \coloneqq r$ and $k \coloneqq d_r(n)$. Let $\thresh_0$ and $\potdeg_0$ denote the threshold and the set of potential neighbors of $v_0$, respectively. By definition, $\thresh_0$ and $\potdeg_0$ are given by
\begin{align}
\thresh_0 &= \text{${k+1}^{\mathrm{th}}$ smallest value in }\left\{C_{n}(\{v_0,j\}) : j\in[n] \setminus \{v_0\} \right\},\\
\potdeg_0 &= \{j\in[n]\setminus \{v_0\}:C_{n}(\{v_0,j\}) < \thresh_0 \}.
\end{align}
Next, we present an alternative way to realize $\thresh_0$ and $\potdeg_0$ without realizing the cost of $\{v_0,j\}$ for all $j\in[n] \setminus \{v_0\}$. This alternative construction of the finite graph is an essential part of the proof of the weak convergence result, and is used at all time steps $t\geq 0$ as well.\newline
Pick a vertex $z_0\in[n]\setminus \{v_0\}$ uniformly at random and assume the threshold of the vertex $v_0$ is equal to the cost of the edge $\{v_0,z_0\}$, i.e., $\thresh_0 = C_{n}(\{v_0,z_0\})$. Realize the value of $C_{n}(\{v_0,z_0\})$; according to Lemma~\ref{lem:orderstat}, the density function of $C_{n}(\{v_0,z_0\})$ is given by
\begin{align}
&f_{C_{n}(\{v_0,z_0\})}(w) = (k+1){n-1 \choose k+1} \times \frac{1}{n}\eexp^{-w(n-k-1)/n} \times (1-\eexp^{-w/n})^{k}.
\end{align}
Next, pick $\It_0=\{z_1,z_2,\dots,z_{k}\}$, a subset of size $k$, from $[n]\setminus(\{z_0\}\cup \{v_0\})$ uniformly at random and assume $\It_0$ is the set of potential neighbors of $v_0$, i.e., $\potdeg_0 = \It_0$. Pick a permutation $\permt_0$ over $[|\It_0|]$ uniformly at random and for all $i\in[k]$ define $\phi(z_i) = \permt_0(i)$. Realize the values of $\left\{C_{n}(\{v_0,z_i\})\right\}_{i=1}^k$; by Corollary~\ref{cor:randorderstat}, the conditional joint density function of these random variables is given by
\begin{align}
&f_{C_{n}(\{v_0,z_{1}\}),C_{n}(\{v_0,z_{2}\}),\dots,C_{n}(\{v_0,z_k\})|\thresh_0}(w_1,w_2,\dots,w_k|w_0) = \frac{\prod_{i=1}^{k}\frac{1}{n}\eexp^{-w_i/n}}{(1-\eexp^{-w_0/n})^k}.\allowdisplaybreaks
\end{align}
Start the exploration process with
\begin{subequations}
\begin{align}
\At_0 &= \left\{\left(\{v_0,j\},c_n\left(\{v_0,j\}\right)\right):j\in \potdeg_0 \right\}\label{eq:prepA0},\allowdisplaybreaks\\
\Ct_0 &= \{\}\label{eq:prepC0},\allowdisplaybreaks\\
\Dt_0 &= \{\left(\{v_0,z_0\},c_n\left(\{v_0,z_0\}\right)\right)\}\label{eq:prepD0},\allowdisplaybreaks\\
\Rt_0 &= \{\} \label{eq:prepR0},\allowdisplaybreaks\\
\Ut_0 &= E_n \setminus \left\{\{v_0,j\}:j\in \It_0 \cup \{z_0\}\right\}.\label{eq:prepU0}
\end{align}
\end{subequations}
The description of the above equations is as follows:
\begin{enumerate}[]
	\item Equation \cref{eq:prepA0}: The vertex $v_0$ is the root of the connected component. All the potential neighbors of $v_0$ are included in $\At_0$.
	\item Equation \cref{eq:prepC0}: Although the vertex $v_0$ is the root of the connected component, there is no edge in the connected component yet; hence, the set $\Ct_0$ is set to be empty at the initial stage.
	\item Equation \cref{eq:prepD0}: The connection $\{v_0,z_0\}$ determines the threshold of the vertex $v_0$; hence, the vertex $z_0\notin \potdeg_0$ and the edge $\{v_0,z_0\}$ does not survive.
	\item Equation \cref{eq:prepR0}: The vertex $v_0$ is the root of the connected component; hence, none of the edges of form $\{v_0,z\}$ belongs to $\Rt_0$. The set $\Rt_0$ is set to be empty at the initial stage.
	\item Equation \cref{eq:prepU0}: All the edges $\{v_{0},j\}$ such that $C_n{\{v_{0},j\}}$ has been realized are removed from $E_n$ to construct $\Ut_0$.
\end{enumerate}
Figure~\ref{fig:Prep} depicts the preparation step for the exploration process. Define $\widehat{\thresh}_0$ to be equal to $\thresh_0$. These two values might be different for $t>0$. The definition and the role of $\widehat{\thresh}_t$ will become clear later on. \newline
\begin{figure}
	\centering
	\begin{subfigure}[t]{.3\linewidth}
	\centering
	\includegraphics*[width = \textwidth]{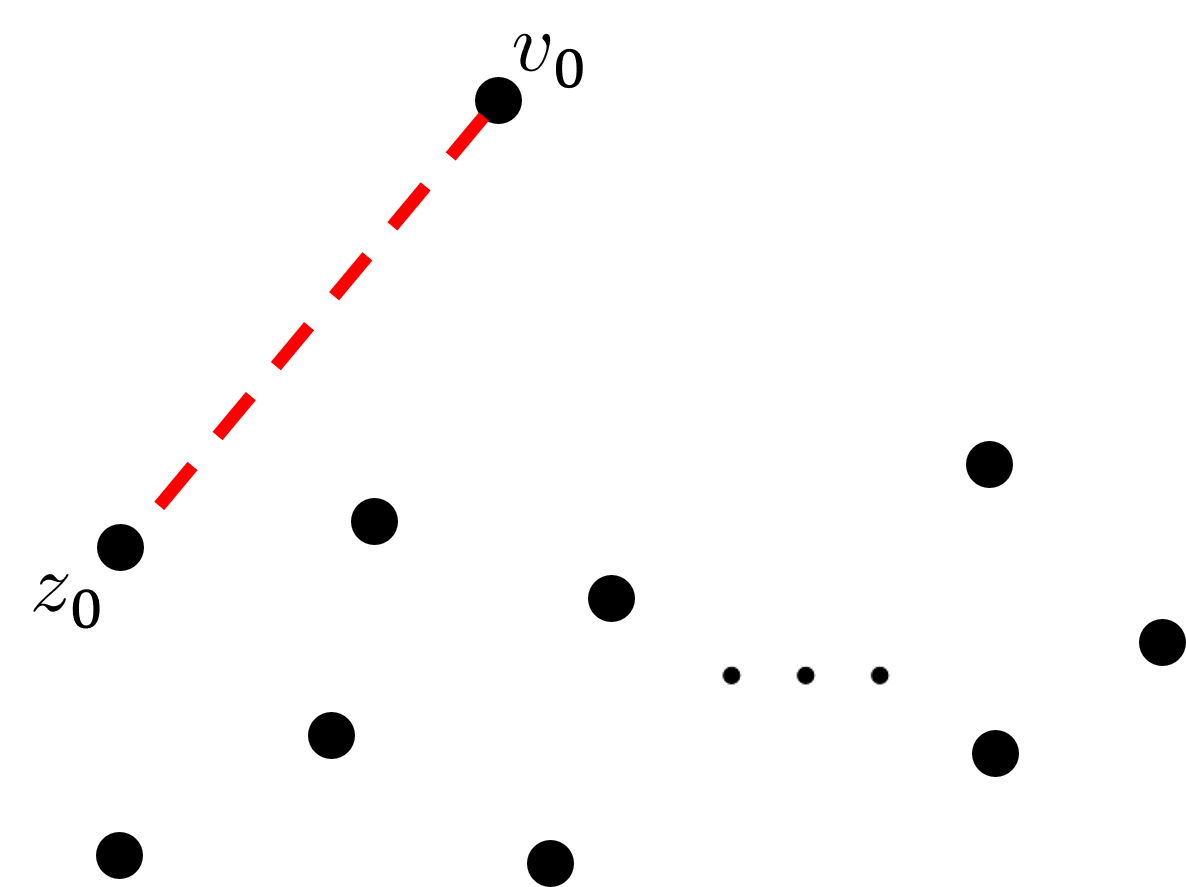}
	\caption{Set $\phi(v_0) = \root$, pick $z_0$ uniformly at random, realize cost of $\{v_0,z_0\}$ such that $\thresh_0 = C_n(\{v_0,z_0\})$}
\end{subfigure}\hfill
\begin{subfigure}[t]{.3\linewidth}
	\centering
	\includegraphics*[width = \textwidth]{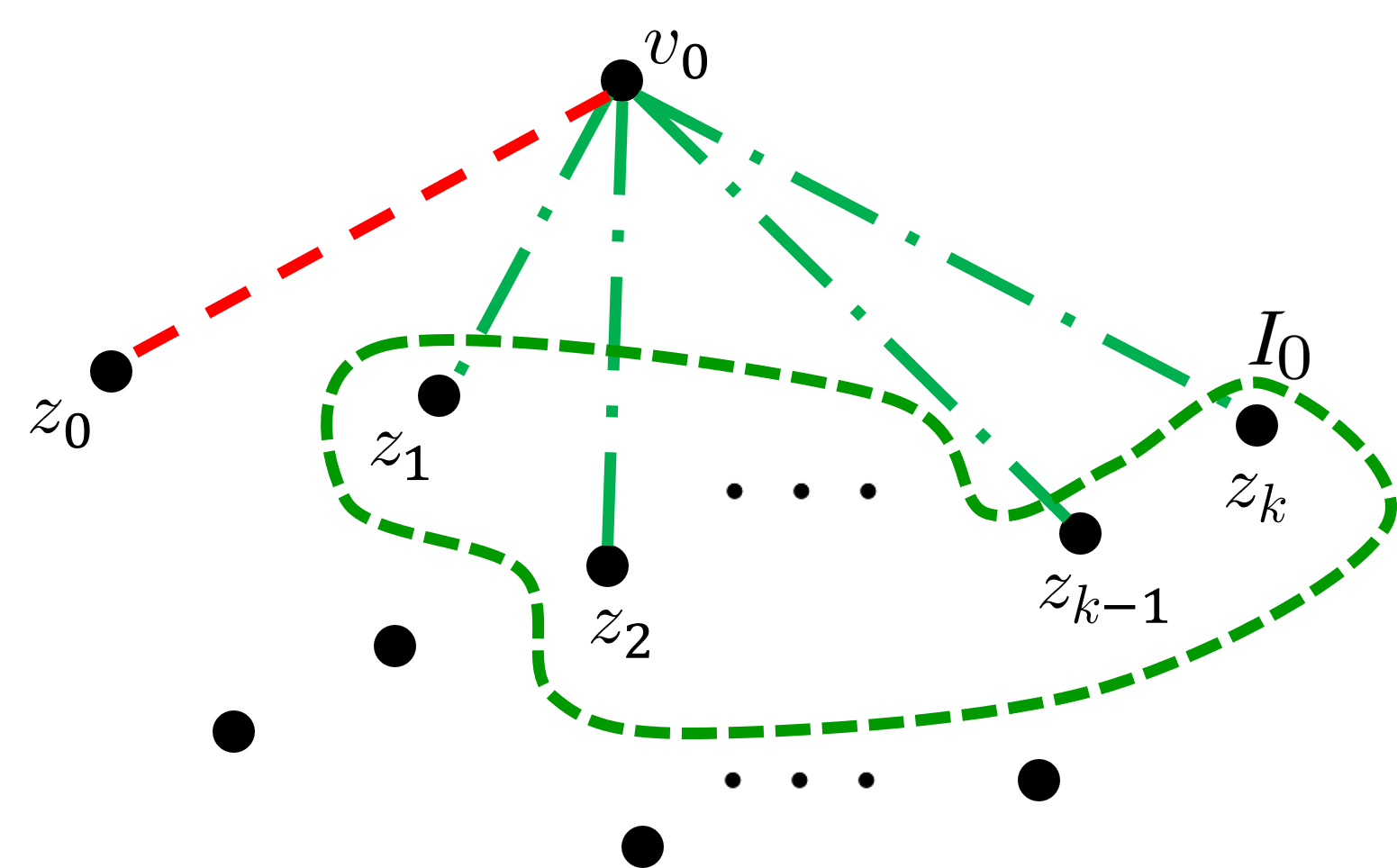}
	\caption{Pick $\It_0 = \{z_1,z_2,\dots,z_k\}$ uniformly at random. For all $i\in[k]$, set $\phi(z_i) = \permt_0(i)$ and realize the cost of $\{v_0,z_i\}$ such that $\potdeg_0=\It_0$}
	\end{subfigure}\hfill
	\begin{subfigure}[t]{.3\linewidth}
		\centering
		\includegraphics*[width = \textwidth]{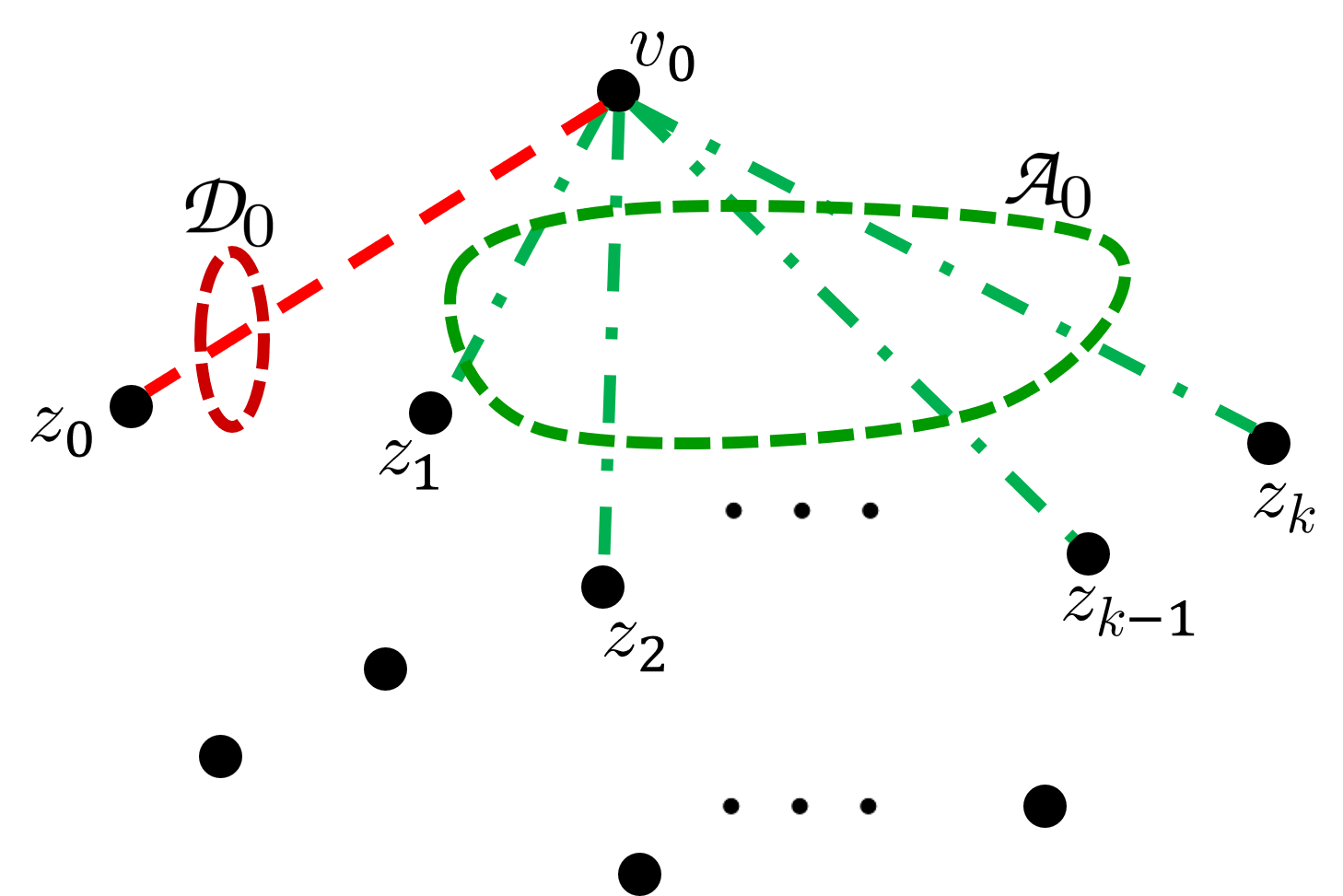}
		\caption{Initialize the sets $\At_0$ and $\Dt_0$. Set $\Ut_0 = E_n \setminus (\At_0\cup \Dt_0)$ and define $\Ct_0 =\Rt_0= \{\}$}
	\end{subfigure}
	\caption{Preparation step for the exploration process}
	\label{fig:Prep}
\end{figure}
Before proceeding with the exploration process, we need to define an order on $\mathbb{N}^f$: for two sequences $\boldsymbol{i} = (i_1,i_2,\dots,i_l)$ and $\boldsymbol{j} = (j_1,j_2,\dots,j_{l^\prime})$, we say $\boldsymbol{i} \prec \boldsymbol{j}$ if $l < l^\prime$ or $l = l^\prime$ and there exist some $g \in \mathbb{Z}_+$ such that $(i_1,i_2,\dots,i_{g-1})=(j_1,j_2,\dots,j_{g-1})$ and $i_g < j_g$.
\begin{remark}
	For the sake of notational simplicity, we denote the set of potential neighbors and the threshold of the vertex $v_t$ by $\potdeg_t$ and $\thresh_t$ instead of $\potdeg_{v_t}$ and $\thresh_{v_t}$. We may also use $\potdeg_j$ as the set of potential neighbors of the vertex $j$. The distinction is clear from the context.
\end{remark}
The exploration process for $t\geq 0$ is as follows; let $e_{t+1} = \{\phi^{-1}(\boldsymbol{i}),\phi^{-1}(\boldsymbol{j})\}\in \At_t$ such that $\boldsymbol{i}$ is minimal among $\{\phi(v):v\in \At_t\}$ and $\boldsymbol{j}$ is minimal among $\{\phi(z):\{\phi^{-1}(\boldsymbol{i}),z\}\in \At_t\}$. The choice of $e_{t+1}$ corresponds to the breadth-first search algorithm. As an example for $t=0$, the set $\{\phi(v):v\in \At_0\}$ equals to $\{\root,1,2,\dots,k\}$; hence $\boldsymbol{i} = \root$ and $\phi^{-1}(\root) = r$. Moreover, the set $\{\phi(z):\{\phi^{-1}(\root),z\}\in \At_t\}$ equals $\{1,2,\dots,k\}$; hence $\boldsymbol{j}=1$ and $\phi^{-1}(1) = z_{\permt_0^{-1}(1)}$. Hence, $e_{1} = \{\phi^{-1}(\root),\phi^{-1}(1)\} = \{r,z_{\permt_0^{-1}(1)}\}$.
\begin{remark}\label{rem:pari}
	Let $\phi(v) = (i_1,i_2,\dots,i_g)$ and define $par(v) \coloneqq \phi^{-1}(i_1,i_2,\dots,i_{g-1})$. The exploration process ensures that $par(v)$ belongs to the connected component of $r$; moreover, $par(v)$ is the first vertex in the connected component such that $v$ belongs to the set of the potential neighbors of $par(v)$, i.e., for every $z$ in the connected component if $v\in \potdeg_z$ then $par(v)$ is attached to the connected component before $z$. However, it is possible to have $\{par(v),v\}\in \Dt_t$ for some $t>0$, which is the case if $par(v)\notin \potdeg_v$ and the vertex $v$ has been explored by time step $t$. Still, $v$ may connect to the connected component through some other vertex $v^\prime$, i.e., $\{v^\prime,v\}\in \Ct_{t^\prime}$ for some $t^\prime > t$. Figure~\ref{fig:counterexp} illustrates such a situation, where $par(b) = r$, $\{par(b),b\}\in \Dt_2$ and $\{d,b\}\in \Ct_5$. Notice that the labeling is based on being a ``potential neighbor'' rather than being an actual neighbor.
\end{remark}
\begin{figure}[!ht]
	\centering
	\begin{subfigure}[t]{.3\linewidth}
	\centering
	\includegraphics*[width = \textwidth]{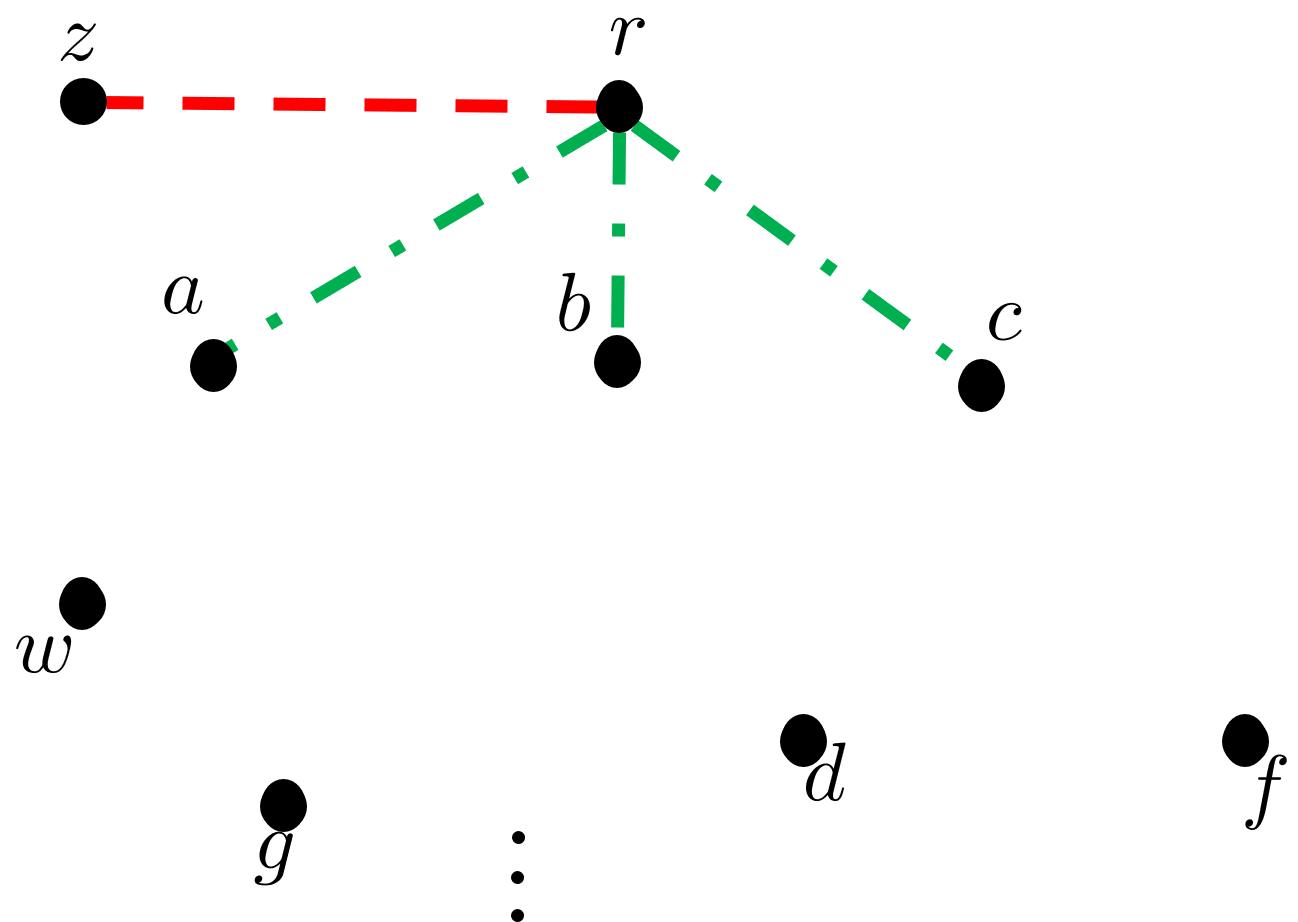}
	\caption*{(Prep): Set $\phi(r)=\root$. Realize $\potdeg_0 = \{a,b,c\}$ and $\thresh_0 = c_n(\{r,z\})$. Set $\phi(a) = (1)$, $\phi(b) = (2)$, $\phi(c) = (3)$.}
	\end{subfigure}\hfill
	\begin{subfigure}[t]{.3\linewidth}
		\centering
		\includegraphics*[width = \textwidth]{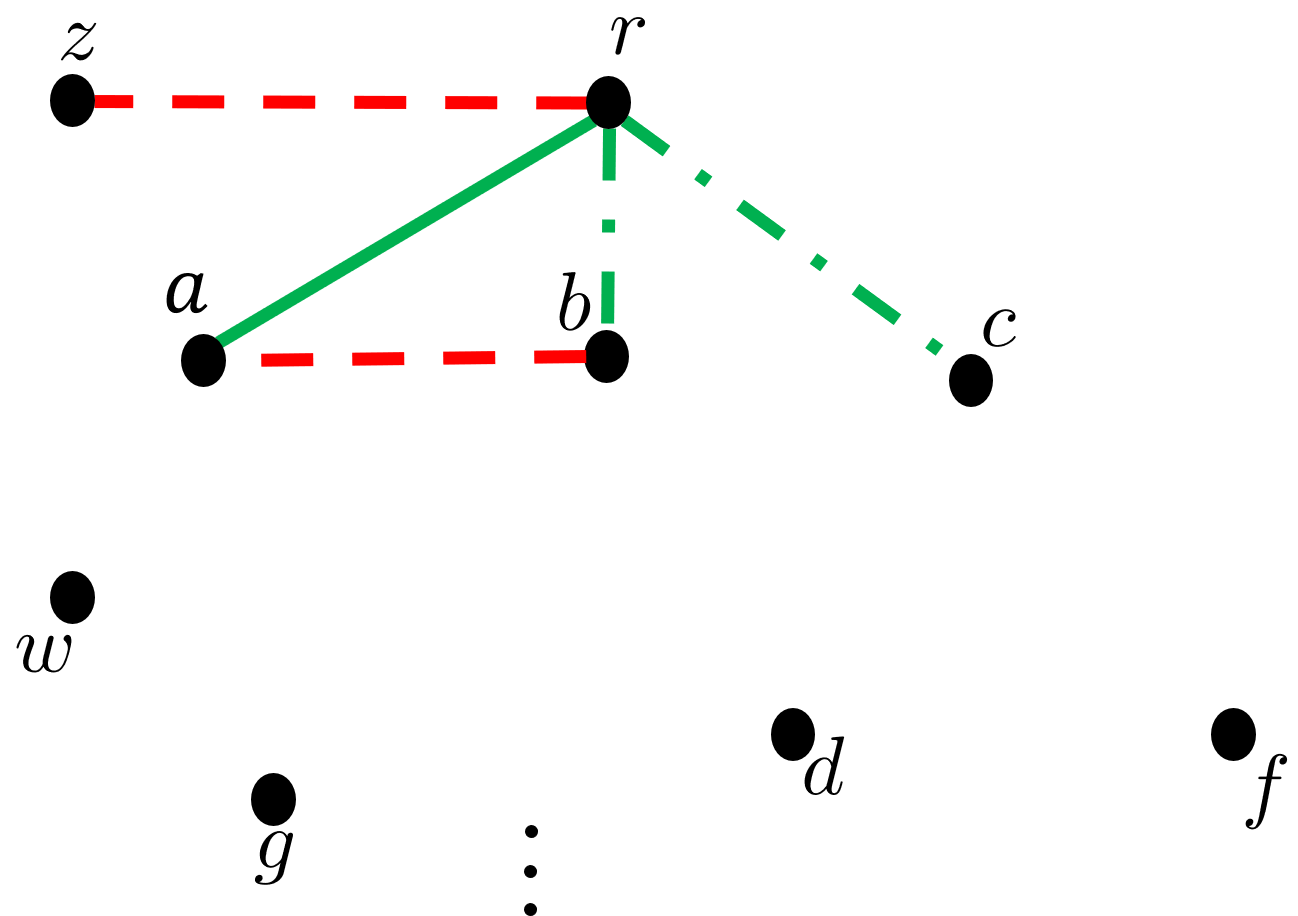}
		\caption*{($t=0$): Pick $e_1 = \{r,a\}$. Realize $\potdeg_1 = \{r\}$ and $\thresh_1 = c_n(\{a,b\})$.}
	\end{subfigure}\hfill
	\begin{subfigure}[t]{.3\linewidth}
		\centering
		\includegraphics*[width = \textwidth]{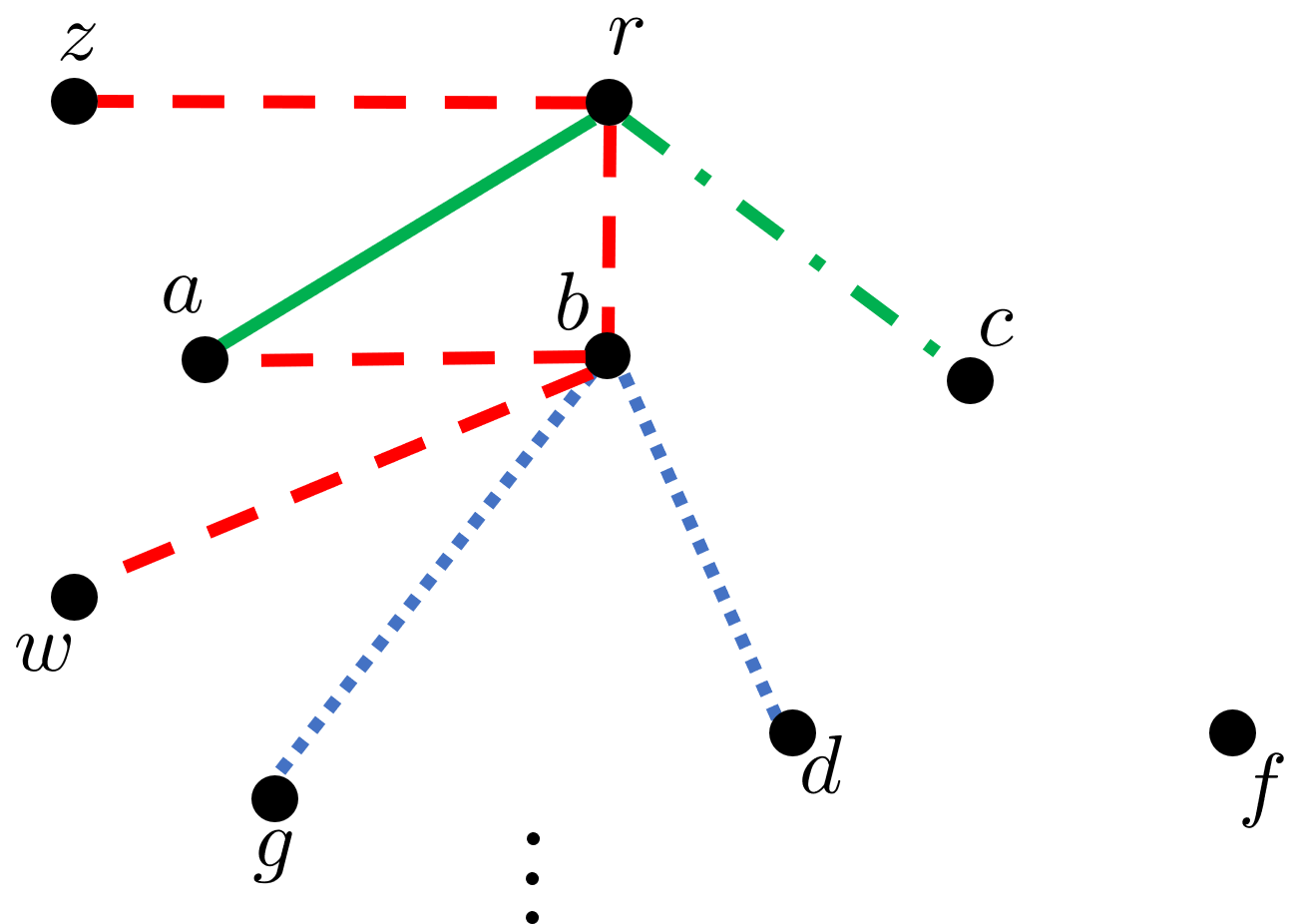}
		\caption*{($t=1$): Pick $e_2 = \{r,b\}$. Realize $\potdeg_2 = \{g,d\}$ and $\thresh_2 = c_n(\{b,w\})$.}
	\end{subfigure}
	\medskip
	\begin{subfigure}[t]{.3\linewidth}
		\centering
		\includegraphics*[width = \textwidth]{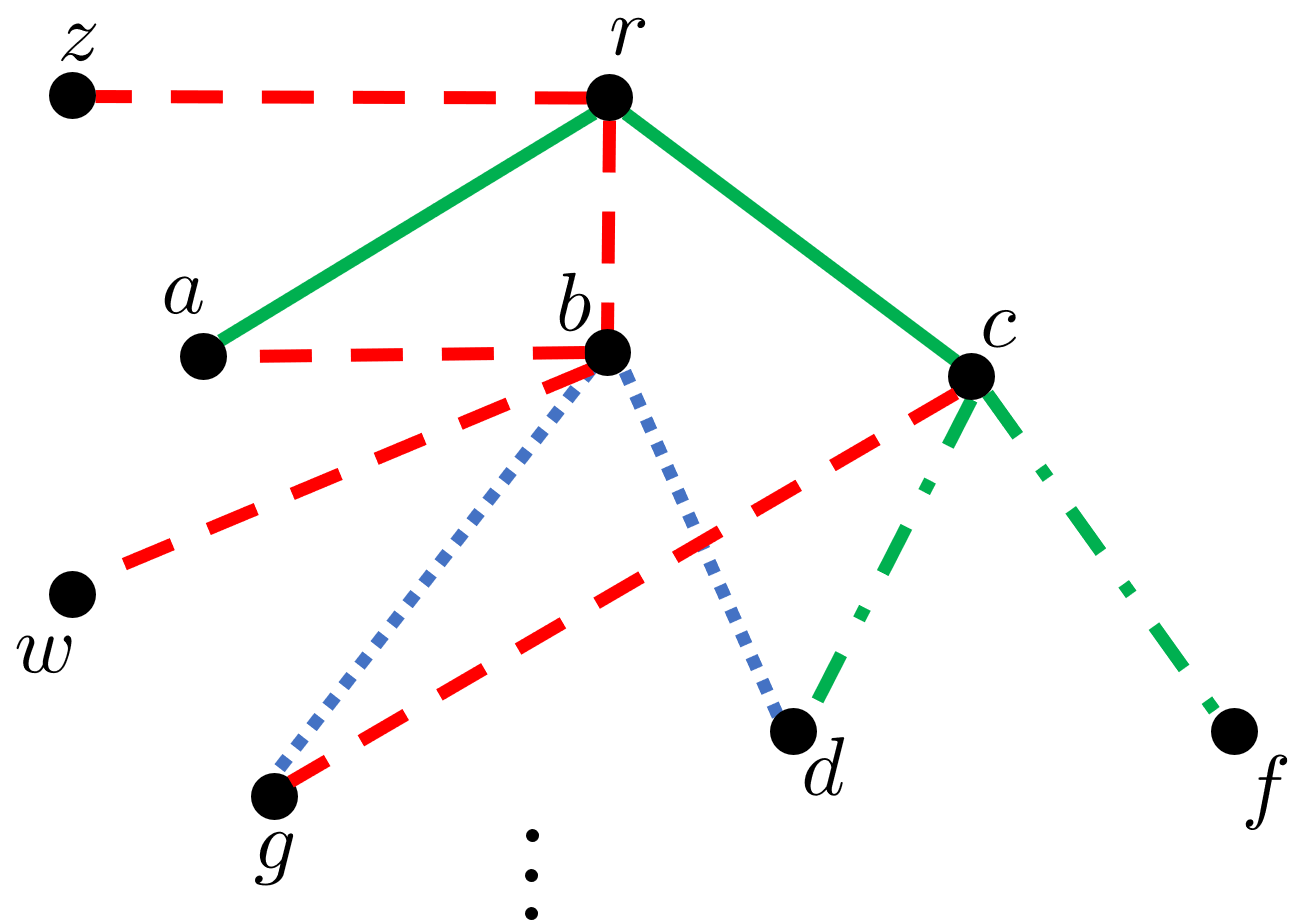}
		\caption*{($t=2$): Pick $e_3 = \{r,c\}$. Realize $\potdeg_3 = \{r,d,f\}$ and $\thresh_3 = c_n(\{c,g\})$. Set $\phi(d) = (3,1)$, $\phi(f) = (3,2)$.}
	\end{subfigure}\hfill
	\begin{subfigure}[t]{.3\linewidth}
		\centering
		\includegraphics*[width = \textwidth]{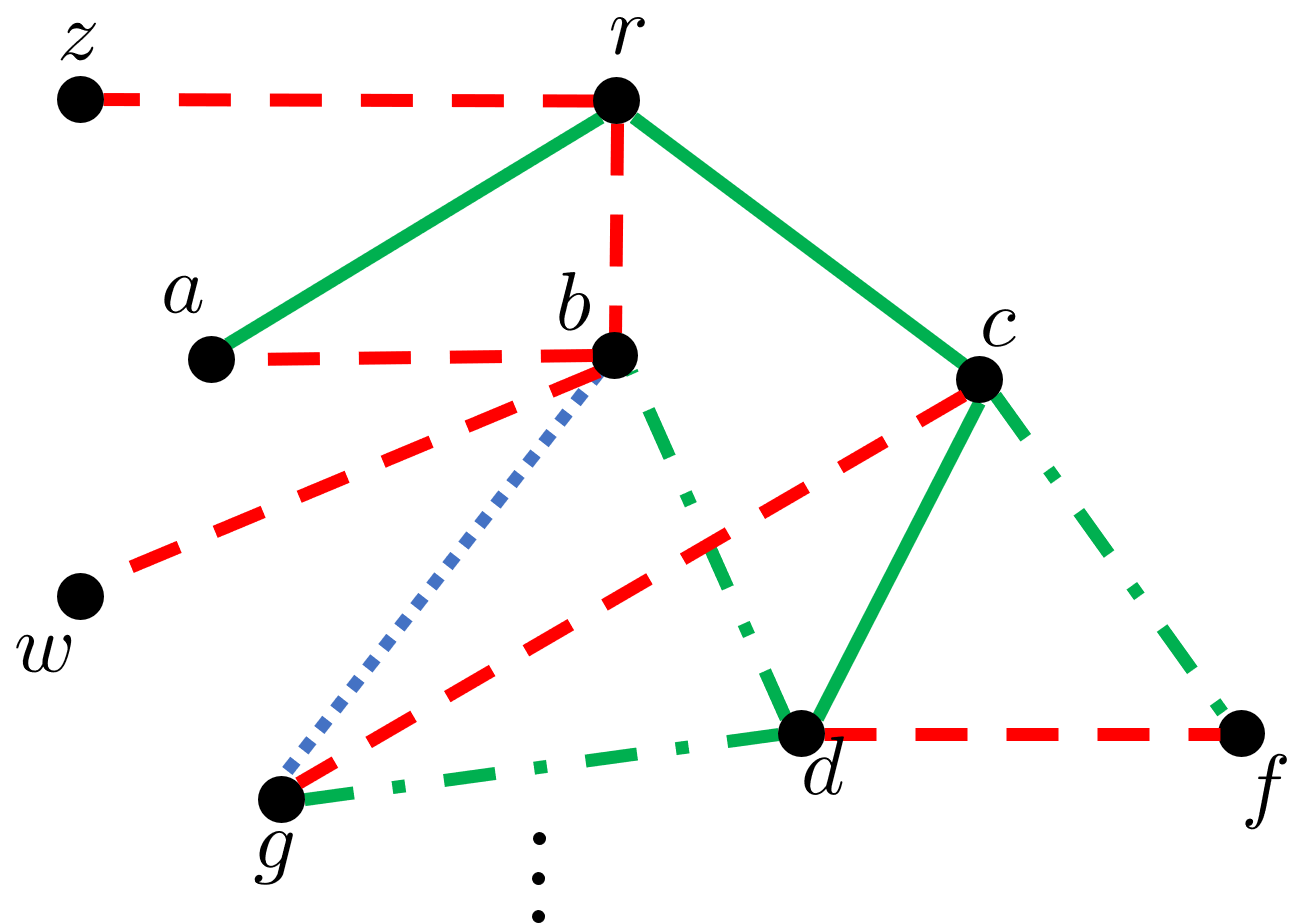}
		\caption*{($t=3$): Pick $e_4 = \{c,d\}$. Realize $\potdeg_4 = \{b,c,g\}$ and $\thresh_4 = c_n(\{d,f\})$. Set $\phi(g) = (3,1,1)$.}
	\end{subfigure}\hfill
	\begin{subfigure}[t]{.3\linewidth}
		\centering
		\includegraphics*[width = \textwidth]{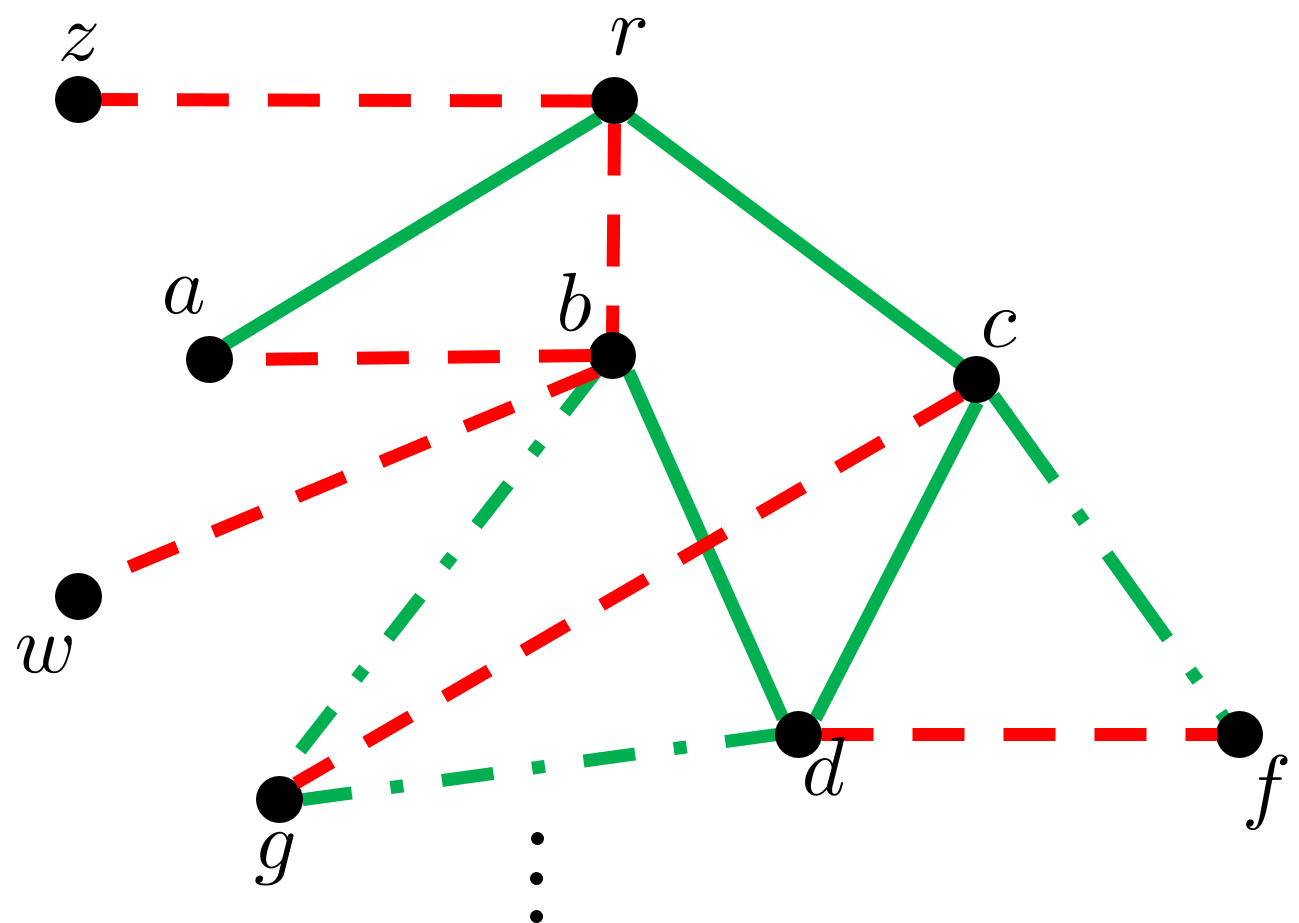}
		\caption*{($t=4$): Pick $e_5 = \{b,d\}$. We know $\potdeg_5 = \{g,d\}$ and $\thresh_5 = c_n(\{b,w\})$.}
	\end{subfigure}
	\caption{A realization of the exploration process up to $t=5$. Cost of the edges, threshold of the vertices, potential degree of the vertices, the permutation $\permt_t$, and the sets $\At_t$, $\Ut_t$, $\Ct_t$, and $\Dt_t$ are not mentioned. Solid green edges belong to $\Ct_t$, dashed red edges belong to $\Dt_t$, dashed dotted green edges belong to $\At_t$, and dotted blue edges belong to $\Rt_t$. Notice that $par(b)$ is defined to be $r$ although $b$ is connected to the root via $d$ at time $t=5$. Moreover, $par(g)$ is $d$ since the vertex $d$ is the first vertex in the connected component such that $g\in \potdeg_d$; although $g\in \potdeg_b$, the vertex $b$ is connected to the connected component after the vertex $d$. Based on the exploration process, $\{b,g\}\in \At_5$ and $e_6 = \{b,g\}$.}
	\label{fig:counterexp}
\end{figure}
\begin{remark}\label{rem:propExp_t}
	A vertex $v\neq r$ belongs to the connected component of $r$ by time step $t$ if and only if $v\in \Ct_t$. A vertex $v\in[n]\setminus \{r\}$ has been explored by time step $t$ if and only if $v$ belongs to the connected component, or there is a vertex $v^\prime\in \Ct_t$ such that $\{v^\prime,v\}\in \Dt_t\cup \Ct_t$ and $v$ belongs to the set of potential neighbors of $v^\prime$, i.e., $v\in \potdeg_{v^\prime}$. Notice that in the later case, the vertex $v^\prime$ may not be the vertex $par(v)$; As an example, in Figure~\ref{fig:counterexp} the vertex $g$ is explored at time step $t=5$ (since $e_6 = \{b,g\}$), but $\{par(g),g\} = \{d,g\}\in \At_6$.
\end{remark}
\begin{remark}\label{rem:propA_t}
	An important observation is that for every $\{v,z\}\in \At_t$ exactly one of $v$ or $z$ (but not both) belongs to the connected component of the vertex $r$ at time $t$. Moreover, at least one of the vertices $v$ or $z$ has been explored; hence, at each time step we may explore at most one vertex.
\end{remark}
\noindent Based on the exploration strategy the vertex $\phi^{-1}(\boldsymbol{i})$ has been explored, but it may not belong to the connected component. More explicitly, $par(\phi^{-1}(\boldsymbol{i}))$ belongs to the connected component (Remark~\ref{rem:pari}), and $\phi(par(\phi^{-1}(\boldsymbol{i}))) \prec \boldsymbol{j}$; hence, the edge $\{par(\phi^{-1}(\boldsymbol{i})),\phi^{-1}(\boldsymbol{i})\}\in \Ct_t\cup \Dt_t$ or equivalently, $\phi^{-1}(\boldsymbol{i})$ has been explored by time $t$ (Remark~\ref{rem:propExp_t}).
However, the vertex $\phi^{-1}(\boldsymbol{j})$ has two different possibilities,
\begin{itemize}
	\item \textbf{\textit{Subcase I}}, where $\phi^{-1}(\boldsymbol{j})$ has not been explored: in this case, the vertex $\phi^{-1}(\boldsymbol{i})$ belongs to the connected component. Let $v_{t+1} = \phi^{-1}(\boldsymbol{j})$. Let $m\leq t+1$ denote the number of explored vertices by time step $t$. Notice that at time $t=0$, the root vertex has already been explored and for each $t>0$, we may explore at most one vertex at each time step (Remark~\ref{rem:propA_t}). Define $k^\prime \coloneqq \min(n-m-2,d_{v_{t+1}}(n))$. If $n-m-2 < 0$, which may happen if the graph is fully connected and the process is reaching to its end, then let $k^\prime = 0$. In order to explore $v_{t+1}$, the first step is to choose $\Bt_{t+1} = \{z_1,z_2,\dots,z_{k^\prime}\}$, a subset of size $k^\prime$, uniformly at random from the set of unexplored vertices (there are $n-m-1$ unexplored vertices other than $v_{t+1}$). Next, pick a vertex $z_0$ out of remaining unexplored vertices uniformly at random (there are $n-m-1-k^\prime$ option for $z_0$). Assume that the cost of the edges $\{v_{t+1},z_i\}_{i=1}^{k^\prime}$ are the least ${k^\prime}$ values in $\left\{C_n(\{v_{t+1},z\}):z\text{ is not explored} \right\}$ and the cost of $\{v_{t+1},z_0\}$ is exactly the ${{k^\prime}+1}^{\mathrm{th}}$ smallest one. As in $t=0$, we do not realize the cost of $\{v_{t+1},z\}$ for all unexplored vertices $z\in[n]$. Using Lemma~\ref{lem:orderstat} and Corollary~\ref{cor:exprrderstat}, the joint density function of $\left\{C_n(\{v_{t+1},z\})\right\}_{i=0}^{k^\prime}$ is given by,
	\begin{align}
	&f_{C_n(\{v_{t+1},z_0\}),C_n(\{v_{t+1},z_1\}),\dots,C_n(\{v_{t+1},z_{k^\prime}\})} (w_0,w_1,\dots,w_{k^\prime})= \allowdisplaybreaks\\
	&\qquad \qquad ({k^\prime}+1) {n-m-1 \choose {k^\prime}+1} \times \prod_{i=0}^{k^\prime} \frac{1}{n}\eexp^{-w_i/n} \times \eexp^{-w_{0}(n-m-1-({k^\prime}+1))/n}
	\end{align}
	where $w_i\leq w_{0}$ for all $i\in[k']$. Notice that for every vertex $v\notin \Bt_{t+1}\cup\{z_0\}$ such that $v$ has not been explored and the cost of $\{v_{t+1},v\}$ has not been realized, the value of $C_n(\{v_{t+1},v\})$ is greater than $c_{n}(\{v_{t+1},z_0\})$. Define $\widehat{\thresh}_{t+1}$ to be $c_{n}(\{v_{t+1},z_0\})$,
	\begin{align}
	\widehat{\thresh}_{t+1} \coloneqq c_{n}(\{v_{t+1},z_0\}).
	\end{align}
	\begin{remark}\label{rem:T_hat}
		If, after realizing $\Bt_{t+1}\cup\{z_0\}$, the set of unexplored vertices $v$ such that $\{v_{t+1},v\}$ has not been realized is non-empty, then $d_{v_{t+1}}(n) < n-m-2$ and $\thresh_{t+1} \leq \widehat{\thresh}_{t+1}$.
	\end{remark}

	The second step to explore $v_{t+1}$ is to realize the cost of all the edges between $v_{t+1}$ and the explored vertices; by Corollary~\ref{cor:exprrderstat}, for every explored vertex $v$ such that $\{v_{t+1},v\}\in \Ut_t$, the density of $C_{n}(\{v_{t+1},v\})$ conditioned on $\widehat{\thresh}_v = w_v$ is given by
	\begin{align}
	f_{C_{n}(\{v_{t+1},v\})|\widehat{\thresh}_v}(w|w_v) = \frac{1}{n}\eexp^{-(w-w_v)/n}
	\end{align}
	\begin{remark}
		Assume the vertex $v$ has been explored but the value of $C_{n}(\{v_{t+1},v\})$ has not been realized. Since $v$ has been explored, we already know that $v_{t+1}\notin \potdeg_v$ and $C_{n}(\{v_{t+1},v\})\allowbreak> \thresh_v$. However, by the first step of the exploration process for the vertex $v$ we have $C_{n}(\{v_{t+1},v\}) > \widehat{\thresh}_v$. Moreover, Remark~\ref{rem:T_hat} suggests $\widehat{\thresh}_v\geq \thresh_v$ since $\{v_{t+1},v\}\in \Ut_t$.
	\end{remark}

	Notice that the potential neighbors of $v_{t+1}$ are either explored or belongs to $\Bt_{t+1}\cup\{z_0\}$. Define $k\coloneqq d_{v_{t+1}}(n)$ and set the threshold and the set of potential neighbors of $v_{t+1}$,
	\begin{align}
	\thresh_{t+1} &= \text{${k+1}^{\mathrm{th}}$ smallest value in }\left\{c_{n}(\{v_{t+1},j\}) : j\in[n]\text{ is explored or } j\in \Bt_{t+1}\cup\{z_0\} \right\}\allowdisplaybreaks\\
	\potdeg_{t+1} &= \left\{j\in [n]: c_{n}(\{v_{t+1},j\}) < \thresh_{t+1} \text{ and }j\in[n]\text{ is explored or } j\in \Bt_{t+1}\cup\{z_0\}\right\}
	\end{align}
	\begin{remark}
		The value of $k^\prime$ is less than or equal to $k$. As the process reaches to its end or if $d_{v_{t+1}}(n) > n-m-2$, we have $k^\prime < k$; hence, it is possible to have $z_0 \in \potdeg_{t+1}$.
	\end{remark}
	\textbf{\textit{Sub-subcase I.1}}: If $c_n(e_{t+1}) \geq \thresh_{t+1}$, then the connection $e_{t+1}$ does not survive; however, all the potential neighbors of $v_{t+1}$ has been realized and the vertex $v_{t+1}$ has been explored. In this case, update the sets as follows:
	\begin{subequations}
	\begin{align}
	&\At_{t+1} = \At_t \setminus \left\{\left(\{v_{t+1},j\},c_n\left(\{v_{t+1},j\}\right)\right): j\notin \potdeg_{t+1} \text{ and } \{v_{t+1},j\}\in \At_t \right\} \label{eq:1.1 At}\allowdisplaybreaks\\
	&\Ct_{t+1} = \Ct_t \label{eq:1.1 Ct}\allowdisplaybreaks\\
	&\begin{aligned}
	\Dt_{t+1} &= \Dt_t \cup \left\{\left(\{v_{t+1},j\},c_n\left(\{v_{t+1},j\}\right)\right): j\notin \potdeg_{t+1} \text{ and } C_n(\{v_{t+1},j\}) \text{ is realized} \right\}\allowdisplaybreaks\\
	&\myquad[2] \cup \left\{\left(\{v_{t+1},j\},c_n\left(\{v_{t+1},j\}\right)\right):\text{ $j$ has been explored and } v_{t+1}\notin \potdeg_j \right\}
	\end{aligned}\label{eq:1.1 Dt}\allowdisplaybreaks\\
	&\begin{aligned}
	\Rt_{t+1} = \left(\Rt_t \cup \left\{\left(\{v_{t+1},j\},c_n\left(\{v_{t+1},j\}\right)\right): j\in \potdeg_{t+1} \text{ and $j$ has not been explored}\right\}\right) \allowdisplaybreaks\\
	\qquad \setminus \left\{\left(\{v_{t+1},j\},c_n\left(\{v_{t+1},j\}\right)\right): j\notin \potdeg_{t+1} \text{ and } \{v_{t+1},j\}\in \Rt_t \right\}
	\end{aligned} \label{eq:1.1 Rt} \allowdisplaybreaks\\
	&\Ut_{t+1} = \Ut_t \setminus \{\{v_{t+1},j\}: C_n(\{v_{t+1},j\}) \text{ is realized}\} \label{eq:1.1 Ut}
	\end{align}
	\end{subequations}
	The description of the above equations is as follows:
	\begin{enumerate}[]
		\item Equation \cref{eq:1.1 At}: All the active edges $\{v_{t+1},j\}$ in $\At_t$ such that $j\notin \potdeg_{t+1}$ are removed, including $e_{t+1}$. Notice that if $\{v_{t+1},j\}\in \At_t$, then $v_{t+1}\in \potdeg_j$ (Remark~\ref{rem:propA_t}); however, after exploring the vertex $v_{t+1}$, it is clear whether $j$ is a potential neighbor of $v_{t+1}$ or not. If $j\notin \potdeg_{t+1}$ then the edge $\{v_{t+1},j\}$ is moved to $\Dt_{t+1}$. On the other hand, if $j\in \potdeg_{t+1}$, then $\{v_{t+1},j\}$ survives; however, this edge needs to be revisited at a later time in order to add new members to the set of active edges.
		\item Equation \cref{eq:1.1 Ct}: The vertex $v_{t+1}$ is not connected to the connected component through the edge $e_{t+1}$. Notice that there might be some other vertex $j$ such that $\{v_{t+1},j\}\in \At_t$ and $j \in \potdeg_{t+1}$, i.e., $\{v_{t+1},j\}$ survives (Remark~\ref{rem:propA_t}); however, the exploration of the edge $\{v_{t+1},j\}$ is postponed to some $t^\prime > t$.
		\item Equation \cref{eq:1.1 Dt}: All the edges $\{v_{t+1},j\}$ such that $C_n{\{v_{t+1},j\}}$ has been realized and $j\notin \potdeg_{t+1}$ do not survive. Moreover, for all explored vertices $j$ such that $v_{t+1} \notin \potdeg_j$, the edge $\{v_{t+1},j\}$ does not survive as well.
		\item Equation \cref{eq:1.1 Rt}: For all $j\in \potdeg_{t+1}$ such that the vertex $j$ has not been explored, $\{v_{t+1},j\}$ is added to $\Rt_{t+1}$. Notice that the cost of $\{v_{t+1},j\}$ has been realized and neither $v_{t+1}$ nor $j$ belong to the connected component. Moreover, for each explored vertex $j$, if $\{v_{t+1},j\}\notin \Rt_{t}$ then either $v_{t+1}\notin \potdeg_j$ or $j$ belongs to the connected component; hence, $\{v_{t+1},j\}$ need not be included in $\Rt_{t+1}$. Finally, for all edges $\{v_{t+1},j\}\in \Rt_{t}$, the vertex $v_{t+1}$ is a potential neighbor of the vertex $j$; however, if $j\notin \potdeg_{t+1}$ then $\{v_{t+1},j\}$ does not survive.
		\item Equation \cref{eq:1.1 Ut}: All the edges $\{v_{t+1},j\}$ such that $C_n{\{v_{t+1},j\}}$ has been realized are removed from $\Ut_{t+1}$.
	\end{enumerate}
	\begin{remark}
	Consider an edge $e = \{v_{t+1},j\}$ such that the cost of $e$ has been realized. If the vertex $j\notin \potdeg_{t+1}$, then the edge $e$ does not survive and it belongs to $\Dt_{t+1}$. Now assume $j\in \potdeg_{t+1}$. If the vertex $j$ has not been explored, then $e$ belongs to $\Rt_{t+1}$. If the vertex $j$ has been explored and $v_{t+1}\notin \potdeg_j$, then the edge $e$ does not survive and it belongs to $\Dt_{t+1}$. Assume $j$ has been explored and $v_{t+1}\in \potdeg_j$. If $j$ belongs to the connected component, then $e\in \At_t$. If $j$ does not belong to the connected component, then $e\in \Rt_t$. In either case, $e$ needs no update, and it is included in the corresponding set at time step $t+1$.
	\end{remark}
	\textbf{\textit{Sub-subcase I.2}}: If $c_n(e_{t+1}) < \thresh_{t+1}$, then the connection $e_{t+1}$ survives and $v_{t+1}$ belongs to the connected component. Define $\It_{t+1} = \{p\in \potdeg_{t+1}: \phi(p) \text{ is not defined}\}$. Let $\It_{t+1} = \{p_1,p_2,\dots,p_{|\It_{t+1}|}\}$. Pick a permutation $\permt_{t+1}$ over $[|\It_{t+1}|]$ uniformly at random and set $\phi(p_l) = (\boldsymbol{j},\permt_{t+1}(l))$ for all $l\in\left[|\It_{t+1}|\right]$, where $\boldsymbol{j} = \phi(v_{t+1})$. Update the sets as follows,
	\begin{subequations}
	\begin{align}
	&\begin{aligned}
	\At_{t+1} &= \big( \At_t \cup \left\{\left(\{v_{t+1},j\},c_n\left(\{v_{t+1},j\}\right)\right):j\in \potdeg_{t+1} \text{ and $j$ has not been explored}\right\}\\
	&\myquad[1]\cup \!\left\{\left(\{v_{t+1},j\},c_n\left(\{v_{t+1},j\}\right)\right):\text{$j$ has been explored and } j\in \potdeg_{t+1}, v_{t+1}\in \potdeg_j\right\}\!\big)\\
	&\myquad[1]\setminus \left\{\left(\{v_{t+1},j\},c_n\left(\{v_{t+1},j\}\right)\right):\text{$j$ belongs to the connected component} \right\}
	\end{aligned}\label{eq:1.2 At}\allowdisplaybreaks\\
	&\Ct_{t+1} = \Ct_t\cup \left\{\left(\{v_{t+1},j\},c_n\left(\{v_{t+1},j\}\right)\right):j\in \potdeg_{t+1} \text{ and }\{v_{t+1},j\}\in \At_t \right\}\label{eq:1.2 Ct}\\
	&\begin{aligned}
	\Dt_{t+1} &= \Dt_t \cup \left\{\left(\{v_{t+1},j\},c_n\left(\{v_{t+1},j\}\right)\right):j\notin \potdeg_{t+1} \text{ and } C_n(\{v_{t+1},j\}) \text{ is realized} \right\}\\
	&\qquad\cup \left\{\left(\{v_{t+1},j\},c_n\left(\{v_{t+1},j\}\right)\right):\text{$j$ has been explored and } v_{t+1}\notin \potdeg_j \right\}
	\end{aligned}\label{eq:1.2 Dt}\allowdisplaybreaks\\
	&\Rt_{t+1} = \Rt_t \setminus \left\{\left(\{v_{t+1},j\},c_n\left(\{v_{t+1},j\}\right)\right):\{v_{t+1},j\}\in \Rt_t \right\} \label{eq:1.2 Rt}\allowdisplaybreaks\\
	&\Ut_{t+1} = \Ut_t \setminus \{\{v_{t+1},j\}: C_n(\{v_{t+1},j\}) \text{ is realized}\}\label{eq:1.2 Ut}
	\end{align}
	\end{subequations}
	The description of the above equations is as follows:
	\begin{enumerate}[]
		\item Equation \cref{eq:1.2 At}: All the edges $\{v_{t+1},j\}$ such that $j\in \potdeg_{t+1}$ and $j$ has not been explored are added to $\At_t$. Moreover, all the edge $\{v_{t+1},j\}$ such that $j$ has been explored, $j$ do not belongs to the connected component, $j\in \potdeg_{t+1}$ and $v_{t+1}\in \potdeg_j$ are also included in $\At_{t+1}$.
		\item Equation \cref{eq:1.2 Ct}: The vertex $v_{t+1}$ is connected to the connected component through the edge $e_{t+1}$; however, all the edges $\{v_{t+1},j\}\in \At_t$ such that $j \in \potdeg_{t+1}$ are also included in $\Ct_{t+1}$; since for each edge $\{v_{t+1},j\}\in \At_t$ the vertex $j$ belongs to the connected component and $v_{t+1}\in \potdeg_j$.
		\item Equation \cref{eq:1.2 Dt}: All the edges $\{v_{t+1},j\}$ such that $C_n{\{v_{t+1},j\}}$ has been realized and $j\notin \potdeg_{t+1}$ do not survive. Moreover, for all explored vertex $j$ such that $v_{t+1} \notin \potdeg_j$, the edge $\{v_{t+1},j\}$ does not survive as well.
		\item Equation \cref{eq:1.2 Rt}: Since $v_{t+1}$ is connected to the connected component, no edge needs to be added to $\Rt_{t}$; however, all the edges $\{v_{t+1},j\}\in \Rt_t$ is removed from $\Rt_t$, since one end of such an edge belongs to the connected component.
		\item Equation \cref{eq:1.2 Ut}: All the edges $\{v_{t+1},j\}$ such that $C_n{\{v_{t+1},j\}}$ has been realized is removed from $\Ut_{t+1}$.
	\end{enumerate}
	\begin{remark}
		Consider an edges $e = \{v_{t+1},j\}$ such that the cost of $e$ has been realized. If the vertex $j\notin \potdeg_{t+1}$, then the edge $e$ does not survive and it belongs to $\Dt_{t+1}$. Assume $j\in \potdeg_{t+1}$. If the vertex $j$ has not been explored, then $e$ belongs to $\At_{t+1}$. If the vertex $j$ has been explored and $v_{t+1}\notin \potdeg_j$, then the edge $e$ does not survive, and it belongs to $\Dt_{t+1}$. Assume $j$ has been explored and $v_{t+1}\in \potdeg_j$. If $j$ belongs to the connected component, then $e\in \At_t$ and $e$ is moved to $\Ct_{t+1}$. If $j$ does not belong to the connected component, then $e\in \Rt_t$ and $e$ is moved to $\At_{t+1}$.
	\end{remark}
	Figure~\ref{fig:case1} illustrates the update process for the case where only $\phi^{-1}(\boldsymbol{i})$ has been explored.
	\begin{figure}
		\centering
		\begin{subfigure}[t]{.48\linewidth}
			\centering
			\includegraphics*[width = 0.8\textwidth]{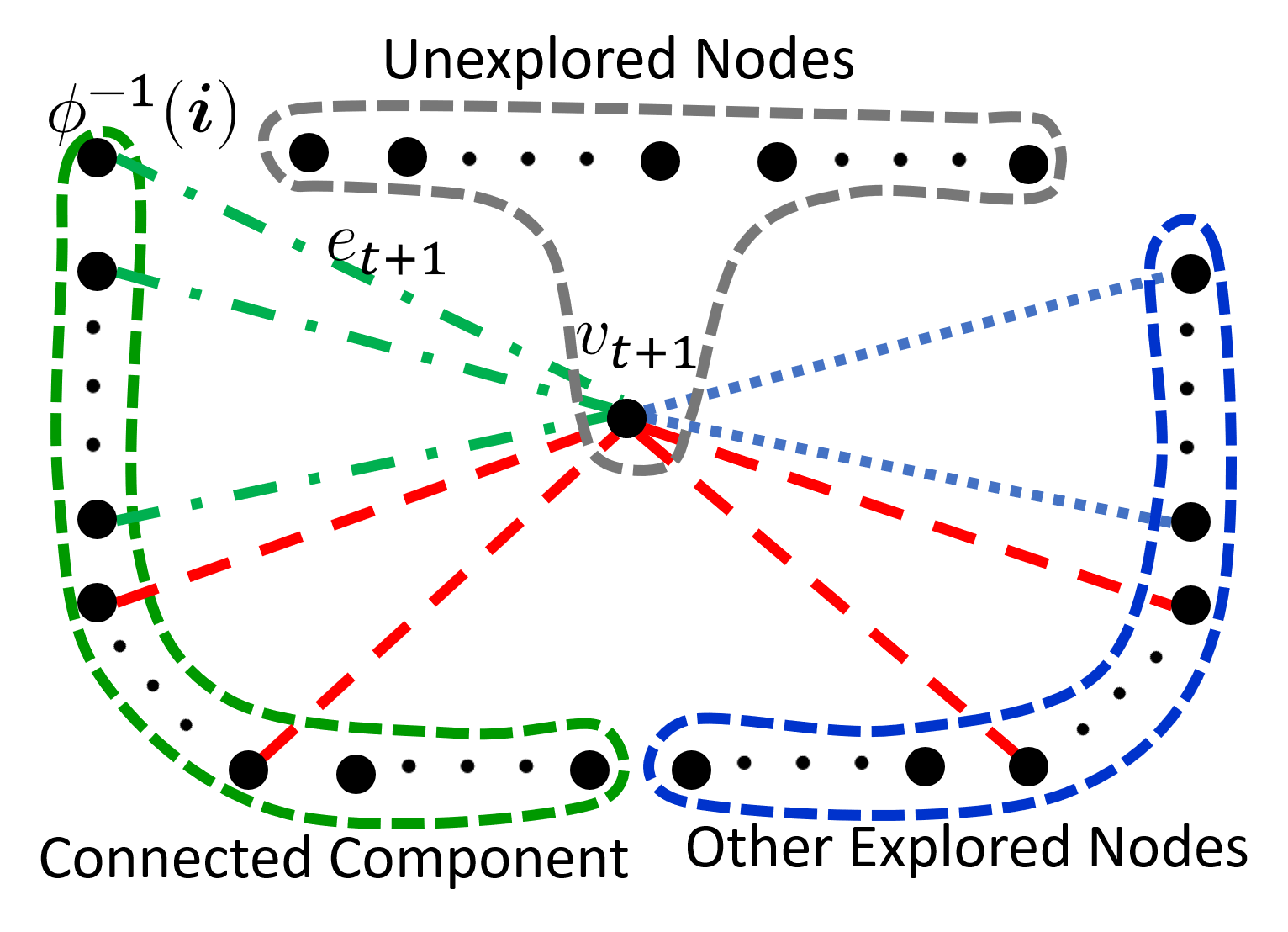}
			\caption{Let $e_{t+1} = \{\phi^{-1}(\boldsymbol{i}),\phi^{-1}(\boldsymbol{j})\}$ and $v_{t+1} = \phi^{-1}(\boldsymbol{j})$. Dashed red edges belong to $\Dt_t$, dashed dotted green edges belong to $\At_t$, and dotted blue edges belong to $\Rt_t$.}
		\end{subfigure}\hfill
		\begin{subfigure}[t]{.48\linewidth}
			\centering
			\includegraphics*[width = 0.8\textwidth]{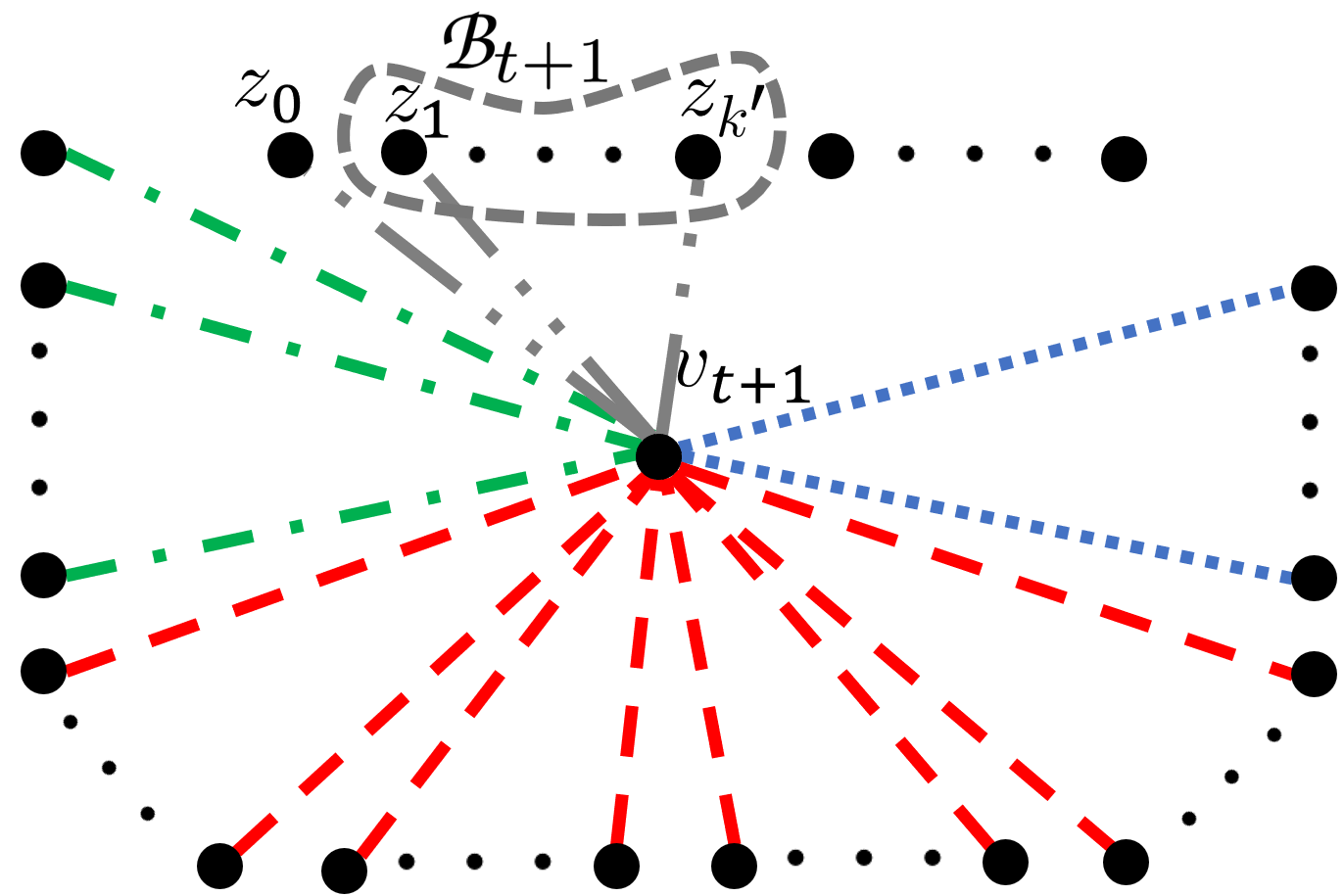}
			\caption{Realize $\Bt_{t+1}$, $z_0$ and the corresponding edge costs. Realize the cost of all edges $\{v_{t+1},j\}$ for explored vertices $j$ as well. For each explored vertex $j$ such that $\{v_{t+1},j\}\in \Ut_t$, we have $\{v_{t+1},j\} \in \Dt_{t+1}$.}
		\end{subfigure}\hfill
		\medskip
		\begin{subfigure}[t]{.48\linewidth}
			\centering
			\includegraphics*[width = 0.8\textwidth]{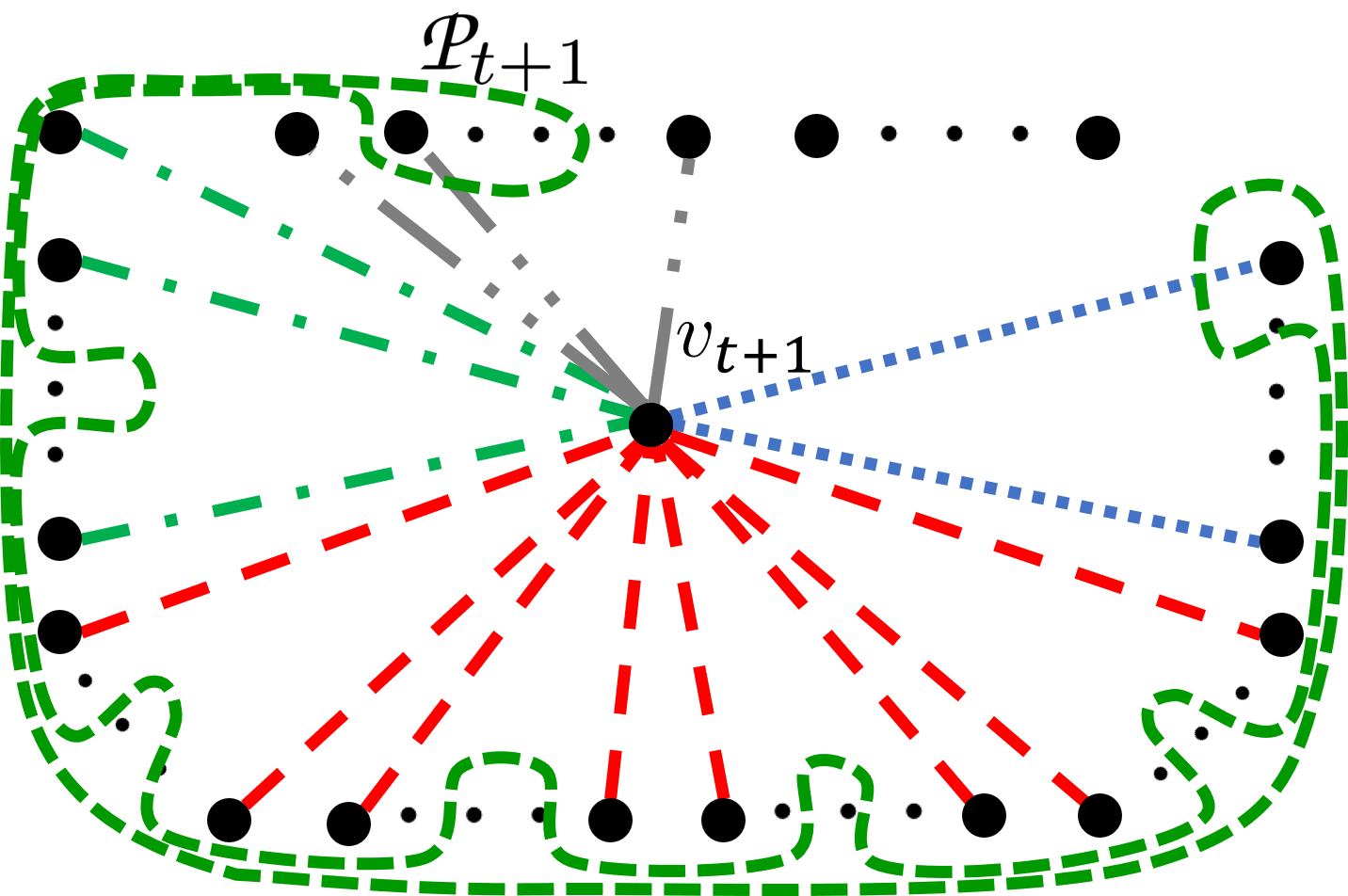}
			\caption*{(c1) Realize $\thresh_{t+1}$ and $\potdeg_{t+1}$. Consider the case where $e_{t+1}$ does not survive, i.e., $c_n(e_{t+1}) \geq \thresh_{t+1}$.}
		\end{subfigure}\hfill
		\begin{subfigure}[t]{.48\linewidth}
			\centering
			\includegraphics*[width = 0.8\textwidth]{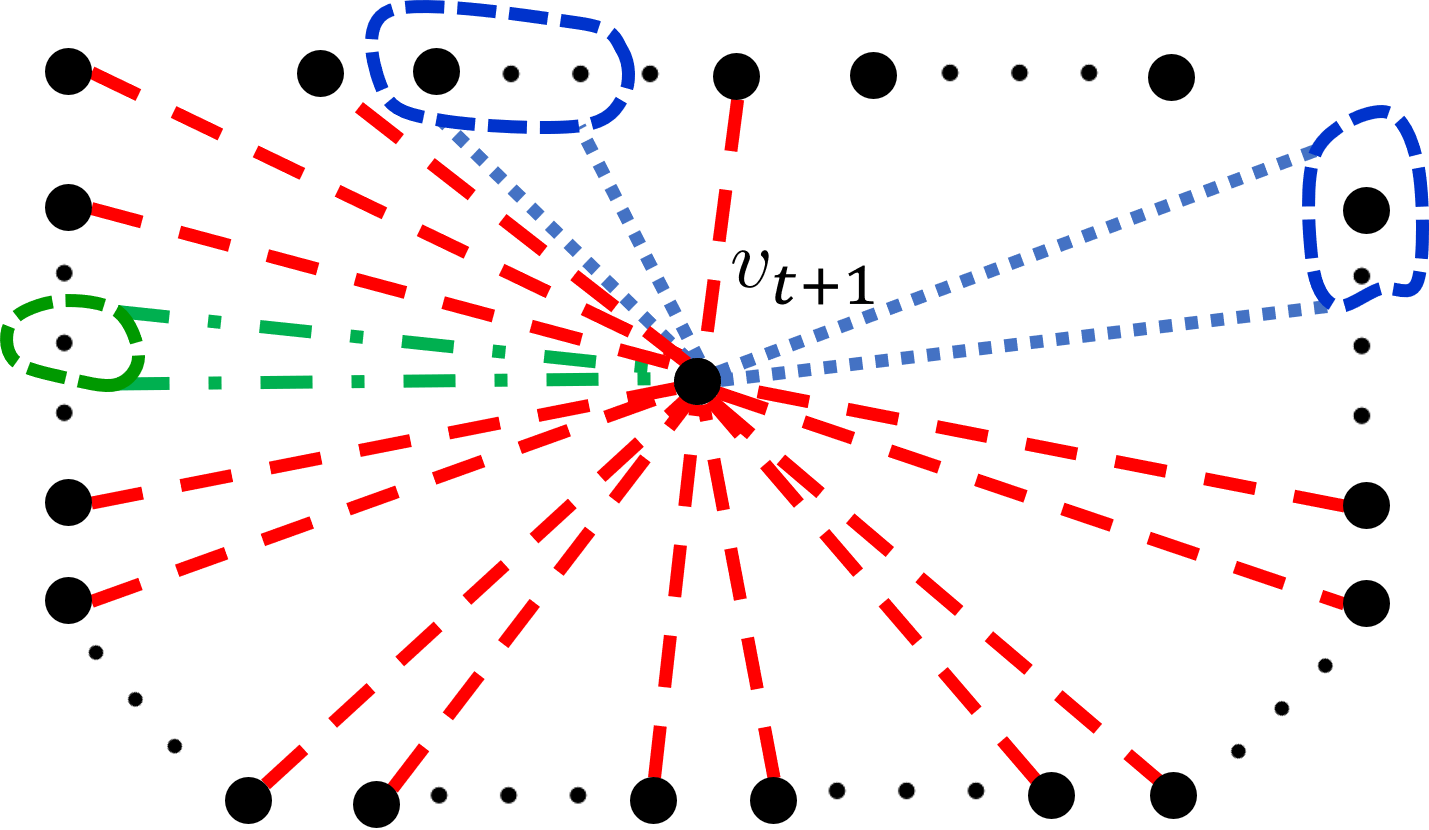}
			\caption*{(d1) Update the sets for time step $t+1$. Dashed red edges belong to $\Dt_{t+1}$, dashed dotted green edges belong to $\At_{t+1}$, and dotted blue edges belong to $\Rt_{t+1}$.} \label{subfig:case1step2}
		\end{subfigure}\hfill
		\medskip
		\begin{subfigure}[t]{.48\linewidth}
			\centering
			\includegraphics*[width = 0.8\textwidth]{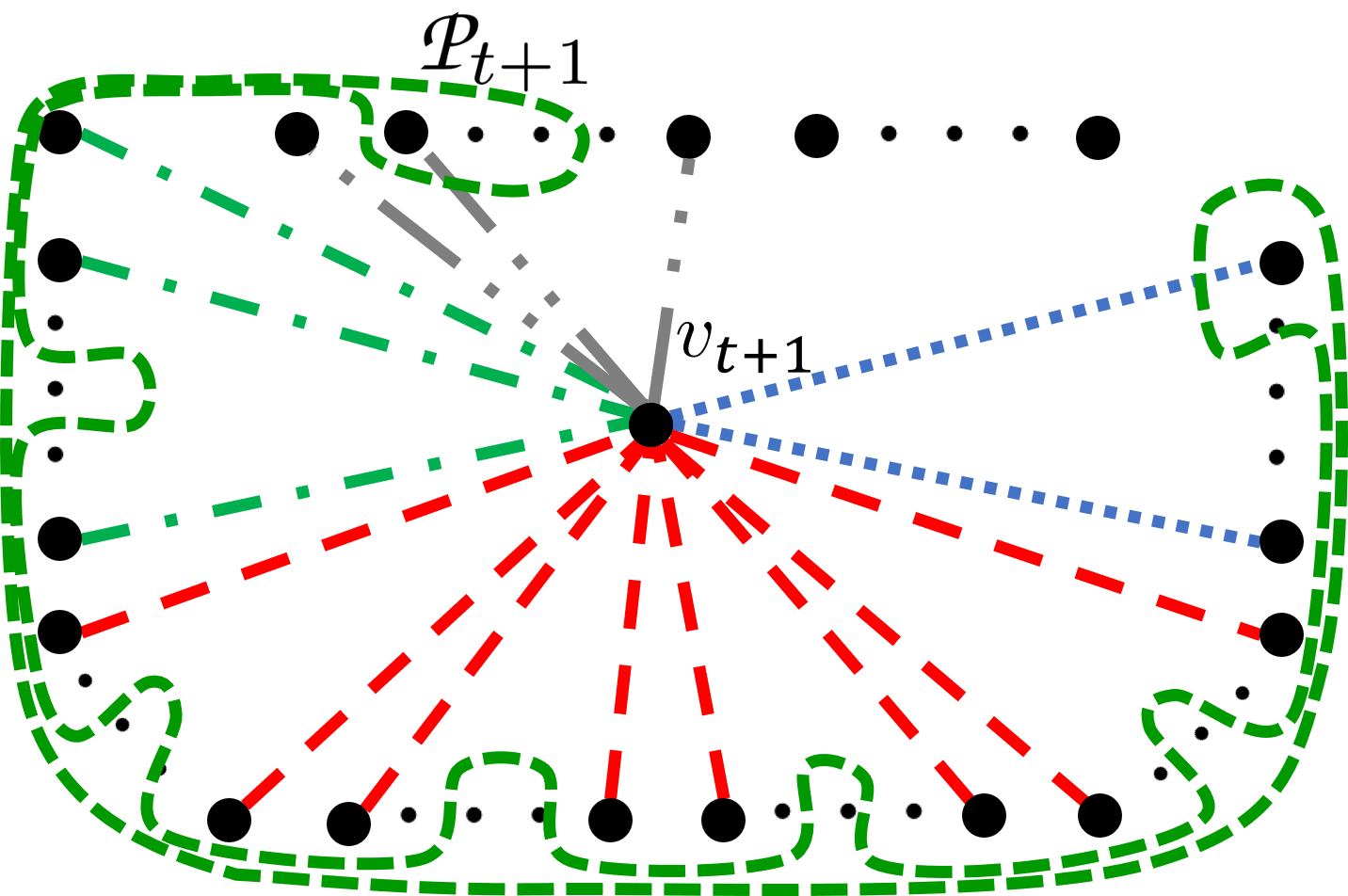}
			\caption*{(c2) Realize $\thresh_{t+1}$ and $\potdeg_{t+1}$. Consider the case where $e_{t+1}$ survives, i.e., $c_n(e_{t+1}) < \thresh_{t+1}$. Define $\It_{t+1} = \{p_1,p_2,\dots,p_{|\It_{t+1}|}\}$ such that $\phi(p)$ is not defined for all $p\in \It_{t+1}$. Set $\phi(p_l) = (\boldsymbol{j},\permt_{t+1}(l))$ for all $l\in\left[|\It_{t+1}|\right]$.}
		\end{subfigure}\hfill
		\begin{subfigure}[t]{.48\linewidth}
			\centering
			\includegraphics*[width = 0.8\textwidth]{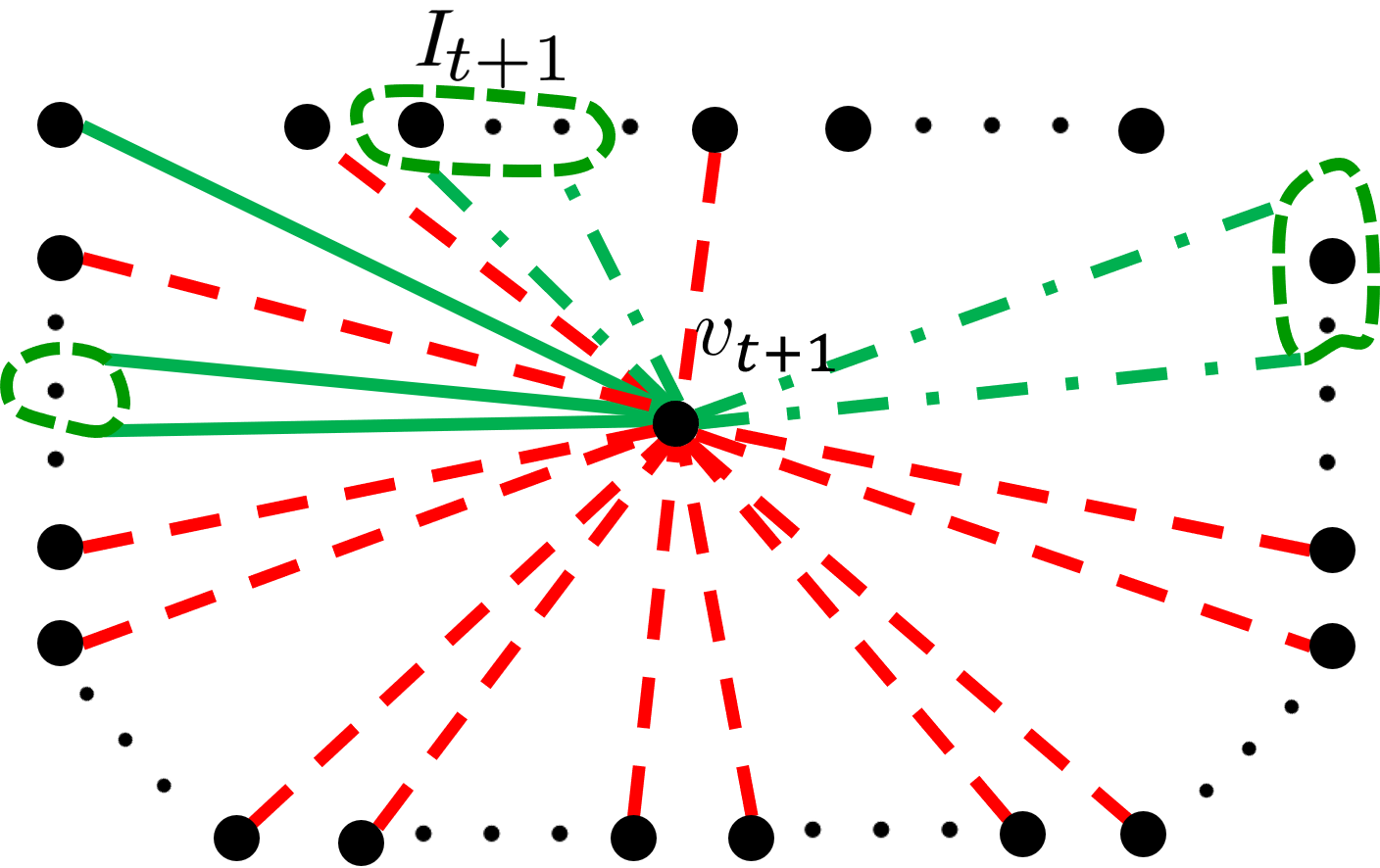}
			\caption*{(d2) Update the sets for time step $t+1$. Solid green edges belong to $\Ct_{t+1}$, dashed red edges belong to $\Dt_{t+1}$, and dashed dotted green edges belong to $\At_{t+1}$.}
		\end{subfigure}
		\caption{The exploration process at time step $t$, when the vertex $\phi^{-1}(\boldsymbol{j})$ has not been explored. (c1) and (d1) illustrate the case when $c_n(e_{t+1}) \geq \thresh_{t+1}$ while (c2) and (d2) illustrate the case when $c_n(e_{t+1}) < \thresh_{t+1}$.}
		\label{fig:case1}
	\end{figure}
	\item \textbf{\textit{Subcase II}}, $\phi^{-1}(\boldsymbol{j})$ has been explored: Let $v_{t+1}$ denote the one, amongst $\phi^{-1}(\boldsymbol{j})$ and $\phi^{-1}(\boldsymbol{i})$, which is not connected to the connected component. Since $v_{t+1}$ has already been explored, all the potential neighbors of the vertex $v_{t+1}$ has been realized.
	\begin{remark} Since the vertex $v_{t+1}$ has been explored and it does not belong to the connected component by time $t$, there is a vertex $v\in[n]$, which belongs to the connected component of $r$ by time $t$ and $v_{t+1}\in \potdeg_v$ and $\{v,v_{t+1}\}\in \Dt_t$. Notice that $v$ may or may not be $par(v_{t+1})$. To clarify the reason, consider the following cases,
		\begin{enumerate}
			\item Consider the case where $\phi^{-1}(\boldsymbol{j})$ belongs to the connected component. As is mentioned in Remark~\ref{rem:pari}, the vertex $par(\phi^{-1}(\boldsymbol{i}))$ has been explored; hence, $\{par(\phi^{-1}(\boldsymbol{i})),\phi^{-1}(\boldsymbol{i})\}\allowbreak\in \Dt_t$. In Figure~\ref{fig:counterexp}, at $t=4$, we have $\boldsymbol{i} = (2)$ and $\phi^{-1}(2) = b$, and $\boldsymbol{j}=(3,1)$ and $\phi^{-1}(\boldsymbol{j}) = d$; however, $d$ belongs to the connected component and $b$ does not and the edge $\{par(b),b\}=\{r,b\}\in \Dt_4$.
			\item Consider the case where $\phi^{-1}(\boldsymbol{i})$ belongs to the connected component. In this case, the edge $\{par(\phi^{-1}(\boldsymbol{j})),\phi^{-1}(\boldsymbol{j})\}$ may belong to $\At_t$. In Figure~\ref{fig:counterexp}, at $t=5$, we have $\boldsymbol{i} = (2)$ and $\phi^{-1}(2) = b$, and $\boldsymbol{j}=(3,1,1)$ and $\phi^{-1}(\boldsymbol{j}) = g$; assuming $b\notin \potdeg_g$ and $d\in \potdeg_g$, the connection $e_6$ does not survive but the vertex $g$ is explored and $\{par(g),g\}=\{d,g\}\in \At_6$.
		\end{enumerate}
	\end{remark}
	Without loss of generality, assume $\phi^{-1}(\boldsymbol{i})$ belongs to the connected component; hence, $v_{t+1} = \phi^{-1}(\boldsymbol{j})$. Define $k\coloneqq d_{v_{t+1}}(n)$ and set the threshold and the set of potential neighbors of $v_{t+1}$,
	\begin{align}
	\thresh_{t+1} &= \text{${k+1}^{\mathrm{th}}$ smallest value in }\left\{c_{n}(\{v_{t+1},j\}) : j\in[n]\text{ and } \{v_{t+1},j\}\in \Rt_t\cup \At_t\cup \Dt_t \right\}\allowdisplaybreaks\\
	\potdeg_{t+1} &= \{j\in [n]:\{v_{t+1},j\}\in \Rt_t \cup \At_t\cup \Dt_t \text{ and } c_{n}(\{v_{t+1},j\}) < \thresh_{t+1} \}
	\end{align}
	\begin{remark}\label{rem:survcase2}
	Given that both $\phi^{-1}(\boldsymbol{i})$ and $\phi^{-1}(\boldsymbol{j})$ have been explored and one of them does not belong to the connected component, the survival of $\{\phi^{-1}(\boldsymbol{i}),\phi^{-1}(\boldsymbol{j})\}$ should have been determined, i.e., it survives. The edge $\{\phi^{-1}(\boldsymbol{i}),\phi^{-1}(\boldsymbol{j})\}$ has been added to the set of active edges to revisit the vertex $v_{t+1}$ and add new potential edges to $\At_t$.
	\end{remark}
	As is mentioned in Remark~\ref{rem:survcase2}, the connection $e_{t+1}$ survives and $v_{t+1}$ belongs to the connected component. Define $\It_{t+1} = \{z\in \potdeg_{t+1}: \phi(z) \text{ is not defined}\}$. Let $\It_{t+1} = \{z_1,z_2,\dots,z_{|\It_{t+1}|}\}$. Pick a permutation $\permt_{t+1}$ over $[|\It_{t+1}|]$ uniformly at random and set $\phi(z_l) = (\boldsymbol{j},\permt_{t+1}(l))$ for all $l\in[|\It_{t+1}|]$, where $\boldsymbol{j}= \phi(v_{t+1})$. Update the sets as follows,
	\begin{subequations}
	\begin{align}
	&
	\begin{aligned}
	\At_{t+1}&= \left( \At_t \cup \left\{\left(\{v_{t+1},j\},c_n\left(\{v_{t+1},j\}\right)\right): \{v_{t+1},j\}\in \Rt_t \right\}\right) \label{eq:2 At}\allowdisplaybreaks\\
			&\qquad\qquad\setminus \left\{\left(\{v_{t+1},j\},c_n\left(\{v_{t+1},j\}\right)\right):\{v_{t+1},j\}\in \At_t \right\}
	\end{aligned}\allowdisplaybreaks\\
	&\Ct_{t+1} = \Ct_t\cup \left\{\left(\{v_{t+1},j\},c_n\left(\{v_{t+1},j\}\right)\right):\{v_{t+1},j\}\in \At_t\right\}\label{eq:2 Ct}\allowdisplaybreaks\\
	&\Dt_{t+1}= \Dt_t \label{eq:2 Dt}\allowdisplaybreaks\\
	&\Rt_{t+1}= \Rt_t \setminus \left\{\left(\{v_{t+1},j\},c_n\left(\{v_{t+1},j\}\right)\right):\{v_{t+1},j\}\in \Rt_t \right\} \label{eq:2 Rt}\allowdisplaybreaks\\
	&\Ut_{t+1}= \Ut_t \label{eq:2 Ut}
	\end{align}
	\end{subequations}
	The description of the above equations is as follows:
	\begin{enumerate}[]
		\item Equation \cref{eq:2 At}: All the edges $\{v_{t+1},j\}\in \Rt_t$ is added to $\At_t$; since, for every $\{v_{t+1},j\}\in \Rt_t$, the vertex $j$ is a potential neighbor of $v_{t+1}$ and if $j$ has been explored, then $v_{t+1}\in \potdeg_j$ as well. In addition, all the edges $\{v_{t+1},j\}\in \At_t$ are removed from $\At_t$; since, $j$ belongs to the connected component at time $t$ (Remark~\ref{rem:propA_t}), the edge $\{v_{t+1},j\}$ survives (Remark~\ref{rem:survcase2}) and we do not need to revisit the vertex $v_{t+1}$ at a later time.
		\item Equation \cref{eq:2 Ct}: All the edges $\{v_{t+1},j\}\in \At_t$ are moved to $\Ct_{t+1}$; since, if $\{v_{t+1},j\}\in \At_t$ then $j\in \potdeg_{t+1}$, $v_{t+1}\in \potdeg_j$ and the vertex $j$ belongs to the connected component (Remark~\ref{rem:propA_t} and Remark~\ref{rem:survcase2}).
		\item Equation \cref{eq:2 Dt}: Notice that both $\phi^{-1}(\boldsymbol{i})$ and $\phi^{-1}(\boldsymbol{j})$ have been explored; hence, the cost of none of the edges in $\Ut_t$ is realized and the set $\Dt_t$ needs no update.
		\item Equation \cref{eq:2 Rt}: All the edges $\{v_{t+1},j\}\in \Rt_t$ are removed from $\Rt_t$, since exactly one end of such an edge belongs to the connected component. All of these edges are moved to $\At_{t+1}$.
		\item Equation \cref{eq:2 Ut}: The cost of none of the edges in $\Ut_t$ is realized; hence, $\Ut_t$ needs no update.
	\end{enumerate}
	\begin{remark}
	Consider an edges $e = \{v_{t+1},j\}$ with realized cost. If $e\in \At_t$, then $j$ belongs to the connected component, $v_{t+1}\in \potdeg_j$ (Remark~\ref{rem:propA_t}) and $j\in \potdeg_{t+1}$ (vertex $v_{t+1}$ has been explored); hence, $e$ is moved to $\Ct_{t+1}$. If the edge $e\in \Dt_t$, then $e$ needs no update. If the edge $e\in \Rt_t$, then $e$ is moved to $\At_{t+1}$ since $v_{t+1}$ belongs to the connected component. Finally, $e$ does not belong to $\Ut_t$ nor $\Ct_t$.
	\end{remark}
	\begin{remark}
		Recall that for any $\{v,z\}\in \Rt_t$, if $v$ has been explored then $z\in \potdeg_v$. Moreover, neither $z$ nor $v$ belongs to the connected component of $r$ by time $t$.
	\end{remark}
	Figure~\ref{fig:case2} illustrates the updating process for the case where both $\phi^{-1}(\boldsymbol{i})$ and $\phi^{-1}(\boldsymbol{j})$ have been explored.
	\begin{figure}[h]
		\centering
		\begin{subfigure}[t]{.48\linewidth}
			\centering
			\includegraphics*[width=0.8\textwidth]{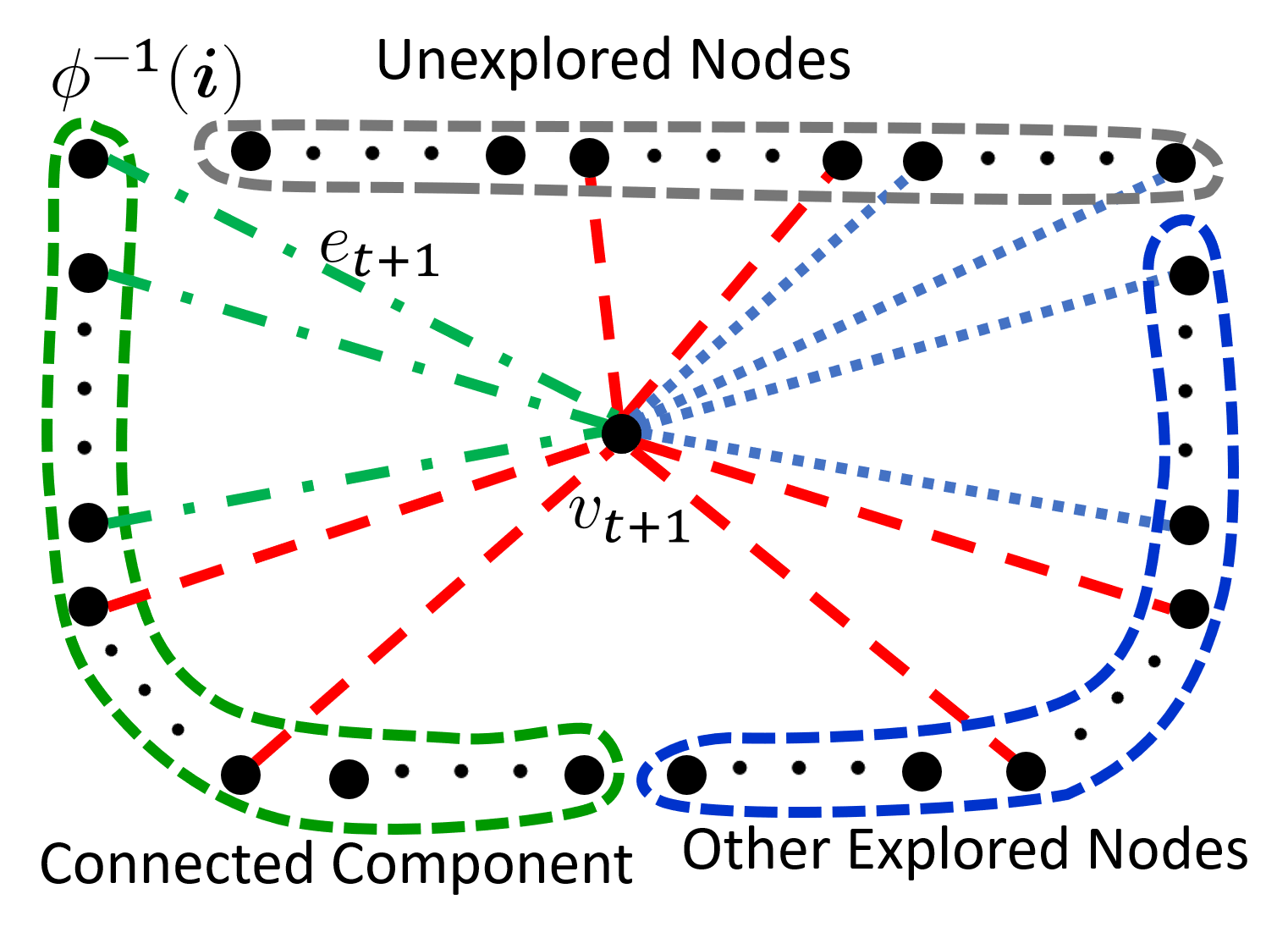}
			\caption{Let $e_{t+1} = \{\phi^{-1}(\boldsymbol{i}),\phi^{-1}(\boldsymbol{j})\}$ and $v_{t+1} = \phi^{-1}(\boldsymbol{j})$. Dashed red edges belong to $\Dt_t$, dashed dotted green edges belong to $\At_t$, and dotted blue edges belong to $\Rt_t$.}
		\end{subfigure}\hfill
		\begin{subfigure}[t]{.48\linewidth}
			\centering
			\includegraphics*[width=0.8\textwidth]{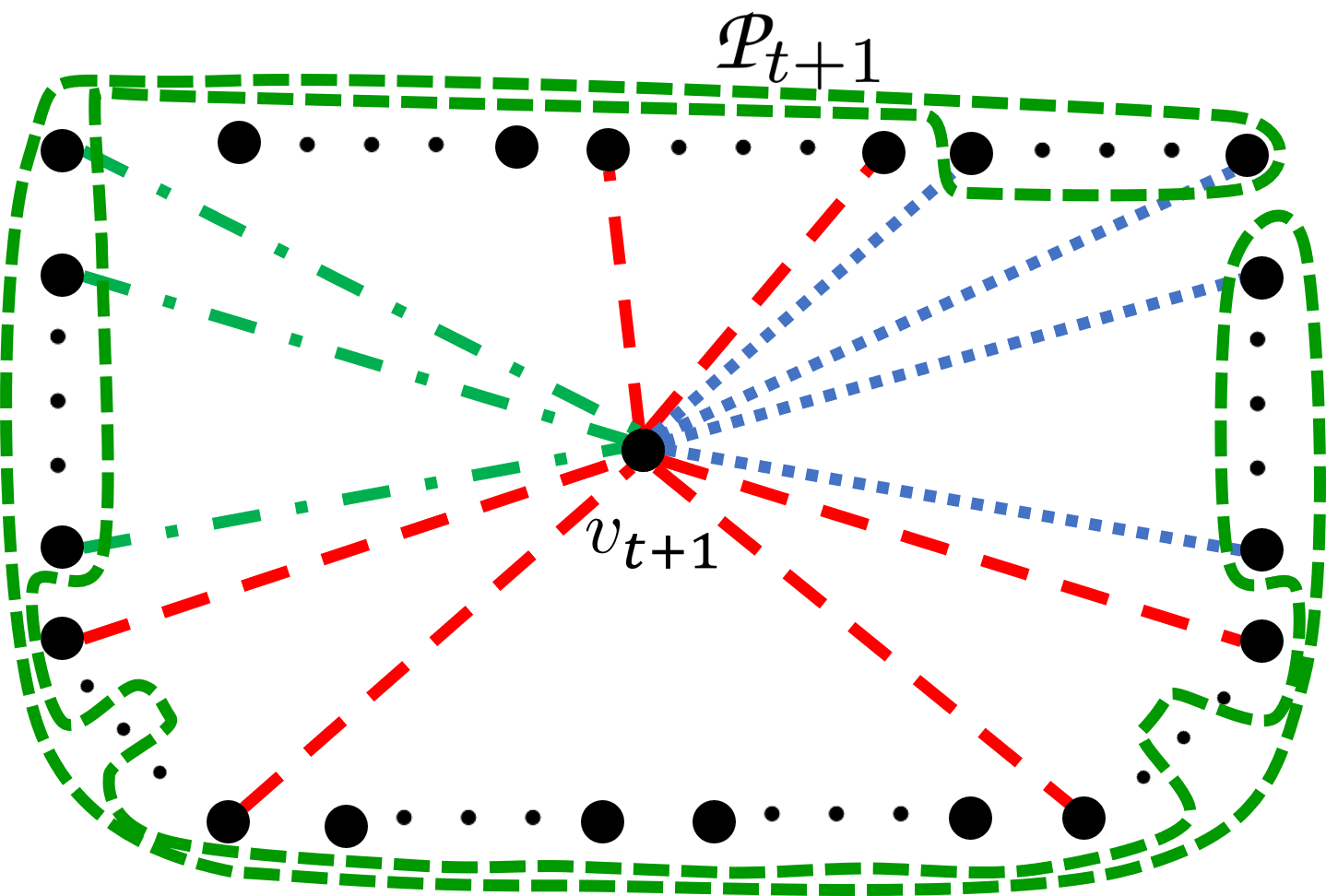}
			\caption{Determine the set of potential neighbors of $v_{t+1}$, i.e., $\potdeg_{t+1}$. Define $\It_{t+1} = \{z_1,z_2,\dots,z_{|\It_{t+1}|}\}$ such that $\phi(z)$ is not defined for all $z\in \It_{t+1}$. Set $\phi(z_l) = (\boldsymbol{j},\permt_{t+1}(l))$ for all $l\in\left[|\It_{t+1}|\right]$.}
		\end{subfigure}
		\medskip
		\begin{subfigure}[t]{.48\linewidth}
			\centering
			\includegraphics*[width=0.8\textwidth]{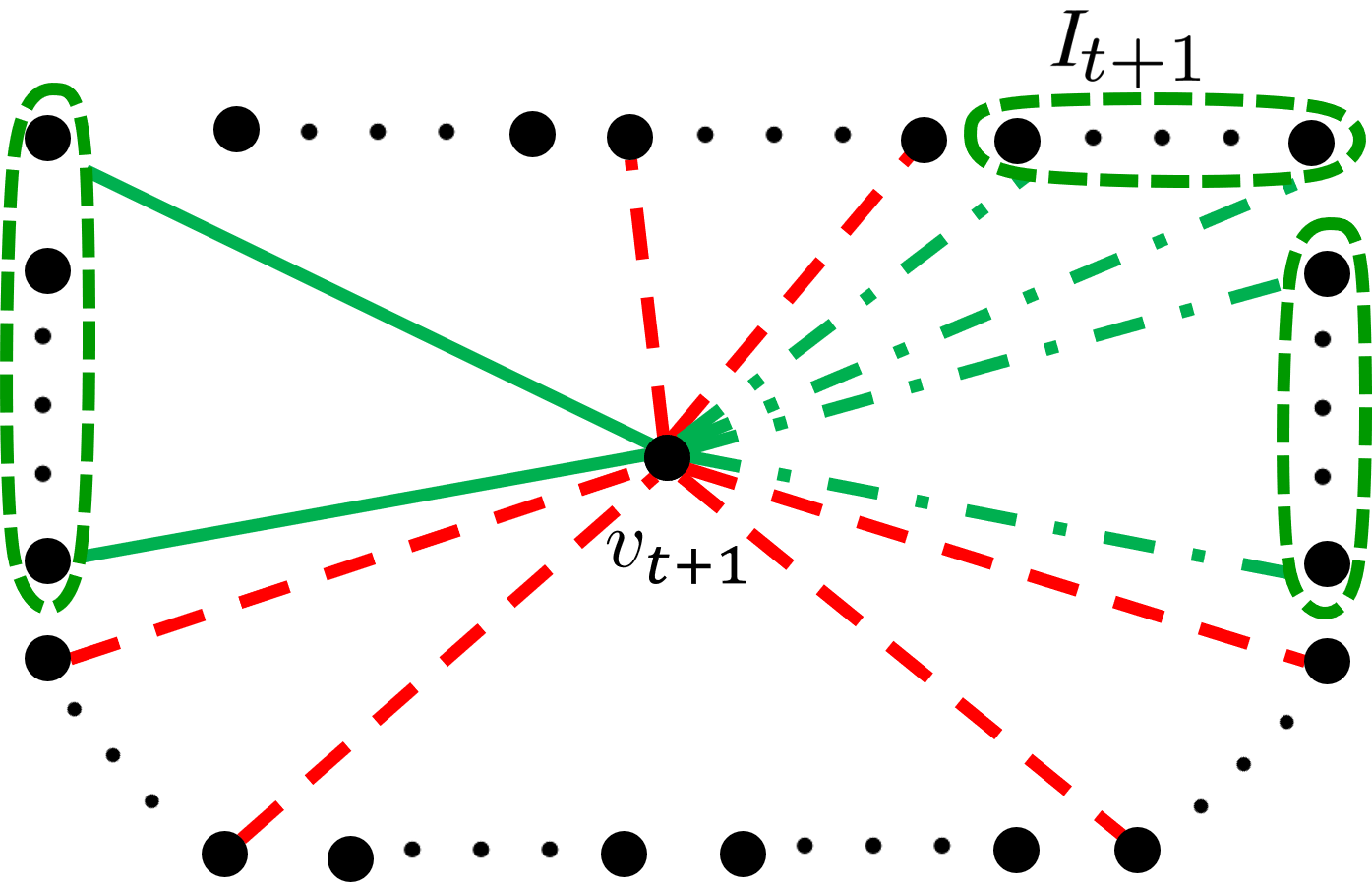}
			\caption{Update the sets for time step $t+1$. Solid green edges belong to $\Ct_{t+1}$, dashed red edges belong to $\Dt_{t+1}$, and dashed dotted green edges belong to $\At_{t+1}$.}
		\end{subfigure}
		\caption{The exploration process at time step $t$, when both the vertices $\phi^{-1}(\boldsymbol{j})$ and $\phi^{-1}(\boldsymbol{i})$ have been explored.}
		\label{fig:case2}
	\end{figure}
\end{itemize}
\textbf{\textit{Exploration phase}}: The exploration terminates when $\At_{t} = \emptyset$. Consider the following filtration,
\begin{align}
\mathcal{F}_t = \sigma((\At_0,\Ct_0,\Dt_0,\Rt_0,\Ut_0),\dots,(\At_t,\Ct_t,\Dt_t,\Rt_t,\Ut_t))
\end{align}
Let $\tau$ denote the time that the algorithm terminates. Indeed, $\tau$ is a stopping time of the filtration where
$\tau = \inf\{t\geq 1: \At_t = \emptyset \}.$\\
\smallskip\\
{\bf Step 2: Locally tree-like property}\\
In the second step, the goal is to show that the rooted graph induced by $\Ct_{t\wedge\tau}$ for any fixed $t$ becomes a tree as the number of vertices, $n$, goes to infinity. This implies that the graph $G_n$, induced by the network $N_n$ after removing the marks, is asymptotically locally tree-like. In fact, a stronger property holds: for every fixed $t>0$, the probability that the vertex $v_l$, for all $l\in[t\wedge \tau]$, has been {\it touched} twice during the exploration process prior to time step $l$ goes to zero as $n\to\infty$. The term ``touching'' is defined as follows,
\begin{definition}
A vertex $v$ is said to be touched at time $t^\prime\leq\tau$ if the cost of $\{v_{t^\prime},v\}$ is realized at time $t^\prime$, $i.e.,$ $\{v_{t^\prime},v\}\in \Ut_{t^\prime-1}\setminus \Ut_{t^\prime}$. The vertex $v_{t^\prime}$ is chosen according to the exploration process. Notice that the vertex $v$ may have or may not have been explored.
\end{definition}
If for every $l\in[t\wedge \tau]$, the vertex $v_l$ has been touched only once before the time step $l$, then $e_{l} = \{par(v_l),v_l\}$; moreover, for every $l^\prime < l$ such that $v_{l^\prime}\neq par(v_l)$, the vertex $v_l$ is not the potential neighbor of the vertex $v_{l^\prime}$. This implies that the rooted graph induced by $\Ct_{t\wedge\tau}$ is a tree. A stronger condition is proved in the following lemma: with high probability, for all $l\in[t\wedge\tau]$ the potential neighbors of the vertex $v_l$ are touched for the first time, except maybe $par(v_l)$.
\begin{lemma}{\bf Locally tree-like property}\label{lem:localtree}
	For $t^\prime>0$, let $J_{t^\prime}$ denote the set of vertices $j$ such that $C_n\left(\{v_{t^\prime},j\}\right) \allowbreak\leq \thresh_{t^\prime}$ and $j$ has been touched at least twice during the exploration process up to time $t^\prime$, once at time step $t^\prime$ and at least once at some time step $\tilde{t}<t^\prime$, i.e.,
	\begin{align}
	&J_{t^\prime} = \{j\in[n]: C_n\left(\{v_{t^\prime},j\}\right)\leq \thresh_{t^\prime}, \{v_{t^\prime},j\}\in \Ut_{t^\prime-1}\setminus \Ut_{t^\prime},\\
	&\myquad[20]\text{ and } \exists \tilde{v}\neq v_{t^\prime} \text{ such that } \{\tilde{v},j\}\notin \Ut_{t^\prime-1}\}\qquad
	\end{align}
	Consider a fixed value $t>0$, then we have,
	\begin{align}\label{eq:localtree}
	\lim_{n\to\infty} \prob\left( \exists l\in[t\wedge\tau] \text{ such that } |J_l|\neq 0 \right) = 0.
	\end{align}
\end{lemma}
\begin{remark}
	Consider the event $\{J_l = \emptyset\}$ for all $l\in[t\wedge\tau]$. This implies that for every vertex $j$ such that $C_n\left(\{v_{l},j\}\right)\leq \thresh_{l}$, either $j$ is touched for the first time at time step $l$ or the value of $C_n\left(\{v_{l},j\}\right)$ has been realized by time step $l-1$. However, if $j \neq par(v_{l})$, then the later case is impossible; otherwise, the vertex $v_l$ should have been touched at least twice during the exploration process up to time $l-1$: once when we realized $C_n\left(\{v_{l},j\}\right)$ and once when we realized $C_n\left(\{par(v_{l}),v_{l}\}\right)$.
\end{remark}
\begin{remark}\label{rem:localtree}
	Even if the rooted graph induced by $\Ct_{t\wedge\tau}$ is a tree, it does not mean that the exploration process satisfies the property which is mentioned in Lemma~\ref{lem:localtree}. In Figure~\ref{fig:counterexp}, vertex $b$ has been touched twice during the exploration process up to time step $t=1$: at time steps $t=0$ (by the vertex $a$) and $t=1$; however, $\Ct_1$ is a tree.
\end{remark}
\begin{proof}
	Observe that $J_0 = \emptyset$. Fix $t>0$. An obvious upper-bound for the left-hand side of \cref{eq:localtree} is given by applying the union bound:
	\begin{align}
		\prob\left( \exists l\in[t\wedge \tau] \text{ such that } |J_l|\neq 0 \right) \nonumber
		&= \prob\left(\bigcup_{l=1}^{t\wedge \tau} \{|J_l|\neq 0\}\right)\allowdisplaybreaks\nonumber\\
		&= \prob\left(\bigcup_{l=1}^{t} \left(\{l\leq\tau\}\cap \{|J_l|\neq 0\}\right)\right)\allowdisplaybreaks\nonumber\\
		&\leq \sum_{l=1}^{t} \prob\left(\{l\leq\tau\}\cap \{|J_l|\neq 0\}\right)\allowdisplaybreaks\nonumber\\
		&= \sum_{l=1}^{t} \expect\left[{\onefunc}\{l\leq\tau\}\prob\left( \{|J_l|\neq 0\}\left|\mathcal{F}_{l-1} \right.\right)\right]. \label{eq:slocaltree}
	\end{align}
	We provide an upper-bound for each term on the right-hand side. If the vertex $v_l$ has been explored by time step $l-1$, then we do not need to touch any vertex at time $l$ and $J_l = \emptyset$. In Figure~\ref{fig:counterexp}, the vertex $b$ has already been explored at time step $t=2$ and $J_5 = \emptyset$. Hence, we only need to consider sample paths where $v_l$ has not been explored. Thus,
	\begin{align}
	\prob(|J_l|= 0\,\vert\,\mathcal{F}_{l-1}) &={\onefunc}{\{\text{$v_l$ has been explored}\}} + {\onefunc}{\{\text{$v_l$ has not been explored}\}} \prob(|J_l| = 0\,\vert\,\mathcal{F}_{l-1}).
	\end{align}
	Consider the sets $\delta_l$, $\varepsilon_l$, $\epsilon_l$ $\in\mathcal{F}_{l-1}$ defined as follows,
	\begin{enumerate}[]
		\item $\delta_l$: set of vertices $j\neq v_l$ with at least one incident edge such that the cost of the edge has been realized prior to time $l$. Equivalently, $\delta_l$ is the set of all vertices except $v_l$ that have been touched prior to time $l$,
		\begin{align}
		\delta_l &= \left\{j\in[n]\setminus\{v_l\}: \exists i \neq v_l \text{ such that, }\{i,j\} \notin \Ut_{l-1} \right\}.
		\end{align}
		\item $\varepsilon_l$: set of all vertices except $v_l$ that have been explored prior to time $l$,
		\begin{align}
		\varepsilon_l &= \left\{j\in[n]\setminus\{v_l\}: j \text{ has been explored by }l-1 \right\}.
		\end{align}
		\item $\epsilon_l$: set of vertices $j\neq v_l$ such that the cost of $\{v_l,j\}$ has been realized prior to time $l$,
		\begin{align}
		\epsilon_l &= \left\{j\in[n]\setminus\{v_l\}: \{v_l,j\} \notin \Ut_{l-1} \right\}.
		\end{align}
		\end{enumerate}
	Observe that $|\epsilon_l|\geq 1$ since $v_l\in \At_l$. Moreover, at each step of the time we may explore at most one vertex (there might be cases in which we revisit an explored vertex); hence, $|\varepsilon_l| \leq l$. Furthermore, for all sample paths in $\mathcal{F}_{l-1}$ in which $v_l$ has not been explored, $\epsilon_l \subseteq \varepsilon_l$ since if $\{v_l,j\}$ has been realized and $v_l$ has not been explored, then $j$ has been explored. Finally, at each time step $l^\prime$, we may touch at most $d_{v_{l^\prime}}(n) + 1$ new vertices; hence, $|\delta_l| \leq 1+ \sum_{i=0}^{l-1}\left(d_{v_i(n) }+ 1\right)$

	Let $k\coloneqq d_{v_l}(n)$ denote the potential degree of the vertex $v_l$. Let $\widetilde{k}\coloneqq \min(k,n-|\epsilon_l|-2)$, where $n-|\epsilon_l|-1$ equals to the number of vertices $j$ such that $\{v_l,j\}\in \Ut_{l-1}$. Notice that $n-|\epsilon_l|-1 > 0$ if $v_l$ has not been explored and $n>l$. Define $\widetilde{\thresh}_l$ and $\widetilde{\potdeg}_l$ to be modified versions of $\thresh_l$ and $\potdeg_l$, i.e.,
	\begin{align}
	\widetilde{\thresh}_{l} &= \text{${\widetilde{k}+1}^{\mathrm{th}}$ smallest value in }\left\{C_{n}(\{v_{l},j\}) : j\in[n]\text{ and } \{v_{l},j\}\in \Ut_{l-1} \right\}\allowdisplaybreaks\\
	\widetilde{\potdeg}_{l} &= \{j\in [n]:\{v_{l},j\}\in \Ut_{l-1} \text{ and } C_{n}(\{v_{l},j\}) < \widetilde{\thresh}_l \}
	\end{align}
	Recall that $\thresh_l$ and $\potdeg_l$ are defined as follows:
	\begin{align}
	\thresh_{l} &= \text{${k+1}^{\mathrm{th}}$ smallest value in }\left\{C_{n}(\{v_{l},j\}) : j\in[n] \right\}\allowdisplaybreaks\\
	\potdeg_{l} &= \{j\in [n]:C_{n}(\{v_{l},j\}) < \thresh_l \}.
	\end{align}
	In the definition of $\thresh_l$, all possible vertices are considered; however, the definition of $\widetilde{\thresh}_{l}$ skips all the vertices $j$, such that $\{v_l,j\}$ has been realized prior to time step $l$. Hence, if $\widetilde{k} = k$, then $\thresh_l \leq \widetilde{\thresh}_{l}$. Moreover, for every vertex $j\in \potdeg_l$ such that the cost of $\{v_l,j\}$ is realized at time $l$, i.e., $\{v_l,j\}\in \Ut_{l-1}$, we have $j\in\widetilde{\potdeg}_l$. To see this, consider the two cases: $1)$ If $\widetilde{k} = k$, then $j\in \potdeg_l$ implies $C_n(\{v_l,j\})<\thresh_l\leq\widetilde{\thresh}_l$. $2)$ If $\widetilde{k} < k$, then $\widetilde{\potdeg}_l$ contains all the vertices $j$ such that $\{v_l,j\}\in \Ut_{l-1}$.

	To realize $\widetilde{\thresh}_l$ and $\widetilde{\potdeg}_l$, we need to pick the $\widetilde{k}+1$ closest vertices to $v_l$, based on the cost of the connection. For an unexplored vertex $j$, the cost of $\{v_l,j\}$ is an exponentially distributed random variable with parameter $\frac{1}{n}$. For an explored vertex $j$ such that $\{v_l.j\}\in \Ut_{l-1}$, the cost of $\{v_l,j\}$ conditioned on $\widehat{\thresh}_j$ is a shifted exponentially distributed random variable with parameter $\frac{1}{n}$ (Corollary~\ref{cor:exprrderstat}); that is $C_{n}(\{v_{l},j\}) \equiv \widehat{\thresh}_j + \expdist(\frac{1}{n})$, where $\widehat{\thresh}_j$ is defined before Remark~\ref{rem:T_hat}. Hence, we need to pick the $\widetilde{k}+1$ smallest value in $H_1\cup H_2$ where,
	\begin{align}
		H_1& = \left\{C_{n}(\{v_{l},j\}) : j\in[n], j\text{ has not been explored and }\{v_{l},j\}\in \Ut_{l-1} \right\} \allowdisplaybreaks\\
		& \equiv \left\{Y_1,Y_2,\dots,Y_{n-|\varepsilon_l\cup \epsilon_l|-1}:Y_i \widesim[2.5]{i.i.d.} \expdist(\frac{1}{n}) \right\}
	\end{align}
	and,
	\begin{align}
		H_2& = \left\{C_{n}(\{v_{l},j\}) : j\in[n], j\text{ has been explored and}\{v_{l},j\}\in \Ut_{l-1} \right\} \allowdisplaybreaks\\
		& \equiv \left\{\expdist(\frac{1}{n}) + \widehat{\thresh}_j: j\in[n], j\text{ has been explored and}\{v_{l},j\}\in \Ut_{l-1} \right\}
	\end{align}
	Instead of $H_2$ we consider $\widehat{H}_2$, defined as follows,
	\begin{align}
		\widehat{H}_2& = \left\{C_{n}(\{v_{l},j\})-\widehat{\thresh}_j : j\in[n], j\text{ has been explored and}\{v_{l},j\}\in \Ut_{l-1} \right\} \allowdisplaybreaks\\
		& \equiv \left\{Y_1^{\prime},Y_2^{\prime},\dots,Y_{|\varepsilon_l\setminus\epsilon_l|}^{\prime}: Y_i^{\prime} \widesim[2.5]{i.i.d.} \expdist(\frac{1}{n}) \right\}
	\end{align}
	In fact, $\widehat{H}_2$ is obtained by replacing $C_{n}(\{v_{l},j\})$ with $C_{n}(\{v_{l},j\})-\widehat{\thresh}_j$ for all explored vertices $j$ such that $\{v_l,j\}\in \Ut_{l-1}$. Notice that if $\widetilde{k}+1$ smallest values in $H_1\cup H_2$ correspond to $\{u_0,u_1,\dots,u_{\widetilde{k}}\}$, then the $\widetilde{k}+1$ smallest values in $H_1\cup \widehat{H}_2$ correspond to $\{\widehat{u}_0,\widehat{u}_1,\dots,\widehat{u}_{\widetilde{k}}\}$ where $\widehat{u}_i$ is either $u_i$ or $u_i - \thresh_j$ for some explored vertex $j\in[n]$. Notice that if a member of $H_2$ is amongst $\widetilde{k}+1$ smallest values in $H_1 \cup H_2$, then the corresponding element is also amongst $\widetilde{k}+1$ smallest values in $H_1 \cup \widehat{H}_2$.
	Collecting everything together, we have
	\begin{align}
	\prob(|J_l| = 0 |\mathcal{F}_{l-1}) &= \prob\left(\left.\left\{
	\begin{minipage}{0.5\textwidth}
		$\forall \{v_l,u\}\in \Ut_{l-1}$ such that $C_n(\{v_l,u\})\leq \thresh_l$, the vertex $u$ is touched for the first time at time step $l$
	\end{minipage}
	\right\}
	\right\vert \mathcal{F}_{l-1}\right) \allowdisplaybreaks\\
	& \geq \prob\left(\left.\left\{
	\begin{minipage}{0.5\textwidth}
	$\forall \{v_l,u\}\in \Ut_{l-1}$ such that $C_n(\{v_l,u\})\leq \widetilde{\thresh}_l$, the vertex $u$ is touched for the first time at time step $l$
	\end{minipage}
	\right\}
	\right\vert \mathcal{F}_{l-1}\right) \allowdisplaybreaks\\
	& = \prob\left(\left.\left\{
	\begin{minipage}{0.5\textwidth}
	The $\widetilde{k}+1$ smallest values in $H_1\cup H_2$ correspond to the vertices that are touched for the first time at time step $l$
	\end{minipage}
	\right\}
	\right\vert\mathcal{F}_{l-1}\right) \allowdisplaybreaks\\
	& \geq \prob\left(\left.\left\{
	\begin{minipage}{0.5\textwidth}
	The $\widetilde{k}+1$ smallest values in $H_1\cup \widehat{H}_2$ correspond to the vertices that are touched for the first time at time step $l$
	\end{minipage}
	\right\}
	\right\vert \mathcal{F}_{l-1}\right),
	\end{align}
	where the last inequality follows from the fact that members of $H_2$ correspond to the vertices that have been touched before time step $l$ (notice that some members of $H_1$ may also correspond to the vertices that have already been touched.).
	However, all the values in $H_1\cup\widehat{H}_2$ are independent and exponentially distributed with parameter $\frac{1}{n}$. There are $n-|\epsilon_l|-1$ vertices $j\neq v_l$ such that $\{v_l,j\}\in \Ut_{l-1}$ and the number of the vertices $j$ that has not been touched prior to time step $l$ is $n - |\delta_l| - 1$; hence,
	\begin{align}
	&\prob\left(\left.\left\{
	\begin{minipage}{0.5\textwidth}
	The $\widetilde{k}+1$ smallest values in $H_1\cup \widehat{H}_2$ correspond to the vertices that are touched for the first time at time step $l$
	\end{minipage}
	\right\}
	\right\vert \mathcal{F}_{l-1}\right) \allowdisplaybreaks\\
	&\qquad = \frac{\binom{n - |\delta_l| - 1}{\widetilde{k}+1}}{\binom{n - |\epsilon_l| - 1}{\widetilde{k}+1}}\geq \frac{\binom{\max\left(0,n - \left(\sum_{i=0}^{l-1}\left(d_{v_i }(n)+ 1\right)\right) - 2\right)}{\widetilde{k}+1}}{\binom{n - 1}{\widetilde{k}+1}} \allowdisplaybreaks\\
	&\qquad\geq \left(\max\left(0,\frac{n - l - \sum_{i=0}^{l}d_{v_i }(n)- 2}{n}\right) \right)^{d_{v_l}(n)+1}.
	\end{align}
	Recall that $\widetilde{k} \leq k=d_{v_l}(n)$. Finally, $\sum_{i=0}^{l}d_{v_i }(n)< M$ with arbitrary high probability for a large enough constant $M$ since the unique elements of the sequence $(d_{v_i }(n))_{i=0}^l$ are chosen uniformly at random (without replacement) from $\bs{d_n}$ and empirical distribution of $\bs{d_n}$ converges to $P$; hence,
	\begin{align}
		&\expect\left[{\onefunc}\{l\leq\tau\}\prob\left( \{|J_l|\neq 0\}\left|\mathcal{F}_{l-1} \right.\right)\right] \allowdisplaybreaks\\
		&\myquad[5]= \expect\left[{\onefunc}\{l\leq\tau\}\left(1 - \prob\left( \{|J_l|= 0\}\left|\mathcal{F}_{l-1} \right.\right)\right)\right] \allowdisplaybreaks\\
		&\myquad[5]\leq 1 - \expect\left[\left(\max\left(0,1 - \frac{ l+\sum_{i=0}^{l}d_{v_i }(n)+2}{n}\right) \right)^{d_{v_l}(n)+1}\right] \xrightarrow{n\to\infty} 0,
	\end{align}
	using the law of total probability. Now, the result follows from the fact that the summation in \cref{eq:slocaltree} has only $t$ summands, each of which converges to zero as $n$ goes to $\infty$.
\end{proof}
$~$\smallskip\\
{\bf Step 3: Convergence of the Exploration}\\
In the third step, we study the local structure of the rooted graph induced by $\Ct_{t\wedge\tau}$ for any fixed $t$. The goal is to analyze the joint distribution of the sequence $(X_0^{(n)},X_1^{(n)},X_2^{(n)},\dots,X_{t\wedge\tau}^{(n)})$ as $n$ goes to infinity, where
\begin{align}
X_{0}^{(n)}\coloneqq\big(d_{v_0}(n),\thresh_0,C_n(\{v_0,j_1\}),C_n(\{v_0,j_2\}),C_n(\{v_0,j_3\}),\dots,C_n(\{v_0,j_{d_{v_0}(n)}\})\big)
\end{align}
such that $C_n(\{v_0,j_s\}) < \thresh_0$ for all $s\in[d_{v_0}(n)]$ and $\phi(j_1)\prec \phi(j_2) \prec \cdots \phi(j_{d_{v_0}(n)})$, and for all $l\in[t\wedge\tau]$ the random vector $X_l^{(n)}$ is given by
\begin{align}
X_{l}^{(n)}\coloneqq\big(d_{v_l}(n),\overline{\thresh}_l,C_n(\{v_l,j_1\}),C_n(\{v_l,j_2\}),C_n(\{v_l,j_3\}),\dots,C_n(\{v_l,j_{d_{v_l}(n)-1}\})\big)
\end{align}
such that $\phi(j_1)\prec \phi(j_2) \prec \cdots \phi(j_{d_{v_l}(n)-1})$ and for all $s\in[d_{v_l}(n)-1]$ we have $C_n(\{v_l,j_s\}) < \overline{\thresh}_l$;
the term $\overline{\thresh}_l$ is the ${d_{v_l}(n)}^{\mathrm{th}}$ smallest value in the set $\left\{C_{n}(\{v_{l},j\}) : j\in[n]\text{ and } \{v_{l},j\}\neq e_l \right\}$. Notice that the second component of $X_0^{(n)}$ equals the threshold of the vertex $v_0$ and the remaining components correspond to the cost of connections between $v_0$ and its potential neighbors. Recall that $d_{v_l}(n)$ is the potential degree of vertex $v_l$, and that the edge $e_l$ is picked according to the exploration process.

An important observation is that for each $l\in[t\wedge\tau]$, $\thresh_l = \overline{\thresh}_l$ if $C_n(e_l) < \overline{\thresh}_l$; moreover, if $C_n(e_l) > \overline{\thresh}_l$, then the edge $e_l$ does not survive (notice that by Remark~\ref{rem:naivevalues}, we have $C_n(e_l) \neq \overline{\thresh}_l$.). Hence, the first two components of $X_{l}^{(n)}$ together are the type of the vertex $v_l$ if and only if the edge $e_l$ survives. Notice that the value of $X_l^{(n)}$ depends on the number of vertices.

Let us extend the sequence to $(X_0^{(n)},X_1^{(n)},X_2^{(n)},\dots,X_t^{(n)})$: for each $l > t\wedge\tau$, the first component of $X_l^{(n)}$ is defined to be $d_v(n)$ where the vertex $v$ is chosen uniformly at random such that $v\notin \{v_0,v_1,\dots,v_{l-1}\}$, the second component is set to be ${d_v(n)}^{\mathrm{th}}$ smallest value in $\St_l^{(n)} = \{s_1,s_2,\dots,s_n:s_i \widesim[2.5]{i.i.d.} \expdist(\frac{1}{n})\}$, and, the remaining components are defined to be $(s_{l_1},s_{l_2},\dots,s_{l_{d_v(n)-1}})$ such that $l_1<l_2<\dots<l_{d_v(n)-1}$ and $s_{l_i}$ is among the $d_v(n)-1$ smallest values in $\St_l^{(n)}$.

The following Lemma states that the sequence $(X_0^{(n)},X_1^{(n)},X_2^{(n)},\dots,X_t^{(n)})$ has the same distribution as the corresponding sequence $(X_0,X_1,X_2,\dots,X_t)$ generated by $\ertreedist(P)$ and extended up to time $t$. The proof is given by using a coupling argument.
\begin{lemma}{\bf Convergence of the Exploration Process} \label{lem:convergexplor}
	The sequence $(X_0^{(n)},X_1^{(n)},\dots,\allowbreak X_t^{(n)})$ converges to the sequence $(X_0,X_1,\dots,X_t)$,
	\begin{align}
		X_0 &\coloneqq (\mathrm{D}_0,\mathrm{T}_0,\mathrm{C}^{(0)}_1,\mathrm{C}^{(0)}_2,\dots,\mathrm{C}^{(0)}_{\mathrm{D}_l})\allowdisplaybreaks\\
		X_l &\coloneqq (\mathrm{D}_l,\mathrm{T}_l,\mathrm{C}^{(l)}_1,\mathrm{C}^{(l)}_2,\dots,\mathrm{C}^{(l)}_{\mathrm{D}_l-1})\qquad\forall l>0,
	\end{align}
	in distribution where $\mathrm{D}_l$ is distributed as $P(\cdot)$ for all $l\in\mathbb{Z}_+$, $\mathrm{T}_l$ is distributed as $\erlangdist(\mathrm{D}_l)$ for all $l\in[t]$ and $\mathrm{T}_0$ is distributed as $\erlangdist(\mathrm{D}_0+1)$, $\{\mathrm{C}^{(l)}_i\}_{i}$ are {\em i.i.d.} random variables uniformly distributed on $[0,\mathrm{T}_l]$ for all $l\in\mathbb{Z}_+$, and $X_i$s are independent.
\end{lemma}
\begin{proof}
Fix the value of $n$. Let $l>0$ and consider the random vector
\begin{align}
\widetilde{X}_l^{(n)} = (\widetilde{\mathrm{d}}_{(l)}(n),\widetilde{\mathrm{T}}_l,\widetilde{\mathrm{C}}_1,\widetilde{\mathrm{C}}_2,\dots,\widetilde{\mathrm{C}}_{\widetilde{\mathrm{d}}_{(l)}(n)-1})
\end{align}
where $\big(\widetilde{\mathrm{d}}_{(i)}(n)\big)_{i=0}^{n-1}$ is a random reordering of $\big(d_i(n)\big)_{i=1}^n$, $\widetilde{\mathrm{T}}_l$ is the ${\widetilde{\mathrm{d}}_{(l)}(n)}^{\mathrm{th}}$ smallest value in $\St_l^{(n)} = \{s_1,s_2,\dots,s_{n-2}: s_i\widesim[2.25]{i.i.d.} \expdist(\frac{1}{n}) \}$ and $\widetilde{\mathrm{C}}_i$ equals to $s_{l_i}$ where $l_1<l_2<\dots<l_{\widetilde{\mathrm{d}}_{(l)}(n)-1}$ and $s_{l_i}<\widetilde{\mathrm{T}}_l$. Using Corollary~\ref{cor:exprrderstat}, it is easy to see that for any fixed $l>0$, $\widetilde{X}_l^{(n)}$ converges in distribution to $(\mathrm{D}_l,\mathrm{T}_l,\mathrm{C}^{(l)}_1,\mathrm{C}^{(l)}_2,\dots,\mathrm{C}^{(l)}_{\mathrm{D}_l-1})$ as $n$ goes to infinity. Similarly, for a proper definition of $\widetilde{X}_0^{(n)}$, the same property holds. Notice that the distribution of $\widetilde{X}_l^{(n)}$ depends on $n$.

The idea of the proof is to first construct a coupling between $(X_l^{(n)})_{l=0}^{t}$ and $(Y_l^{(n)})_{l=0}^{t}$ where conditioned on $\bigcap_{l\in[t\wedge\tau]}\{J_l = \emptyset\}$, $(Y_l^{(n)})_{l=0}^{t}$ has the same distribution as $(\widetilde{X}_l^{(n)})_{l=0}^{t}$, and then show that
\begin{align}
\lim_{n\to\infty}\prob\left((X_0^{(n)},X_1^{(n)},X_2^{(n)},\dots,X_{t}^{(n)})\neq (Y_{0}^{(n)},Y_1^{(n)},Y_2^{(n)},\dots,Y_{t}^{(n)})\right)=0.
\end{align}
For all $l>t\wedge \tau$, let $Y_l^{(n)} = X_l^{(n)}$. Moreover, let $Y_{0}^{(n)} = X_0^{(n)}$. For all $l\in[t\wedge\tau]$, let the first component of $Y_l^{(n)}$ to be equal to the first component of $X_l^{(n)}$. Conditioned on $\mathcal{F}_{l-1}$, construct the set $\St_l^{(n)}$ as follows,
\begin{itemize}
	\item For each vertex $j$ such that the vertex $j$ has not been explored and the value of $C_n(\{v_l,j\})$ has not been realized by time step $l-1$, include $C_n(\{v_l,j\})$ in $\St_l^{(n)}$.
	\item For each vertex $j$ such that the vertex $j$ has been explored, but the value of $C_n(\{v_l,j\})$ has not been realized by time step $l-1$, include $C_n(\{v_l,j\}) - \widehat{\thresh}_j$ in $\St_l^{(n)}$, where $\widehat{\thresh}_j$ is defined before Remark~\ref{rem:T_hat}.
	\item For each vertex $j$ such that the value of $C_n(\{v_l,j\})$ has been realized by time step $l-1$ and $\{v_l,j\}\neq e_l$, add an exponentially distributed random variable with mean $n$ to $\St_l^{(n)}$.
\end{itemize}
Now define the second component of $Y_l^{(n)}$ be the ${Y_l^{(n)}(1)}^{\mathrm{th}}$ smallest value in $\St_l^{(n)}$ and let the remaining $Y_l^{(n)}(1)-1$ components of $Y_l^{(n)}$ to be the $Y_l^{(n)}(1)-1$ smallest values in $\St_l^{(n)}$ (randomly ordered). Clearly conditioned on $\bigcap_{l\in[t\wedge\tau]}\{J_l = \emptyset\}$, $(Y_l^{(n)})_{l=0}^{t}$ and $(\widetilde{X}_l^{(n)})_{l=0}^{t}$ are equidistributed.

The event $X_l^{(n)}\neq Y_l^{(n)}$ for some $l\in[t\wedge\tau]$ may happen if $1)$ the vertex $v_l$ has been touched twice during the exploration process up to time step $l-1$ or if $2)$ the value of $C_n(\{v_l,j\})-\widehat{\thresh}_j$ for an explored vertex $j$ is smaller than $\overline{\thresh}_l$. Recall that in the proof of the Lemma~\ref{lem:localtree}, we replaced the set $H_2$ with the set $\widehat{H}_2$ where each value in $\widehat{H}_2$ corresponds to $C_n(\{v_l,j\})-\widehat{\thresh}_j$ for an explored vertex $j$ such that $\{v_l,j\}\in \Ut_{l-1}$. We also proved the following inequality:
\begin{align}
&\prob\left(\left.\left\{
\begin{minipage}{0.5\textwidth}
The $\widetilde{k}+1$ smallest values in $H_1\cup \widehat{H}_2$ corresponds to the vertices that are touched for the first time at time step $l$
\end{minipage}
\right\}
\right\vert \mathcal{F}_{l-1}\right) \geq \\
&\myquad[15]
 \left(\max\left(0,\frac{n - l - \sum_{i=0}^{l}d_{v_i }(n)- 2}{n}\right) \right)^{d_{v_l}(n)+1}
\end{align}
Hence, using the above inequality and the union bound, for all $l\in[t\wedge\tau]$ we have
\begin{align}
	&\prob(X_l^{(n)}\neq Y_l^{(n)}|\mathcal{F}_{l-1}) \allowdisplaybreaks\\
	&\myquad[2]\leq {\onefunc}{\{\text{$v_l$ has been touched at least twice}\}} +\allowdisplaybreaks\\
	&\myquad[10] 1 - \left(\max\left(0,\frac{n - l - \sum_{i=0}^{l}d_{v_i }(n)- 2}{n}\right) \right)^{d_{v_l}(n)+1} \allowdisplaybreaks\\
	& \myquad[2]\leq {\onefunc}\left\{\bigcup_{i=1}^{l-1}\{|J_i|\neq 0 \} \right\} + 1- \left(\max\left(0,\frac{n - l - \sum_{i=0}^{l}d_{v_i }(n)- 2}{n}\right) \right)^{d_{v_l}(n)+1}
\end{align}
Using Lemma $\ref{lem:localtree}$ and the same reasoning as in its proof, we get
\begin{align}
	&\prob((X_0^{(n)},X_1^{(n)},\dots,X_{t}^{(n)})\neq (Y_0^{(n)},Y_1^{(n)},\dots,Y_{t}^{(n)})) \leq \\
	&\myquad[10]\sum_{l=1}^{t} \expect\left[{\onefunc}\{l\leq\tau\}\prob\left( \{X_l^{(n)}\neq Y_l^{(n)}\}\left|\mathcal{F}_{l-1} \right.\right)\right] \xrightarrow{n\to\infty} 0 .
\end{align}
\end{proof}
\noindent{\bf Step 4: Portmanteau Theorem}\\
The final step is to prove the weak convergence of $\expect U(N_n)$ to $\ertreedist(P)$ by using the Portmanteau theorem.
Let $\overline{\rho}_n = \expect U(N_n)$ and $\rho = \ertreedist(P)$. The goal is to prove $\overline{\rho}_n\xrightarrow{w}\rho$. For a finite rooted tree $[\rtreealt]\in G_*$ of depth $R$, define the set $A_{\rtreealt}$ as follows,
\begin{align}
	A_{\rtreealt} = \left\{[\N]\in G_*: d_{G_*}([\N],[\rtreealt])< \frac{1}{1+R} \right\}
\end{align}
Notice that if $[\N]\in A_{\rtreealt}$, then the rooted subgraph $(\G)_R$ obtained by removing the marks as well as all the vertices of depth more than $R$ from $\N$ is homeomorphic to the graph structure of ${\rtreealt}$. Moreover, the first component of the mark of each vertex in $\N$ up to depth $R$ is equal to the one in ${\rtreealt}$. Recall that a non-root vertex $\boldsymbol{i}$ with the mark $(n_{\boldsymbol{i}}+1,v_{\boldsymbol{i}})$ is referred to as a vertex of \textit{type} $(n_{\boldsymbol{i}},v_{\boldsymbol{i}})$ where $n_{\boldsymbol{i}}$ denotes the number of potential descendants of the vertex $\boldsymbol{i}$.

The first step is to prove that the measure assigned to $A_{\rtreealt}$ by $\overline{\rho}_n$ converges to the measure assigned by $\rho$. Let $l<\infty$ denote the sum of the first component of the type of the vertices in ${\rtreealt}$. To see whether the rooted network generated by $\overline{\rho}_n$ is in $A_{\rtreealt}$ or not, we need to look at the first $l$ steps of the exploration process; however, by Lemma~\ref{lem:convergexplor} the sequence corresponds to the first $l$ steps of the exploration process converges to the one generated by $\rho$ in distribution. Therefore we have
\begin{align}
&\left|\overline{\rho}_n (A_{\rtreealt}) - \rho(A_{\rtreealt})\right| = \\
&\qquad\left|\prob\left((X_0^{(n)},X_1^{(n)},X_2^{(n)},\dots,X_l^{(n)})\in\mathcal{K}\right) - \prob\left((X_0,X_1,X_2,\dots,X_l)\in\mathcal{K}\right) \right| \xrightarrow{n\to\infty}0,
\end{align}
where $\mathcal{K}$ is defined such that $(X_0^{(n)},X_1^{(n)},X_2^{(n)},\dots,X_l^{(n)})\in\mathcal{K}$ if and only if the rooted network induced by $\Ct_{l\wedge\tau}$ belongs to the set $A_{\rtreealt}$.

The second step is to prove that for any bounded uniformly continuous function $f$,
\begin{align}
\left|\int f d\overline{\rho}_n - \int f d\rho\right| \xrightarrow{n\to\infty}0.
\end{align}
Fix the value of $\varepsilon > 0$. Since $f$ is continuous, there exists a $\delta>0$ such that for every $\N$ and $\N^\prime$ in $G_*$, $d_{G^*}(\N,\N^\prime) < \delta$ implies $\left|f(\N)-f(\N^\prime)\right| < \varepsilon$. Let $t>0$ to be large enough such that $(t+1)^{-1} < \delta$.

Notice that the space $G_*$ is separable; hence the restriction of $G_*$ to the rooted trees is also separable. Moreover, $\rho$ assigns zero measure to the set of rooted networks in $G_*$ that are not rooted trees. Hence, there exists a finite set $\mathcal{S}$ of rooted trees of depth less than or equal to $t$ in $G_*$ such that,
\begin{align}
\sum_{{\rtreealt}\in\mathcal{S}} \rho(A_{\rtreealt}) > 1-\varepsilon
\end{align}
Moreover, since $\overline{\rho}_n (A_{\rtreealt})$ converges to $\rho(A_{\rtreealt})$, for large enough $n$ we have $\sum_{{\rtreealt}\in\mathcal{S}} \overline{\rho}_n(A_{\rtreealt}) \allowbreak > 1-2\varepsilon$. Using all of these points, we have
\begin{align}
\left|\int f d\overline{\rho}_n - \int f d\rho\right| \leq 3\varepsilon\norm{f}_\infty + \sum_{{\rtreealt}\in\mathcal{S}} f({\rtreealt}) \left|\overline{\rho}_n (A_{\rtreealt}) - \rho(A_{\rtreealt})\right| + 2\varepsilon
\end{align}
Finally, let $n$ go to infinity and then $\varepsilon$ to zero, and the apply the Portmanteau Theorem to complete the proof.

%% file: Sections/Unimod.tex
Using the involution invariance property (Lemma~\ref{lem:involution}), we need to prove for all Borel measurable non-negative functions $f:G_{**}\to\mathbb{R}_+ $,
\begin{align}
\expect\left(\sum_{v\thicksim\root}f(G,\root,v)\right) = \expect\left(\sum_{v\thicksim\root}f(G,v,\root)\right)\label{proof:Unimod},
\end{align}
where the expectation is with respect to ${\ertreedist(P)}$.
Let us expand the left-hand side of \cref{proof:Unimod} by conditioning on the potential degree of the root vertex. By linearity of expectation, we have,
\begin{align}
\expect\left(\sum_{v\thicksim\root}f(G,\root,v)\right) &= \sum_{m=1}^\infty P(m) \expect\left(\sum_{i\thicksim\root} f(G,\root,i)\,\vert\, n_{\root} = m\right) \allowdisplaybreaks\\
&=\sum_{m=1}^\infty P(m) \expect\left(\sum_{i=1}^m f(G,\root,i){\onefunc}_{i\thicksim\root}\,\vert\, n_{\root} = m\right) \allowdisplaybreaks\\
&=\sum_{m=1}^\infty P(m) \sum_{i=1}^m \expect\left( f(G,\root,i){\onefunc}_{i\thicksim\root}\,\vert\, n_{\root} = m\right) \allowdisplaybreaks\\
&=\sum_{m=1}^\infty mP(m) \expect\left( f(G,\root,1){\onefunc}_{1\thicksim\root}\,\vert\, n_{\root} = m\right).
\end{align}
where the last equality is based on the symmetric and conditionally independent structure of $\{\zeta_{j}\}_{j=1}^{n_{\root}}$ and $\{(n_j,v_j)\}_{j=1}^{n_{\root}}$ conditioned on $n_{\root}$.
We now expand the term $\expect( f(G,\root,1){\onefunc}_{1\thicksim\root}\,\vert\, n_{\root} = m)$ by realizing the values of $v_{\root}$, $\zeta_1$, $n_1$, and $v_1$:
\begin{align}
&\expect( f(G,\root,1){\onefunc}_{1\thicksim\root}\,\vert\, n_{\root} = m) \allowdisplaybreaks\\
&\qquad= \sum_{k=1}^\infty \widehat{P}(k-1) \int_{x=0}^{\infty}\frac{ \eexp^{- x}( x)^m}{m!}\int_{y=0}^{x}\frac{1}{x}\int_{z=y}^{\infty}\frac{ \eexp^{- z}( z)^{k-1}}{(k-1)!} \times \allowdisplaybreaks\\
&\qquad\qquad\qquad\qquad\expect( f(G,\root,1)\,\vert\, n_{\root} = m, v_{\root} = x, \zeta_1 = y, n_1 = k-1, v_1 = z )\,dz\,dy\,dx \allowdisplaybreaks\\
&\qquad= \sum_{k=1}^\infty P(k) \int_{x=0}^{\infty}\int_{z=0}^{\infty}\int_{y=0}^{\min(x,z)}\frac{ \eexp^{- (x+z)}( x)^{m-1}( z)^{k-1}}{m!(k-1)!} \times \allowdisplaybreaks\\
&\qquad\qquad\qquad\qquad\expect( f(G,\root,1)\,\vert\, n_{\root} = m, v_{\root} = x, \zeta_1 = y, n_1 = k-1, v_1 = z )\,dy\,dz\,dx,
\end{align}
where the last equality is obtained by changing the order of the integration and replacing $\widehat{P}(k-1)$ by $P(k)$. Putting it all together, we have
\begin{align}
&\expect\left(\sum_{v\thicksim\root}f(G,\root,v)\right)=\\
&\qquad \sum_{m=1}^\infty\sum_{k=1}^\infty \frac{P(m)P(k)}{(m-1)!(k-1)!} \int_{x=0}^{\infty}\int_{z=0}^{\infty}\int_{y=0}^{\min(x,z)} \eexp^{- (x+z)} x^{m-1} z^{k-1} \times \allowdisplaybreaks\\
&\qquad\qquad\expect( f(G,\root,1)\,\vert\, n_{\root} = m, v_{\root} = x, \zeta_1 = y, n_1 = k-1, v_1 = z )\,dy\,dz\,dx.\numberthis\label{eq:unimod_left}
\end{align}
Similarly,
\begin{align}
&\expect\left(\sum_{v\thicksim\root}f(G,v,\root)\right)= \allowdisplaybreaks\\
&\qquad \sum_{m=1}^\infty\sum_{k=1}^\infty \frac{P(m)P(k)}{(m-1)!(k-1)!} \int_{x=0}^{\infty}\int_{z=0}^{\infty}\int_{y=0}^{\min(x,z)} \eexp^{- (x+z)}x^{m-1}z^{k-1} \times \allowdisplaybreaks\\
&\qquad\qquad\expect( f(G,1,\root)\,\vert\, n_{\root} = m, v_{\root} = x, \zeta_1 = y, n_1 = k-1, v_1 = z )\,dy\,dz\,dx.\numberthis \label{eq:unimod_right}
\end{align}
In order to complete the proof, the following observation is crucial. Let $(G,\root)$ be a realization of $\ertreedist(P)$; conditioned on $n_{\root} = m$, $v_{\root} = x$, $\zeta_1 = y$, $n_1 = k-1$ and $v_1 = z$ such that $\min(x,z)>y$, the structure and distribution of the doubly rooted graph $(G,\root,1)$ is the same as the structure and distribution of the doubly rooted graph $(G,1,\root)$ conditioned on $n_{\root} = k$, $v_{\root} = z$, $\zeta_1 = y$, $n_1 = m-1$ and $v_1 = x$. This symmetry property is evident from Figure~\ref{fig:Unimod}.
\begin{figure*}[t!]
	\centering
	\begin{subfigure}{.5\textwidth}
\centering
	\includegraphics[width=1.1\linewidth]{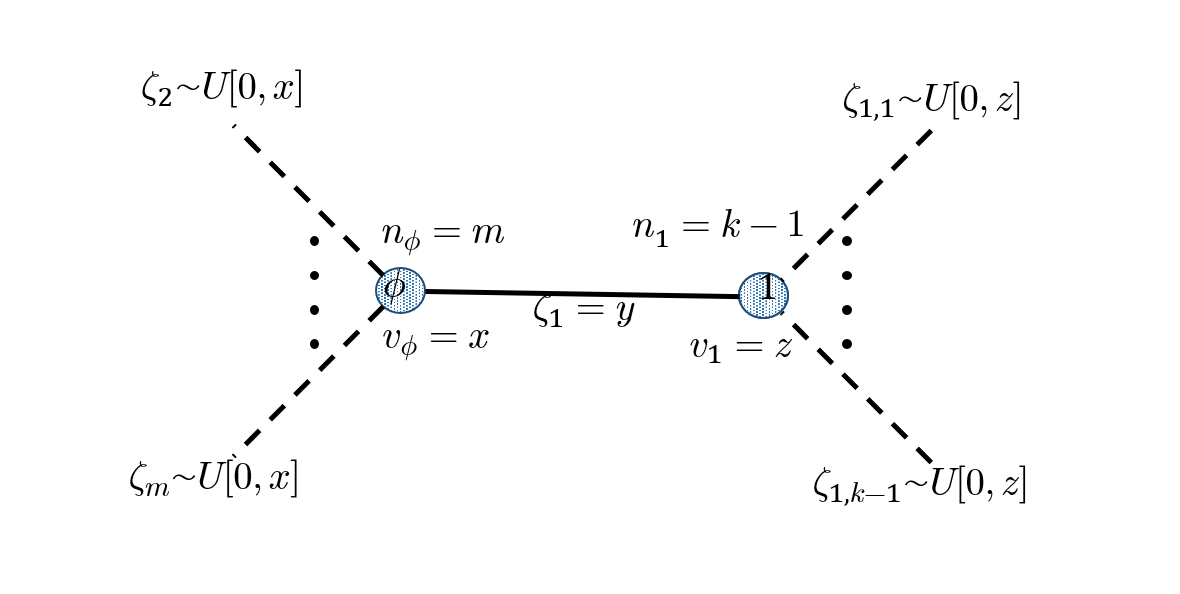}
	\caption{Structure of $N_{\circ\circ}(\root,1)$}
	\end{subfigure}%
	\begin{subfigure}{.5\textwidth}
		\centering
		\includegraphics[width=1.1\linewidth]{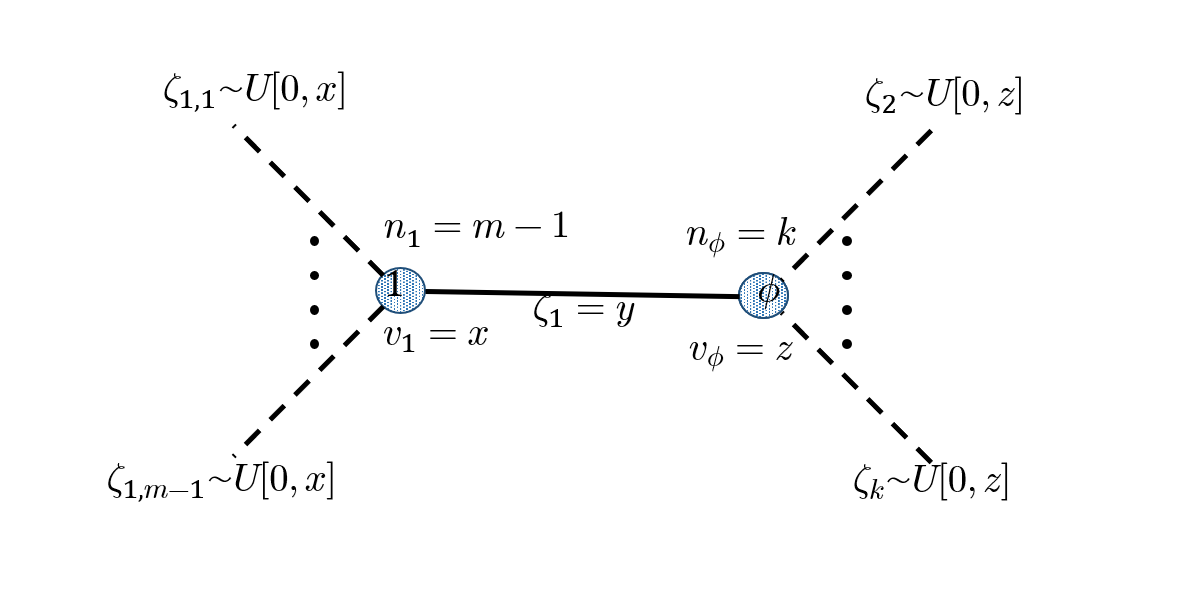}
		\caption{Structure of $N_{\circ\circ}(1,\root)$}
	\end{subfigure}
	\caption{Structure of doubly rooted graphs $(G,1,\root)$ and $(G,\root,1)$ conditioned on a realization of $n_{\root}$, $v_{\root}$, $\zeta_1$, $n_1$ and $v_1$ such that $\zeta_1 < v_1$, where $(G,\root)$ is distributed as $\ertreedist(P)$}
	\label{fig:Unimod}
\end{figure*}
Based on the above discussion, we have
\begin{align}
	\MoveEqLeft \expect( f(G,\root,1)\,\vert\, n_{\root} = m, v_{\root} = x, \zeta_1 = y, n_1 = k-1, v_1 = z ) \allowdisplaybreaks\\
	& = \expect( f(G,1,\root)\,\vert\, n_{\root} = k, v_{\root} = z, \zeta_1 = y, n_1 = m-1, v_1 = x ).
\end{align}
which implies \cref{eq:unimod_left} and \cref{eq:unimod_right} are equal. This completes the proof.

%% file: Sections/ProofLemmaExtinc.tex
\begin{enumerate}[label=(\roman*)]
	\item As the first step, we want to show the range of $T(f)(\cdot)$ is $[0,1]$. The inequality $T(f)(x) \geq 0$ is trivial. For the other side of the inequality, notice that for all $x\in\mathbb{R}_+$, we have $f(x)\leq 1$; hence,
	\begin{align}
	T(f)(x) &\leq \frac{1}{x}\sum_{k=1}^{\infty}P(k)\int_{y=0}^{x} \left(\int_{z=0}^{y}\frac{\eexp^{-z}z^{k-1}}{(k-1)!}\,dz + \int_{z=y}^{\infty}\frac{\eexp^{-z}z^{k-1}}{(k-1)!}\,dz\right)dy \numberthis \label{eq:Tbound}\allowdisplaybreaks\\
	&=\frac{1}{x}\sum_{k=1}^{\infty}P(k)\int_{y=0}^{x} \,dy = 1.
	\end{align}
	The equality holds if and only if $f(x) =1$ for almost every $x\in\mathbb{R}_+$. To see $T(f)(\cdot)$ is non-decreasing, we show that it has a continuous non-negative derivative. Let $x > 0$. We then have
	\begin{align}
	\dv{T(f)(x)}{x} &=	-\frac{1}{x^2}\sum_{k=1}^{\infty}P(k)\int_{y=0}^{x} \left(\int_{z=0}^{y}\frac{\eexp^{-z}z^{k-1}}{(k-1)!}\,dz + \int_{z=y}^{\infty}\frac{\eexp^{-z}z^{k-1}}{(k-1)!}\left(f(z)\right)^{k-1}\,dz\right)\,dy \allowdisplaybreaks\\
	&~~ +\frac{1}{x}\sum_{k=1}^{\infty}P(k) \left(\int_{z=0}^{x}\frac{\eexp^{-z}z^{k-1}}{(k-1)!}\,dz + \int_{z=x}^{\infty}\frac{\eexp^{-z}z^{k-1}}{(k-1)!}\left(f(z)\right)^{k-1}\,dz \right) \allowdisplaybreaks\\
	&=	-\frac{1}{x^2}\sum_{k=1}^{\infty}P(k)\int_{y=0}^{x} \left(\int_{z=0}^{x}\frac{\eexp^{-z}z^{k-1}}{(k-1)!}\,dz + \int_{z=x}^{\infty}\frac{\eexp^{-z}z^{k-1}}{(k-1)!}\left(f(z)\right)^{k-1}\,dz\right)\,dy \allowdisplaybreaks\\
	&~~	-\frac{1}{x^2}\sum_{k=1}^{\infty}P(k)\int_{y=0}^{x} \left(\!-\int_{z=y}^{x}\frac{\eexp^{-z}z^{k-1}}{(k-1)!}\,dz + \int_{z=y}^{x}\frac{\eexp^{-z}z^{k-1}}{(k-1)!}\left(f(z)\right)^{k-1}\,dz\!\right)\!\,dy \allowdisplaybreaks\\
	&~~ +\frac{1}{x}\sum_{k=1}^{\infty}P(k) \left(\int_{z=0}^{x}\frac{\eexp^{-z}z^{k-1}}{(k-1)!}\,dz + \int_{z=x}^{\infty}\frac{\eexp^{-z}z^{k-1}}{(k-1)!}\left(f(z)\right)^{k-1}\,dz \right) \allowdisplaybreaks\\
	&= \frac{1}{x^2}\sum_{k=1}^{\infty}P(k)\int_{y=0}^{x} \int_{z=y}^{x}\frac{\eexp^{-z}z^{k-1}}{(k-1)!}(1 - \left(f(z)\right)^{k-1})\,dz\,dy
	\allowdisplaybreaks\\
	&= \frac{1}{x^2}\sum_{k=1}^{\infty}P(k)\int_{z=0}^{x} \frac{\eexp^{-z}z^{k}}{(k-1)!}(1 - \left(f(z)\right)^{k-1})\,dz
\geq 0 .
	\end{align}
	Observe that the derivative exists and is continuous for all $x > 0$. Taking the limit $x \downarrow 0$, we have
	\begin{align}
	\lim_{x \downarrow 0}\dv{T(f)(x)}{x} = \lim_{x \downarrow 0} \frac{1}{2x}\sum_{k=1}^{\infty}P(k)\frac{\eexp^{-x}x^{k}}{(k-1)!}(1 - f(x)^{k-1}) = 0.
	\end{align}
	Since $f(0)\coloneqq \lim_{x\to 0}f(x)$, the have right-hand derivative of $T(f)(\cdot)$ at $x=0$ is zero.	Hence, $T(f) \in C^1(\mathbb{R}_+;[0,1])$ is non-decreasing which completes the proof of part~\ref{part:propt_i}.
	\item It is easy to see that $\boldsymbol{1}(\cdot)$ is the largest fixed point of $T$. Moreover, for any other fixed point of $T(\cdot)$ say $f(\cdot)\in C^1(\mathbb{R}_+;[0,1])$, from \cref{eq:Tbound} the function $T(f)(\cdot)$ is strictly less than $1$; hence, for all $x\in\mathbb{R}_+$, we have $f(x) < 1$. Using the proof of part~\ref{part:propt_i}, it is easy to see that the derivative of $T(f)$ is strictly positive; hence, the fixed point $f(\cdot)$ is strictly increasing.
	\item The proof is straightforward.
	\item Using part~\ref{part:propt_iii}, since for all $x > 0$, $\boldsymbol{0}(x) < T(\boldsymbol{0})(x) < 1$, we have
	\begin{align}
	\forall x\in\mathbb{R}_+,l\in\mathbb{N},\qquad 0 \leq T^l(\boldsymbol{0})(x) < T^{l+1}(\boldsymbol{0})(x)<1.
	\end{align}
	Let $f_l(x) = T^l(\boldsymbol{0})(x)$. Since, for every fixed value of $x\in\mathbb{R}_+$, the sequence $\{f_l(x)\}_{l=0}^{\infty}$ is strictly increasing, it converges. For all $x\in \mathbb{R}_+$, define $q(x) \coloneqq \lim_{l\to\infty} f_l(x)$. We then have
	\begin{align}
	q(x) &= \lim_{l\to\infty} \frac{1}{x}\sum_{k=1}^{\infty}P(k)\int_{y=0}^{x} \left(\int_{z=0}^{y}\frac{\eexp^{-z}z^{k-1}}{(k-1)!}\,dz + \int_{z=y}^{\infty}\frac{\eexp^{-z}z^{k-1}}{(k-1)!}\left(f_l(z)\right)^{k-1}\,dz\right)\,dy\allowdisplaybreaks\\
	&= \frac{1}{x}\sum_{k=1}^{\infty}P(k)\int_{y=0}^{x} \left(\int_{z=0}^{y}\frac{\eexp^{-z}z^{k-1}}{(k-1)!}\,dz + \int_{z=y}^{\infty}\frac{\eexp^{-z}z^{k-1}}{(k-1)!} \lim_{l\to\infty} \left(f_l(z)\right)^{k-1}\,dz\right)\,dy\allowdisplaybreaks\\
	&=T(q)(x).
	\end{align}
	The second equality follows from monotone convergence theorem, which allows interchanging the order of the summation, the integration, and the limit.

	To show that $q(\cdot)$ is the smallest fixed point of $T$, consider any other fixed pint of $T$, $\tilde{q} = T(\tilde{q})$. Notice that for all $x\in\mathbb{R}_+$, we have $\boldsymbol{0}(x) < \tilde{q}(x)$; hence, for all $x\in\mathbb{R}_+$ and $l\in\mathbb{N}$, we have $\tilde{q}(x) = T(\tilde{q})(x) > f_l(x)$. Passing to the limit as $l\to\infty$, we get $q(x) \leq \tilde{q}(x)$.
\end{enumerate}

%% file: Sections/ProofDegreeDist.tex
Conditioned on $n_{\boldsymbol{i}} = m$ and $v_{\boldsymbol{i}}=x$, the probability of the event $\left\{\zeta_{(\boldsymbol{i},j)} < v_{(\boldsymbol{i},j)}\right\}$ is given as follows,
\begin{align}
\prob\left(\left\{\zeta_{(\boldsymbol{i},j)} < v_{(\boldsymbol{i},j)}\right\}\,\vert\,n_{\boldsymbol{i}} = m, v_{\boldsymbol{i}}=x\right) &=\int_{y=0}^{x} \frac{1}{x}\left(\sum_{k=1}^{\infty}\widehat{P}(k-1)\int_{y}^{\infty} \frac{ \eexp^{- z}z^{k-1}}{(k-1)!}\,dz\right)\,dy\allowdisplaybreaks\\
&=\int_{y=0}^{x} \frac{1}{x}\sum_{k=1}^{\infty}P(k)\bar{F}_{k}(y)\,dy.
\end{align}
The symmetric and conditionally independent structure of the EWT implies that the random variable $D_{\boldsymbol{i}}$ conditioned on $n_{\boldsymbol{i}} = m$ and $v_{\boldsymbol{i}}=x$ has the binomial distribution. Hence,
\begin{align}
\prob\left(D_{\boldsymbol{i}}=d\,\vert\,n_{\boldsymbol{i}} = m, v_{\boldsymbol{i}}=x\right) &= \prob\left(\sum_{j=1}^{n_{\boldsymbol{i}}} {\onefunc}\left\{\zeta_{(\boldsymbol{i},j)} < v_{(\boldsymbol{i},j)}\right\} = d \,\Big\vert\, n_{\boldsymbol{i}} = m, v_{\boldsymbol{i}}=x \right) \allowdisplaybreaks\\
&=Bi\left(d;m,\int_{0}^{x} \frac{1}{x}\sum_{k=1}^{\infty}P(k)\bar{F}_{k}(y)\,dy\right).
\end{align}
The degree distribution of the root follows immediately by integrating/summing over all possible values of $v_{\root}$ and $n_{\root}$. The mean of $D_{\root}$ is obtained as follows:
\begin{align}
\expect[D_{\root}] &= \sum_{d=1}^\infty d \times \prob(D_{\root}=d) \allowdisplaybreaks\\
&= \sum_{m=1}^\infty P(m) \int_{0}^{\infty} \frac{ \eexp^{- x}x^m}{m!} \sum_{d=1}^m d \times Bi\left(d;m,\int_{0}^{x} \frac{1}{x}\sum_{k=1}^{\infty}P(k)\bar{F}_{k}(y)\,dy\right) \,dx \allowdisplaybreaks\\
&= \sum_{m=1}^\infty P(m) \int_{0}^{\infty} \frac{ \eexp^{- x}x^m}{m!} \times m \int_{0}^{x} \frac{1}{x}\sum_{k=1}^{\infty}P(k)\bar{F}_{k}(y)\,dy \,dx \allowdisplaybreaks\\
&= \sum_{m=1}^\infty P(m) \sum_{k=1}^{\infty}P(k) \int_{0}^{\infty} \bar{F}_{k}(y) \int_{y}^{\infty} \frac{ \eexp^{- x}x^{m-1}}{(m-1)!} \,dx \,dy \allowdisplaybreaks\\
&= \sum_{m=1}^\infty\sum_{k=1}^{\infty} P(m)P(k) \int_{0}^{\infty} \bar{F}_{k}(y) \bar{F}_{m}(y) \,dy.
\end{align}
Then, the series expansion follows using
\begin{align}
 \int_{0}^{\infty} \bar{F}_{k}(y) \bar{F}_{m}(y) \,dy = \sum_{n=0}^{k-1} \sum_{l=0}^{m-1} {n+l \choose n} 2^{-n-l-1}.
\end{align}

%% file: Sections/ProofEZ.tex
Notice that the backbone tree is a Galton-Watson tree in which the degree distribution of the root vertex is given by $P(\cdot)$, and the degree distribution of the descendants are given by the shifted distribution $\widehat{P}(\cdot)$. Hence, $\expect[W_l] = \expect[n_{\root}]\times \left(\expect[(n_{\root}-1)]\right)^{l-1} \geq 0$. Recall that the support of $P(\cdot)$ is $\mathbb{N}$, and in particular, $n_{\root} \geq 1$.

For the expected number of vertices at depth $l$, rewrite $Z_l$ as the sum of indicator functions of survival over the potential vertices at depth $l$. A vertex at depth $l$ survives if and only if all the potential edges on its path to the root survive. Writing $\boldsymbol{t}^j = (t_1,t_2,\dots,t_j)$ and $\boldsymbol{t}^0 = \root$ by convention and following the notation introduced in Section~\ref{sec:intro}, we have
\begin{align}
\expect[Z_l] &= \expect\left[ \sum_{\substack{(t_1,t_2,\dots,t_l) \\ \textit{s.t. } t_j \in [n_{\boldsymbol{t}^{j-1}}]}} \onefunc\left(\bigcap_{j=1}^{l}\left\{\zeta_{\boldsymbol{t}^j} < v_{\boldsymbol{t}^j} \right\} \right) \right] \allowdisplaybreaks\\
&= \sum_{m=1}^{\infty} P(m) \times \expect\left[ \sum_{\substack{(t_1,t_2,\dots,t_l) \\ \textit{s.t. } t_j \in [n_{\boldsymbol{t}^{j-1}}]}} \onefunc\left(\bigcap_{j=1}^{l}\left\{\zeta_{\boldsymbol{t}^j} < v_{\boldsymbol{t}^j} \right\} \right) \,\Bigg\vert\, n_{\root} = m \right].
\end{align}
Using the symmetric structure of the EWT, we have
\begin{align}
\expect[Z_l] &= \sum_{m=1}^{\infty} mP(m) \times \expect\left[ 	\sum_{\substack{(t_1=1,t_2,\dots,t_l) \\ \textit{s.t. } t_j \in [n_{\boldsymbol{t}^{j-1}}]}} \onefunc\left(\bigcap_{j=1}^{l}\left\{\zeta_{\boldsymbol{t}^j} < v_{\boldsymbol{t}^j} \right\} \right) \,\Bigg\vert\, n_{\root} = m \right] \allowdisplaybreaks\\
&~~\vdots\allowdisplaybreaks\\
&= \sum_{m=1}^\infty m P(m) \sum_{k_1=2}^{\infty} (k_1-1) P(k_1)\dots \sum_{k_{l-1}=2}^\infty (k_{l-1}-1) P(k_{l-1}) \sum_{k_{l}=1}^\infty P(k_{l}) \times \allowdisplaybreaks\\
&\qquad \expect\left[ \onefunc \left(\bigcap_{j=1}^{l}\left\{\zeta_{\boldsymbol{1}^j} < v_{\boldsymbol{1}^j} \right\} \right) \,\Bigg\vert\, n_{\root} = m, \bigcap_{j=1}^{l} \left\{n_{\boldsymbol{1}^j} = k_j-1 \right\}\right] \allowdisplaybreaks\\
&= \sum_{m=1}^\infty m P(m) \sum_{k_1=2}^{\infty} (k_1-1) P(k_1)\dots \sum_{k_{l-1}=2}^\infty (k_{l-1}-1) P(k_{l-1}) \sum_{k_{l}=1}^\infty P(k_{l}) \times \allowdisplaybreaks\\
&\qquad\int_{x=0}^{\infty}f_{m+1}(x) \int_{y_1 = 0}^{x} \frac{1}{x} \int_{z_1=y_1}^{\infty} f_{k_1}(z_1)\int_{y_2=0}^{z_1} \frac{1}{z_1} \int_{z_2=y_2}^{\infty} f_{k_2}(z_2) \int_{y_3=0}^{z_2} \frac{1}{z_2} \\
&\qquad \dots \int_{y_l=0}^{z_{l-1}} \frac{1}{z_{l-1}}\int_{z_l=y_l}^{\infty} f_{k_l}(z_l) \,dz_l\,dy_l\dots \,dz_1\,dy_1\,dx
\end{align}
where $f_{l}(\cdot)$ is the probability density function of $\erlangdist(l)$ and $\boldsymbol{1}^j\in \mathbb{N}^f$ is a sequence of all ones of length $j$. Using the equality $f_{k}(x)\times (k-1)/x = f_{k-1}(x)$, interchanging order of integration in pairs, e.g., $z_l$ and $y_{l-1}$, and using the complementary cumulative distribution functions to simplify the integrals involving the $z_l$s, we have
\begin{align}
\expect[Z_l] = &\sum_{m=1}^\infty P(m) \sum_{k_1=2}^{\infty} P(k_1)\dots \sum_{k_{l-1}=2}^\infty P(k_{l-1})\sum_{k_l=1}^{\infty}P(k_l)\allowdisplaybreaks\\
&\qquad \int_{y_l=0}^{\infty}\int_{y_{l-1}=0}^{\infty}\dots \int_{y_1=0}^{\infty} \bar{F}_{m}(y_1) \bar{F}_{k_1-1}(\max(y_1,y_2)) \dots \allowdisplaybreaks\\
&\qquad\qquad\qquad\qquad\qquad\qquad\qquad \bar{F}_{k_{l-1}-1}(\max(y_{l-1},y_l)) \bar{F}_{k_l}(y_l) \,dy_1\,dy_2 \dots \,dy_l.
\end{align}

%% file: Sections/ProofGeoDist.tex
\begin{enumerate}[label=(\roman*)]
	\item Let $\{X_i\}_{i=1}^\infty$ denote a set of independent and exponentially distributed random variables with mean $1$. Let $N\sim \geomdist(p)$ be independent of $\{X_i\}_{i=1}^n$. Let $\bar{F}_{k}(\cdot)$ denote the complementary cumulative distribution function of $\erlangdist(k)$. It is easy to see that
	\begin{align}
	\prob\left(\sum_{i=1}^N X_i > y\right) = \sum_{k=1}^{\infty}P(k)\bar{F}_{k}(y),
	\end{align}
	since $\erlangdist(k)$ is the distribution of a sum of $k$ independent exponential variables with mean $1$. On the other hand,
	\begin{align}
	\expect\left[\eexp^{t\sum_{i=1}^N X_i }\right] = \expect\left[ \expect\left[\eexp^{t\sum_{i=1}^N X_i }\,\vert\,N\right]\right] = \expect\left[\left(\frac{1}{1-t}\right)^N\right] = \frac{p}{p-t},
	\end{align}
	which is the moment generating function of an exponentially distributed random variable with rate parameter $p$. Hence,
	\begin{align}
	\sum_{k=1}^{\infty}P(k)\bar{F}_{k}(y) = \prob\left(\sum_{i=1}^N X_i > y\right) = \eexp^{-py}.
	\end{align}
	We treat the case $d\geq 1$ and $d=0$ separately. Assume $d\geq 1$. Using Theorem~\ref{thm:ConDegree}, we have
	\begin{align}
	\prob(D_{\root}=d) &= \sum_{m=1}^\infty P(m) \int_{0}^{\infty} \frac{ \eexp^{- x}x^m}{m!} Bi\left(d;m,\int_{0}^{x} \frac{1}{x}\sum_{k=1}^{\infty}P(k)\bar{F}_{k}(y)\,dy\right) \,dx\allowdisplaybreaks\\
	&= \sum_{m=1}^\infty p(1-p)^{m-1} \int_{0}^{\infty} \frac{ \eexp^{- x}x^m}{m!} Bi\left(d;m,\frac{1 - \eexp^{-px}}{px}\right) \,dx\allowdisplaybreaks\\
	&= \int_{0}^{\infty} \sum_{m=d}^\infty p(1-p)^{m-1} \frac{ \eexp^{- x}}{d!\,(m-d)!} \frac{\left(1 - \eexp^{-px}\right)^d\,\left(px -1 + \eexp^{-px}\right)^{m-d}}{p^m} \,dx\allowdisplaybreaks\\
	&=\int_{0}^{\infty} \frac{\eexp^{- x}}{d!} \left(\frac{1-p}{p}\right)^{d-1} \left(1 - \eexp^{-px}\right)^d \times \allowdisplaybreaks\\
	&\myquad[10]\sum_{m=d}^\infty \left(\frac{1-p}{p}\right)^{m-d} \frac{\left(px -1 + \eexp^{-px}\right)^{m-d}}{(m-d)!} \,dx\allowdisplaybreaks\\
	&= \int_{0}^{\infty} \frac{\eexp^{- px}}{d!} \left(\frac{1-p}{p}\right)^{d-1} \left(1 - \eexp^{-px}\right)^d \exp\left(-\frac{1-p}{p}\left(1 - \eexp^{-px}\right)\right)
	\,dx\allowdisplaybreaks\\
	&= \frac{p}{(1-p)^2} \int_{0}^{1} \frac{ \left(\frac{1-p}{p}\right)^{d+1} z^d \exp\left(-\frac{1-p}{p}z\right)}{d!}
	\,dz\allowdisplaybreaks\\
	&= \frac{p}{(1-p)^2} \left(1 - \sum_{m=0}^d \frac{\left(\frac{1-p}{p}\right)^m \eexp^{-\frac{1-p}{p}}}{m!}\right),
	\end{align}
	where the penultimate equality follows by a change of variable, and the last equality follows by the fact that the integrand is the probability density function of Erlang distribution with parameters $d+1\in\mathbb{N}$ and $\frac{1-p}{p} > 0$. Notice that the third equality does not hold for the case $d=0$.

	Next, consider the case $d=0$. Using Theorem~\ref{thm:ConDegree} and similar to the above, we have
	\begin{align}
	\prob(D_{\root}=0) &= \sum_{m=1}^\infty P(m) \int_{0}^{\infty} \frac{ \eexp^{- x}x^m}{m!} Bi\left(d;m,\int_{0}^{x} \frac{1}{x}\sum_{k=1}^{\infty}P(k)\bar{F}_{k}(y)\,dy\right) \,dx\allowdisplaybreaks\\
	&= \int_{0}^{\infty} \sum_{m=1}^\infty p(1-p)^{m-1} \frac{\eexp^{- x}\,\left(px -1 + \eexp^{-px}\right)^{m}}{p^m\,m!} \,dx\allowdisplaybreaks\\
	&= \frac{p}{1-p} \int_{0}^{\infty}\eexp^{-x} \left(\exp\left(\frac{1-p}{p}\left(px - 1 + \eexp^{-px}\right)\right) - 1\right) \,dx\allowdisplaybreaks\\
	&= \frac{p}{(1-p)^2} \left(1 - \eexp^{-\frac{1-p}{p}}\right) - \frac{p}{1-p}.
	\end{align}
	\item Let us consider the case $x>0$. We have,
	\begin{align}
	T(f)(x) &= \frac{1}{x}\sum_{k=1}^{\infty}P(k)\int_{y=0}^{x} \left(\int_{z=0}^{y}\frac{\eexp^{-z}z^{k-1}}{(k-1)!}\,dz + \int_{z=y}^{\infty}\frac{\eexp^{-z}z^{k-1}}{(k-1)!}\left(f(z)\right)^{k-1}\,dz\right)\,dy\allowdisplaybreaks\\
	&= \frac{p}{x}\int_{y=0}^{x} \Bigg(\int_{z=0}^{y}\eexp^{-z}\sum_{k=1}^{\infty}\frac{(1-p)^{k-1}z^{k-1}}{(k-1)!}\,dz \\ &\myquad[10]+\int_{z=y}^{\infty}\eexp^{-z}\sum_{k=1}^{\infty}\frac{(1-p)^{k-1}z^{k-1}}{(k-1)!}\left(f(z)\right)^{k-1}\,dz\Bigg)\,dy\allowdisplaybreaks\\
	&=\frac{px-1+\eexp^{-px}}{px} + \frac{p}{x}\int_{z=0}^{\infty} \min(x,z) \exp\left(-z\left(1 - (1-p)f(z) \right)\right)\,dz.
	\end{align}
	The derivation for $x=0$ is similar and is omitted.
    \item[(iv)] Using Theorem~\ref{thm:asymdegdist}, and part~\ref{prop:geodist_3} of Proposition~\ref{prop:geodist}, we have
    \begin{align}
        &\lim_{l \to\infty} \prob(D_l = d\,\vert\,Z_l > 0)\allowdisplaybreaks\\
        &\myquad[2]=\frac{1}{\textrm{C}_{\mathrm{asmp}}}\sum_{k=1}^\infty\int_0^\infty P(k)\frac{\eexp^{-z}z^{k-1}}{(k-1)!}\, J_0\left({r_0}\eexp^{-\frac{p}{2}z}\right) Bi\left(d;k-1,\int_{0}^{z} \frac{1}{z}\sum_{k'=1}^{\infty}P(k')\bar{F}_{k'}(y)\,dy\right) \,dz
    \end{align}
    where $\textrm{C}_{\mathrm{asmp}}$ is the normalization factor. Following the same analysis as in the proof of part~\ref{prop:geodist_1}, we have
    \begin{align}
        &\lim_{l \to\infty} \prob(D_l = d\,\vert\,Z_l > 0)\allowdisplaybreaks\\
        &\myquad[4]= \frac{1}{\textrm{C}_{\mathrm{asmp}}}\sum_{k=1}^\infty\int_0^\infty P(k)\frac{\eexp^{-z}z^{k-1}}{(k-1)!}\, J_0\left({r_0}\eexp^{-\frac{p}{2}z}\right) Bi\left(d;k-1,\frac{1 - \eexp^{-pz}}{pz}\right) \,dz\allowdisplaybreaks\\
        &\myquad[4]= \frac{p}{\textrm{C}_{\mathrm{asmp}}} \int_{0}^\infty \frac{\eexp^{-pz}}{d!} \left(\frac{1-p}{p}\right)^{d} (1-\eexp^{-pz})^d \exp\left(-\frac{1-p}{p} (1-\eexp^{-pz})\right)J_0\left({r_0}\eexp^{-\frac{p}{2}z}\right)\,dz\allowdisplaybreaks\\
        &\myquad[4]= \frac{1}{\textrm{C}_{\mathrm{asmp}}} \int_{0}^{1} \frac{ \left(\frac{1-p}{p}\right)^{d} w^d \exp\left(-\frac{1-p}{p}w\right) J_0\left({r_0}\sqrt{1-w}\right))}{d!}\,dw\label{eq:asympdegdist_integralform}
    \end{align}
    Notice that
    \begin{align}
        \textrm{C}_{\mathrm{asmp}} &= \sum_{d=0}^\infty \int_{0}^{1} \frac{ \left(\frac{1-p}{p}\right)^{d} w^d \exp\left(-\frac{1-p}{p}w\right) J_0\left({r_0}\sqrt{1-w}\right))}{d!}\,dw\allowdisplaybreaks\\
        &= \int_{0}^{1} J_0\left({r_0}\sqrt{1-w}\right)\,dw = \frac{2}{r_0} J_1(r_0).
    \end{align}
    Expanding the Bessel function in terms of series in \cref{eq:asympdegdist_integralform}, we have
    \begin{align}
        &\lim_{l \to\infty} \prob(D_l = d\,\vert\,Z_l > 0)\allowdisplaybreaks\\
        &\myquad[4]= \frac{r_0}{2J_1(r_0)}  \int_{0}^{1} \frac{ \left(\frac{1-p}{p}\right)^{d} w^d \exp\left(-\frac{1-p}{p}w\right) }{d!} \times \sum_{i=0}^\infty \left(\frac{r_0^2(1-w)}{4}\right)^i \frac{(-1)^i}{i!i!} \,dw \allowdisplaybreaks\\
        &\myquad[4]= \frac{r_0}{2J_1(r_0)} \sum_{i=0}^\infty \left(\frac{r_0}{2}\right)^{2i} \frac{(-1)^i}{i!i!} \left(\frac{1-p}{p}\right)^{d} \int_{0}^{1} \frac{ w^d (1-w)^i \exp\left(-\frac{1-p}{p}w\right) }{d!}   \,dw \allowdisplaybreaks\\
        &\myquad[4]= \frac{r_0}{2J_1(r_0)}\! \sum_{i=0}^\infty \left(\frac{r_0}{2}\right)^{2i} \frac{(-1)^i}{i!(d+i+1)!} \left(\frac{1-p}{p}\right)^{d}\allowdisplaybreaks\\
        &\myquad[8]\times \left(1 + \sum_{k=1}^\infty \left(\prod_{j=0}^{k-1} \frac{d+j+1}{(d+i+j+2)}\right) \frac{1}{k!}\left(-\frac{1-p}{p}\right)^{k} \right)
    \end{align}
    where the last equality follows by the fact that $\int_{0}^{1} w^d (1-w)^i \exp\left(-\frac{1-p}{p}w\right) dw$ is related to the moment generating function of the beta distribution with parameters $d+1$ and $i+1$. We can further rewrite the right-hand side of the above equation as follows:
    \begin{align}
        &\lim_{l \to\infty} \prob(D_l = d\,\vert\,Z_l > 0)\allowdisplaybreaks\\
        &\myquad[4]= \frac{r_0}{2J_1(r_0)}\! \sum_{i=0}^\infty \left(\frac{r_0}{2}\right)^{2i} \frac{(-1)^i}{i!(d+i+1)!} \left(\frac{1-p}{p}\right)^{d} \sum_{k=0}^\infty \frac{(d+k)!(d+i+1)!}{k!d!(d+i+k+1)!} \times \left(-\frac{1-p}{p}\right)^{k}\allowdisplaybreaks\\
        &\myquad[4]= \frac{1}{J_1(r_0)} \left(\frac{2(1-p)}{r_0p}\right)^{d} \sum_{k=0}^\infty \left(-\frac{2(1-p)}{r_0p}\right)^{k}{ d+k \choose k} J_{d+k+1}(r_0). \label{eq:asympdegdist_seriesform}
    \end{align}
    Notice that we have derived a Neumann series expansion of $\lim_{l \to\infty} \prob(D_l = d\,\vert\,Z_l > 0)$: see \cite[Chapter XVI]{watson1944treatise} for details. The results in \cite[Chapter XVI]{watson1944treatise} also show that the expression in \eqref{eq:asympdegdist_seriesform} is absolutely summable if and only if the series obtained by replacing each $J_{d+k+1}(r_0)$ by its approximation around $0$, namely $\tfrac{1}{(d+k+1)!}\left(\tfrac{r_0}{2}\right)^{d+k+1}$ (which is accurate when $r_0 \ll \sqrt{d+k+2}$), is also absolutely summable. Using this approximation in \eqref{eq:asympdegdist_seriesform}, we have
    \begin{align}
    \tilde{f}(d) & = \frac{2}{r_0\textrm{C}_{\mathrm{asmp}}} \left(\frac{2(1-p)}{r_0p}\right)^{d} \sum_{k=0}^\infty \left(-\frac{2(1-p)}{r_0p}\right)^{k}{ d+k \choose k} \frac{1}{(d+k+1)!}\left(\frac{r_0}{2}\right)^{d+k+1}\allowdisplaybreaks\\
    & = \frac{1}{J_1(r_0)} \left(\frac{2(1-p)}{r_0p}\right)^{d} \frac{1}{d!} \sum_{k=0}^\infty \left(-\frac{2(1-p)}{r_0p}\right)^{k} \frac{1}{k!} \frac{1}{(d+k+1)}\left(\frac{r_0}{2}\right)^{d+k+1}\allowdisplaybreaks\\
    & = \frac{r_0}{2J_1(r_0)} \left(\frac{1-p}{p}\right)^{d} \frac{1}{d!} \sum_{k=0}^\infty \left(-\frac{1-p}{p}\right)^{k} \frac{1}{k!} \frac{1}{(d+k+1)}\allowdisplaybreaks\\
    & = \frac{r_0}{2J_1(r_0)} \frac{p}{(1-p)} \left(1-\exp(-\frac{1-p}{p}) \sum_{k=0}^d \frac{1}{k!} \left(\frac{1-p}{p}\right)^k \right)\allowdisplaybreaks\\
    & = \frac{r_0}{2J_1(r_0)} \frac{p}{(1-p)} F_{d+1}\left(\frac{1-p}{p}\right),
    \end{align}
    which, given the association with the CDF of the Erlang distribution, is finite and decreases to $0$ geometrically fast as $d\rightarrow\infty$. Notice that this is also a good approximation to the asymptotic degree distribution when $d$ is large.

    Finally, using \eqref{eq:asympdegdist_seriesform} the probability generating function of the asymptotic degree distribution is given by
    \begin{align}
    \bar{\mathfrak{F}}(z) & = \frac{1}{J_1(r_0)}\sum_{d=0}^\infty z^d \left(\frac{2(1-p)}{r_0p}\right)^{d} \sum_{k=0}^\infty \left(-\frac{2(1-p)}{r_0p}\right)^{k}{ d+k \choose k} J_{d+k+1}(r_0)\allowdisplaybreaks\\
    & = \frac{1}{J_1(r_0)} \left(\sum_{l=0}^\infty J_{l+1}(r_0) \sum_{d=0}^l { l \choose d} z^d \left(\frac{2(1-p)}{r_0p}\right)^{d} \left(-\frac{2(1-p)}{r_0p}\right)^{l-d}\right)\allowdisplaybreaks\\
    & = \frac{1}{J_1(r_0)}\left(\sum_{l=0}^\infty \left(z  \left(\frac{2(1-p)}{r_0p}\right) -  \left(\frac{2(1-p)}{r_0p}\right)\right)^l J_{l+1}(r_0)\right).
    \end{align}
\end{enumerate}